\documentclass[12pt,reqno]{amsart}
\usepackage[v2,all]{xy}

\usepackage{graphicx} %package to manage images
\graphicspath{ {FCPic/} }
\usepackage{amsfonts}
\usepackage{mathrsfs}
 \usepackage{amssymb}
 \usepackage{pifont}

 \usepackage[in]{fullpage}

 \usepackage{bold}
 \def\unit{\Eins}
 \def\gh{\mathbbnew{\Gamma}}
 \def\DDelta{\mathbbnew \Delta}

 \input{epsf.sty}
 \catcode `\@=11
\def\numberbysection{\@addtoreset{equation}{section}
         \renewcommand{\theequation}{\thesection.\arabic{equation}}}
\numberbysection
\def\subsubsection{\@startsection{subsubsection}{3}%
  \normalparindent{.5\linespacing\@plus.7\linespacing}{-.5em}%
  {\normalfont\bfseries}}

 \newtheorem{thm}{Theorem}[section]
 \newtheorem{theorem}[thm]{Theorem}
 \newtheorem{lem}[thm]{Lemma}
 \newtheorem{lemma}[thm]{Lemma}
 \newtheorem{prop}[thm]{Proposition}
 \newtheorem{proposition}[thm]{Proposition}
 \newtheorem{cor}[thm]{Corollary}
 \newtheorem{corollary}[thm]{Corollary}
 
 \newtheorem{dfprop}[thm]{Definition-Proposition}

 \newtheorem{innercustomthm}{Theorem}
 \newenvironment{customthm}[1]
   {\renewcommand\theinnercustomthm{#1}\innercustomthm}
   {\endinnercustomthm}

 \newtheorem{innercustomlem}{Lemma}
 \newenvironment{customlem}[1]
   {\renewcommand\theinnercustomlem{#1}\innercustomlem}
   {\endinnercustomlem}

 \newtheorem{innercustomcor}{Corollary}
 \newenvironment{customcor}[1]
   {\renewcommand\theinnercustomcor{#1}\innercustomcor}
   {\endinnercustomcor}

 \theoremstyle{definition}

 \newtheorem{df}[thm]{Definition}
 \newtheorem{definition}[thm]{Definition}
 
 \newtheorem{rmk}[thm]{Remark}
 \newtheorem{remark}[thm]{Remark}
 \newtheorem{nota}[thm]{Notation}
 \newtheorem{ex}[thm]{Example}
 \newtheorem{example}[thm]{Example}
 \newtheorem{assump}[thm]{Assumption}

 \newtheorem{innercustomdf}{Definition}
 \newenvironment{customdf}[1]
   {\renewcommand\theinnercustomdf{#1}\innercustomdf}
   {\endinnercustomdf}

 \newcommand{\leftsub}[2]{{\vphantom{#2}}_{#1}{#2}}

 \def\egr{\unit_{\emptyset}}

 \def\eps{\epsilon}
 \def\G{\Gamma}
 
 \def\Vect{\mathcal{V}ect}
 \def\dgVect{dg\Vect}
 \def\dgvect{\dgVect}
 \def\gVect{\Vect^{\Z}}

 \def\Fass{{\FF}_{Assoc}}
 \def\Flie{{\FF}_{Lie}}
 \def\Fprelie{{\FF}_{pre-Lie}}

 \def\Graphs{{\mathcal G}raphs}
 \def\Gr{\Agg}

 \def\CalC{{\mathcal C}}
 \def\CalD{{\mathcal D}}
 \def\CalE{{\mathcal E}}
 \def\base{{B}}
 \def\CalB{{\mathcal B}}
 \def\CO{{\mathcal O}}
 \def\Agg{{\mathcal A}gg}
 \def\Set{{\mathcal S}et}
 \def\Top{{\mathcal T}op}
 \def\SSet{\mathcal {SS}et}

 \def\Crl{{\mathcal C}rl}
 \def\nCrl{\Crl_{\N}}

 \def\Cyclic{{\mathcal C}yc}
 \def\opd{\mathcal{O}pd}
 \def\CCyclic{\mathfrak  {C}}
 
 \def\operads{{\mathfrak O}}
 \def\props{{\mathfrak P}}
 \def\properads{{\mathfrak P}^{ctd}}
 \def\dioperads{{\mathfrak D}}
 \def\modular{{\mathfrak M}}

 \def\F{\mathcal F}
 \def\FF{\mathfrak F}
 
 \def\GG{\mathfrak G}

 \def\FV{\F_{\V}}
 \def\VV{\V_{\V}}
 \def\FFV{\FF_{\V}}
 \def\iV{\imath_{\V}}
 \def\fL{\mathfrak{L}}

 \def\C{\CalC}
 \def\Z{{\mathbb Z}}
 \def\N{{\mathbb N}}
 \def\G{\Gamma}

 \def\del{\partial}
 \def\colim{\mathrm{colim}}

 \def\FinSet{\mathcal{F}in\mathcal{S}et}
 \def\Surj{\mathcal{S}urj}
 \def\FSurj{\mathfrak{Sur}}
 \def\Inj{FI}

 \def\O{{\mathcal O}}
 \def\P{{\mathcal P}}
 \def\Sn{{\mathbb S}_n}
 \def\SS{{\mathbb S}}

 \def\odo{\otimes \cdots \otimes}
 \def\K{\mathfrak K}
 \def\crl{*}

 \def\Gra{\mathfrak{G}ra}
 \def\day{\circledast}

 \def\wt{\mathrm{wt}}
 \def\deg{\mathrm{deg}}

 \def\kill{trun}
 \def\leaf{leaf}

 \newcommand\ccirc[2]{\, \leftsub{#1}{\circ}_{#2}}
 \def\scirct{\ccirc{s}{t}}
 
 \newcommand\mge[2]{\, \leftsub{#1}{\boxminus}_{#2}}

 \newcommand{\ds}{\displaystyle}

 \def\V{\asts}
 \def\asts{{\mathcal V}}
 \def\F{\clusters}
 \def\clusters{{\mathcal F}}
 \def\isoclusters{Iso(\F)}
 \def\Ab{{\mathcal Ab}}
 \def\CalG{{\mathcal G}}
 \def\opcat{{\mathcal O }ps}
 \def\op{\mathcal}
 \def\smodcat{{\mathcal M}ods}
 \def\forget{\mathit {forget}}
 \def\free{\mathit {free}}
 \def\T{{\mathbb T}}
 \def\I{{\mathcal I}}
 \def\D{\mathcal D}

 \def\oper{op}
 \def\opers{\opcat}
 \def\FT{\mathsf{FT}}
 \def\fops{\F\text{-}\opcat}
 \def\foddops{\F^{odd}\text{-}\opcat}
 \def\fopsc{\F\text{-}\opcat_\C}
 \def\vmods{\V\text{-}\smodcat}
 \def\vmodsc{\V\text{-}\smodcat_\C}
 \def\vseq{\V\text{-}\text{Seq}_\C}

 \def\Cobar{\Omega}
 \def\Bar{\mathsf{B}}
 \def\tensor{\otimes}
 \def\fopst{\F\text{-}\opcat_{\Top}}

 \newcommand{\id}{\mathbb{I}}

 \newcommand{\fr}{\mathfrak}

 \newcommand{\cdc}{,\dots,}
 \newcommand{\fdops}{\F\text{-}\opcat_\op{D}}
 \newcommand{\eops}{\op{E}\text{-}\opcat_\op{C}}

 \newcommand{\vemods}{\V_\fr{E}\text{-}\smodcat_\op{C}}
 \newcommand{\vfmods}{\V_\FF\text{-}\smodcat_\op{C}}

 \newcommand{\adj}[4]{#1\negmedspace: #2\rightleftarrows #3:\negmedspace #4}

 \newcommand{\Fpair}{\F_{dec\O}}
 
 \newcommand{\Fepair}{\Fe_{dec\O}}
 \newcommand{\Fepairtwo}{\Fe'_{dec f_{\ast}(\O)}}

 \newcommand{\Fe}{\mathfrak{F}}

 \def\FFdeco{\FF_{dec\O}}
 \def\Fdeco{\F_{dec\O}}
 \def\Vdeco{\V_{dec\O}}
 \def\ideco{\iota_{dec\O}}
 
 \def\Po{\P}
 \def\final{\mathcal T}

\begin{document}

 \title[Feynman Categories]{Feynman Categories}

 \author
 [Ralph M.\ Kaufmann]{Ralph M.\ Kaufmann}
 \email{rkaufman@math.purdue.edu}

 \address{Purdue University Department of Mathematics,
  West Lafayette, IN 47907}

 \author
 [Benjamin C.~Ward]{Benjamin C.~Ward}
 \email{bward@math.su.se}

 \address{Stockholm University, Stockholm, SE 106-91}

 \begin{abstract}

 In this paper we give a new foundational, categorical formulation for operations and relations and objects parameterizing them.  This generalizes and unifies the theory of operads and
 all their cousins including but not limited to PROPs, modular operads, twisted (modular) operads, properads, hyperoperads, their colored versions, as well as algebras over operads and an abundance of other related structures, such as crossed simplicial groups, the augmented simplicial category or FI--modules.

 The usefulness of this approach is that it allows us to handle all the classical as well as more esoteric structures under a common framework and we can treat all the situations simultaneously.
 Many of the known constructions simply become Kan extensions.

 In this common framework, we also
 derive universal operations, such as those underlying Deligne's conjecture, construct Hopf algebras as well as perform resolutions, (co)bar transforms
 and Feynman transforms which are related to master equations. For these applications, we construct the relevant model category structures. This produces many new examples.
 \end{abstract}

 \maketitle

\section*{Introduction}

\subsection{General overview and  background}
We introduce Feynman categories as a universal foundational framework for treating operations and their relations.
It unifies a plethora of theories which have previously  each been treated individually. We give the axiomatics, the essential as well as novel, universal constructions
which provide a new conceptual level to a sweeping set of examples.
In this sense Feynman categories present a new {\it lingua} or {\it characteristica universalis} for  objects governing operations.
We use a categorical formulation in which all constructions become Kan extensions, true to the dictum of MacLane \cite{MacLane}.

The history of the  problem of finding universal languages goes back to Leibniz \cite{leib}  and in modern times it was first considered by Whitehead \cite{Whitehead}. Depending on the situation, the special types of operations that have appeared were encoded by various objects fitting the setup. Classically these came from the PROPs of \cite{MacLanePROP} and their very successful specialization to operads \cite{Mayoperad}. As the subject entered its renaissance in the 90s \cite{renaissance} fueled by the connection to physics, especially (string) field theory and correlation functions, more constructions appeared, such as cyclic operads \cite{GKcyclic}, modular \& twisted modular operads \cite{GKmodular} and even in the 2000s there are additions such as properads \cite{Vallette}, wheeled versions \cite{wheeledprops} and many more, see e.g. \cite{MSS,KWZ}, Table \ref{univtable} and \S2. There are still more examples to which our theory applies, such as FI--modules \cite{Farb}, which appear in representation theory, as well as operations appearing in  open/closed field theories \cite{KP} or those in a future direction, homological mirror symmetry \cite{KOOO}, where the relevant framework has perhaps not even been defined yet axiomatically.

The beauty of Feynman categories is that all these examples ---and many more--- can
be treated on equal footing.  Although differing greatly in their details the above structures do have a common fabric, which we formalize by the statement that these are all functors from Feynman categories.   This enables us to replace ``looks like'', ``feels like'', ``operad--like'' and ``is similar/kind of analogous to'' by
the statement that all examples are examples of functors from Feynman categories. The upshot being that there is no need to perform constructions and prove theorems in every example separately, as they can now be established in general once and for all.

What we mean by operations and relations among them can be illustrated by considering a group (which yields an example of a Feynman category) and its representations. There are the axioms of a group, the abstract group itself, representations in general and representation of a particular group. The operations when defining a group are given by the group multiplication and inverses, the relations are given by associativity and unit relations. Thus the operations and relations deal with groups and representations in general. These specialize to those of a particular group. Universal constructions in this setting are for instance the free group, resolutions, but also induction and restriction of representations.  This situation may be interpreted categorically by viewing a group as a groupoid with one object. Here the morphisms are the elements of the group, with the group operation being composition. The fact that there are inverses makes the category into a groupoid by definition. The representations then are functors from this groupoid.

In many situations however, there are many operations with many relations. One main aim in the application of the theory is to gather information about the object  which admits these operations. A classical paradigmatic topological example is for instance the recognition theorem for loop spaces, which states that a connected space is a loop space if and only if it has all the operations of a topological homotopy associative algebra \cite{Stasheff}.
A more modern example of operations are Gromov--Witten invariants \cite{KoMa,ManinBook} and string topology \cite{CS} which are both designed to study the geometry of spaces via operations and compare spaces/algebraic structures admitting these operations, see also \cite{KLP,woods,CHV, KWZ}.
It then becomes necessary to study the operations/relations themselves as well as the objects and rules governing them.

After giving the definition of a Feynman category, we give an extensive list of examples. This can be read in two ways. For those familiar with particular examples they can find their favorite in the list and see that the theory applies to it. Importantly, for those not familiar with operads and the like, this can serve as a very fast and clear definition. Some, but not all, examples are built on graphs.
This gives a  connection to Feynman diagrams which is responsible for the terminology.
 In our setup the graphs are ---crucially--- the morphisms and not the objects.

Having encoded the theories in this fashion, we can make use of categorical tools, such as Kan-extensions.
This enables us to define pull--backs, push--forwards (left Kan extensions), which are always possible as well as extension by zero (right Kan extension) which sometimes exist, {\it et cetera}. Indeed we show that ``all'' classical constructions, such as free objects, the PROP generated by an operad, the modular envelope, etc., are
 all push--forwards. To further extend the applicability, we also consider the enriched version of the theory.

Among the universal constructions are universal operations, Hopf algebras, model structures, and bar/co-bar and Feynman transforms. Examples of universal operations are Gerstenhaber's pre--Lie algebra structure for operads
and its cyclic generalization, as well as BV operators etc. The classical examples are collected in Table \ref{univtable}. We show that given any Feynman category there is a universal construction for this type of universal operation, which is actually highly calculable.

The reason for giving the bar/co-bar and Feynman transforms is two-fold. On one hand, they give resolutions. For this to make sense, one needs a model category structure, which we provide.  The construction of this model structure unifies and generalizes a number of examples found in the literature (see section $\ref{htsec}$).
It for instance establishes these for all the presented examples, which includes the whole zoo of known operad--like structure.
On the other hand, said transforms give rise to master equations as we discuss. A case-by-case list of these was given in \cite{KWZ} and here we show that they all fit into the same framework. We furthermore give a W--construction which under some technical assumptions, which the interesting examples satisfy, gives a cofibrant replacement.

The Hopf algebra structures  give a new source for such algebras and go back to an observation of Dirk Kreimer. They are expounded in \cite{GKT} where we treat the Hopf algebra structures of Connes--Kreimer \cite{CK}, Goncharov \cite{Gont} and Baues \cite{Baues} as examples.

Our new axiomatization rests on the shoulders of  previous work, as mentioned above,
but it was most influenced by \cite{BM} who identified Markl's \cite{Markl} hereditary condition as essential.
For the examples based on graphs, we also use their language.
Our construction, however, is one category level higher. For the examples involving graphs, they are not objects, but underlie morphisms.
With hindsight there are many similar but different
generalizations and specializations like \cite{getzler},\cite{BergerMoerdijk},\cite{multicategory} and others, such as  Lavwere theories \cite{Lawvere} and
FI--modules \cite{Farb} and  crossed--simplicial groups \cite{Lodaycrossed}, whose relationship to Feynman categories we discuss in \S\ref{otherpar}.

\subsection{Main definition}

The main character is the notion of a Feynman category. It has the following concise definition.
 Fix a  symmetric monoidal category $\F$ and let $\asts$ be a category that is a groupoid, that  is $\asts=Iso(\asts)$. Denote the free symmetric monoidal category on $\V$ by $\V^{\otimes}$.
Furthermore let $\imath\colon\V\to \F$ be a functor and let $\imath^{\otimes}$ be the induced symmetric monoidal functor $\imath^{\otimes}\colon \V^{\otimes}\to \F$.

\begin{customdf}{\ref{feynmandef}}A triple $\FF=(\V,\F,\imath)$ of objects as above is called a Feynman category
if

\begin{enumerate}
\renewcommand{\theenumi}{\roman{enumi}}

\item (Isomorphism condition)
%\label{objectcond}
The monoidal functor $\imath^{\otimes}$ induces an equivalence of symmetric monoidal categories between $\V^{\otimes}$ and $Iso(\F)$.

\item (Hereditary condition) The monoidal functor $\imath^{\otimes}$ induces an equivalence of symmetric monoidal categories between $Iso(\F\downarrow \V)^{\otimes}$ and
$Iso(\F\downarrow\F)$.
%\label{morcond}

\item (Size condition) For any $\ast\in \asts$, the comma category $(\clusters\downarrow\ast)$ is
essentially small,
viz.\ it is equivalent to a small category.
\end{enumerate}
\end{customdf}
How to think about these conditions is discussed in detail in \S\ref{explainsec}. A very brief description is as follows: Condition (i) means that the objects of $\F$ are up to isomorphism words in objects of $\V$ and the isomorphisms in $\F$ are words of isomorphisms in $\V$ and permutations of letters. Condition (ii) is the vital condition. It states that up to isomorphism any morphism can be decomposed into simpler, elementary, or as we call them one--comma generating morphisms. Namely, any morphism $\phi:X\to Y$ in $\F$ can be factored up to isomorphism into a tensor product $\bigotimes_{v\in I}\phi_v$, where $\phi_v: X_v\to \imath(\ast_v)$, whenever an isomorphism of  $Y$  with a word $\imath^\otimes(\bigotimes_{v\in I} \ast_v)$  in elements of $\V$ has been picked. The notation $\V$ is chosen for vertices and a vertex is called $\ast_v$.
Such a word exists by (i).
It is implied that $X\simeq \bigotimes_v X_v$. The condition (ii) actually asks more, namely that these decompositions are unique up to a unique block isomorphism which is a tensor product of isomorphisms on the sources of the $\phi_v$. The last condition (iii) is technical and ensures that certain colimits exist.
Hence the comma category $(\F\downarrow\V)$, mapping into words of length one, generates the morphisms under
tensor product up to (unique) isomorphism.

Feynman categories are interesting on their own, just like groups, the simplicial category $\DDelta_+$ or the crossed simplicial group $\DDelta S$, which are all examples. This intrinsic interest also manifests itself by the construction of Hopf algebras from their morphisms. These Hopf algebras include the ones of Connes and Kreimer, Goncharov and a Hopf algebra considered by Baues for the chains on a double loop space. This is discussed in \S\ref{hopfsec}.

A main motivator for the study of Feynman categories, however, arises by considering functors out of them, aka.\ algebras over them, or $ops$ as we call them to avoid confusion arising by the overused term algebra. The technical definition is as follows.
For a given $\FF=(\V,\F,\imath)$ we denote by
$\F$-$\opcat_\C$ the strong symmetric monoidal functors $Fun_{\otimes}(\F,\C)$ where $\C$ is a symmetric monoidal category, which is taken to be cocomplete and satisfying that tensor and colimits commute.
We also consider $\V$-$\smodcat_\C$ which is $Fun(\V,\C)$. Both $\fopsc$ and $\vmodsc$ are themselves symmetric monoidal categories using the ``levelwise'' monoidal structure inherited from $\C$.

Of course, there is a forgetful functor from
$\F$-$\opcat_\C$ to $\V$-$\smodcat_\C$
given by restricting along $\imath$.
The first set of results then deals with a free functor left adjoint to the forgetful one and an equivalence of categories between the category  of functors
from a Feynman category and the category algebras over the monad (aka.\ triple) given by the forget and free functors. This is known as monadicity and gives two equivalent ways of characterizing the operations. The definition as algebras over a triple (aka.\ monad)  is usual for instance in (modular) operad theory \cite{MSS,GKmodular}.

\subsection{Examples}
To illustrate the concept, we provide a few examples with increasing difficulty. Some of the examples are based on the premise that some of these structures are known already to the reader. For the main text, no such knowledge is required. All notions are introduced to make the presentation self--contained and furthermore the theory goes beyond the collection of known examples. Indeed this is one of the main aims.

\subsubsection{Tautological examples} Given a groupoid $\V$, we can let $\F=\V^{\otimes}$ and $\imath:\V\to \F$ the standard inclusion. This yields a Feynman category whose category of $\fops_\C$ is equivalent to $\vmodsc$ which is the category of groupoid representations of $\V$. If the groupoid $\V$ only has one element $*$, then these are the group representations of $G=Aut(*)$ mentioned above. It is in this setting that the push--forward, or left Kan extension, becomes induction, see  Theorem \ref{pushthm} below.

\subsubsection{Trivial $\V$}
These examples are treated in detail in \S\ref{trivsec}.
We consider $\V=Triv$  the category with one object $*$ with $Hom_{\V}(*,*)=id_*$.
Then  let $\V^{\otimes}$ be the  free symmetric category and let $\bar\V^{\otimes}$ be its strict version. It has  objects $n:=1^{\otimes n}$ and each of these has an $\Sn$ action: $Hom(n,n)\simeq \SS_n$.
Let $\bar\imath$ be the functor from $\V^{\otimes}\to\bar\V^{\otimes}$.
 $\F$ will be a category with $Iso(\F)$ equivalent to $\bar\V^{\otimes}$, which itself
 is equivalent to the skeleton of  $Iso(\FinSet)$ where $\FinSet$ is the category of finite sets with $\amalg$ as monoidal structure.  Here,  for  convenience, we do use the strictification of $\amalg$.

A particular example, which is discussed \S\ref{surjpar}, is given by choosing $\F=\Surj$ which is the category of finite sets with surjections as morphisms and disjoint union as monoidal structure.  Using $\imath$ as the natural inclusion, this yields the Feynman category $\FF_{surj}=(Triv,\Surj,\imath)$.
In this case $\fopsc$ is the category of  non-unital commutative associative monoids.

If we are in the non--symmetric version, then $\V^{\otimes}$ again has the natural numbers as objects, but is discrete, which means that the only morphisms are identities. The analog of $\Surj$ in this case is $\Surj_<$ the category of order preserving surjections. In this case $\fopsc$ will be non--unital associative monoids in $\C$.

Other examples are crossed simplicial groups, the simplicial category or the category FI, which appears in the threory of FI--modules.

These examples become even more interesting if one allows enrichment. For instance enriching $\FF_{surj}$ amounts to specifying an operad $\O$ (see below) with trivial $\O(1)$. The category $\fopsc$ in this case is the category of
algebras over the specified operad $\O$ (with trivial $\O(1)$).
The algebras over operads with non-trivial $\O(1)$ also form examples of $\opcat$ for a Feynman category, but these Feynman categories do not have a trivial $\V$.

There is an abundance of other related structures discussed throughout the text.

\subsubsection{Graphs, Trees and Operads}
There is a Feynman category $\operads$ such that the objects of $\operads$-$\opcat_\C$ are the operads in $\C$ and the objects of $\V$-$\smodcat_\C$ are the $\Sigma$-- aka.\ $\SS$-modules. In general, we will say that a Feynman category is the Feynman category for a certain structure $X$, e.g.\ operads, algebras, etc., if $\fopsc$ is the category of $X$--structures in $\C$, e.g.\ operads in $\C$, algebras in $\C$, etc..

 $\operads$ and all
 the Feynman categories for the beasts in the known zoo of operad--like structures,   including but not limited to PROPs, modular operads, twisted (modular) operads, properads, hyperoperads, operads with multiplication, their colored versions arise from {\it one fundamental} example $\GG=(\Crl,\Agg,\imath)$, see \S\ref{ggpar}.
 Here  $\Crl$ is the full subgroupoid of corollas and $\Agg$ is the full subcategory of aggregates of corollas in $\Graphs$,  Borisov--Manin's category of graphs \cite{BM}, which is reproduced in Appendix A.
 These examples are discussed in detail in \S\ref{examplesec}. That section also contains practical information about how to include special elements, like units, multiplications, differentials, {\it et cetera}.

 We will briefly introduce these categories informally.  A graph consists of a collection of vertices, a collection of half edges, aka.\ flags and two morphisms. The first says to which vertex a half edge is attached and the second says which half edges are glued to an edge and which ones remain unpaired, i.e.\ are tails. Objects of the Borisov--Manin category of graphs are these graphs. The definition of morphism is a bit tricky. This allows one to have edge contractions, merging of vertices and gluing of tails and all their compositions as  morphisms.

 A morphism actually has three parts. A surjection on vertices from those of the source graph to those of the target graph. An injection from the flags of the target graph to the flags of the source graphs. The third part is an involution, which keeps track of the flags that are glued to edges and then contracted.
 These are called ghost edges of the morphisms.  Technically this is a pairing of all the flags that are not in the image of the injection on flags. This category is monoidal under disjoint union.

 Our key observation is  that the ghost edges of a morphism then define an underlying ghost graph of the morphism, by using the vertices and flags with their incidence of the source graph and the ghost edges as edges. It is these ghost graphs {\em underlying morphisms} (as opposed to graphs as objects of $\Graphs$) that play the role of the graphs in the usual description of the zoo; e.g.\ trees, rooted trees, directed graphs, etc. We will now give some details on this.

Namely, with this preparation, we can introduce several of the beasts. A corolla is a one vertex graph without loops. It only has tails attached to it. The objects of $\Crl$ are corollas whose vertex is $*$ and whose tails are the set $S$, the standard notation for this graph being $*_S$. $\Crl$ is then the full subgroupoid on corollas. That is, the morphisms are the isomorphisms between corollas. An isomorphism of corollas $*_S$ and $*'_T$ is uniquely determined by  a  bijection  $S\leftrightarrow T$.  An aggregate is a disjoint union of corollas. This explains all the data that defines the Feynman category $\GG$.

By restricting morphisms in $\Agg$, via restriction of the type of underlying ghost graphs of the morphisms, one obtains several of the beasts. If one, for instance, says that the underlying ghost graphs for the morphisms from an aggregate to a corolla have to be trees, one arrives at the Feynman category for cyclic operads. A general morphism then will have a forest as the ghost graph. To obtain the Feynman category $\operads$ one has to do this restriction, but also add the additional data of a root. One way is to just define morphisms with this extra data. This is generalized by the notion of a Feynman category indexed over another Feynman category, see Definition \ref{overdef}.

There is a more general way to do this in several steps. First, one decorates all the flags as ``in'' or ``out''.
Then one can restrict to morphisms whose ghost edges have exactly one ``in'' and one ``out'' flag. This yields PROPs. If one further restricts to connected ghost graphs for the morphisms from aggregates to corollas, one obtains properads. Finally, if one restricts to aggregates whose corollas all have exactly one ``out'' flag one obtains operads.
Decoration here is a technical term \cite{decorated}, which we introduce in \S\ref{decopar}.

Adding a unit or an associative multiplication can again be done in two different ways. In the first, one adjoins  special morphisms to the Feynman category and then quotients by the relations defining the structure. This is explained in \S\ref{unitmultsec}. The special morphisms all have the unit of the Feynman category as the source.  In $\GG$ this is the empty corolla.
The units or multiplication will then be the image of these morphisms under a functor $\O\in \fopsc$. This is again a case of indexing.

The second way to view this is as a decoration of the vertices. A familiar example are operads with multiplication, here the underlying ghost graphs have the additional decoration of vertices being black or white, where a black vertex indicates the morphism whose image under $\O$ is multiplication, see e.g.\cite{KS,del,KSchw}. These decorations are discussed in Example \ref{bwex}.

The whole zoo of examples are derived from $\GG$ by adding special morphisms, decoration and restriction. (See Table \ref{table1}).

\subsection{Discussion of the results}

Here we give a short overview of the results providing their logical interconnectedness, while giving examples of the structures under discussion.
This synopsis is not intended to be exhaustive and there are more constructions, (technical) details and applications in the main text.  See the organization of the text below for more particulars.

We prove the following structural theorem, which shows that there are two equivalent ways to characterize  $\fopsc$. Either as the category of strong symmetric monoidal functors, or as the category of algebras over a triple.
This is known in the theory of operads, where the triple description is in terms of rooted trees.

%Then main theorem is
\begin{customthm}{\ref{freethm}}
Let $\CalC$ be a cocomplete symmetric
monoidal category such that $\otimes$ preserves colimits in each variable, then there is a left adjoint (free) functor  $F=\free$
to the forgetful functor $G=\forget$ which is comonoidal.
That is, these functors are adjoint  symmetric (co)monoidal  functors for the symmetric monoidal categories of $\smodcat$ and $\opcat$.
\end{customthm}

This theorem has been proven over and over again for each special case and is now available in full generality.

It is actually a special case of a more general statement about the existence of push--forwards along morphisms of Feynman categories.
Morphisms between Feynman categories are given by functors compatible with $\iota$, see Definition \ref{mordef} for a more nuanced definition. A more general theorem for such functors is as follows. Let $\O\in \fopsc$, consider  a morphism of Feynman categories $\FF$ to  $\FF'$ and let $f:\F\to \F'$ be a strong symmetric monoidal functor which is part of the data of such a morphism. Denote the left Kan extension of $\O$ along $f$ by $Lan_f\O$.
This is a functor from $\F'$ to $\C$ whose existence is guaranteed by the the condition (iii) and assumption that $\C$ is cocomplete.

\begin{customthm}{\ref{pushthm}}
$Lan_{f}\O$ is a strong symmetric  monoidal functor, which we denote  as push--forward
 $f_*\O:=Lan_{f}\O$, and hence $f$ defines a functor $f_*:\fopsc\to \F'$-$\opcat_\C$. Moreover $f_*,f^*$   form an adjunction of symmetric monoidal functors between the symmetric
 monoidal categories $\fopsc$ and $\F'$-$\opcat_\C$.
\end{customthm}
The free construction follows by applying this theorem to the tautological Feynman category $(\V,\V^{\otimes},\jmath)$ and its inclusion into $\FF$. Further examples include the modular envelope of a cyclic operad, the PROP generated by an operad or the restriction of a cyclic operad to an operad or a modular one to a cyclic one. There are more exotic ones like the PROP or properad generated by a modular operad. Again, the theorem is universally applicable and the usual case--by--case scenarios can be avoided.

The most natural setting is to regard strong symmetric monoidal functors, which is what is considered when using $\fopsc$. However, given a Feynman category $\FF$, we show in \S\ref{constructionsec} that there are Feynman categories
$\FF^{\boxtimes}$ called the free Feynman category and  $\FF^{nc}$ called the non--connected Feynman category,
such that $\FF^{\boxtimes}$-$\opcat_\C$ is equivalent to the category of all functors $Fun(\F,\C)$.
and $\FF^{nc}$-$\opcat_\C$  is equivalent to the category of all lax symmetric monoidal functors $Fun_{lax}(\F,\C)$.

As mentioned previously, there are several levels at which one can consider Feynman categories. The lowest corresponds essentially to restricting the monodial structure to be like that of the category $\Set$ and
disjoint union, we will call this combinatorial. This includes all the classical operad--like examples (except the twisted ones). Here one can weaken condition (ii) to the existence of the tensor factorizations of morphisms. Namely,  any morphism $\phi:X\to Y$ can be factored up to isomorphism into a tensor product $\bigotimes_{v\in I}\phi_v$ of one--comma generators.
The difference is that the condition (ii) also requires such a decomposition to be unique up to a unique isomorphism that has to be in ``block form''.  Here ``block'' form means that the $\phi_v$ may be permuted and replaced only individually by isomorphic morphisms. This is discussed in detail in \S\ref{explainsec} and \S\ref{othersec}. This condition is easily checked for instance in  graph based examples which include the zoo of operad--like examples. We emphasize that the graphs are not objects, but underly morphisms. The more general view is given by passing from categories with graph morphisms to  symmetric monoidal categories satisfying additional conditions, which is our basic definition. There is also a non--symmetric version that arises when demanding only that $\F$ is simply a monoidal category.

\subsubsection{Decoration}
As mentioned before, many examples come from decorations. Decoration is defined in the following way.

\begin{customthm}{\ref{decoexistthm}} Given an $\O\in \F$-$\opcat_\C$ with $\C$ Cartesian, there exists the following Feynman category $\FFdeco=(\Vdeco,\Fdeco,\ideco)$ indexed over $\FF$. The objects of $\F_{dec\O}$ are pairs $(X,dec\in \O(X))$ and $Hom_{\F_{dec\O}}((X,dec),(X',dec'))$ is the set of $\phi:X\to X'$, s.t. $\O(\phi)(dec)= (dec')$. $\Vdeco$ is given by the pairs $(\ast,a_{\ast})$, $\ast \in \V$, $a_\ast\in \O(\imath(\ast))$ and $\ideco$ is the obvious inclusion. The indexing
is given by forgetting the decoration.
\end{customthm}

There is a slightly more involved construction for decorations in the case that $\C$ is not Cartesian.
Examples are the genus decoration for modular operads, the direction for PROPs and properads, the colors for colored versions and so on. Other types of decorations come from functors such as the associative operad, its cyclic version, or its modular envelope which yield the non-$\Sigma$ versions of operads, cyclic operads and modular operads respectively. This is discussed in detail in \S\ref{decopar}.

\subsubsection{Enrichment and twists}
A higher level of generalization is given by considering enriched monoidal categories and subsequently twists.
To this end, we give an alternative formulation of the conditions, which is rather highbrow and allows one to generalize to the enriched case, where the enrichment is over a Cartesian category. Note that here we are not talking about, say operads in a category $\C$, but Feynman categories enriched over a Cartesian category $\CalE$. Examples will then be the Feynman category whose $\fops$ are algebras over an operad in $\CalE$, cf.\ \ref{operfeypar}, e.g.\ algebras over a simplicial operad.
In order to add twisted (modular) operads to the story, we have to go beyond Cartesian enrichment, since these beasts live in $k$--linear settings. The answer to this is to work with so--called indexed enriched Feynman categories.  These enrichments will be needed starting in Chapter \ref{enrichedsec}  and the reader not interested in this level of generality can skip these rather heavy categorical considerations and return to them later if needed. An indexed enrichment is given by a lax 2--functor $\D$ from $\F$ to $\CalE$ satisfying additional axioms, where both monoidal categories are viewed as 2--categories. This theory allows us to also include the twisted versions. In fact, this allows us to prove the analogues of the two theorems (monadicity and push-forwards) as above also in the twisted case. Indeed the full theory of twists and co-cycles for modular operads is generalized to arbitrary Feynman categories. These twists, in particular the odd twists, are necessary in order to define the bar/cobar or Feynman transforms.

An equivalent description of the enrichment functors $\D$ is given as elements of $\F^{hyp}$-$\opcat$, where $\FF^{hyp}$, called the hyper Feynman category, is a category constructed from $\FF$ via a two--step process. The first is a $+$ construction, similar to the opetopic principle \cite{BaezDolan}, and the second is a quotient in which the value of $\D$ on isomorphisms becomes invertible. This construction is done in \S\ref{Fplussec} and \S\ref{hypersec}.
Examples are that if $\modular$ is the Feynman category for modular operads, then $\modular^{hyp}$ is the Feynman category for hyper--operads. These are what can be used to twist modular operads. Other examples are $\FF_{surj}^+=\FF_{\rm May-Operads}$, the Feynman category for (non--unital) May operads, while $\FF_{surj}^{hyp}=\operads_0$ the Feynman category of reduced operads, i.e.\ whose $\O(1)=\unit_\C$. The $+$ and $hyp$ constructions lead to a boot-strap. There are however several sources for this boot-strap. As far as is known $\modular$ can not be obtained as $\FF^+$ or $\FF^{hyp}$.

Importantly for any Feynman category $\FF=(\V,\F,\imath)$ and a  $\D$ thought of equivalently as a 2--functor or as an element of $\FF^{hyp}$-$\opcat_\C$, we define an {\it indexed enriched} Feynman category $\FF_{\D}=(\V_\D,\F_\D,\imath_\D)$. With this definition, generalizations of the two structural theorems hold.

\begin{customthm}{\ref{enrichedtriplethm}}  $\FF_{\D}$ is a weak Feynman category.  The forgetful functor from $\FF_{\D}$-$\opcat$ to $\V_{\CalE}$-$\smodcat$ has a left adjoint and more generally push-forwards among indexed enriched Feynman categories exist. Finally there is an equivalence of categories between the category of algebras over the triple $GF$ and $\F_{\D}$-$\opcat$.
\end{customthm}

Here a weak Feynman category defined in Definition \ref{weakdef} is a weakening of the condition (i).  This construction also allows one to give a definition of algebras of a given element of $\fopsc$ as follows. Suppose $\FF=\bar\FF^{hyp}$ for some underlying Feynman category $\bar\FF$,
and fix $\D\in \fopsc$. Then $\bar \FF_{\D}$ is an indexed enriched Feynman category such that $\bar \FF_{\D}$-$\opcat_\C$ are the {\it algebras over $\D$}.
In particular for a reduced operad $\O$, $\FF_{surj,\O}$ is the Feynman category for algebras over $\O$. There are constructions to deal with non--reduced operads given in \S\ref{examplesec} and \S\ref{operfeypar}.

\subsubsection{Feynman categories through generators and relations and their odd versions}
Going back to considering algebras and operads, there are in fact three ways of defining them. For operads, first one may say that for each tree there is a composition, this corresponds to being an element in $\operads$-$\opcat$. The second is as an algebra over the monad (or triple) of trees, which is the setting of the monadicity theorem.  Lastly, one may give the operations $\circ_i$ and impose $\SS_n$ equivariance and associativity axioms.  This generators and relations approach is also familiar for defining associative algebras. Namely, one postulates a multiplication and then adds its associativity as a relation.

For Feynman categories there is also such a generators and relations approach. This is described in \S\ref{genrelpar}. Here the generators are a subset of the one--comma generators, which can be used to write any one--comma generator  using composition (aka.\ concatenation) and tensor powers of identities. I.e.\ they are generators for the partial algebra structure of morphisms under composition and the algebra structure under the monoidal product. The way this works is most easily seen in the graph based examples, with simple morphisms. Here a simple morphism is either a simple edge contraction or simple loop contraction, i.e.\ by a morphism whose underlying graph has one edge, i.e.\ a tree with two vertices or a one vertex graph with one loop, see Figure \ref{generatorfig}. If the graphs are not connected there is another operation on graphs called a merger, whose image under a functor will be a ``horizontal'' composition. These remarks are detailed in \S\ref{graphstrucsec}, where $\GG$  as the paradigmatic example is studied in depth. Up to isomorphism there are the three generating morphisms above. Restricting to the connected case, the possible ways to decompose a morphism into a sequence of simple morphisms up to isomorphism correspond to ordering the edges of the underlying graph. If there are mergers, then one can agree to  do the mergers last. These decompositions are not unique and hence yield relations. A typical such relation is that when performing two edge contractions their order does not matter for the resulting composition.

 For general Feynman categories this leads to the notion of ordered presentations as well as ordered and oriented Feynman categories, which allow one to decompose morphisms in a standard fashion up to isomorphisms and permutations. When passing to an enrichment over $\Ab$, the oriented versions become the odd versions $\FF^{odd}$, after one assigns parity to the relations. These parities have to be compatible, a fact that can most easily be checked by introducing 2--morphisms in the diagrams yielding the relations, see \S\ref{2catsec}.  The full data of this is called an ordered presentation. This gives an action of the symmetric group and allows one to define $\FF^{odd}$, the odd twisted version of $\FF$. That is, $\FF^{odd}$ is dependent on such a choice. There are additional conditions of being ``resolving'' (Definition \ref{cocompsec}), and being ``cubical'' (Definition \ref{cubicaldef}), which will become important for the existence of differentials and resolutions.

\subsubsection{Feynman category of universal operations}
One of the most fruitful observations in operad theory is the one by Gerstenhaber that given an operad $\O$ in a $k$--linear category there is a pre--Lie structure given by $a\circ b=\sum a\circ_i b$ here the sum runs from $1$ to $n$ if $a$ is in $\O(n)$.
The pre-Lie structure is on $\bigoplus_n \O(n)$. In section \S\ref{universalsec}, we generalize this type of observation to arbitrary Feynman categories.
To set this up, we have to first go to the cocompletion of $\F$ denoted by $\hat \F$ as spelled out in \S\ref{cocompsec}. Let $y: \F\to \hat \F$ be the inclusion, known as the Yoneda embedding. Inside the cocompletion $1:=\colim_{\V} y \circ \imath$
makes sense. By general theory, which we review, $\hat \F$ is a symmetric monoidal category with respect to Day convolution $\day$. Thus, we can regard the symmetric monoidal subcategory $\F_\V$ generated by $1$ = inside $(\hat \F,\day)$. Let $\V_\V$ be the trivial subcategory with object $1$, then:

\begin{customthm}
{\ref{universalthm}}
The category $\FFV=(\VV,\FV,\iV)$ is a Feynman category. This analogously holds in the enriched
and indexed-enriched case.
\end{customthm}
The upshot is that any $\O\in \fopsc$ to a cocomplete $\C$ factors through $\hat \F$ and by restriction yields a functor $\hat\O:\F_\V\to \C$. In the particular case of an operad this is $\hat\O(1)=\bigoplus \O(n)_{\SS_n}$. But there is more structure,
since there are morphisms in the Feynman category $\FFV$. These morphisms give universal operations. After applying $\hat \O$ they operate on the $\hat\O(n)=\hat\O(1)^{\otimes n}$. They are highly calculable, see e.g.\ \S\ref{preliesec}.
Table \ref{univtable} contains results for many of the usual suspects. We note that these operations live by construction on the full coinvariants. If there are planar versions, e.g.\ given by an extra decoration, forgetting the planar structure/decoration gives a morphism of Feynman categories along which the operations can be pulled back. This explains why the pre-Lie structure for an operad lifts to $\bigoplus_n \O(n)$ while the Lie algebra structure for an anti-cyclic operad lifts from $\bigoplus_n \O((n))_{\SS_n}$ only to $\bigoplus_n \O((n))_{C_n}$. In that case the $\V$ for the Feynman category for non-$\Sigma$ cyclic operads is not discrete.

\subsubsection{Bar/cobar and Feynman transforms}

A striking fact is that these operations also appear in master equations associated to the Feynman transforms, see Table \ref{Ftable}.  This is no coincidence, as we show.

The first step is the construction of another type of universal operation called $d_{\Phi}$ associated to a set of morphisms $\Phi$, which is an element in $Mor(\F_{\V'}^{\prime op})$. Passing to the odd version, we can ask if this element squares to zero, then the set $\Phi$ is called resolving and $d_\Phi$ yields a differential, see \S\ref{unidiffsec}. This is the avatar of the differential for the Feynman transform, in which there is one term for each element in $\Phi$ and accordingly there is one term in the master equation corresponding to each such element.  Moreover in the odd version, these elements are exactly the ones that behave as if they have degree one. In particular for the graph based examples, these are usually the edges and loops.

To be more explicit, we shortly describe the bar/cobar and Feynman transforms for Feynman categories. They generalize the bar and cobar transforms for algebras or operads. The Feynman transform is a hybrid which exists when there is a duality. The archetypical situation is a finite dimensional algebra $A$. The dual or Feynman transform would be the co-bar transform on the dual co-algebra $\Omega \check A$. In order to be able to talk about the transforms, we have to pass to complexes in $\C$. Notice that the classical bar construction is a basically free construction, but there is a degree shift to make the differential work. One can interpret this as assigning degree $1$ to the symbols $|$. For (most of) the graph examples, this will be achieved by making the edges have degree $1$,  see Figure \ref{barfig}. This is because in the generators and relations picture, the edges of the ghost graphs correspond to  simple morphisms (edge and loop contractions) making up a morphism. The relation between two such contractions is quadratic, i.e.\ two such contractions commute. Giving the generators degree $1$ will then make two contractions anti--commute, whence the name odd.

Things become more difficult when including mergers, since there is a non--homogeneous relation, see \S\ref{graphstrucsec}, in particular equation \eqref{triangleeq}. This means that although simple morphisms, mergers will have degree $0$ assigned to them.

In general to define what ``odd'' means, one has to have a Feynman category $\FF$ with a so--called ordered presentation, see Definition \ref{opdef}. This extra datum will determine $\FF^{odd}$ and hence the degrees.

\begin{figure}
    \centering
    \includegraphics[width=0.5\textwidth]{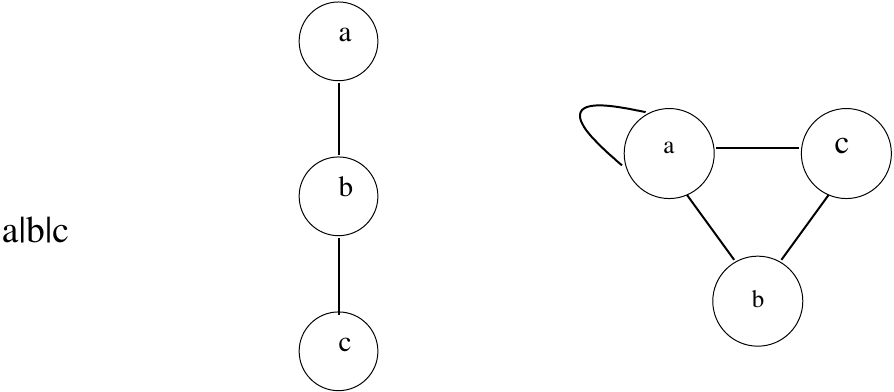}
    \caption{The sign mnemonics for the bar construction, traditional version with the symbols $|$ of degree $1$, the equivalent linear tree with edges of degree $1$,  and a more genaral graph with edges of degree $1$. Notice that in the linear case there is a natural order of edges, this ceases to be the case for more general graphs}
    \label{barfig}
\end{figure}

In order to give the definition, we need a bit of preparation.
Since $\V$ is a groupoid, we have that $\V\simeq \V^{op}$. Thus, given a functor $\Phi:\V\to \C$, using the equivalence we get a functor from $\V^{op}$ to $\C$ which we denote by $\Phi^{op}$. Since the bar/cobar/Feynman transform adds a differential, the natural target category from $\fops$ is not $\C$, but complexes in $\C$, which we denote by $Kom(\C)$. Thus any $\O$ may have an internal differential $d_\O$.

\begin{customdf}{\ref{ftdef}}
Let $\FF$ be a Feynman category enriched over $\Ab$ and with an ordered presentation and let $\FF^{odd}$ be its corresponding odd version.
Furthermore let $\Phi^1$ be a resolving subset of one--comma generators and let $\op{C}$ be an additive category.  Then:
\begin{enumerate}
\item  The bar construction is a functor
\begin{equation*}
\Bar \colon \fops_{Kom(\C)}\to \foddops_{Kom(\C^{op})}
\end{equation*}
defined by
\begin{equation*}
\Bar(\O):=\imath_{\FF^{odd} \; *}(\imath_{\FF}^*(\O))^{op}
\end{equation*}
together with the differential $d_{\op{O}^{op}}+d_{\Phi^1}$.

\item  The cobar construction is a functor
\begin{equation*}
\Cobar \colon \foddops_{Kom(\C^{op})}\to \fops_{Kom(\C)}
\end{equation*}
defined by
\begin{equation*}
\Cobar(\op{O}):=\imath_{\FF \; *}(\imath^\ast_{\FF^{odd}}(\op{O}))^{op}
\end{equation*}
together with the co-differential $d_{\op{O}^{op}}+d_{\Phi^1}$.

\item Assume there is a duality equivalence $\vee\colon \CalC\to \CalC^{op}$.
The Feynman transform is a pair of functors, both denoted $\FT$,
\begin{equation*}
\FT\colon \fops_{Kom(\C)}  \leftrightarrows \foddops_{Kom(\C)}\colon \FT
\end{equation*}
defined by
\begin{equation*}
 \FT(\O):=\begin{cases} \vee\circ \Bar(\O) & \text{ if } \O \in \fops_{Kom(\C)} \\ \vee\circ \Cobar(\O) & \text{ if } \O \in \foddops_{Kom(\C)}
\end{cases}
\end{equation*}
\end{enumerate}
\end{customdf}
The differentials are explained in more detail in \S\ref{diffsec}.

\begin{customlem} {\ref{adjunctionlem}} The bar and cobar construction form an adjunction.
\begin{equation*}
\adj{\Cobar}{\foddops_{Kom(\op{C}^{op})}}{\fops_{Kom(\op{C})}}{\Bar}
\end{equation*}
\end{customlem}

The quadratic relations in the graph examples are a feature that can be generalized to the notion of cubical Feynman categories, see Definition \ref{cubicaldef}. The name reflects the fact that in the graph example the $n!$ ways to decompose a morphism whose ghost graph is connected and has $n$ edges into simple edge contractions correspond to the edge paths of $I^n$ going from $(0,\dots,0)$ to $(1,\dots, 1)$. Each  edge flip in the path represent one of the quadratic relations and furthermore the $\SS_n$ action on the coordinates is transitive on the paths, with transposition acting as edge flips.

This is a convenient generality in which to proceed.

\begin{customthm}{\ref{resthm}}  Let $\FF$ be a cubical Feynman category and $\op{O}\in \fops_{Kom(\C)}$.  Then the counit $\Cobar\Bar(\op{O})\to\op{O}$ of the above adjunction is a levelwise quasi-isomorphism.
\end{customthm}

\subsubsection{Master equations}
By construction the Feynman transform is a quasi-free element in $\FF^{odd}$-$\opcat_{Kom(\C)}$ and hence one may ask what extra equation governs compatibility of morphisms from the Feynman transform with the differentials.  Or even more generally how to characterize the Feynman transform by Yoneda. The answer is as follows in the classical cases:
\begin{customthm}{\ref{methm}}(\cite{Bar},\cite{MerkVal},\cite{wheeledprops},\cite{KWZ})  Let $\op{O}\in \fopsc$ and $\op{P}\in \foddops_\op{C}$ for an $\F$ represented in Table $\ref{Ftable}$.  Then there is a bijective correspondence:
\begin{equation*}
Hom(\FT(\op{P}),\op{O})\cong ME(lim(\op{P} \tensor\op{O}))
\end{equation*}
\end{customthm}

Here ME is the set of solutions of the appropriate master equation set up in each instance.  With Feynman categories this tabular theorem can be compactly written and generalized.  The first step is the realization that the differential specifies a natural operation, in the above sense, for each arity $n$.  This natural operation living on a space associated to an $\op{O}p$ $\op{Q}$ is denoted $\Psi_{\op{Q},n}$ and is formally defined by:

\begin{customdf}
{\ref{MEdef}}
For a Feynman category $\FF$ admitting the Feynman transform and for $\op{Q}\in\fopsc$ we define the formal master equation of $\FF$ with respect to $\op{Q}$ to be the completed cochain $\Psi_{\op{Q}}:= \prod \Psi_{\op{Q},n}$.  If there is an $N$ such that $\Psi_{\op{Q},n}=0$ for $n>N$, then we define the master equation of $\FF$ with respect to $\op{Q}$ to be the finite sum:
\begin{equation*}
d_{\op{Q}}+\ds\sum_{n}\Psi_{\op{Q},n} = 0
\end{equation*}
We say $\alpha\in lim(\op{Q})$ is a solution to the master equation if $d_\op{Q}(\alpha)+\sum_{n}\Psi_{\op{Q},n}(\alpha^{\tensor n}) = 0$, and we denote the set of such solutions as $ME(lim(\op{Q}))$.
\end{customdf}
Here the first term is the internal differential and the term for $n=1$ is the differential corresponding to $d_{\Phi^1}$, where $\Phi^1$ is the subset of odd generators.  We then show the following generalization of Theorem \ref{methm}:

\begin{customthm}{\ref{methm2}}  Let $\op{O}\in \fopsc$ and $\op{P}\in \foddops_\op{C}$ for an $\F$ admitting a Feynman transform and master equation.  Then there is a bijective correspondence:
\begin{equation*}
Hom(\FT(\op{P}),\op{O})\cong ME(lim(\op{P} \tensor\op{O}))
\end{equation*}
\end{customthm}

\subsubsection{Homotopy theory for $\fopsc$}
Finally, let us turn to the homotopy theory of $\fopsc$ using model categories. We refer to \S\ref{htsec} for full details. The motivation is two-fold.  First, having a model category structure is of independent interest and much effort has gone into constructing these.  We thus view it as an application of our theory that this type of  work which up to now has been in a case-by-case analysis, can now be done very broadly (independent of $\FF$). %It is maybe indicative of the fragmentation of treatments, that particular cases that are actually covered by our theory have been worked out even after our theory has become available.
The second motivation is to show that the above transforms yield cofibrant replacements as they do for algebras.  Among other applications this allows us to study homotopy classes of ME solutions.

Our general theorem on  model category structures is:

\begin{customthm}{\ref{modelthm} and \ref{feymodthm}}  Let $\FF$ be a Feynman category.  There exist conditions on a symmetric monoidal model category $\op{C}$, which in particular include $\Top$, $\SSet$, and chain complexes in characteristic $0$, such that
$\fopsc$ is a model category where a morphism $\phi\colon \op{O}\to\op{Q}$ of $\F$-$\opers$ is a weak equivalence (resp. fibration) if and only if $\phi\colon \op{O}(v)\to\op{Q}(v)$ is a weak equivalence (resp. fibration) in $\op{C}$ for every $v\in \V$.
\end{customthm}

The proof involves transfer across the forgetful-free adjunction.  The non-smallness of objects in $\Top$ makes that case more complicated and we are careful to give full details in Appendix B.

Notice that we have formulated this result to put the restrictions on $\C$, and have no restriction on $\FF$.  In particular we have a model structure on $\fopsc$ for any $\FF$ including cyclic operads, modular operads, colored properads, colored PROPs, and the whole bestiary, regardless of any particular details of said $\FF$ such as existence of units etc.

A benefit of formulating the problem this way is that the adjunction associated to a morphism of Feynman categories $f$ is now always an adjunction between model categories:
\begin{equation*}
f_*\colon\eops \leftrightarrows \fopsc\colon f^*
\end{equation*}
where $f^*(\op{A}):= \op{A}\circ\alpha$ is the right adjoint and $f_*(\op{B}):=Lan_f(\op{B})$ is the left adjoint, see Remark \ref{shriekrmk} for the use of $f_*$ vs.\ $f_!$.  Under mild conditions this is a Quillen adjunction:

\begin{customlem}{\ref{qalem}} Suppose $f^*$ restricted to $\vfmods\to \vemods$ preserves fibrations and acyclic fibrations (see Remark $\ref{vmodmodelrmk})$.  Then the adjunction  $(f_*, f^*)$ is a Quillen adjunction.
\end{customlem}
An example of such an adjunction is the one between  $\CCyclic$ and $\modular$, i.e.\ the Feynman categories for cyclic and modular operads respectively and that given by the morphism $i\colon\CCyclic\to\modular$ by including as genus zero. The push-forward is called the modular envelope. This is a Quillen adjunction.  We give further examples of such adjunction in \S\ref{Quillensec}.

Turning to the issue of cofibrancy we show that, under suitable finiteness conditions, quasi-free objects are cofibrant.  In particular this implies:

\begin{customcor}
{\ref{cofcor1}}
The Feynman transform (Definition $\ref{ftdef}$) of a non-negatively graded dg $\F$-$\oper$ is cofibrant.
\end{customcor}

\begin{customcor}{\ref{cofcor}}  The double Feynman transform of a non-negatively graded dg $\F$-$\oper$ in a cubical Feynman category is a cofibrant replacement.
\end{customcor}

In light of Theorem $\ref{methm2}$ discussed above, this tells us how to make sense of homotopy classes of ME solutions, since the set of such solutions is now given by maps from a cofibrant object to a fibrant object.  We can thus define $\pi_0$ of the associated ME simplicial set without appealing to its fibrancy; although fibrancy of this simplicial set is an interesting future direction.
%Ben wrote MC here.

We then turn to the topological situation. Whereas the Feynman transform yields the replacement in the linear case, the W-construction is the correct replacement in the topological setting. As is expected there are a few technical conditions, which are encoded in the definition of simple (Definition \ref{simpledef}) and $\rho$--cofibrant (Definition \ref{rhocofdef}).  In the graphical cases these conditions reflect the fact that, unlike for trees and operads, we must take into account the existence of tail-fixing automorphisms.

\begin{customthm}
{\ref{Wthm}}
Let $\FF$ be a simple Feynman category and let $\op{P}\in\fopst$ be $\rho$-cofibrant.  Then $W(\op{P})$ is a cofibrant replacement for $\op{P}$ with respect to the above model structure on $\fopst$.
\end{customthm}

The W-construction $W(\op{P})$ is naturally an object in $\fopst$, ie it is a symmetric monoidal functor (see Proposition \ref{fopstprop}), and it is defined by considering a certain colimit of weighted chains of morphisms in $\FF$.

\begin{customdf}
{\ref{Wdef}}
Let $\op{P}\in\fopst$.  For $Y \in ob(\F)$ we define
\begin{equation*}
W(\op{P})(Y):= colim_{w(\FF,Y)}\op{P}\circ s (-)
\end{equation*}
\end{customdf}
The idea, going back to \cite{BoVo}, is to replace a sequence of morphism of length $n$ by a
simplex.  Since we have symmetries, under the conditions above these sequences naturally define an $n$-cube and the symmetric group action is permutation of the axis. The quotient then is the sought after simplex.

This for instance gives novel W-constructions for cyclic and modular operads.

\subsection{Organization of the text}
We give a short outline as a reference guide.

\subsubsection*{\S 1 Feynman categories}
%{\sc The paper is organized as follows:}
In \S1 we directly give the definition of a Feynman category in its most useful concise form as
a triple of objects satisfying certain axioms.
We then go on to discuss simple examples, which play a role later on, and give a detailed study of the axioms. After this, the main theorems about monadicity, push--forwards and pull--backs are given and proven. The section continues with more elaborate definitions involving Feynman categories.
These include a weakening that is equivalent in the combinatorial situation as well as Cartesian enriched and indexed enriched Feynman categories. The full discussion of the latter is postponed to Chapter \S\ref{enrichedsec}.
The section ends with a discussion of similar and related structures that have appeared in the literature.

\subsubsection*{\S 2 Examples}

Chapter \S2 is devoted to examples.
They basically come in four types. The first are ones based on graphs and give the usual suspects, operads, cyclic/modular operads, PROPs etc, by varying the types of graphs. The second type is the enriched version for these, which are introduced in \S4. This is an essential addition to the theory. It allows us for the first time to give a definition of a twisted modular operads on par with their non--twisted versions.
The third type of example is the one in which the symmetries of the Feynman category are trivial.
We classify these examples, which include crossed--simplicial groups, the simplicial category and
the category FI of finite sets with injection as well as the category of finite sets with surjections. The forth type are free constructions on these.
We first treat the graph based examples, which yield all the known operad--like structures, cf.\ Table \ref{table1}. These all arise from {\it one fundamental} example $\GG=(\Crl,\Agg,\imath)$ where $\Agg$ is the full subcategory of aggregates of corollas in Borisov--Manin's category of graphs. This category is described in Appendix A.
All the examples are derived from $\GG$ by restriction and decoration. Decoration here is a technical term which is defined in \S\ref{decopar}.

\subsubsection*{\S3 General constructions}
Before giving the enriched theory, we provide several essential general constructions which produce Feynman categories from Feynman categories   in \S3. Among them are,
iterations of Feynman categories $\FF'$, free Feynman categories $\FF^{\boxtimes}$, non--connected Feynman  categories $\FF^{nc}$, level Feynman categories (or plus construction) $\FF^+$ and hyper Feynman categories $\FF^{hyp}$. Each of these constructions has a purpose in the general theory. The latter  two for instance are essential for enrichment in linear categories. We also review the construction of decorated Feynman categories $\Fdeco$ from \cite{decorated}.

\subsubsection*{\S4 Indexed Enriched Feynman categories, (odd) twists and Hopf algebras}
\S4 contains the discussion about enrichment in the general case. Here the key definition is that of a Feynman category with an indexed enrichment. Such an enrichment is given by a type of functor $\D$.
To be precise this is a 2--functor from the Feynman category $\F$ to the enrichment category $\CalE$,
which satisfies several axioms given in Definition \ref{enrichmentfunctordef}. Here monoidal categories are treated as 2--categories.
For each such $\D$, we define an indexed enriched Feynman category $\FF_\D$. An example of such an enrichment is given by trivial enrichment and by twisted modular operads. In fact, we can generalize the construction of twisted modular operads to arbitrary Feynman categories. In this way, we also obtain the Feynman category of algebras over a fixed operad.

As an application, we review the construction of a Connes--Kreimer type Hopf algebra out of any $\Ab$ enriched Feynman category from \cite{GKT}.

\subsubsection*{\S5 Feynman categories given by generators and relations}
In \S5 we pay attention to the construction of Feynman categories and their functors via generators and relations. We start by working out  the example $\GG$.
For the special subcase of (pseudo)-operads this is the consideration that an operad structure can be given by the simple morphisms $\circ_i$ and
the associativity equations.  In general, we need such a presentation to define the transforms of \S6. In particular, we  show that given a presentation and a compatible assignment of degrees to relations, there are natural ordered, oriented and odd versions of the Feynman category.  This is the case for instance if the Feynman category is quadratic.
 For the basic graph examples the generators are simply morphisms whose underlying graphs are one edge graphs or more succinctly the  ``edges''.

\subsubsection*{\S6 Universal operations}
In \S6 we treat universal operations on colimits and limits. The paradigmatic example being pre-Lie operations on operads and Brace operations on operads with multiplication. This is generalized to any Feynman category. In particular, we recover the known pre-Lie structure for operads, the newer Lie structure for anti-cyclic operads and other structures that are known or new, cf.\ Table \ref{universaltable}. This now works for any Feynman category.

\subsubsection*{\S7 Feynman transform, the (co)bar construction and Master Equations}

\S7 introduces the (co)bar and Feynman transforms for presented  Feynman categories, i.e.\ Feynman categories given by generators and relations with a consistent choice of degrees for relations (or generators) as set forth in \S6. The examples encompass all graph examples and the theory is set up in great generality. The object underlying the Feynman transform is by definition free. We then show that under the mentioned assumptions there is a natural morphism $d_{\Phi^1}$ defined by the presentation which yields a differential and hence the Feynman transform is quasi-free dg.

\subsubsection*{\S 8 Homotopy theory of $\fopsc$}

\S8 then contains the homotopy theory needed to state that the transforms give resolutions which are cofibrant. Here the case of topological spaces is especially tricky and
we study it in detail. The upshot is a model category for $\F$-$\opcat_\C$ for any $\F$ with values in $\C$, where $\C$ is simplicial sets, chain complexes in characteristic $0$, topological spaces and any $\C$ that satisfies the condition of Theorem \ref{modelthm}.
This of course includes all the $\FF$ in Table \ref{table1}.

Quasi-free objects, and hence the Feynman transform, are cofibrant and, under stated conditions on $\FF$,  the double Feynman transform is a cofibrant replacement. In the topological category we also define a W construction and show that this is a cofibrant replacement, again under stated conditions on $\FF$.

\subsubsection*{Appendices}
We conclude with two appendices for the convenience of the reader. The first, Appendix A, is a reference section on graphs and sets up all the language of graphs that is used throughout the paper.  The second, Appedix B, is a careful analysis of the topological model category structure as Theorem \ref{modelthm} is not directly applicable and a more subtle study is needed.

\subsection*{Acknowledgments}
We would like to thank Yu.~Manin, D.~Kreimer, M.~Kontsevich, C.~ Berger, D.~Borisov, B. Fresse, M.~Fiore and M. Batanin and M.~Markl for enlightening discussions.

RK gratefully acknowledges support from NSF DMS-0805881, the Humboldt Foundation and the Simons Foundation.
He also thanks the
Institut des Hautes Etudes Scientifiques, the Max--Planck--Institute
for Mathematics in Bonn and the Institute for Advanced Study
for their support and the University of Hamburg for its hospitality.
RK also thankfully acknowledges the Newton Institute where a large part
extending the original scope of the paper came into existence during
the activity on ``Grothendieck-Teichm\"uler Groups, Deformation and Operads''.

Part of this work was conceptualized when RK visited the IAS in 2010.
At that time at the IAS RKs work was supported by the NSF under agreement
DMS--0635607. It was written in large parts while staying at the aforementioned locations.
The final version of this paper  was  finished at the IAS during a second stay which is supported by a Simons Foundation Fellowship and during stays at the Simons Center for Geometry and Physics.

 Any opinions, findings and conclusions or
recommendations expressed in this
 material are those of the authors and do not necessarily
reflect the views of the National Science Foundation.

BW gratefully acknowledges support from the Purdue Research Foundation as well as from the Simons Center for Geometry and Physics.

We also thank the referee for the many detailed comments which greatly improved the text.
\vfill

\pagebreak
\tableofcontents

\listoftables
\listoffigures
\pagebreak

\subsection*{Conventions and notations}
We will use the notion of disjoint union of sets. This gives a monoidal structure to sets. It is not
strict, but sometimes it is useful to use a strict monoidal structure. If this is the case we will say so and tacitly use MacLane's Theorem \cite{MacLane} to do this.

We will denote the comma categories by $(\,\downarrow\,)$. Recall that for two functors $\imath:\D\to \C$ and $\jmath:\CalE\to \C$
$(\jmath\downarrow \imath)$ is the category whose objects are morphisms $\phi\in Hom_{\C}(\jmath(X), \imath(Y))$.
A morphism between such $\phi$ and $\psi$ is given by a commutative diagram.
$$
\xymatrix{
\jmath(X)\ar[d]_{\jmath(f)}\ar[r]^{\phi}&\imath(Y)\ar[d]^{\imath(g)}\\
\jmath(X')\ar[r]_{\psi}&\imath(Y')
}
$$
with $f\in Hom_{\D}(X,X'), g\in Hom_{\CalE}(Y,Y')$. We will write $(\imath(f),\imath(g))$ for such morphisms or simply $(f,g)$.

If a functor, say $\imath: \V\to \F$ is fixed, we will just write  $(\F\downarrow \V)$,
and given a category $\CalG$ and an object $X$ of $\CalG$, we denote by $(\CalG\downarrow X)$ the respective comma category. I.e.\ objects are morphisms $\phi: Y\to X$ with $Y$ in $\CalG$ and
morphisms are morphisms over $X$, that is morphisms $Y\to Y'$ in $\CalG$ which commute with the base
maps to $X$.

For any category $\mathcal E$ we denote by $Iso({\mathcal E})$ the category whose objects are those of $\mathcal E$, but whose morphisms are only the isomorphisms of $\mathcal E$.

For a morphism $\phi$ in a given category, we will use the notation $s(\phi)$ for the source of $\phi$ and $t(\phi)$ for the target of $\phi$.

For a natural number $n$, we define the sets $\bar n=\{1,\dots,n\}$, $[n]=\{0,\dots, n\}$.

Finally, we will consider categories enriched over a monoidal category $\CalE$. Unless otherwise
stated, we assume that the monoidal product of $\CalE$ preserves colimits in each variable and
that the objects of $\CalE$ have underlying elements. The first assumption is often essential while the latter
is for the sake of exposition.

\section{Feynman categories}
In this section, we start by giving the basic definition of a Feynman category. After discussing the axioms, we show that weaker conditions lead to an equivalent notion in the combinatorial case. We then prove several general theorems for Feynman categories, such as monadicity, and give an adjoint pair of functors (push--forward and pull--back) for maps between Feynman categories. We finish the section by discussing related concepts which might be familiar to the initiated reader.

\subsection{ Feynman categories -- the definition}

Fix a  symmetric monoidal category $\F$ and let $\asts$ be a category that is a groupoid, that  is $\asts=Iso(\asts)$. Denote the free symmetric monoidal category on $\V$ by $\V^{\otimes}$.
Furthermore let $\imath\colon\V\to \F$ be a functor and let $\imath^{\otimes}$ be the induced monoidal functor $\imath^{\otimes}\colon \V^{\otimes}\to \F$.

\begin{df}
\label{commadef}
\label{feynmandef}
A triple $\FF=(\V,\F,\imath)$ of objects as above is called a Feynman category (FC)
if

\begin{enumerate}
\renewcommand{\theenumi}{\roman{enumi}}

\item (Isomorphism condition)
\label{objectcond}
The monoidal functor $\imath^{\otimes}$ induces an equivalence of symmetric monoidal categories between $\V^{\otimes}$ and $Iso(\F)$.

\item (Hereditary condition) The monoidal functor $\imath^{\otimes}$ induces an equivalence of symmetric monoidal categories between $(Iso(\F\downarrow \V))^{\otimes}$ and
$Iso(\F\downarrow\F)$.
\label{morcond}

\item (Size condition) For any $\ast\in \asts$, the comma category $(\clusters\downarrow\ast)$ is
essentially small,
viz.\ it is equivalent to a small category.
\end{enumerate}

\end{df}

Here the induced functor from $(Iso(\F\downarrow \V))^{\otimes}$ to $Iso(\F\downarrow\F)$ is given as follows. Given
a collection $\phi_v:X_v\to \imath(\ast_v), v\in I$ of objects of $(\F\downarrow \V)$, the free tensor product $\bigotimes \phi_v$ is sent to the tensor product
$\bigotimes\phi_v:\bigotimes_{v\in I}X_v\to \bigotimes_{v\in I} \imath(\ast_v)=\imath^{\otimes}(\bigotimes_{v\in I}\ast_v)$  in $(\F\downarrow \F) $. The morphisms in $(Iso(\F\downarrow \V))^{\otimes} $ are generated by permutation of the $\phi_v$ and the
collection of isomorphisms $(\sigma,\sigma')$
\begin{equation}
\xymatrix{
X_v\ar[d]_\sigma\ar[r]^{\phi_v}&\imath(\ast_v)\ar[d]^{\imath(\sigma')}\\
X_{v'} \ar[r]_{\psi_{v'}}&\imath(\ast'_{v'})
}
\end{equation}
and these are sent to permutations and to $(\bigotimes_{v\in I} \sigma_v,\bigotimes_{v\in I} \imath(\sigma'_v))$, using the notation introduced in the conventions.

Given a Feynman category $\FF$ we will sometimes write $\V_{\FF}$ and $\F_{\FF}$ for the underlying groupoid and monoidal category and often
take the liberty of dropping the subscripts if we have already fixed $\FF$.

We will say that $\asts$ is a {\em one--comma generating subcategory},  and call its objects stars or vertices.
The objects of the comma category $(\F\downarrow\V)$ will be referred to as one--comma generating morphisms or simply as one--comma generators.
The  objects of $\V^{\otimes}$ or $\imath^{\otimes}(\V^{\otimes})$
 are sometimes called clusters or aggregates (of stars or vertices).

\subsection{(Re)--Construction}
The above axioms can be used to reduce the data needed to give a putative Feynman category or to encode one. In particular, up to equivalence, we only need to know  $\V$ to know the objects and isomorphisms of $\F$ by (i). For the morphisms, (ii) tells us that up to equivalence  the morphisms in $(\F\downarrow \V)$ determine all the morphisms. To construct a Feynman category, one must thus give the data $\V$ and the morphisms from $\V^{\otimes}\to \V$, aka.\ the one--comma generators. Finally, one has to add in their tensor powers and compositions compatibly with the axioms.

Vice-versa, to test if $\F$ gives rise to a Feynman category, one has to find
a $\V$ such that $\V^{\otimes}$ is equivalent to $Iso(\F)$,
and then test whether morphisms to the included $\V$ one--comma generate.
To check (i) one then needs to check if the isomorphisms of $\V$ generate all the isomorphisms up to equivalence. $\V$ also  defines
$(\F\downarrow \V)$ and one can check condition (ii), and (iii).

We will now give examples and details.

\subsubsection{Tautological Feynman category on $\V$}
\label{groupex}
One way to construct a Feynman category is to fix a groupoid $\V$. Then the inclusion of $\imath: \V\to \V^{\otimes}$  already tautologically yields a Feynman category structure by setting $\F=\V^{\otimes}$.
If $\V$ only has one object then it is of the form $\underline{G}$, that is the groupoid corresponding to a group and $\F$ is the free symmetric monoidal category on $\underline{G}$.

By general theory, we know that strong monoidal functors from $\V^{\otimes}$ are the same as ordinary functors from $\V$; $Fun_{\otimes}(\V^{\otimes},\C)\simeq
Fun(\V,\C)$.  In the special case of a groupoid defined by a group, $\V=\underline{G}$, these functors are exactly group representations in $\C$.

\subsubsection{Adding one--comma generating morphisms} After having selected $\V$,
to get to any other Feynman
category structure with this $\V$ up to equivalence, one can now select new morphisms in $(\V^{\otimes}\downarrow \V)$ and generate the sets of morphisms of $\F$ by
taking tensor products of these to satisfy (ii). Then one has to define a composition and check that it is associative. For compositions,
one only needs to define the compositions for diagrams of the type:

\begin{equation}
\label{basicmoreq}
\xymatrix{ X=\bigotimes_{v\in I} X_{v} \ar[rr]^{\otimes_{v\in I}(\phi_v)}\ar@/_{2pc}/[rrr]_{\phi_1} &&Y=\bigotimes_{v\in I}\ast_v \ar[r]^{\ \ \ \ \phi_0}&\ast
}
\end{equation}
with $\ast \in \V$. Usually this amounts to showing that the natural candidate for $\phi_1$ is again of the chosen type of morphism.

Concretely, for instance in the case of underlying graphs, inserting  a graph  of a certain type (e.g.\ a tree) with tails
 into a vertex of another graph of the same type which has the same number of flags as legs of the first graph yields a graph, again of the chosen type. See Appendix A and \S\ref{examplesec} for examples.

Finally, one has to check
 the compatibility with the isomorphisms that are already present. This means that the morphisms
 transform into each other under isomorphisms and in particular they carry an action of the automorphisms.  Often these new morphisms can be defined by using generators and relations. This will be discussed in detail in \S
\ref{genrelpar}.

%MOVE !!
%Another approach is to use a discrete $\V$ and use known (colored) operads for the morphisms -
%together with the isomorphism operations mentioned above. This is done in
%\S \ref{operfeypar}. We should already warn that this theory will then yield algebras over
%these operads, not the operads themselves.

Vice-versa, given a Feynman category, the data of $\imath$ and the one--comma generating morphisms are enough to reconstruct it, by using the procedure above. Moreover, many times one can even identify a smaller set of generators and give generators and relations for the one--comma generators.

\subsection{Surjections: A simple, but not too simple, example $\FF_{surj}$}
\label{surjpar}
The following example is helpful to keep in mind. Let $\V=\underline{\unit}$ be the trivial category,  that is one object $\ast$ with only the identity morphism.
Let $\F=FinSet_{surj}$ be the category of finite sets with surjections and disjoint union $\amalg$ as the monoidal structure.
By realizing that $\ast \amalg \dots \amalg \ast=\bar n=\{1,\dots,n\}$ one sees $\V^{\otimes}\simeq \aleph_0$, the category of finite cardinals with symmetric group actions. This also defines $\imath$, and $\imath^{\otimes}$ is
an equivalence, but not an isomorphism of categories. Nevertheless (i) is satisfied and so is (iii). It remains to check (ii).
 For this, notice that a surjection decomposes into its fibers. Namely, for every surjection $f:S\twoheadrightarrow T$, we get a diagram

$$
\xymatrix{S\ar@{->>}[rr]^f\ar@{<->}[d] && T\ar@{<->}[d]\\
\coprod_{t\in T} f^{-1}(t)\ar[rr]_{\amalg_{t\in T}f_t} && \coprod_{t\in T}
\ast
}
$$
where the vertical arrows are bijections and $f_t\colon f^{-1}(t)\to \ast$ is the unique surjection to a point. These are the one--comma generators.
Here we identified $\{t\}$ with $\ast$ and hence $T=\amalg_{t\in T} \{t\}\leftrightarrow \amalg_{t\in T} \ast$.

This diagram proves (ii), since (a), every morphism is isomorphic in the comma category to a disjoint union of morphisms $S\to \ast$, which are the
one--comma generators, i.e. elements of $(\F\downarrow \V)$. And (b), this decomposition is unique up to unique isomorphism, which necessarily preserves the fibers.
We call the resulting Feynman category $\FF_{surj}=(\underline{\unit}, FinSet_{surj} , \imath)$.

This example shows the concept of one--comma generating morphisms and also has the special features of the monoidal structure, the disjoint union being the restriction of the co-product of the category finite sets and set maps, and the existence of fibers.  These features characterize a less general situation in which a check of diagrams such as above is sufficient, see below.

%?
Also in this example, up to isomorphism the one--comma generators are generated by one morphism for each $n$, namely the surjection $\bar n\twoheadrightarrow \bar 1$.

\subsection{Induced structures}
\label{explainsec}
To give a better understanding of Definition $\ref{feynmandef}$, we wish to make the following remarks which
 address consequences from the different axioms.

\begin{rmk}  (Choice and change of base.) Due to the condition (\ref{objectcond})
for each $X\in \clusters$ there exists an isomorphism
\begin{equation}
\label{objdecompeq}
\phi_X\colon X\stackrel{\sim}{\rightarrow} \otimes_{v\in I} \imath(\ast_v) \text{ with } \ast_v\in \asts
\end{equation}
 for a finite index set $I$. Moreover, fixing a quasi--inverse (i.e.\ a functor $\jmath:Iso(\F)\to \V^\otimes$ which exhibits the equivalence) entails fixing such a decomposition for each $X$. We will call this a choice of basis.
The decomposition (\ref{objdecompeq}) has the following property: For any two such isomorphisms (choices of basis) there is a bijection  of the two index sets $\psi\colon I\to J$ and a diagram

\begin{equation}
\xymatrix
{&\bigotimes_{v\in I} \imath(\ast_v) \ar[dd]^{\simeq \bigotimes \imath(\phi_v)}\\
X\ar[ur]^{\phi_X}_{\simeq}\ar[dr]^{\simeq}_{\phi'_X}&\\
&\bigotimes_{w\in J} \imath(\ast'_w)
}
\end{equation}
where
$\phi_v\in Hom_{\asts}(\ast_v,\ast'_{\psi(v)})$ are isomorphisms.
We call the unambiguously defined value $|I|$ the {\em length of $X$}.

\end{rmk}

\begin{rmk}(Hereditary diagram)
\label{hereditaryrem}
The hereditary condition means that the comma category $(\clusters\downarrow\asts)$ monoidally generates the morphisms in the following way.
% Any morphism $X\to X'$ is
%part of a commutative diagram

Given $\phi:X\to X'$ and a decomposition $X'\simeq \bigotimes_{v\in I} \imath(\ast_v)$, $\ast_v\in \asts$, there are objects $X_v$ of $\F$,
morphisms $\phi_v: X_v\to \imath(\ast_v)$  in $\F$ and an isomorphism $X\simeq \bigotimes_{v\in I} X_v$ in $\F$ such that the diagram \eqref{morphdecompeq} commutes.

\begin{equation}
\label{morphdecompeq}
\xymatrix
{
X \ar[rr]^{\phi}\ar[d]_{\simeq}&& X'\ar[d]^{\simeq} \\
 \bigotimes_{v\in I} X_v\ar[rr]^{\bigotimes_{v\in I}\phi_{v}}&&\bigotimes_{v\in I} \imath(\ast_v)
}
\end{equation}
%where , $X_v\in \clusters$ and $\phi_v\in Hom(X_v,\ast_v)$.

Furthermore, given {\em any} two decompositions of a morphism, (\ref{morcond}) also implies that
 there is a unique isomorphism in $(\imath^{\otimes}\downarrow \imath^{\otimes})$
giving an isomorphism between the two decompositions --- i.e.\ an isomorphism between the lower rows.

The condition of equivalence of comma--categories furthermore  implies that

\begin{enumerate}
\item  Any two such decompositions $\bigotimes_{v\in I} \phi_v$ and $\bigotimes_{v'\in I'}\phi'_{v'}$ are isomorphic via an isomorphism of block type.
This means that
there is a bijection
$\psi:I\to I'$ and isomorphisms
$\sigma_v:X_v\to X'_{\psi(v)}$ , $\sigma':\ast_v\to \ast_{\psi(v)}$, s.t. $\bigotimes \phi'_{v}=
P_{\psi}\circ \bigotimes_v( \sigma'_v\circ \phi_v \circ\sigma_v^{-1})\circ P^{-1}_{\psi}$ where
$P_{\psi}$ is the permutation corresponding to $\psi$.

\item These are the only isomorphisms between morphisms.
\end{enumerate}
The decomposition and the conditions above (namely that the isomorphisms are unique up to unique block decomposition) are equivalent to (ii).

Notice that if the vertical isomorphisms in (\ref{morphdecompeq}) are fixed, then so is the lower morphism.  Hence, a choice of basis also fixes a  particular diagram of  type (\ref{morphdecompeq}).
 Namely, we can further decompose the $X_v$ as $X_v\simeq \bigotimes_{w\in I_v}\imath(\ast_w)$ in (\ref{morphdecompeq}), so that
for $J= \amalg_{v\in I}I_v$: $X\simeq \bigotimes_{w\in J}\ast_w=
\bigotimes_{v\in I} \left(\bigotimes_{w\in I_v}\imath(\ast_w)\right)$. And, we can also assume that this is the decomposition of $X'$ and that of $X$ (up to permutation) are those given by the choice of a base functor $\jmath$. This means there is a diagram

\begin{equation}
\label{morphdecompeq2}
\xymatrix
{
X \ar[rr]^{\phi}\ar[d]_{\simeq}&& X'\ar[d]^{\simeq} \\
 \bigotimes_{v\in I} \left(\bigotimes_{w\in I_v}\imath(\ast_w)\right)
 \ar[rr]^{\bigotimes_{v\in I}\phi_{v}}&&\bigotimes_{v\in I} \imath(\ast_v)
}
\end{equation}
with $\phi_v:\bigotimes_{w\in I_v}\imath(\ast_w)\to \imath(\ast_v)$ and the vertical isomorphisms given by $\jmath$ up to a possible permutation.

It can happen that $|J|=|I|$, but this is usually not the case.

\end{rmk}

\begin{rmk}(Monoidal unit in $\F$ and morphisms)
Notice that the empty monoidal product in $\V^{\otimes}$ is the monoidal unit in $\V^{\otimes}$. Hence after applying $\imath^\otimes$ it becomes $\unit_{\F}$ of $\F$ the monoidal unit of $\F$, which has length $0$. It is therefore possible with the above definition
 to have one--comma generators $\psi$ in $Hom_\F(\unit_{\F},\iota(\ast))$ for $\ast$ in $\V$. In particular, in the decomposition (\ref{morphdecompeq}) there can be several factors of length $0$.
Such a $\psi$
can modify any morphism $\phi\colon X\to Y$ via $X\to X\otimes \unit_{\F} \stackrel{\phi\otimes \psi}{\longrightarrow} Y\otimes \ast$.
This will play an important role in the Feynman category of finite sets and injections and thus with FI--modules.
It will also allow us to ensure the existence of certain morphisms, like the morphisms which represent units or multiplications.

In contrast to this freedom for morphisms with source $\unit_{\F}$,  by  axiom (\ref{morcond}), any morphism with target $\unit_{\F}$,  must also have $\unit_{\F}$ as the source, up to isomorphism.

\end{rmk}
\subsection{Functors as a generalization of operads and $\SS$--modules: $\opcat$ and $\smodcat$}
\begin{df}\label{opdef1}
Let $\CalC$ be a symmetric monoidal category and $\FF=(\V,\F,\imath)$ be a Feynman category.
Consider the category of strong symmetric monoidal functors
$$\fopsc:=Fun_{\otimes}(\clusters,\CalC)$$
which we will call $\F$--$\opers$ in $\CalC$ and a particular object will be called an $\F$-$\oper$ in $\C$. The category
of functors
$$\vmodsc:=Fun(\asts,\CalC)$$ will be
called $\V$-modules in $\CalC$ with particular objects being called a $\V$--mod in $\C$.

If $\CalC$ and $\FF$ are fixed,
we will only write $\opcat$ and $\smodcat$.  We occasionally choose to emphasize the role of the triple $\mathfrak{F}$ by writing $\mathfrak{F}$-$\opcat$ in place of $\mathcal{F}$-$\opcat$.
\end{df}

There is the symmetric monoidal product in $\opcat$ and $\smodcat$ given levelwise by that of $\CalC$. That is $(\O\otimes \O')(X)=\O(X)\otimes \O'(X)$, with the identity being the trivial functor $\unit(X)=\unit_{\CalC}$.

\begin{ex}
In the case of (pseudo)-operads, $\opcat$ is the category
of (pseudo)-operads and $\smodcat$ is the category of $\SS$--modules.
A longer list of classical notions is given in \S\ref{examplesec}.
\end{ex}

Notice that since $\isoclusters$ is equivalent to the free symmetric
monoidal category on $\asts$ we have an equivalence of
categories between $Fun(\asts,\CalC)$ and $Fun_{\otimes}(\isoclusters,\CalC)$. This, or the pull-back along $\imath$, yields a natural forgetful functor $\forget:\opcat\to\smodcat$.

In the following, we will often have to assume that the monoidal product $\otimes$ in $\CalC$ preserves colimits in each variable. This is automatic if $\CalC$ is cocomplete and closed monoidal.
\begin{thm}
\label{freethm}
Let $\CalC$ be a cocomplete symmetric
monoidal category such that $\otimes$ preserves colimits in each variable, then there is a left adjoint (free) functor  $F=\free$
to the forgetful functor $G=\forget$ which is comonoidal.
That is, these functors are adjoint  symmetric (co)monoidal  functors for the symmetric monoidal categories of $\smodcat$ and $\opcat$.
\end{thm}

\begin{proof}(Short proof.) Here we give a brief sketch. More details are given below.
First by the remark above $Fun(\V; \C), Fun_{\otimes}(\V^{\otimes}; \C)$ and
$Fun_{\otimes}(Iso(\F); \C)$ are all equivalent. Since $\C$ is cocomplete, the left Kan extension  along $\imath^{\otimes}$ exists, and we define the image of the functor $F$ via left Kan extension.  (See \cite{MacLane} for the definition of Kan extensions and below for a calculation). Given a  $\Phi \in \smodcat$, we need to check if $F(\Phi)$
 is a strong monoidal functor. It is clear that the extension to $\V^{\otimes}$ is strong monoidal.
 By definition of the Kan extension it is clear that
the functor $F(\Phi)$ is monoidal, since $\otimes$  has been assumed to
preserves colimits in each variable. This structure can be shown to be strong because of the hereditary condition (ii), see below, and so $F(\Phi)$ is strong monoidal. The compatibility of the
commutativity constraints  is easily checked  and hence $F(\Phi)$ is symmetric monoidal.
The adjunction follows immediately. The second statement follows by general theory \cite{Kelly}, since $f^*$ is a right adjoint of $f_*$ and $f^*$ is monoidal, it follows that $f_*$ is comonoidal (aka colax monoidal). The structural morphisms are simply induced by the adjunction. Furthermore by \cite{Kelly} it is also symmetric, thus we obtain an adjoint pair for the symmetric monoidal categories $\vmodsc$ and $\fopsc$.
\end{proof}

\begin{rmk}
In fact this is just one instance of the more general Theorem \ref{pushthm} below. Our proof follows the lines of the proofs for operads, see e.g.\ \cite{MSS} and modular operads \cite{GKmodular}. The theorem also works in the enriched case. For that, we use a different approach  involving more category theory. First we show that one can translate condition (ii) to another more categorical condition (ii'), which allows to pass to the enriched setting; see  Proposition \ref{altprop}. Then
 Theorem \ref{freethm} and the Monadicity Theorem \ref{triplethm} below follow from general category theory as first observed in \cite{getzler}.
\end{rmk}

\subsubsection{Proof details} We would like to provide the details of the arguments above for reference and to expose the inner workings of our concepts.
The expert may want to skip this section.

Given a $\asts$--module $\Phi$, we extend $\Phi$
to all objects of $\clusters$ by picking a functor $\jmath$ which yields the equivalence of
$\V^{\otimes}$ and $Iso(\F)$.  Then, if $\jmath(X)=\bigotimes_{v\in I}\ast_v$, we set
\begin{equation}
\Phi(X):=\bigotimes_{v\in I}\Phi(\ast_v)
\end{equation}

Now, for any $X\in\clusters$ we set
\begin{equation}
\label{kandefeq}
F(\Phi)(X)=\colim_{Iso(\clusters\downarrow X)}\Phi\circ s
\end{equation}
where $s$ is the source map in $\clusters$ from $Hom_{\clusters}\to Obj_{\clusters}$ and on the right hand
side, we mean the underlying object.
For a  given morphism $X\to Y$ in $\F$, we get an induced morphism of the colimits  and it is straightforward that this defines a functor.

We claim that the functor is strong monoidal.  On objects this means that $F(\Phi)(X)\simeq \bigotimes_{v\in I}F(\Phi)(\ast_v)$. This identification then also shows the existence of the colimit due to (iii) and the fact that $\C$ was assumed to be cocomplete.

The argument is basically the following: assembling the maps $\phi_v$ into a map $\phi$ we get a morphism  $\bigotimes_{v\in I}F(\Phi)(\ast_v)\to F(\Phi)(X)$.
On the other hand, we can decompose each $\phi$ into $\bigotimes_{v\in I}\phi_v$ using the condition (\ref{morcond}) and the functor $\jmath$.
As any morphism decomposes as in (\ref{morcond}) up to unique block isomorphism, this gives a morphism going in the other direction $F(\Phi)(X)\to \bigotimes_{v\in I}F(\Phi)(\ast_v)$.
 These maps are then easily checked to be inverses of each other.

Technically this can be seen as follows.
Let $(C(X),\psi^X_{ \phi}:\Phi\circ s(\phi)\to C(X))$ with the $ \phi$  running through the objects of $Iso(\clusters\downarrow X)$ be a collection (cocone) defining a colimit
$F(\Phi(X))$. Likewise let  $(C(\ast_v),\psi^{\ast_v}_{ \phi}:\Phi\circ s(\phi)\to C(\ast_v))$ with the $ \phi_v$  running through the objects of $Iso(\clusters\downarrow *_v)$ be a cocone for the colimit.

Since we assumed that tensor preserves colimits in each variable and $\Phi$ is monoidal, the tensor of the colimits $(\bigotimes_{v \in I} C(\ast_v),\bigotimes_v \psi^{\ast_v}_{ \phi}:\bigotimes_v \Phi\circ s( \phi_v)\to \bigotimes_v C(\ast_v)) $ represents the colimit $\colim_{\times_{v \in I} Iso(\F\downarrow \imath(\ast_v))}\Phi\circ s\simeq \bigotimes_{v \in I} \colim_{Iso(\F\downarrow \imath(*_v))}\Phi\circ s=\bigotimes_{v\in I}F(\Phi)(\imath(\ast_v))$.
Due to condition (\ref{morcond}) the colimit over the indexing category $Iso(\F\downarrow X)$   is equivalent to the colimit over the idexing category $\times_v Iso(\F\downarrow \imath(\ast_v))$
from which we obtain:
\begin{align*}
F(\Phi)(X) & =\colim_{Iso(\clusters\downarrow X)}\Phi\circ s  \\
& \simeq \colim_{\times_{v \in I} Iso(\F\downarrow \imath(\ast_v))}\Phi\circ s  \simeq \bigotimes_{v \in I} \colim_{Iso(\F\downarrow \imath(\ast_v))}\Phi\circ s=\bigotimes_{v\in I}F(\Phi)(\imath(\ast_v))
\end{align*}
Notice that restricting $(C(X),\psi^X_{ \phi}:\Phi\circ s(\phi)\to C(X))$ to $X=\bigotimes_{v\in I} \imath (\ast_v)$, $\phi=\bigotimes_{v\in I}\phi_v$ and block isomorphisms is always a cocone over  $\times_v Iso(\F\downarrow \imath(\ast_v))$ and hence there is always a morphism
$\bigotimes_{v\in I}F(\Phi)(\imath(\ast_v))\to F(\Phi)(\bigotimes_{v\in I}\imath(\ast_v))$.
To get the map the other way around, we need (ii) and in particular that all the isomorphisms are of block type.

To check the monoidal structure on morphisms, we note that for a morphism $\phi\in Hom_{\clusters}(X,Y)$, the map $F(\Phi)(\phi)$ is the map induced
by composing with $\phi$. Namely,
 $(C(Y),\psi^Y_{\phi\circ \tilde \phi})$ is also a co-cone over $Iso(\clusters\downarrow X)$ and hence by the universal property of the colimit we have a map:
$C(X) \to C(Y)$. The monoidal structure on maps is now easily checked using this description.
The symmetric monoidal structure is then again given by the universal properties
of the co-cones and that tensor preserves colimits in each variable.

The adjointness of the functors follows from the following pair of adjunction morphisms.
For  $X\in \clusters$ and $\CO\in \opcat$,
let $\sigma(X)\colon (FG\CO)(X)\to \CO(X)$ be
the morphism induced by the universal property of the co-cone $\CO(C(X)))$ applied to $\CO(X)$, which is made
into a co-cone by mapping for each $\G\in (\F\downarrow X)$
each $\CO(s(\G))$ with $\CO(\G)$ to $\CO(X)$.
For  $\ast\in \asts$ and $\Phi\in \smodcat$,
let $\tau:\Phi(\ast)\to (GF\Phi)(\ast)$ be the morphism given by the inclusion of the identity,
that is the morphism $\psi^{\ast}_{id_{\ast}}$ in the notation above.

It is straightforward to check that these morphisms satisfy the needed conditions using the assumptions. This boils down to the fact that composing or
precomposing with the identity leaves any map invariant.

Finally, it is clear that $G$ is a strict symmetric monoidal functor $\fopsc\to \vmodsc$ for the levelwise tensor product. This means that its left adjoint $F$ is automatically colax. This is a case of doctrinal adjunction \cite{Kelly}.
%It is easy to check that this colax structure is strict for the levelwise monoidal product.

\begin{cor}
\label{triplecor}
The triple (aka monad) generated by $F=\free$ and $G=\forget$ in $\smodcat$ given by
$\T=GF$ has the following explicit description: On objects it is given by
$\T\Phi(\ast_v)=\colim_{(\clusters\downarrow\iota(\ast_v))}\Phi\circ s$
and on morphisms $\phi\in Hom_{\V}(\ast_v,\ast_v)$,
$\T$ acts by composing with $\phi$.

Furthermore the monadic product $\mu:\T\circ \T\to \T$ is given
by composition of morphisms in $\mathcal{F}$.\qed
\end{cor}

\begin{proof}
The map $\mu\colon G(FG)F\to GF$ is defined by the adjoint map above. We will make this explicit and show how the conditions of Definition \ref{commadef} are needed.

Let us compute $FGF(\Phi)(\ast)$.
\begin{equation}
\label{doublecolimeq}
FGF(\Phi)(\ast) = \colim_{Iso(\clusters\downarrow\ast)}( GF\Phi\circ s )
=  \colim_{Iso(\clusters\downarrow\ast)}(\colim_{Iso(\F\downarrow s(\,.\,))}\Phi\circ s)
\end{equation}
where $\colim_{Iso(\F\downarrow s(\,.\,))}\Phi\circ s$ is the functor that sends an object $\phi$ of $(\clusters\downarrow\ast)$ to $\colim_{Iso(\F\downarrow s(\phi))}$ $\Phi\circ s$. Choosing a representation, we see that the double colimit
is indexed by elements $(\phi, \amalg_{v\in I} \phi_v)$,
where $s(\phi)\simeq \bigotimes_{v\in I}\ast_v$
and the $\phi_v$ are morphisms $X_v=s(\phi_v)\to \ast_v$.
Composing the morphisms to $\phi':=(\amalg_v \phi_v)\circ \phi:\bigotimes X_v\to \ast$ yields a co-cone (by commutativity of tensor and colimits)
 and in turn
a map $\mu' \colon FGF(\Phi)(\ast) \to F(\ast)$ such that $\mu = G(\mu')$.

\end{proof}

From this description, we immediately get a monadicity theorem in parallel to the arguments found in e.g.\ \cite{MSS,GKmodular}.
\begin{thm}
\label{triplethm}
There is a equivalence of categories between $\opcat$ and algebras over the triple (aka monad) $\T$.
\qed
\end{thm}

\begin{proof}
Given an element $\O\in\opcat$ we use the $\O(\phi)$ to define maps $\T(\Phi(\ast))\to \Phi(\ast)$. Vice-versa,
for an algebra over $\T$, and for $\phi\in Hom_{\F}(X,\ast)$, we let $\O(\phi)$ be the component corresponding to $\phi$ of the given map
$\T(\Phi(\ast))\to \Phi(\ast)$. The fact that this yields an equivalence in now a straightforward computation using Corollary \ref{triplecor} and in particular \eqref{doublecolimeq}.
\end{proof}

\begin{nota}\label{sumdef}
If $\CalC$ is cocomplete, then for $\O\in \opcat$ we
let $\O^{\oplus}=colim_{\asts}\O\circ \imath$.
\end{nota}

\subsection{Morphisms of Feynman categories}
Morphisms of Feynman categories will explain many of the standard operations, such as
the PROP generated by an operad, the modular envelope, or the operad contained in a PROP.
It also covers the restriction of modular to cyclic, cyclic to non-cyclic or less well known structures such as the PROP or properad generated by a modular operad, see \cite{hoch2}.

\begin{df}
\label{mordef}
A {\em morphism of Feynman} categories $(\V,\F,\imath)$ and $(\V',\F',\imath')$
is basically a pair $(v,f)$ of  a functor $v\colon \V\to \V'$ and  monoidal functor $f\colon \F\to \F'$ which preserves all
the structures, in particular they commute with $\imath$ and $\imath'$, the induced functor
 $v^\otimes:\V^{\otimes}\to \V'^{\otimes}$ is compatible with $f$,
and the decompositions of (\ref{morcond}) are preserved.
\end{df}

In particular this means that in
$$
\xymatrix{
\\
\V \ar[d] _{v}
\ar@(ur,ul)[rr] ^{\imath}
\ar[r]&\V^{\otimes}\ar[r]^{\imath^{\otimes}}\ar[d]_{v^{\otimes}}&\F\ar[d]_{f}\\
\V'\ar[r]
\ar@(dr,dl)[rr] ^{\imath'}
&\V'^{\otimes}\ar[r]^{\imath'^{\otimes}}&\F'\\
}
$$
where $v^{\otimes},\imath^{\otimes}, \imath^{\prime\otimes}$ are defined by the universal property of $\V^{\otimes}$ (respectively $\V'^{\otimes}$),
the right square, and hence the outer square, 2--commute, which means that the two compositions of functors are isomorphic. To be very precise, the datum of a morphism will be a triple $(v,f,N)$ with $N:f\circ \imath^\otimes\simeq \imath'^\otimes \circ v^\otimes$.   In the examples $N$ is usually the identity. Hence, we will tacitly assume such a choice  and abuse notation by only referencing the pair. All compatibilities for $N$ are a straightforward check.
Together with the natural notion of natural transformations of such pairs, Feynman categories form a 2--category.

\subsubsection{Pull--backs and push--forwards}
\label{pushpar}

Given a morphism $(v,f)$
from one Feynman category $\FF=(\V,\F,\imath)$ to another $\FF'=(\V',\F',\imath')$,
there is a pull-back $f^*:\F'$--$\opcat_{\CalC}\to{\F}$--$\opcat_{\CalC}$ given by sending $\CO$ to $\CO\circ f$.
The same is true for $\smodcat$ using $v$. These are compatible by the conditions on morphisms of FCs.

Suppose that the category $\CalC$ is cocomplete.  Then it is possible to construct a left Kan  extension $Lan_f\O$ or push--forward $f_*\O$ along a
functor $f\colon \F\to \F'$. Further assuming that the monoidal product  in $\CalC$ preserves colimits in both variables,
 we will show that this transforms symmetric monoidal functors to symmetric monoidal functors.

The value on $X'\in {\clusters}'$ is then
\begin{equation}
f_*\O(X')= \colim_{(f\downarrow X')}\O\circ P
\end{equation}
where $P$ is the projection $P(Y,\phi:f(Y)\to X)=Y$ \cite{MacLane}.

\begin{thm}\label{amorthm}
\label{pushthm}
$Lan_{f}\O$ is a strong symmetric  monoidal functor, which we denote  as push--forward
 $f_*\O:=Lan_{f}\O$, and hence $f$ defines a functor $f_*:\fopsc\to \F'$-$\opcat_C$. Moreover $f_*,f^*$   form an adjunction of symmetric (co)monoidal functors between the symmetric
 monoidal categories $\fopsc$ and $\F'$-$\opcat_\C$.
\end{thm}
\begin{proof} The proof is parallel to that of Theorem \ref{freethm}. We will use the notation as above. It is clear that the functor $f_*$ exists, since we assumed $\C$ to be cocomplete. We, however, still need to check that the functor $f_*$ is  monoidal, which means that each $f_*\O$ is a monoidal functor, as this is not guaranteed.
Let $X'\simeq\bigotimes_{v\in V} \ast^\prime_v$ then each $\phi:f(Y)\to X$ splits into
$\phi_v\colon f(Y)_v\to \ast'_v$ since $\clusters'$ is a Feynman category.
Decomposing $Y\simeq \bigotimes_{w\in W} \ast_w$ and using the fact that $f$
is a morphism of Feynman categories, we have  $f(Y)\simeq \bigotimes f(\ast_w)\simeq \bigotimes f(Y)_v$. Now
since the $f(\ast_w)$ have length one, we see that there is
an isomorphism $f(Y)_v\simeq \bigotimes_{w\in W_v}f(\ast_w)$ for each $v$ such that $\bigotimes W_v=W$.
Summing up we showed that $f(Y)_v\simeq f(Y_v)$ with $Y_v=\bigotimes_{w\in W_v}\ast_w$. Thus
we decomposed the objects $(Y,\phi)$ in the comma
category  into factors of the form $(Y_v,\phi_v)$.
The morphisms in the comma category factor likewise. Taking
the colimits, we get the desired result by using that the monoidal product respects colimits in each variable and again
using composition of morphisms in one direction and decomposition according to condition (\ref{morcond}) in the other direction.

The symmetric structure is straightforward. For the adjointness, we notice
that $f^*$ is surely a symmetric monoidal functor. Now $f_*$ is by construction adjoint as a functor. This can also be checked explicitly. The morphisms $\eps$ and $\eta$  again are by inclusion of the identity component and composition of morphism.  Like before, writing out the adjunction maps one obtains the induced comonoidal structure. %The fact that this is
\end{proof}

\begin{ex}(Free functor as push--forward)
In this language the free functor can be understood as follows.
The identity $id_{\V}$ together with inclusion functor $inc:\isoclusters\to \clusters$ is a morphism
of Feynman categories  $(\V,\isoclusters,\imath) \stackrel{i=(id_{\V},inc)}{\longrightarrow}(\V,\F,\imath)$.
Noticing that
$\colim_{Iso(\clusters\downarrow *)}\O\circ s= \colim_{(i\downarrow *)}\O\circ P$,
one obtains that the pull-back $i^*$ is the forgetful functor
 $\forget$ and the push-forward $i_*$ is the free functor $\free$.
\end{ex}
\begin{rmk}

One could also consider right Kan extensions, but they are not
always well behaved with respect to the monoidal structure. They do provide an ``extension by zero'' in some cases.
\end{rmk}

\begin{rmk}
\label{shriekrmk}
We will stick with the algebraic geometric notation and intuition for sheaves with $f_*$ denoting the natural push--forward and $f_!$ the extension by zero for special morphisms. Categorically one could argue that the symbols should be switched for left and right Kan extensions, but this may be counter--intuitive for geometers. The categorically minded reader will however know to make this substitution.
\end{rmk}

\begin{ex}
A simple example is given by the following. Let $H\to G$ be a group morphism. This induces a morphism of tautological Feynman categories $(\underline{H},\underline{H}^\otimes,\imath)\to(\underline{G},\underline{G}^\otimes,\imath)$.
The push--forward is then induction and the pull--back restriction of representations.
\end{ex}
\subsection{Other relevant notions}

\subsubsection{$\V$-Sequences}
 For a given Feynman category $\FF$ let $\V_{disc}$ be the discrete version of $\V$; that is the same objects but just identity morphisms.
We will call
$$
Fun(\V_{disc},\CalC)$$
 $\V$-sequences. These will play a role in the chapter \ref{htsec}.

The trivial construction on $\V_{disc}$ yields the Feynman category $(\V_{disc},\V_{disc}^{\otimes},\imath)$ which has a morphism to $\F$ given by the inclusion maps $i=(inc_{\V_{disc}},\imath^{\otimes} \circ inc_{\V_{disc}^\otimes})$. The pull--back is the forgetful functor from $\F$--$\opcat$ down to $\V$-sequences. The push-forward $i_*$ is the free functor, which factors as first making a $\V$-sequence into a free $\V$--module and then into a free
$\F$-$\oper$. Notice that this now has a free $\V$ action. In the example of operads, these are the regular operads in the terminology of Loday.

%more i_*

\subsubsection{Indexed Feynman categories}
Many of the familiar examples involve some sort of graphs with extra structure. To capture
this, we look at indexed Feynman categories. The paradigmatic example for the indexing category will
then be the Feynman category $\GG$ introduced in \S\ref{ggpar}.

\begin{df}
\label{overdef}
Let ${\mathcal B}$ be a Feynman category.
A {\em Feynman category  $\FF$ indexed over  ${\mathcal B}$} is a morphism
of  Feynman categories from $\clusters$ to $\CalB$,
which is surjective on objects. We will write
$\base$ for the underlying functor: $\clusters\to\mathcal B$.
\end{df}

\subsection{Weaker, alternative and Cartesian enriched notions of Feynman categories}
\label{othersec}
There are several modifications one can make on the conditions of a Feynman category. The first deals with the non--symmetric situation. The second is a strictification. The third replaces the hereditary condition by the weak hereditary one. This is usually indeed weaker, but under other categorical
conditions of $\F$, which can be called combinatorial, again equivalent.

Furthermore, the hereditary axiom can be translated into a pre--sheaf condition. In this form one can generalize to enriched Cartesian monoidal categories. To obtain a notion of Feynman categories for non--Cartesian monoidal categories, one has to use this pre--sheaf condition and has to specify that in the isomorphism condition a groupoid means a freely enriched groupoid. This is rather technical and not needed for the examples in \S\ref{examplesec} and the construction in \S\ref{constructionsec}, and so we will return to this in \S\ref{enrichedsec}.

\subsubsection{Non--symmetric Feynman categories}
It is sometimes useful to not use symmetric monoidal categories, but just monoidal categories. In this case $\F$ is assumed to be monoidal and
$\V^\otimes$ means the free monoidal category $\V$. With this adjustment the definition carries over.

This adjustment does change things. For instance, if $\V$ is  the trivial category, then the free monoidal $\V$ is just the discrete category of natural numbers, that is objects $\bar n$ with only the identities as morphisms.

Non--symmetric Feynman categories are the basic objects to construct the Hopf algebras.

\subsubsection{Strict Feynman categories}

Since we will be mainly considering categories of functors from Feynman categories to other monoidal categories, up to equivalence of categories one can assume that $\V$ is a subcategory of $\F$.  We call the process of passing to such an equivalent Feynman category reduction:

\begin{definition}  Let $\FF=(\V,\F,\imath)$ be a Feynman category.  The reduction of $\FF$ is defined to be the triple $\tilde{\FF}=(\imath(\V),\tilde{\F},in)$, where $\tilde{\F}$ is defined to be the full symmetric monoidal subcategory of $\F$ generated by $\imath(\V)$ and where $in$ is the inclusion $\imath(\V)\to\tilde{\F}$. We call a Feynman category strict if it is equal to its reduction.
\end{definition}

\begin{lemma}\label{redlem}  The reduction of a Feynman category is itself a Feynman category, and these two Feynman categories are equivalent.
\end{lemma}
\begin{proof}  In the above notation, $in^\tensor\colon \imath(\V)^\tensor\to Iso(\tilde{\F})$ is simply the identity functor, and is fully faithful since $\imath$ is, and thus is an equivalence.  Moreover, applying the hereditary condition for $\FF$ twice gives the necessary decomposition of morphisms in $\tilde{\FF}$.

To see that $\FF$ and $\tilde{\FF}$ are equivalent we may construct an explicit equivalence.  Let $\pi\colon \tilde{\F}\to\F$ be the inclusion functor.  Let $\jmath\colon Iso(\F)\to \V^{\tensor}$ be a quasi-inverse to $\imath^\tensor$.  Then define a symmetric monoidal functor $\bar{\jmath}\colon \F\to \tilde{\F}$ by taking $X \mapsto \imath^\tensor \jmath(X)$ on objects, and by taking $\phi\colon X\to X^\prime$ to the composite
\begin{equation}
\imath^\tensor \jmath(X)\stackrel{\nu_X}\to X\stackrel{\phi}\to X^\prime \stackrel{\nu^{-1}_{X^\prime}}\to \imath^\tensor \jmath(X^\prime)
\end{equation}
where $\nu, \nu^{-1}$ are natural transformations giving the isomorphism of functors $id_{Iso(\F)}\cong \imath^\tensor\jmath$.  Then one easily sees that $\pi$ and $\bar{\jmath}$ form an equivalence by using the natural transformations $\nu^{-1}$ and $\nu$.
\end{proof}

\subsubsection{Weakly hereditary Feynman categories}
\begin{df} We say that a triple $(\V,\F,\imath)$ satisfies the weak hereditary condition if each morphism $\phi$ of $\F$ admits a hereditary diagram and such a diagram is unique up to isomorphism.  We say that a triple $(\V,\F,\imath)$ is a weakly hereditary Feynman category if it satisfies
(i) and (iii) and the weak hereditary diagram condition.
\end{df}

\begin{rmk} Weakly hereditary Feynman categories where the first definition of Feynman categories \cite{talks}. It turns out however
that this notion is too weak in general and has to be replaced by the stronger condition (ii) in order for the free construction to work properly.
While weaker in general, if we are in a combinatorial situation, that is if $(\F,\otimes)$ has a fully faithful strong symmetric monoidal functor to $(\Set,\amalg)$, the two notions coincide.
\end{rmk}

%\item  There is a version, where we replace symmetric monoidal with just monoidal.
%We will call these non-symmetric Feynman categories.

%\item These is an enriched version; see \S\ref{enrichedsec}.

\begin{lem}
\label{oldcritlem}
If $(\F,\otimes)$ has a  faithful strong symmetric monoidal functor to $(\Set,\amalg)$, then (i) and the weak hereditary condition imply (ii).
\end{lem}
\begin{proof}
Given two decompositions $\otimes_{v\in I} \phi_v: \bigotimes_{v\in I}X_v\to \bigotimes_{v\in I}\imath(\ast_v)$
and $\otimes_{v\in I} \phi'_v: \bigotimes_{v\in I}X'_v\to \bigotimes_{v\in I}\imath(\ast_v)$
completing (\ref{morphdecompeq}), we have that after passing to $\Set$: $\phi_v^{-1}(\ast_v)=X_v$ and
$\phi_v^{\prime -1}(\ast_v)=X'_v$.
Then we can decompose further $X_v\simeq\bigotimes_{w_v \in I_v}\imath(\ast_{w_v})$
and $X_v\simeq\bigotimes_{w'_{v'} \in I'_v}\imath(\ast_{w_{v'}})$. By (i) the isomorphism between the full decomposition is a permutation followed by an isomorphism of the individual $\ast_{v
_w}$. This means in particular that under the isomorphism given by the decompositions:
$\bigotimes_{w_v\in I_v}\imath(\ast_{w_v})\simeq \bigotimes_{w'_{v'}\in I_{v'}}\imath(\ast_{w'_{v'}})$
since both are the pre--images of $\imath(\ast_v)$. It also follows that all isomorphisms are of this type.
\end{proof}

This result implies:
\begin{prop} If $(\F,\otimes)$ has a  faithful strong symmetric monoidal functor to $(\Set,\amalg)$ and $\FF$ is a weakly hereditary Feynman category then it is a Feynman category.
\qed
\end{prop}

\begin{rmk}
Instead of asking for the  faithful functor to $(\Set,\amalg)$, we can ask that the category has fiber products. It is then clear that isomorphisms can only be fiberwise, the fibers over $\imath(\ast_v)$ being the $X_v$, and hence isomorphisms between two diagrams are unique block isomorphisms.
\end{rmk}

\subsubsection{(Non)--examples}
\label{nonexamplesec}
The following examples show
 that in general the existence of the hereditary diagrams \eqref{morphdecompeq} is weaker than the hereditary condition (ii) and moreover that weakly hereditary is indeed weaker than hereditary. Their background is from spin chains, but this can be ignored for the present purposes.

As a first example consider $\V=\mathbb{C}^2$  with just the identity morphism and let the objects of $\F$ be $(\mathbb{C}^2)^{\otimes n}\simeq \mathbb{C}^{2n}$ considered in $\mathbb{C}$ vector spaces. Let $e_1, e_2$ be the two standard basis elements of $\mathbb{C}^2$.
We add the one--comma generators $\phi_n:(\mathbb{C}^2)^{\otimes n}\to \mathbb{C}^2$ which are defined by
$$\phi_n(e_{i_1}\odo e_{i_n})=\begin{cases}
e_1&\text{ if } e_{i_1}\odo e_{i_n}=e_1\odo e_1\\
0&else
\end{cases}
$$
It is  easily checked that these morphisms   compose correctly: $\phi_k\circ (\phi_{n_1}\odo \phi_{n_k})=\phi_{n_1+\dots+ n_k}$. Now, diagrams of the type \eqref{morphdecompeq} exist. But since $\phi_2\otimes \phi_2=\phi_3\otimes \phi_1\in Hom((\mathbb{C}^2)^{\otimes 4},(\mathbb{C}^2)^{\otimes 2})$,
we see that there is no isomorphism of block type between the two decompositions. The decompositions are however isomorphic and this isomorphism is unique up to general permutations.

% such diagrams are not unique up to unique isomorphism. Furthermore,

As a second example consider the  same $\V$, but now with the one--comma generators

$$\psi_n(e_{i_1}\odo e_{i_n})=\begin{cases}
e_1&\text{ if } e_{i_1}\odo e_{i_n}=e_1\otimes e_{i_2}\odo e_{i_n} \\
e_2&\text{ if } e_{i_1}\odo e_{i_n}=e_2\otimes e_{i_2}\odo e_{i_n}\\
\end{cases}
$$
closing the morphisms by tensor and decomposition, we obtain a weak hereditary Feynman category that is not strong.
Indeed one can check that any morphism is a tensor product of $\psi_n$ and units. On the standard basis elements the $\psi_n$ reads out the left basis element of the given block, so that indeed all the morphisms given by different tensor products are different and hence two decomposition only differ by a block isomorphism permuting the respective blocks. However, there are extra isomorphisms in the comma category which are not of block type, and thus although the weak hereditary condition is satisfied, $\F$ is not a Feynman category. One such example is:

$$
\xymatrix{
(\mathbb{C}^2)^{\otimes 4}\ar[r]^{\phi_2\otimes\phi_2}\ar[d]_{\sigma_{13}}&(\mathbb{C}^2)^{\otimes 2}\ar[d]^{\sigma_{12}}\\
(\mathbb{C}^2)^{\otimes 4}\ar[r]^{\phi_2\otimes\phi_2}&(\mathbb{C}^2)^{\otimes 2}
}
$$
where $\sigma_{ij} \in Iso(\F)$ permutes the $i$th and $j$th factors.
%
%There are two natural choices of isomorphisms for $\F$. Either (a) all the isomorphisms from the symmetric structure of vector spaces, that is all permutations of the $2n$ factors of   $(\mathbb{C}^2)^{\otimes n}\simeq \mathcal{C}^{2n}$ or (b) only  the permutations of the block factors in
% $(\mathbb{C}^2)^{\otimes n}\simeq \mathcal{C}^{2n}$

\subsubsection{An alternative formulation of condition (\ref{morcond}) }
There is another more high-brow way to phrase the condition on morphisms.
We claim that the condition (\ref{morcond}) can equivalently be stated as
\begin{equation}
\label{dayeq}
\imath^{\otimes \wedge}Hom_{\F}(\,\cdot\, , X\otimes Y):=
Hom_{\F}(\imath^{\otimes}\, \cdot\, ,X\otimes Y)=\imath^{\otimes\wedge}Hom_{\F}( \,\cdot\,, X)\day \imath^{\otimes\wedge}  Hom_{\F}(\,\cdot\, , Y)
\end{equation}
where $\day$ is the Day convolution of presheaves on $\V^{\otimes}$, that is functors in $[\V^{\otimes op},\Set]$, and $\imath^{\otimes \wedge}$ is the pull-back induced by $\imath^{\otimes}$.
The Yoneda embedding for $\F$ is a strong monoidal functor, i.e.~ $Hom_{\F}(\,\cdot\, , X\otimes Y)=Hom_{\F}(\,\cdot\, , X)\day
Hom_{\F}(\,\cdot\, , Y)$, where now the convolution is of representable pre--sheaves on $\F$.
Thus (\ref{dayeq}) can be cast into
\begin{itemize}

\item[(ii')] The pull-back of presheaves $\imath^{\otimes \wedge}\colon [\F^{op},Set]\to [\V^{\otimes op},Set]$ {\em restricted to representable pre-sheaves} is  monoidal.
\end{itemize}

In order to compare the two definitions, let us establish the following.
\begin{lem}(Iterative decompositions)
\label{iterativelem}

Axiom (ii) implies that for any morphism $\phi\in Hom_\F(W,X\otimes Y)$ there exists a pair $(Z,Z')$ of objects in $\F$ such
that
\begin{equation}
\xymatrix{W\ar[rr]^\phi\ar[dr]_\simeq&&X\otimes Y\\
&Z\otimes Z'\ar[ur]_{\phi_1\otimes \phi_2}&\\
}
\end{equation}
and this decomposition is unique up to unique block isomorphism.
Moreover this condition is equivalent to (ii).
\end{lem}
\begin{proof}
First notice that we may decompose $X=\bigotimes_{v\in I_X}*_v$ and $Y=\bigotimes_{v\in I_Y}*_v$. Then $\bigotimes_{v\in I_X\amalg I_Y}*_v$
 gives a lower right corner for \eqref{morphdecompeq}. Completing the diagram and setting $Z=\bigotimes_{v\in I_X}W_v$ and  $Z'=\bigotimes_{v\in I_Y}W_v$ gives the diagram. This is unique up to unique block isomorphism by (ii).

In order to show that it is equivalent to (ii), we have to prove the existence and uniqueness of diagrams according to Remark \ref{hereditaryrem}.
We can apply the construction iteratively until the length of the tensor components are all $1$, that means that they are isomorphic to some $\imath(*_v)$
and we arrive at a diagram of the type \eqref{morphdecompeq}.
At each stage the decomposition is unique up to unique block isomorphism, which proves the claim.
\end{proof}

\begin{prop}\label{altprop}

$\FF$ is a Feynman category if and only if (i), (ii') and (iii) hold.
\end{prop}

\begin{proof}
For any choice of $X$ and $Y$ the condition (ii') is equivalent to (\ref{dayeq}).
Using the definition of the Day convolution the right hand side of (\ref{dayeq})  becomes
\begin{multline}
\imath^{\otimes\wedge}Hom_{\F}( \,\cdot\,, X)\day\imath^{\otimes\wedge}
Hom_{\F}( \,\cdot\, , Y)
 =Hom_{\F}(\imath^{\otimes} \,\cdot\,, X)\day Hom_{\F}(\imath^{\otimes} \,\cdot\, , Y)\\
=\int^{Z,Z'}Hom_{\F}(\imath^{\otimes} Z, X)\times
Hom_{\F}(\imath^{\otimes} Z' , Y)\times Hom_{\V^{\otimes}}( \,\cdot\,,Z\otimes Z')
\end{multline}
Now since $\V^{\otimes}$ was a groupoid, deciphering the co-end,
 there are contributions for each decomposition of the argument which decompose any chosen  morphism into two
 pieces. The whole situation is precisely
 equivariant with respect to pairs of  isomorphisms and, due to the definition of a  co-end,  changing the isomorphisms gives the same morphism.
Fix $\jmath\colon Iso(\F)\to \V^{\otimes}$ as before.
Since any $X$  can be decomposed completely into $X\simeq \bigotimes_v{\imath( \ast_v)}$, we can
apply (\ref{dayeq}) iteratively to arrive exactly at (\ref{morcond}) by using Lemma \ref{iterativelem}.
\end{proof}

\begin{rmk} There is a  similarity between the iterative decomposition and one of the conditions for a rigid tensor category \cite{TannakaDel} in the formulation of the 2012 version of \cite{DelMilne}
\begin{equation}
\label{rigideq}
Hom(Z',X)\otimes Hom(Z'',Y')\simeq Hom(Z'\otimes Z'',X\otimes Y)
\end{equation}
The difference is that although in a Feynman category any morphism $\phi$ from $Z=Z'\otimes Z''$ to $X\otimes Y$ will factor through some $W'\otimes W''$, it is not guaranteed that this decomposition is block isomorphic to the given one. This means that neither $Z'\simeq W'$ and $Z''\simeq W''$ nor $Z''\simeq W''$ nor $Z'\simeq W''$ and $Z'\simeq W'$ nor $Z''\simeq W''$ must hold and if this is not the case, $\phi$ is not of the form needed for the isomorphism \eqref{rigideq}. However, due to the presence of natural Hopf algebras in both theories, a possible connection merits further study.
\end{rmk}
\subsubsection{Feynman categories in the Cartesian enriched setting}
In this section, $\mathcal E$ is a symmetric monoidal category that is Cartesian. In this case, the condition of being a groupoid still makes sense, but the hereditary condition should be replaced by (ii'). Likewise, the condition (iii) has to be rephrased in terms of indexing functors.

\begin{df} A Feynman category $\FF$ enriched
over $\mathcal E$ is a triple $(\asts,\clusters,\imath)$ of a category $\clusters$ enriched
over $\mathcal E$ and an enriched category $\asts$ which satisfy the enriched
version of the axioms of Definition \ref{commadef}. That is (i), (ii') as given above and
\begin{equation*}
\text{(iii') For all $*\in \V$, the  indexing functors }\tilde \imath^{\otimes}(*):=Hom_{\F}(\imath^{\otimes} *, -) \text{ are essentially small.}
\end{equation*}
\end{df}

%Notice that in particular, $\V$ being a groupoid means that the $Hom_{\V}(*_v,*_w)$ are $\otimes$--invertible in $\CalE$.

\begin{rmk} It is easy to check that if $\CalE$ is $(\Set,\times)$ then the two definitions coincide.
%We also remark that (iii') implies (iii).
\end{rmk}
Fixing a target category $\CalC$, which is also enriched over $\mathcal E$,
we define the categories $\opcat$ and $\smodcat$ as before, but insisting
that the functors are functors of enriched categories.

This definition allows one to enrich over $\Top$ or simplicial sets.

\subsection{Weak Feynman categories and indexed enriched Feynman categories}
To enrich over the standard Abelian categories, like $\Ab$ or $\Vect$ or any other category that is not Cartesian, there are two choices. First one can simply weaken the conditions.
Secondly, one can use an ordinary Feynman category as an indexing system, which we develop in \S\ref{enrichedsec}.
This is guided by the free construction and is more important for our current purposes.
It will allow us to treat twisted modular operads or more generally twisted Feynman categories and give a second approach to algebras.
Here we will just give the basic idea with the details in \S\ref{enrichedsec}.

\begin{df}
\label{weakdef}
A weak Feynman category is a triple $({\mathcal W},\F,\imath)$, both ${\mathcal W}$ and $\F$ enriched over $\CalE$ and $\mathcal V$ symmetric monoidal tensored over $\CalE$ satisfying: (i') $\imath^{\otimes}$ is essentially surjective, and (ii') and (iii') as above.
\end{df}

Notice, we dropped the condition on $\mathcal V$ that it is a groupoid, see \S\ref{Getzlerpar} for a closely related notion due to Getzler and further discussion.

\begin{rmk}
In the non-enriched case, given a Feynman category $(\V,\F,\imath)$, we can take any
groupoid $\V'$ equivalent to a subgroupoid  between $\V_{disc}\hookrightarrow \bar\V' \hookrightarrow \V$ to form a weak Feynman category $(\V',\F,\imath)$.
Here the maximal and minimal $\V'$ are determined by $Iso(\F)$ and its objects.
\end{rmk}
\begin{rmk}
For $\Ab$ or $\Vect$ (ii') translates to the hereditary condition that each $\phi$ can be decomposed into summands. In general, proceeding as in Proposition \ref{altprop} the analogue of (ii) in the enriched case is that for each choice of base
for $X\simeq\bigotimes_{w\in W}\imath(\ast_w)$ and $Y=\bigotimes_{v \in V} \imath(\ast'_v)$
there is an induced isomorphism
$$Hom_{\F}(X,Y)\simeq \bigoplus_{ (W_v)_{v\in V}:\amalg_v W_v=W}
\bigotimes_{v\in V}  Hom_{\F}(\bigotimes_{w\in W_v}\imath(\ast_w),\imath (\ast'_v))$$
where the coproduct is over all partitions and the isomorphism is given by the monoidal structure, viz.\ the product and
commutativity/associativity constraints.
\end{rmk}

\subsubsection{Freely enriched Feynman categories}
\label{freeenrichedpar}
We recall from \cite{kellybook} that if $\CalE$ is a symmetric monoidal closed category with the underlying $\CalE_0$ being locally small, complete and cocomplete, that there is a left adjoint functor
$( - )_{\CalE}$ to the underlying category functor $( - )_0$.
We will assume these conditions on $\CalE$ from now on.

 A freely enriched Feynman category is then a triple
 $\FF_{\CalE}:=(\V_{\CalE},\F_{\CalE},\imath_{\CalE})$ where $\FF=(\V,\F,\imath)$ is a Feynman category.
This gives the basic idea for indexed enriched Feynman categories.

\subsubsection{Indexed enrichment}
At this point, we will just write the definition for reference and explain it when more of the theory has been developed.
We will need the definition of an enrichment functor as defined in Definition \ref{enrichmentfunctordef}.
Given a monoidal category $\F$ considered as a 2--category $\underline{ \F}$ and lax 2--functor $\D$ to $\underline{\CalE}$,  which is an enrichment functor, we define an enriched monoidal category $\F_{\D}$ as follows. The objects of $\F_{\D}$ are those of $\F$. The morphisms are given by

\begin{equation*}
Hom_{\F_{\CalD}}(X,Y):=\bigoplus_{\phi\in Hom_{\F}(X,Y)}\D(\phi), \text{ where } \D(\phi)\in Obj(\CalE)
\end{equation*}
and composition defined via the functor $\D$.
\begin{rmk}
We obtain the free enriched case for $\D(\phi)\equiv \unit$.
\end{rmk}
\begin{rmk}
As we shall show, the enrichment functors alternatively can be thought of as $\FF^{hyp}$-$\opcat_{\CalE}$ for the hyper--Feynman
category of a given Feynman category defined in \S\ref{hypersec}.
\end{rmk}
%The composition is given by
%\begin{multline}
%Hom_{\F_{\CalD}}(X,Y)\otimes Hom_{\F_{\CalD}}(Y,Z) =
%\bigoplus_{\phi\in Hom_{\F}(X,Y)}\D(\phi)\otimes \bigoplus_{\psi\in Hom_{\F}(Y,Z)} \D(\psi)\\
%\simeq
%\bigoplus_{(\phi,\psi)\in Hom_{\F}(X,Y)\times Hom_{\F}(Y,Z)} \D(\phi)\otimes \D(\psi)
%\stackrel{\bigoplus \D(\circ)}{\longrightarrow} \bigoplus_{\chi\in Hom_{\F}(X,Z)}\D(\chi)=Hom_{\F_{\D}}(X,Z)
%\end{multline}
%The image lies in the components $\chi=\phi\circ\psi$. Using this construction on $\V$, pulling back $\D$ via $\imath$, we obtain $\V_{\D}=\V_{\CalE}$. The functor $\imath$ then is naturally upgraded to an enriched functor $\imath_{\CalE}:\V_{\D}\to \F_{\D}$.

%
%For these two applications mentioned above, we will
%enrich the  category $\clusters$ over
%either the category of Abelian groups
%$\Ab$ ---viz.\ consider Ab--categories---, the category $\Vect_k$ of vector spaces over a field $k$, its graded version
%$\Vect_k^{\Z}$ or  consider it enriched over  the same category
%$\CalC$ in which we wish to
%consider $\opcat$ in.

The definition given in Chapter \S\ref{enrichedsec} is:
\begin{customdf}{\ref{enrichedfeydef}}
Let $\FF$ be a Feynman category and let $\D$ be an enrichment functor.
We call $\FF_{\D}:=(\V_{\CalE},\F_{\D},\imath_{\CalE})$ a
Feynman category enriched
over $\CalE$ indexed by $\D$.

\end{customdf}

 \subsection{Connection to Feynman graphs and physics}
The intuition for Feynman categories
comes from quantum field theory and Feynman graphs, whence the name.
The physical terminology is meant as follows.
Feynman graphs for a quantum field theory contain two pieces of vital information.
First there are fixed vertex types encoding the possible interactions (this role is played by $\asts$)
and secondly these vertices are connected by edges of the graph, which are the propagators or field lines. There is one type of edge for
each particle and the vertices have flags according to the interactions set forth in the action.

The morphisms in the category $\F$ are analogous to an $S$ matrix.  More precisely,
the $S$ matrix with a given external structure $\ast$ is encoded by the slice category $(\F\downarrow \ast)$.
 That is, we fix the external legs as
the target of a morphism and then the possible morphisms with this target are all possible graphs that we can put inside
this effective scattering vertex. Each such graph gives rise to a morphism, where the source is
given by specifying the collection of vertices and the morphism is nailed down by putting in the propagation lines.  This also explains the use of the symbol $\ast$ for the objects of $\V$.

We wish to point out two things though: First, the definition of a Feynman category itself {\em does not need any underlying graphs}.
And secondly, if graphs are present, they are the {\em morphisms} not the objects.  FCs using graphs are formalized below
as FCs indexed over a category of graphs see Definition  \ref{overdef} and section \S\ref{graphsec}.

\subsection{Discussion and relation to other structures}
\label{otherpar}
There are several related notions that have appeared independently in the literature. In this section, we discuss them
as they relate to the notion of Feynman categories.
  %These relations have prompted recent in depth studies by  \cite{BKW} and \cite{Caviglia}.

 One can distinguish two different setups by their presentation of the classical examples like operads.
 In the first, the graphs or composition/pasting schemes are objects of a category, such as in
 the Borisov--Manin setup or the classical colored operad setup as founded by May.
  In the second setup, which is where Feynman categories
 live, the graphs are actually part of morphisms. After a little sleuthing, with the help of C.~Berger, one can find out that this approach
 actually predates the first one and traces back to Boardman--Vogt.

 Other different approaches that have been studied concurrently to ours are polynomial monads \cite{monads} and operadic categories \cite{BatMar}.

 \subsubsection{Boardman--Vogt Categories of operators in standard form}
 The first appearance of treating graphs  or compositions as morphisms can be found in \cite{BV}, under
 the notion of Categories of operators in standard form. Here however, we are only in $\Top$ and the category $\V$ is trivial.

\subsubsection{Multi-categories and colored operads}
\label{coloredpar}
Several people have remarked a certain connection to colored operads or multicategories
These indeed exist. First,
our framework  subsumes the notions of operads and colored operads as specific examples, see \S \ref{classexsec}

Secondly, for those already familiar with colored operads it is  useful to note the following. Consider a discrete
$\V$, then as a subcategory of $\F$, $\V$ is the set of colors of a colored
operad defined by $\F$. For this, set $\O(\ast_1,\dots,\ast_n;\ast_m):=\F(\ast_1\otimes \dots\otimes \ast_n,\ast_m)$
and let the composition of $\O$ be given by $\circ$ and the monoidal structure in $\F$.
This is of course the same way to obtain a multicategory from a monoidal category. What the axiom (\ref{morcond}) is saying is that these ``multicategory morphisms''
are ``all the morphisms''. Of course, we will only get back things up to equivalence.

In this light, our setup generalizes the observation of \cite{KS,BergerMoerdijk} that  a non-symmetric operad
is an algebra over a certain colored operad in several ways.
First, we now look at all possible generalized operadic gadgets --symmetric or not--  that can be formed this way and secondly,
 we include a precise axiomatization of the possible isomorphisms of the colors needed
 e.g. in the symmetric case, that is we enlarge the set of colors to a groupoid of basic colors with the
 all colors given by words in the basic colors.

The main conceptual difference is  that in May's operad theory the $\O(n)$ are thought of as objects,
while in our theory they are morphisms, see \S\ref{operfeypar}.
Another conceptual difference is that we always look at something like the PROP generated
by an operad, which does have underlying objects.
This makes our considerations closer to categories of
operators in standard form of Boardman--Vogt \cite{BoVo} as was remarked by C.~Berger.

The relation to colored PROPs may also make this clearer.
Colored PROPs are $\F$--$\opcat$ for a specific Feynman category, they do not form one.
This is important, since in this view our results on homotopy theory and model structures {\it et cetera} for $\F$--$\opcat$ in general
will pertain equally to colored PROPs and colored operads {\it et cetera}.
The precise relationship between these concepts has been under study in \cite{BKW,Caviglia}.

Procedurally or ontologically, note that if we want to avoid circularity in the definition of concepts,
we {\em cannot define} a Feynman
category using colored operads, as these are secondary and
are defined as functors from a particular Feynman category, see \S \ref{classexsec}.

\subsubsection{Manin-Borisov version}
As we will discuss below, graphs are an important source of examples. Graphs and their morphisms form a category which
was nicely analyzed and presented in \cite{BM}. We will review this below and use this category. Whereas in their generalization of operads
Borisov and Manin used graphs as objects, we will take graphs as {\em morphisms}. More precisely the graphs will be a part of the data of a morphism in the full subcategory whose objects are collections or aggregates of corollas. Recall that a corolla is a graph with one vertex and only flags, i.e.\ no edges or loops.
Thus one could say that we go up one categorical level.

\subsubsection{Getzler's version: patterns}
\label{Getzlerpar}
As we realized after having given our construction, there is a related
   notion of operads starting from patterns \cite{getzler}. Although the connection was not obvious \cite{Getzlernote}, these are related through the condition (ii') and our Proposition \ref{altprop}.
 Part of the condition (i) appears in this context,
as the notion of regularity, namely that the functor $\imath^{\otimes}$ is essentially surjective.
The groupoid condition is not included. A regular pattern would  also require the stronger condition (iv)
that $\imath^{\otimes \wedge}$ is monoidal on all presheaves, which in many cases is not too restrictive.
 This depends on the size of the category. They are equivalent if $\F$ is small.

Using our Proposition \ref{altprop}, we can see that up to representability of presheaves, our notion of weak Feynman categories
coincides with that of a regular pattern. A Feynman category then has an extra condition, namely (i).

So in one sense a pattern is more general as one may omit (i) and in another sense it is more restrictive since (iv)
is stronger than (ii) and might be harder to check, if one is not in the case that the category is suitably small, see  e.g. \cite{kellybook}.
Starting from our work, the exact relationship up to equivalences and strictification has been recently analyzed in detail in the subsequent work \cite{BKW} and also in \cite{Caviglia}.

A big advantage is that our axiom (ii) is very easily checked in all the examples, for instance using hereditary diagrams which go back to Markl \cite{Markl}.

\section{Examples}
\label{examplesec}
In this section, we give examples of Feynman categories.  We treat the usual notions of operad-like structures in our framework
as well as new examples.

We start
with those indexed over a basic category $\GG$ and is variations.
 These yield all the usual suspects like operads, colored
operads, PROPs, properads, cyclic operads, modular operads, etc..
Furthermore, many constructions like the modular envelope or the PROP generated
by an operad can be understood via push-forward and pull-back. This allows us to also study not-so-classical notions and their relations, such as ungraded, nc, and non-$\Sigma$ variants of modular operads which are for instance helpful in string topology constructions, see e.g. \ \cite{hoch1}.
We will start this section by discussing the basic Feynman category $\GG$ of aggregates of corollas which is a full subcategory of the Borisov--Manin
category of graphs. Indeed all classical examples can be obtained via  decorations, in the technical sense as defined in \S\ref{decopar},
 and restrictions from this basic example.

The next set of examples we treat are actually of a simpler kind, namely those that have trivial $\V$ or variations thereof.
This includes the basic example of \S\ref{surjpar} as well as injections, crossed simplicial groups and the augmented simplicial category.  Among these examples are also Feynman categories whose $\F$-$\opcat$ are algebras over (colored) operads.
Thus, in our theory operads and algebras over them are on the same footing.  In particular, we get free algebras this way.
To fully grasp these examples however, one needs the theory of enrichments \S\ref{enrichedsec} and the construction of $\FF^{hyp}$ given in \S\ref{hypersec}.

\begin{nota}
In the following if $\V$ is actually a subcategory of $\F$, we just write $\imath$ for the inclusion without further ado.
Also, since it often happens that the categories we regard are subcategories which contain {\em all objects},
we recall that the standard nomenclature for these subcategories is {\em wide} subcategories.
\end{nota}

All the Feynman categories appearing here are combinatorial, i.e.\ have a fully faithful functor to $\Set,\amalg$ and hence Lemma \ref{oldcritlem} applies.

The condition (iii) will not be explicitly checked as it is obvious in all the cases we discuss here and later on.

\subsection{The Feynman category  $\GG=(\Crl,\Agg,\imath)$ and categories indexed over it}
\label{ggpar}

\begin{df} Let $\Graphs$ be the category whose objects are abstract graphs and
whose morphisms are the morphisms described in Appendix A.
We consider it to be a monoidal category with monoidal product $\amalg$ (see
Remark \ref{disjointrmk}).

We let $\Agg$ be the full subcategory
of disjoint unions (aggregates) of corollas.
And we let $\Crl$ be the groupoid subcategory whose objects are corollas and
whose morphisms are the isomorphisms between them.
We denote the inclusion of $\Crl$ into $\Agg$ by $\imath$.
\end{df}

The main result which allows one to realize ``graph based operad like things''  as functors from Feynman categories is the following:

\begin{prop}
\label{aggprop}
The triple $\GG=(\Crl,\Agg,\imath)$ is a Feynman category.
\end{prop}
\begin{proof}
By construction $Iso(\Agg)$ is equivalent to $\Crl^{\otimes}$ where the morphisms $\imath^{\otimes}$ map the abstract $\otimes$  to $\amalg$.
The condition (i) is obviously
satisfied on objects. For the equivalence, we first
 check the condition (\ref{morcond}).
A morphism from one aggregate of corollas $X=\amalg_{w\in W}\ast_w$
to another $Y=\amalg_{v\in V}\ast_v$ is given by
$\phi=(\phi^F,\phi_V,i_{\phi})$. We let
$\phi_v:\amalg_{w\in {\phi}^{-1}(v)}*_w\to \ast_v$ be given by the restrictions (see Appendix A).
This yields the needed decomposition.  Any isomorphism thus also decomposes in this fashion, and
in this case we see that the decomposition yields a decomposition into isomorphisms of corollas.  Hence
$\imath^{\otimes}$ induces an equivalence of $\Crl^{\otimes}$ with $Iso(\Agg)$.
\end{proof}

As is apparent by the use of only aggregates of corollas as objects, the graph structure is not part of the objects, it is part of the morphisms.  More precisely they are the ghost graph of a morphism. Here, we wish to point out that a morphisms is not completely determined by its ghost graph.
It is only indexed by it. We will now give the details.

\subsubsection{Morphisms in $\Agg$ and graphs: a secondary structure}
\label{grsecondsec}
As many users are used to graphs appearing in operadic contexts, we wish to make their role clear in the current
setup. The objects in $\Agg$ are by definition just aggregates of corollas, so they are very simple graphs. Notice we do not take any other graphs as objects.
All the structure is in the morphisms and these  have underlying graphs, the ghost graphs defined in \S\ref{ghostpar}.
These are, up to {\em important details},
the graphs readers might  expect to appear. Since this is a potential source of confusion, we review the details here carefully.

\begin{enumerate}
\item By definition, a morphism $\phi$ in $\Agg$ from  an aggregate $X=\amalg_{v\in V}\ast_v$ to another aggregate $Y=\amalg_w\ast_w$
gives rise to a (non--connected) graph $\gh(\phi)$, viz.\ its ghost graph (see Appendix A). This is
the graph $\gh(\phi)=(V_X,F_X,\hat \imath_{\phi},\del_X)$ where $\hat \imath_{\phi}$ is
the  extension of $i_{\phi}$ to all flags of $X$ by being the identity
on all flags in the image of $\phi^F$.
This is the graph obtained from $X$ by using orbits of
$i_{\phi}$ as edges.

The ``extreme'' cases occur when either $\gh(\phi)$ is connected, in which case
 the morphism is a called a {\em contraction}, or when $\gh(\phi)$ has no edges, in which
case the morphism is called a merger. Here we follow the terminology of \cite{BM}.
For a merger, $\phi^F$ is a bijection and $i_{\phi}$ is hence
the unique map from $\emptyset$ to itself.
Following the proof of Proposition \ref{aggprop}  shows that in $\Agg$ the associated graph $\gh(\phi)$  is the disjoint union of
the $\gh_v(\phi):=\gh(\phi_v)$, we call these the associated graphs.

Given the source  and target of a morphism and a $\gh(\phi)$ however,
{\it does not} (even up to isomorphism) determine the morphism $\phi$.
One crucial piece of missing information is the map $\phi_V$ which tells which components
need to be merged. That is, we need to know which  components of $\gh(\phi_v)$ correspond to the map $\phi_v$.

More precisely, to reconstruct the map
$\phi$ we need to have
\begin{itemize}
\item[(a)] the graph $\gh(\phi)$

\item[ (b)] the decomposition
$\gh(\phi)=\amalg_{v \in V_Y} \gh_v(\phi)$  (notice the $\gh_v$ do not have to be connected),  and

\item[(c)] the isomorphism $\phi^F$ restricted to its image.
%\item[(d)] The source and target maps $\imathi_{\Gamma)$
\end{itemize}
The data (b) then is equivalent to  the map $\phi_V$,  which is the map sending all vertices in $\phi_v$ to $v$. The map $i_{\phi}$ is just the restriction of $\hat i_{\gh(\phi)}$ to its elements which have 2 element orbits.
If we only have (a) and (b), then the data fixes a map up to the isomorphism $\phi^F$ restricted onto its image, which is the same as a morphism up to isomorphism on the target and source.
Fixing (a) only we do not even get a class up to isomorphism, since the morphism $\phi_V$ is unrecoverable.

For a general graph morphism $\phi\colon X\to Y$, a ghost graph or even the full set of data (a)--(c) do not uniquely determine the source and target, as there is no way to recover $i_X$ or $i_Y$, but in the case of aggregates, we know that both these maps are isomorphisms.

\item
If the target $Y=\ast_v$ is just one corolla, then the situation is slightly better. Namely  there is only one associated graph and this is just the underlying graph $\gh(\phi)=\gh_{v}(\phi)$. So that although on all of $\Agg$ morphisms are not determined just by their underlying graph, up to isomorphism this is possible for the comma category $(\Agg\downarrow \Crl)$.
This fact has the potential to cause confusion and it is the basic reason why, in many descriptions of operad like structures, graphs arise as objects rather than morphisms.

\end{enumerate}

\subsubsection{Compositions in $\Gr$ compared to morphisms of  the underlying graphs}
This section  is intended for the reader that is used to dealing with graphs
as objects not as underlying objects of morphisms and gives details to clarify the relationship and differences.
The morphisms of graphs of this section appear in $\Gr$ as compositions in the following
way.  Let $\phi_0$, $\phi_1$ and $\phi_2$ be morphisms such that $\phi_0=\phi_1\circ \phi_2$

$$
\xymatrix{
X
\ar@(dr,dl)[rr] _{\phi_0}
\ar[r]^{\phi_2}&Y\ar[r]^{\phi_1}&\ast\\
\\
}
$$
Now let $\gh_v$ be the associated (ghost) graphs of $\phi_2$, and $\gh_0$ and $\gh_1$
be the associated graphs to $\phi_0$ and $\phi_1$ respectively.

Picking a basis, $Y\simeq \amalg_{v\in V} \ast_v$, $X_v\simeq \amalg_{w\in V_v} \ast_{w}$, and $\phi\simeq \bigotimes_v \phi_v$, set $\gh_v=\gh(\phi_v)$.  The diagram above corresponds, up to isomorphisms, to the formal diagram of graphs:
$$
\xymatrix{
\amalg_v\amalg_{w\in V_v}\ast_{w}
\ar[r]^--{\amalg_v \gh_v}
\ar@(dr,dl)[rr] _{\gh_0}
& \amalg_v\ar[r]^{\gh_1}
&\ast\\
\\
} \text{  up to isomorphisms.}
$$

Now going through the definitions, we see that there is a morphism of graphs $\psi:\gh_0\to \gh_1$ in $\Graphs$, such that
for $v\in V(\gh_1)$,  the $\gh_v$ are the inverse images of the respective vertices,
that is the given restrictions of $\psi$ to the vertices $v$.

Under the caveat that all statements about the associated graphs are only well defined up to isomorphisms,
the composition
of graphs in $\Agg$ is related to morphisms of graphs in $\Graphs$ as follows.
To compose morphisms of aggregates, we {\em insert} a collection of graphs into the vertices, that is we blow
up the vertex $v$ into the graph $\gh_v$.
Reinterpreting the graphs  $\gh(\phi)$ as objects, {\em as one actually should not}, one
can view the composition as a morphism from $\Gamma_1$ to $\Gamma_0$ as being given by the inverse operation, that is
contracting the subgraphs $\Gamma_v$ of $\Gamma_1$.
\subsubsection{Physics parlance}
The composition by insertion is again related to the physics terminology.
Here we are thinking of the vertices as effective vertices, that is we can expand them by inserting interactions.
As Dirk Kreimer pointed out this leads to the possibility of transferring many aspects of Connes--Kreimer \cite{CK} theory to Feynman categories;
see \S\ref{CKsec}.

\subsubsection{Directed, oriented and ordered versions of $\GG$}
\label{orderedversionspar}

Graphs may come with extra structures. A list of the most commonly used ones is included in Appendix A.
Adding these structures as extra data to the morphisms, we arrive at Feynman categories indexed over
$(\Crl,\Agg,\imath)$. Of course we then have to define how to {\em compose} the extra structure.
Moreover, via pull--back, for any Feynman category indexed over $(\Crl,\Agg,\imath)$ we
can add the extra structure to the morphisms as we will explain.

The most common indexing is over directed graphs. We will also introduce the relevant categories with orders on
 the edges and orientations on the edges. The latter two can be understood as analogous to ordered and oriented simplices.
We will consider other extra structures as needed.

\begin{enumerate}

\item
$\GG^{dir}=(\Crl^{dir},\Agg^{dir},\imath^{dir})$.
$\Crl^{dir}$ has directed corollas as its objects and isomorphisms of directed corollas as its morphisms.
$\Agg^{dir}$ is the full subcategory of aggregates of directed corollas in the category of directed graphs; see Appendix A.

\item  $\GG^{ord}=(\Crl^{ord},\Agg^{ord},\imath^{ord})$ and  $\GG^{or}=(\Crl^{or},\Agg^{or},\imath^{or})$.
The underlying objects  in both cases are the same as in $\Agg$ and $\Crl$.
A morphism is a pair of a  morphism in $\Agg$ or $\Crl$ {\em together} with an
orientation  or order of the set of all the edges of the associated graphs.

\item We of course can have both, direction of edges and an orientation/order of all edges.
\end{enumerate}

\begin{df}
Given a Feynman category $(\asts,\clusters,\jmath)$ indexed over $(\Crl,\Agg,\imath)$
with base functor $\base$,
we define the oriented/ordered version  $(\asts,\clusters^{or/ord},\jmath)$
as the Feynman category whose objects are the same, but whose morphisms are pairs $(\phi,ord)$
where $\phi$ is a morphisms in $\F$ and  $ord$ is an order/orientation on the edges of  $\gh(\base(\phi))$,
and composition on the $ord$ factors is the one given above.

For $\F^{dir}$ and $\V^{dir}$ we take as objects pairs $(X,F_{B(X)}\to\{in,out\})$ with $X$ an original object.
Morphisms $(X,F_{B(X)}\to\{in,out\})$ to $(Y,F_{B(Y)}\to\{in,out\})$
are now those morphisms, for which $B(\phi)$ is a morphism in $\GG^{dir}$ from
$(B(X),F_{B(X)}\to\{in,out\})$ to $(B(Y),F_{B(Y)}\to\{in,out\})$.
\end{df}

Notice that this is just the pull--back  or change of base in the category of
Feynman categories.

The directed versions correspond to a decoration, as in \S\ref{decopar} and the orientations and orders to an indexed enrichment over $\Set$.

\subsection{Feynman categories indexed over $\GG^{dir}$}
\label{classexsec}
The most commonly known types of structures are
operads, PROPs and their variations. These are all indexed over directed graphs, and so we will start with these, although
the more basic ones have less structure. They are however all based on subcategories of $\GG^{dir}$.
Of course $\GG^{dir}$ is indexed over $\GG$, so that these examples are also indexed over $\GG$. For operads the subcategories of $\GG^{dir}$ are actually the most restricted.
\subsubsection{Operads}\label{ops-ex-sec}
There is a Feynman category whose $\fops$ are operads.  More precisely, there are several minor variations in the definition of operad, each encoded by its own Feynman category.  We will begin with the simplest variation and define the Feynman category for pseudo-operads.  Recall \cite{MSS} a pseudo-operad satisfies the conditions of an operad, except is not required to have a unit for the composition.

{\bf Pseudo-Operads.}  The Feynman category for pseudo-operads is denoted $\mathfrak{O}^{ps}=(\Crl^{rt},\opd,\imath)$ and defined as follows:
 The objects of $\Crl^{rt}$ are directed corollas having exactly one ``out'' flag (called the root).  The morphisms of $\Crl^{rt}$ are those isomorphisms which preserve the directed structure; ie permutations of the ``in'' flags.  The objects of $\opd$ are aggregates (disjoint unions) of such corollas and the morphisms of $\opd$ are $\tensor$-generated by those morphisms of oriented graphs whose associated ghost graphs $\gh(\phi_v)$ are (connected) trees.  Notice these trees are consequently both oriented and rooted.  The root of $\gh(\phi_v)$ is given by the image of the root flag of $v$.   A pseudo-operad is by definition a $\mathfrak{O}^{ps}\text{-}\oper$.

Those morphisms in $\mathfrak{O}^{ps}$ whose ghost graphs have one edge are often called partial compositions \cite{MSS} and are denoted $\circ_s$.\footnote{ In fact the $\circ_s$ and the isomorphisms $\sigma\in Aut(S)$ actually generate all morphisms; for a thorough discussion see Lemma \ref{genlem}.}  Explicitly, let $S_1\sqcup r_1$ and $S_2\sqcup r_2$ be two objects in $\Crl^{rt}$, specified by their sets of $in\sqcup out$-flags.  For $s\in S_1$ we define $\circ_s$ to be the morphism which connects the flag $s\in S_1$ to the flag $r_2$.  Contracting this edge yields the morphisms' target; $(S_1\setminus s\amalg S_2)\sqcup r_1$ with $\phi^F$ being $id$.  Pictorially:

\begin{figure}[htb]
\begin{center} \includegraphics[scale=.35]{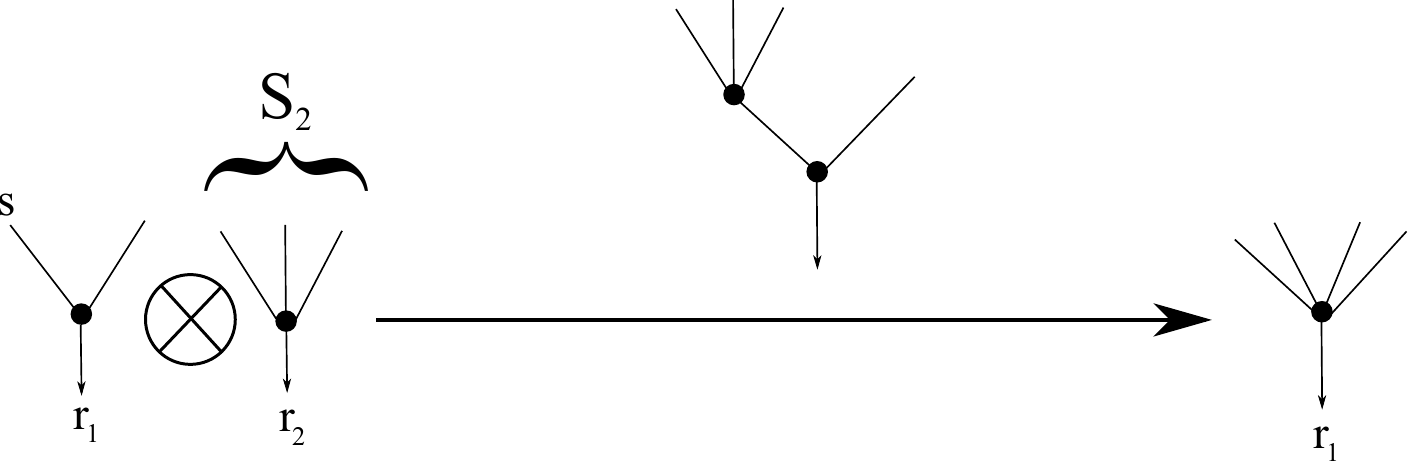} \end{center}
\caption{An edge contraction. The tree is the ghost graph.}
\end{figure}

If $\mathcal{O}$ is a pseudo-operad one more commonly denotes $\mathcal{O}(S\sqcup r)$ by $\mathcal{O}(S)$.  Such an object carries an $Aut(S)$ action and the partial compositions specify morphisms
\begin{equation*}
\mathcal{O}(S_1)\tensor\mathcal{O}(S_2)\stackrel{\circ_s}\longrightarrow \mathcal{O}(S_1\setminus s\amalg S_2)
\end{equation*}
This data satisfies the usual associativity and $\SS$ equivariance axioms (see e.g.\ \cite{MSS}) which are now {\em consequences} of the structure of $\mathfrak{O}^{ps}$.

In the definition of $\mathfrak{O}^{ps}$ we allowed any finite set of flags to form a corolla.   After Markl \cite{Markl} we call this the ``unbiased'' definition of operads.  If we choose to express a bias for the standard skeleton of the category of finite sets via natural numbers, then we get the ``biased'' definition as we presently explain.

{\bf Biased version.}  Let $\Crl^{rt}_{\N}$ be the category of directed corollas with flags $[n]:=\{0,1,2,\dots,n\}$, for $n \in \N_0$ and with $0$ being the ``out'' flag. The automorphism group of such a corolla is naturally identified with $\SS_n$.  Let $\mathfrak{O}^{ps}_\N=(\Crl^{rt}_\N,\opd_\N,\imath)$ be the corresponding full Feynman subcategory of $\mathfrak{O}^{ps}$.

If $\CalC$ is cocomplete and tensor preserves colimits in each variable, the inclusion morphism $i\colon\mathfrak{O}^{ps}_{\N}\to \mathfrak{O}^{ps}$ yields an adjunction between the respective categories of $\mathcal{O}ps$ in $\CalC$ (Theorem $\ref{amorthm}$).  The pull-back $i^*$ is just the restriction, and the push-forward $i_*$ is given by
\begin{equation*}
\CO(S):=\left[ \bigoplus_{S\leftrightarrow \bar n} \CO(n)\right]_{\SS_n}
\end{equation*}
where the action is simultaneously on the sum and summand.  In this case it is easy to see that $i^*$ and $i_*$ give an equivalence of categories, both between the $\smodcat$ and the $\opcat$; in particular the notions of biased and unbiased operads are equivalent.

%Here we can choose the isomorphisms $\sigma\in \SS_n$ and the partial compositions $\circ_i$ as generators.

{\bf May Operads.}  Historically operads were defined by May as having operations parametrized by rooted trees with levels.  This is a property of a tree which is stable under simultaneous composition at all input flags; graphically:

\begin{figure}[htb]
\begin{center} \includegraphics[scale=.4]{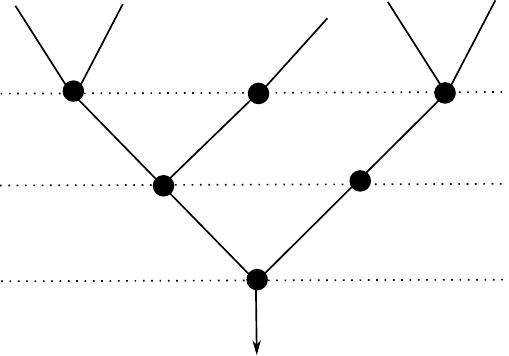} \end{center}
    \caption{A level tree corresponding to the ghost tree of a morphism}
    \end{figure}

Restricting to those morphisms whose ghost graphs are level trees specifies a Feynman subcategory $\mathfrak{O}^{May}\hookrightarrow \mathfrak{O}^{ps}$ and we call $\mathfrak{O}^{May}$-$\mathcal{O}ps$ ``May operads''.  Again we may consider the biased or unbiased versions of these Feynman categories and it is here also the case that biased and unbiased May operads are equivalent.

{\bf Operads.}  Operads are most commonly defined to include a unit for the composition.  We define the Feynman category $\mathfrak{O}=(\Crl^{rt},\opd^\prime,\imath)$ where $\opd^\prime$ has the same objects as $\opd$ but has one additional generating morphism $u\colon \emptyset\to [1]$,
and subject to the requirement that composing generating morphisms with $u$ recovers the monoidal unit isomorphism.  Explicitly $ [n]\tensor \unit \to [n]\tensor [1] \stackrel{\circ_i}\to [n]$ recovers the isomorphism $[n]\tensor k \cong [n]$ and symmetrically.  Notice that $u$ is not a morphism of graphs in our setup, in that such a morphism would require a map of sets (on flags) to the empty set.

It is often practical to denote the unit morphism $u$ graphically by a segment:
\begin{figure}[htb]
\begin{center} \includegraphics[scale=.4]{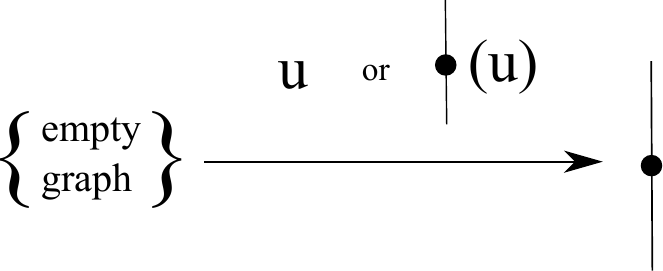} \end{center}
\caption{Graphically adding units}
\end{figure}

This is because pre-composition with the unit morphism has the graphical effect of erasing a unary black vertex.  As an exercise in this graphical calculus one can show that, in the presence of a unit, the partial compositions and the level-wise compositions are equivalent.  In particular any rooted tree maps to a leveled version via the unit:
\begin{figure}[htb]
\begin{center} \includegraphics[scale=.5]{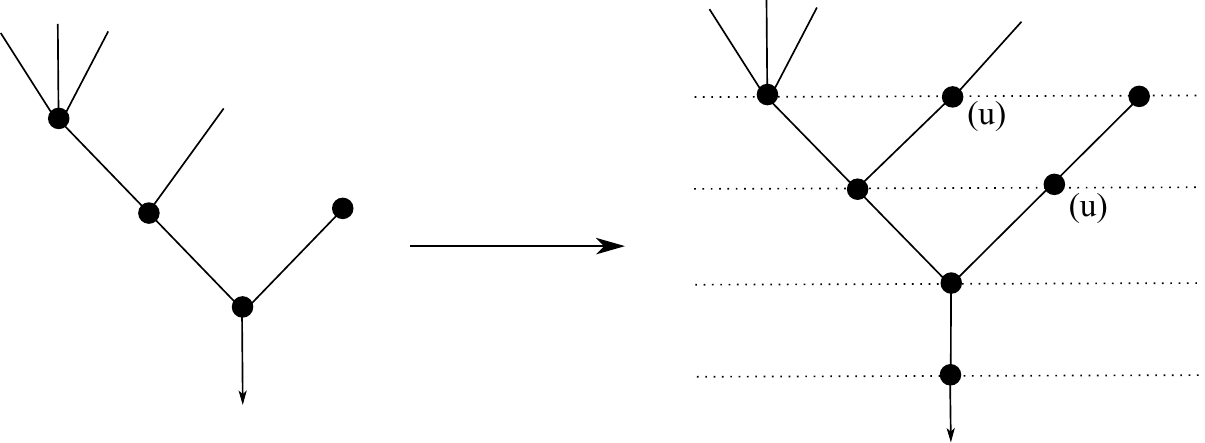} \end{center}
\caption{Leveling a tree using units.}
\end{figure}
Thus the distinction between May operads and pseudo-operads vanishes in the presence of a unit.  Said another way, the Feynman category for operads is isomorphic to the Feynman category for unital May operads.

 \subsubsection{ (Wheeled) PROPs}
 Let $\Agg^{dir,nl}$ be the wide subcategory whose morphisms
 satisfy the condition that the graphs $\gh(\phi_v)$
associated to  $\phi$ do not contain any oriented
cycles.
Set $\props=(\Crl^{dir},\Agg^{dir,nl},\imath)$.

Strictly speaking $\props$--$\opcat$  could be called non--unital pseudo--PROPs. For the classical version of non--unital
PROPs one has to require that the underlying graphs of the morphisms
have levels and all input tails are on one level and all output tails are one  level; see e.g. \cite{Markl}.
We will call the classical version strict PROPs and call the $\props$--$\opcat$  simply PROPs.

{\sc Wheeled version.}
Dropping the condition about oriented cycles, we obtain the Feynman
category underlying the notion of wheeled (non--unital) PROPs, which is simply $\GG^{dir}$.  Oriented one-edged cycles are often called wheels.

{\sc Based Versions.}
In the above two constructions, we can again replace the categories by equivalent  small
categories $(\V,\F,\imath)$ as follows.
For $\V$ these are given by oriented corollas indexed
by $\bar n:=\{1,\dots, n\}$ on the inputs and $\bar m$ on the outputs. Strictly speaking
we use $S=\overline{n+m}\simeq \bar n\amalg \bar m$ together with the function that is $in$ on the
$i\leq n$ and $out$ on the rest. The automorphism groups  are then canonically isomorphic to
 $\SS_{n}\times \SS_m$ and we again get the usual picture. For $\F$ one has to take the full subcategory of the respective categories whose objects are aggregates of these corollas.

 The generators are morphisms from two corollas to one. If we are in the strict case they are given by pairing the inputs of one corolla with the outputs of the other or a merger.
 In the non-strict case, there are given by partial pairings or mergers.

{\sc Functors}  Besides the biased/non--biased inclusion, which again yields an equivalence of $\opcat_{\CalC}$ for suitable $\CalC$, there are other interesting morphisms.  Inclusion gives rise to a morphism of Feynman categories $i\colon\operads\to \props$. The pull--back  $i^*$ is again restriction.
This is the operad contained in a PROP. If it exists $i_*\CO$ is the PROP generated by an operad and if $\CalC$ is say Abelian complete monoidal with the product preserving limits in both variables, then $i_!$ is the extension by $0$ of the operad to a PROP.

\subsubsection{Dioperads.} Consider the wide Feynman subcategory of $\props$ formed by restricting to those morphisms in $\Agg^{dir}$ whose ghost graphs are aggregates of trees.  We call the resulting Feynman category $\dioperads$.  It's ops are called dioperads and the generating (one edged) compositions are called dioperadic.

\subsubsection{(Wheeled) properads.}
We may consider wide Feynman subcategories $\properads\subset \props$ and  $\props^{\circlearrowleft,ctd}\subset \props^{\circlearrowleft}$ by restricting to those morphisms whose ghost graphs are aggregates of connected graphs.  A properad (resp.\ wheeled properad) is then an op over $\properads$ (resp. $\props^{\circlearrowleft,ctd}$). Notice $\props^{\circlearrowleft,ctd}$ is generated by one edged (dioperadic and wheeled) compositions.

%For properads, one uses the $\Crl^{dir}$ and the wide subcategory of $\Agg^{dir,nl}$ whose morphisms have
%connected associated graphs. For the wheeled version, one uses the same condition in $\Agg^{dir}$.

\subsubsection{More classical examples}
Furthermore, one can quickly write down the Feynman categories for
other species such as $\frac{1}{2}$-PROPs and so on.

\subsection{Feynman categories indexed over $\GG$}

\subsubsection{Cyclic operads}
\mbox{}

{\sc Unbiased version.}
Consider
the Feynman category given by $\CCyclic= (\Crl,\Cyclic,\imath)$ where
$\Cyclic$ is the wide subcategory of $\Agg$ whose morphisms are only those morphisms $\phi$
for which each $\gh(\phi_v)$ in the decomposition of condition (\ref{morcond}) of Definition \ref{commadef} is a  tree
and hence by definition connected.  In this definition it is clear that this Feynman category is indexed
over $\Gr$.

{\sc Classical notation.}
 If $\CO$ is a functor from $\Cyclic$ to $\CalC$ then set
$\CO(S):=\CO(\ast_S)$. The basic morphisms are those from two corollas $\ast_S$ and $\ast_T$ to a corolla $\ast_U$
where the ghost tree has a single edge $(s,t)$. If $S$ and $T$ are disjoint and $U=S\setminus \{s\} \cup T\setminus \{t\}$, then we can use $id=\phi^F$ and denoting this morphism by $\phi_{s,t}$, usually one sets $\CO(\phi_{s,t}) =\scirct$.

{\sc Biased version.} Let $\Crl_{\N_0}$ be the corollas indexed by the sets $[n]$ and let $\Cyclic_{\N}$ be the full subcategory of $\Cyclic$ whose objects are aggregates of objects of $\Crl_{\N}$.
Then $\CCyclic_{\N_0}=(\Crl_{\N_0},\Cyclic_{\N_0},\imath)$ is a Feynman category. The classical notations are
$\O(n):=\O((n+1)):=\O([n])$. Notice there is now an $\SS_{n+}=Aut([n])\simeq \SS_{n+1}$ action on $\O(n)$ and furthermore the $\circ_s:=\ccirc{0}{s}$ for $s\neq 0$ do generate with the associativity and $\SS$ equivariance, but due to the $\SS_{n+}$ action there is an additional relation,
which is the well known relation for cyclic operads $(a\circ_1b)^*=b^*\circ_n a^*$, where $*$ is the action of $(01\dots n)\in \SS_{n+}$ and where $a$ is in $\O(n)$. Again in this formalism, this is a {\em consequence}.

{\sc Pull--back and push--forward functors}
There is an obvious inclusion functor of Feynman categories $i:\CCyclic_{\N_0}\to \CCyclic$.
Now $i^*$ is just the restriction and $i_*$ is the usual colimit formula $\O(S)=\bigoplus_{S\leftrightarrow [|S-1|]}
\O((|S|))$ for target categories $\CalC$,  where these exist. As was the case for operads, $i^*$ and $i_*$ are not only adjoint, but yield an equivalence.

There is also a forgetful functor $inc:\operads\to\CCyclic$ which forgets the direction.
Now $inc^*$ is the underlying operad of the cyclic operad. This is often used in the biased version. Furthermore, there is the push-forward which yields the free cyclic operad on an operad.

\subsubsection{Ungraded  (nc)--modular operads}
 We call  $\GG$-$\opcat_{\CalC}$ ungraded nc--modular operads in $\CalC$.
 Here ``nc'' stands for non-connected and
ungraded for the fact that there is no genus or Euler characteristic labeling.

{\sc Connected version.}
Let $\Agg^{ctd}$ be the wide subcategory of $\Agg$  whose morphisms have associated graphs $\gh_v(\phi)$ which are all connected.  Then $\GG^{ctd}=(\Crl,\Agg^{ctd},\imath)$ is a Feynman category and we refer to $\GG^{ctd}$-$\opcat_{\CalC}$ as the category of ungraded modular operads.

{\sc Biased version/MOs.}
For the biased version we consider $\nCrl$, the full subcategory of $\Crl$ with objects $\ast_{\bar n}$ and $\Agg_{\N}$ the full subcategory of $\Agg$ with objects aggregates of $\Crl_{\N}$ then $\GG_{\N}=(\Crl_{\N},\Agg_{\N},\imath)$.

Some variant of these objects has previously been defined by Schwarz in \cite{schw} and were called MOs.
This identification uses the fact that all morphisms in $\Agg$ are generated by merging two corollas and
contracting one small ghost loop on a corolla, see \ref{graphstrucsec}. One has to be careful,
since one can only write down generators using indexing sets which is why
MOs were basically defined on invariants, see \cite{schw} and \cite{HVZ}.

{\sc Functors.}
Via the inclusion the biased and unbiased versions give equivalent $\opcat$, when the push--forwards exist.

There are also inclusion functors $\CCyclic \stackrel{i}{\rightarrow}\GG^{ctd}\stackrel{j}{\rightarrow}\GG$.
The pull-backs are restrictions, if they exist $i_*$ is the ungraded modular envelope and $i_!$ the extension by zero.
The push-forward $j_*$ is analogous to the free PROP construction, it produces the free ungraded nc--modular
operad. Here mergers just go to tensor products. Again, if it exists $j_!$ is the extension by zero.

\subsubsection{Modular and nc--modular operads}
\mbox{}

{\sc NC modular operads.}
Let $\Crl^{\gamma}$ be $\gamma$ labeled corollas with their isomorphisms. A $\gamma$ labeling
in this case is just a number, that is the label of the lone vertex.
The category  $\Agg^{\gamma}$ has as objects aggregates of $\gamma$ labeled corollas.
A morphism $\phi\in\Agg^{g}(X,Y)$ is a morphism of the underlying corollas, with the restriction that
for each vertex $v\in Y$, $\gh(\phi_v)$ has $\gamma(\gh(\phi_v))=\gamma(v)$.  Here the genus labeling of
  $\gh(\phi_v)$ is inherited from $X$ and  $\gamma$ of a graph is defined in (\ref{genuseq}).
We let $\modular^{nc}=(\Crl^{\gamma},\Agg^{\gamma}, \imath)$  and call $\modular^{nc}$--$\opcat_{\C}$
nc modular operads in $\C$.  Forgetting the marking gives a morphism to $\GG$ and hence an indexing.

{\sc Modular operads.}  The category  $\Agg^{\gamma,ctd}$ is the wide subcategory of $\Agg^{\gamma}$,
with the restriction on morphisms $\phi$ that
 each associated graph $\gh(\phi_v)$ is connected.
 Set $\modular=(\Crl^{\gamma},\Agg^{\gamma,ct}, \imath)$  and call $\modular$--$\opcat_{\C}$
modular operads in $\C$.

{\sc Biased versions.} There are biased versions built on $\Crl^{\gamma}_{\N_0}$. As above this yields the standard notation $\CO((g,n)):=\CO(g,n-1):=\CO(\ast_{[ n]}, \gamma(\ast_{[n]}))$ for functors $\CO$.
Via the inclusion the biased and unbiased versions give equivalent $\opcat$, when the push--forwards exist.

 {\sc Functors.}  Assigning $0$ as the value of $\Gamma$ to each unlabeled corolla
yields an inclusion functor $\CCyclic \stackrel{i}{\rightarrow}\modular$, and then there is the natural inclusion $\modular \stackrel{j}{\rightarrow}\modular^{nc}$.
 The pull--backs are restrictions, if they exist $i_*$ is modular envelope and $i_!$ the extension by zero.

\subsection{More functors}  There are additional functors between the Feynman categories discussed above which produce categorical constructions and illuminate the relation between their respective categories of $\fopsc$.

First there is the forgetful functor $\imath\colon\modular^{nc}\to \GG$ which forgets the genus marking. Secondly there is a forgetful functor $\jmath\colon\GG^{dir} \to \GG$ which forgets the direction. So given a wheeled PROP, that is a functor $\CO\colon\GG^{dir}\to\CalC$, we get an associated modular operad $j^*i_*\CO$. Vice--versa given an nc--modular operad, we get an associated wheeled PROP.

If we consider modular operads, then there is a natural map from wheeled properads to unmarked modular operads, given by pushing along the forgetful map $\props^{ctd}\to \GG^{ctd}$ and pulling along the forgetful map $\modular \to \GG^{ctd}$. Vice-versa, we get a properad and even a PROP from a modular operad.  This is for instance the case in string topology and the Hochschild action of \cite{hoch1,hoch2}.

Other interesting relations exist along these lines between operads, dioperads, properads, PROPs and modular operads.

\subsection{Colored versions} We obtain colored versions of the notions above if we use colored corollas and morphisms
of graphs which preserve the coloring, in the sense that the flags of the ghost edges have the same color, i.e.\ that the ghost graph is colored. Here a colored corolla is a corolla together with a map of its flags to a fixed set of colors. For a colored graph one asks that the two flags of an edge have the same color.  Maps between color sets induce functors between the corresponding Feynman categories, and hence adjunctions between their respective ops.

\subsection{Planar versions}
Here we will use the special planar subcategories given in Appendix A \S\ref{planarsec}.
Planar versions are best understood as decorations by various versions of the associative operad, \cite{decorated}, see also \S\ref{decopar}.

\subsubsection{Non-$\Sigma$ cyclic operads}The relevant Feynman category is $\CCyclic^{pl}=(\Crl^{pl},\Cyclic^{pl},\imath)$
There is a biased version using only $\Crl^{pl,dir}_{\N}$.  Forgetting the planar structure gives a morphism
$(v,f):\CCyclic^{pl}\to\CCyclic$ and hence this is again a Feynman category indexed over $\GG$.

\subsubsection{Non-$\Sigma$ operads}
The relevant Feynman category is $\operads^{pl}=(\Crl^{pl,dir},\opd^{pl},\imath)$.  There is a biased version using only $\Crl^{pl,dir}_{\N}$.  Here again there is a forgetful morphism $\operads^{pl}\to \CCyclic^{pl}$ forgetting the root.

\subsubsection{Non-$\Sigma$ modular operads}

Non-$\Sigma$ modular operads were formally introduced in \cite{Marklnonsigma}. They are a specialization of the C/O structures (see section $\ref{cosec}$) having trivial closed sector and only one brane label.

For the Feynman category: the objects of $\V$ are corollas along with an Euler characteristic marking and a poly-cyclic order on the set of flags. A strict poly-cyclic order on a set $F$ is equivalently described by one of the following:
\begin{enumerate}
\item A decomposition $F=F_1\amalg  \dots \amalg F_b$ plus a cyclic order on each of these subsets. Note, here there is no  order on the decomposition, we treat it as a coproduct or as if it were an orbit decomposition.  This formulation is used in \cite{postnikov} and \cite{Marklnonsigma}.
\item An element of the symmetric group $\SS_{|F|}$. This is used in \cite{KontsevichEurope, Bar}.
\item An isomorphism $N:F\to F$ as used in \cite{KP}.
\end{enumerate}
Certainly fixing an order on $F$ (3) and (2) become equivalent. Given a permutation, we can look at the cycle decomposition. The elements in each cycle are exactly the orbits of $N$ and these have a cyclic order. ($N$ stands for ``next'').
% fix points $F_{g,b}^s$.

Notice that one can include empty sets in (1), which is what is called a poly-cyclic order in \cite{Marklnonsigma} and this is equivalent to introducing an extra parameter $s$ in the latter two cases. In a geometric picture this data represents a surface of genus $g$ with $b$ marked boundaries and $b_i$ points on the $i$-th boundary. If $b_i=0$, then this type of boundary is equivalent to a puncture. That is one can reduce to only marked boundaries and $s$ punctures.

The morphisms are then those that respect these poly-cyclic orders as described in \cite{Marklnonsigma,decorated}, see also \S\ref{decopar}, i.e.\  morphisms are those with underlying (connected) graphs such that source decoration and target decoration are compatible as explained in \S\ref{decopar}. In particular, using the formulation with $N$ this decoration is determined as follows. Given a decoration on each of the source vertices and given a morphism $\phi$ with underlying graph $\gh(\phi)$, we have to give a poly-cyclic order on the target vertex. For this, notice that a graph with poly-cyclic orders at the vertices still has cycles, just like a ribbon graph. They are the orbits of $\nu:=\imath\circ N$. These orbits also have a natural cyclic structure on all the flags in the orbit given by $f\succ \nu(f)$. The subset of tails in each such cycle inherits this cyclic order. If $b$ is the total number of cycles, then $s$ is the number of cycles which contain no flags.

\subsubsection{Tabular list}
Before proceeding to the not-so-classical examples, we provide a tabular list of many of the usual subjects, see Table \ref{table1}. These are all
obtained by decorations and restrictions of $\GG$. The decoration ensures the composability of morphisms.
When restricting, one has to be careful to preserve the composition structure i.e.\ make sure that
one is restricting to a subcategory.

\begin{table}[htb]
\begin{tabular}{l|l|l|l}
$\FF$ & Feynman category & additional & restriction \\
& encoding & decoration & on graphs \\ \hline
$\operads$ & operads & direction & trees, one out flag \\
$ \CCyclic$ & cyclic operads & {\it none} & trees  \\
$\dioperads$ & dioperads & direction & trees \\
$\operads^{pl}$ & non-$\Sigma$ operads & dir., planarity & trees, one out flag \\
$ \CCyclic^{pl}$ & non-$\Sigma$ cyclic operads & planarity & trees  \\
$\modular$ & modular operads & genus marking & connected \\
$\modular^{\neg\Sigma}$&non--Sigma modular operads&genus marking&connected\\
&&poly--cyclic structure\\
$\properads$ & properads & direction & connected, no wheels\\
$\props$ & PROPs & direction & no wheels \\
$\props^{\circlearrowleft,ctd}$ & wheeled properads & direction & connected \\
$\props^{\circlearrowleft}$ & wheeled props & direction & {\it none} \\
$\GG^{ctd}$& unmarked modular operads & {\it none} & connected \\
$\modular^{nc,} $ & nc modular operads & genus marking & {\it none} \\
$\GG^{ctd}$& unmarked nc modular & {\it none} & {\it none} \\
\end{tabular}
\caption{\label{table1}List of Feynman categories with conditions and decorations on the graphs, yielding the zoo of examples.}
\end{table}

\subsection{Not so classical examples}

\subsubsection{$k$ truncation} A simple modification to any of the above Feynman category is the restriction to a full subcategory containing vertices of a special type.
E.g.\ only $k$ valent, which appears in $\phi^k$ theory. These morphisms then will have underlying $k$--regular ghost graphs. Likewise, one can restrict to the full subcategory whose objects have $\leq k$ valent vertices. This is used for instance in \cite{Turchin}.

\subsubsection{Orders and Orientations: The main examples.}
We obtain another class of examples from the above, by using extra data
on the morphisms. These will be main examples we will deal with as
their $\opcat$ carry the algebraic structures discussed in \cite{KWZ}. For these one needs an ``odd'' version of the Feynman category. This is discussed in detail in \cite{KWZ} and in section \S\ref{genrelpar}, where we explain the construction in detail and generalize it to Feynman categories with ordered presentations. The cases at hand follow from a canonical ordered presentation on $\GG$.

These are actually the corresponding $\Set$ enriched indexed versions. They should be viewed in analogy to singular and ordered chains, where one can avoid signs on the simplices and  then realize them by putting restrictions on functors. A similar construction is used for involutive functors in TFT.

{\em Ordered cyclic operads.} In this case the category $\asts$ is again $\Crl$
and the objects of $\clusters$ are the disjoint union of corollas.
The new input is that a morphism is a morphism  $\phi$ in $\Cyclic$ together
with an order on the edges of the graphs $\Gamma_v$. Since there are
no mergers or loops this is the same as a choice of decomposition of $\phi$
into no--loop edge contractions. The composition of morphisms is
then just the composition of the two factorizations. In terms of
the edges this is the lexicographical order obtained from the two morphisms.

{\em Oriented cyclic operads.} Analogous to the above, but only retaining
an orientation of the edges, which is a class of total orders under
the equivalence relation of even permutations.

{\em Ordered/oriented NC modular operads.} Again using orders/orientations
on the set of {\em all edges} we obtain ordered/oriented versions of these operads.

We will use these Feynman categories for instance to define $\K$--modular operads.

\subsubsection{C/O structures}\label{cosec}
Other structures, which were up to now not ordered under some overall algebraic
structure, but can be neatly expressed in terms of Feynman categories, are the
{\em C/O structures} of \cite{KP}, and one can now readily define
{\em NC C/O structures} by allowing disconnected surfaces.

Another relevant example are the algebraic structures underlying the paper \cite{HVZ}.
Here we have the Feynman category for two-colored (NC) modular operads
 which has additional morphisms changing the color. These have not been handled
in any previous framework.

Here the vertex structure is relatively complicated. With hindsight it can be handled by decorations. The extra decorations are as follows.

\begin{enumerate}
\item
The set of flags is partitioned in $F=F_{cl}\amalg F_{op}$ which can be viewed as a coloring by open and closed. This partition is called $S\amalg T$ in \cite{KP}.
\item The set $T$ is given a poly-cyclic order.
\item There is a further coloring of all the flags.
\item There is a bi--decoration by $(g,\chi-1)$.
\item There is a possible brane labelling given by two functions $\lambda,\rho:T\to {\mathcal \P}\times {\mathcal \P}$, where $\P$ is a monoid.
\end{enumerate}

We will not discuss the conditions on (5) here. They can be found in \cite{KP}.
Composition along graphs of the decorations follow those of modular operads in the variable $g$ for elements of $T$ and in the variable $\chi-1$ in $S$.

There are several more variants. One can use strict poly-cyclic sets and then there is an additional decoration by $s$, which geometrically corresponds to the brane labelled punctures as mentioned above. One can also leave this extra grading out of the Feynman category by a forgetful functor.

Vice-versa, it is possible to add closed marked points as well, see e.g.\ \cite{KP}. Or to regard the markings $T$ as punctures instead of marked boundaries.

The compositions here are along morphisms with underlying (connected) graphs in a manner that the source decoration and target decoration are compatible morphisms as explained in \S\ref{decopar}, see also \cite{KP}[Appendix A] for the concrete description.

\subsubsection{Boundary type change}
In \cite{HVZ} the authors also consider a similar situation to the above, if not as rigorously axiomatized. They add one operator, which allows for the change of a marking from a marked closed point or boundary to a marked brane labelled point.

\subsubsection{Further applications}
Finally, let us mention that we expect other moduli spaces occurring in homological mirror symmetry to be a source of emerging examples of Feynman categories.  While these notions have not been entirely formalized yet, once this is done we expect our full theory to apply to them.

%\subsubsection{Further applications}
%Finally, in homological mirror symmetry several other moduli spaces occur which give rise to Feynman categories.
%These notions have not been formalized yet, but once this is done our full theory applies to them.

\subsection{$\opcat$ with special elements: units and multiplication}
\label{unitmultsec}

Many operads carry extra structure in the form of distinguished elements.  These can be encoded in our framework, since we have the freedom to have morphisms in $\F$ from the unit element to any other element.  If we already have a Feynman category, we can simply adjoin these morphisms to the monoidal category.

\subsubsection{$\opcat$ for an added multiplication}\label{multpar} A multiplication is generally an element in $\CO(2)$.
Given any of the Feynman categories above, we define its version with multiplication  as follows. Add a new morphism from $\egr$ to $\ast_{[2]}$, for a corolla of the correct type and complete the morphisms by including this element.

Let us be explicit for {\sc operads with a multiplication.} Here we adjoin a new arrow
$m:\egr\to \ast_{[2]}$ monoidally. Call the resulting category $\operads_m$. This will have arrows which are
composable words in old arrows and in $id \amalg\dots \amalg id\amalg m\amalg id\amalg id$. We can depict such a word by a rooted b/w tree, where the
black vertices are trivalent  ---we can think of them as decorated by $m$---
and the white vertices of arity $k$ are decorated by a morphism from $\amalg_{1}^k\ast_{S_i}$ to $\ast_T$.

We can think of the underlying tree as the ghost tree of a particular morphism $w'$ as follows.
 If there are $k$ occurrences of $m$ in a morphism from $X$ to $Y$,  consider
the factorization of $w$
$$
X\simeq X\amalg \egr \amalg \dots \amalg \egr\stackrel{id\amalg m^k}{\rightarrow} X \amalg
\ast_{[2]} \amalg \dots\amalg\ast_{[2]}\stackrel{w'}{\rightarrow} Y
$$
 where now $w'$ sends flags at the black vertices to the new copies of $\ast_{[2]}$.

Picking a functor $\CO$, the morphism $m$ gets sent to $\CO(m):\unit\to \CO(2)$, which is nothing but an element of $\CO(2)$.

{\sc Associative/commutative multiplication}. If the multiplication is supposed to be commutative, we take $m$ to be appropriately symmetric. If it is to be associative, we impose the relation that the
two compositions $\emptyset \amalg \emptyset \stackrel{m\amalg m}\longrightarrow \ast_{[2]}\amalg \ast_{[2]} \rightrightarrows \ast_{[3]}$ are the same.

Graphically, this manifests itself as  a collapse if the  subtrees representing the multiplication and results in trees which are almost bi--partite as in
\cite{del,cact}.

{\sc $\opcat$ with $A_{\infty}$ multiplication.} In an enriched context we may consider operads with an $A_{\infty}$ multiplication.  In this case we adjoin morphisms $m_n:\egr \to \ast_{[ n]}$ for $n\geq 1$ and then impose the standard relations.

\subsubsection{$\opcat$ with a unary morphism}
This is the same procedure as above, using a corolla with $2$ flags of the appropriate type.
The underlying graphs will then again be b/w with black vertices of valence $2$.

{\sc Differential} If we are adding a differential, we impose that $m^2=0$, which can
only be done in the enriched version. This means that we set to $0$ any tree with more that one black vertex inserted into an edge.

{\sc Operadic unit.}  If we are adding a unit, we impose $\phi \circ m= (id \amalg\dots \amalg id\amalg m\amalg id\amalg id)\circ\phi=\phi$, for $\phi\in (\F\downarrow \V)$. In a graph picture, the black vertices can be erased.
Erasing the black vertex from a corolla with 2 flags, leaves a degenerate graph which is just a lone flag/edge (see section $\ref{ops-ex-sec}$). This is the presentation for graphs that Markl has used in \cite{Markl}. Formally, it is actually a colored empty graph.

\subsection{Truncation, stability, and the role of $0,1,2$  flag corollas}

It is often common, for example when considering reduced operads, to truncate by omitting corollas with $0, 1$ or $2$ flags or any other unstably labeled corollas.  Notice that truncated versions are related to the original via push-forward and pull-back under the natural inclusion of Feynman categories.  There operations on the respective categories of ops may also be referred to as truncation or restriction.

We recall the notation $\CO(n-1)=\CO((n))$ for operads and cyclic operads etc. Eg the operadic identity is in $\CO(1)=\CO((2))$.

%combined truncations

\subsubsection{Ground monoid: $\O((0))$} Corollas with no flags may appear.  For example in the non-directed or wheeled versions as the target of a contraction of a corolla with two flags.

If there are no mergers, then there is only the identity map from these corollas and their presence is benign.
If there are mergers then their inclusion introduces a ``ground monoid'' $S=\CO(\ast_{\emptyset})$;
the associative monoidal product on $S$ is given by the merger and all the $\CO(\ast_T)$ become $S$--modules.
If one wants to forget this structure, one can truncate or restrict to functors for which $\CO(\ast_{\emptyset})=\unit$.

\subsubsection{Algebra elements/units $\CO((1))$} We have allowed corollas with only one flag. Given a functor $\CO$, we hence automatically have an object $\CO((1))=\CO(0)$ this will be an algebra or module over the truncation of $\O$. In order to  encode distinguished elements in algebras over $\CO$ one again uses extra morphisms $\unit_{\emptyset}\to \ast_{[0]}$.  Truncation or restriction permits one to encode their non-unital versions.  Note however these algebra units are an important part of the data when one wants to work with pointed spaces or to have $\opcat$ with multiplication {\em and} a unit for the multiplication (as in section $\ref{simppar}$).

%If the category $\CalC$ has an initial object that is the monoidal unit then in lieu of truncating, we can restrict to functors $\CO$ such that $\CO((1))$ is that unit in order to essentially eliminate these types of units.  It however turns out that sometimes this is not a good idea. ?

\subsubsection{$K$--collections $\CO((2))$}
If one includes two-flag corollas which are not represented by special elements as above, then one obtains a second ring structure $K:=\CO((2))$, with multiplication  $\mu=\circ_1:\CO(1)\otimes \CO(1)\to \CO(1)$. Now the other $\CO(n)$ are left monoid modules, but
have a more complicated right multiplication over these. This was formalized as  $K$--collections in the
terminology of \cite{GinzKap}. If one does not want to {\em a priori discard}
this information, one can put this information into the target category {\it loc.\ cit.}.

\subsubsection{Semi--simplicial and (co)simplicial structures}
\label{simppar}
If we allow 1-flag corollas and pick a specific element ---we will call this the pointed case--- in all the cases above, we get a semi--simplicial structure by using the morphism $\delta_i:\ast_{[n]} \to \ast_{[n]}\amalg \emptyset\to\ast_{[n]}\amalg \ast_{[0]}\to \ast_{[n-1]}$ where the first morphism is the  monoidal unit, the second is given by the chosen element and the third  morphism has one virtual edge connecting $0$ to the flag $i$.

If one has an associative multiplication (subsection $\ref{multpar}$), a unit $\emptyset\to \ast_{[1]}$ and an algebra unit $\emptyset\to \ast_{[0]}$ which is a unit for the multiplication, then composition of these morphisms will satisfy (co)simplicial identities.  Here unit for the multiplication means that one imposes the relation that the two morphisms $\emptyset \to \emptyset\amalg\emptyset \to \ast_{[2]} \amalg \ast_{[0]} \to \ast_{[1]}$ given by virtual edges with flags $0$,$1$ and $0$,$2$ coincide with $\emptyset \to \ast_{[1]}$.  In the case of operads the associated cosimplicial structure was studied in \cite{MS}.

%$\operads^{odd}$&odd operads&rooted trees  + orientation of set of edges& odd pre-Lie\\
%$\operads^{pl}$&non-Sigma operads &planar rooted trees & all $\circ_i$ operations\\
% $\operads_{mult}$&operads with mult.&trees\\

%$\CCyclic^{odd}$&odd cyclic operads &trees + orientation of set of edges& odd Lie&\\
%&&++ orientation of the set of edges\\
% $\GG$&unmarked nc modular operads&all graphs \\
% $\GG^{ctd}$&unmarked  modular operads&connected graphs \\

%$\modular^{nc,}$&nc modular operads &all graphs&genus marking \\
%$\modular^{odd}$&$\K$--modular&connected + orientation on set of edges & odd dg Lie \\
%&&+ genus marking&\\
%$\modular^{nc,odd}$&nc $\K$-modular& orientation on set of edges & BV\\
%&&+ genus marking&\\

%$\dioperads^{\circlearrowleft odd}$&odd wheeled dioperads&directed graphs w/o parallel edges &BV\\
%&&+ orientations of edges&\\

%$\props^{\circlearrowleft,ctd, odd}$&odd wheeled properads&connected directed graphs w/o parallel edges &odd Lie admissible\\
%&&+ orientations of edeges&+extra differential\\
%$\props^{\circlearrowleft, odd}$&odd wheeled props &directed graphs w/o parallel edges &BV\\
%&&+ orientations of edeges&\\

\subsubsection{Algebra type examples}
\label{algsec}
We can also consider smaller examples having algebras as their $\opcat$.
 Here $\V$ is trivial and $\V^\otimes$ has as objects the natural numbers with morphisms given by the symmetric groups, see the next section for details on this situation in general.
This ``explains'' the appearance of trees/graphs in defining the relevant structures
in these algebras.

There are several types of algebras encoded by graphs. These now naturally yield examples of $\opcat$ for Feynman
categories whose morphisms are graphs.

\subsubsection{Non--associative algebras}

Let us start with the Feynman category whose $\opcat$ are non--associative algebras.
Now the morphisms from $n$ to $1$ correspond to the possible bracketings of $n$ elements, which are given by planar planted trivalent trees with $n$ labeled leaves. A general morphisms is then a forest of such trees.
The composition of such morphisms is given by gluing roots to flags.

\subsubsection{Associative algebras}
To obtain the Feynman category for associative algebras as graphs, one quotients out the construction in the last paragraph by the associativity equation on morphisms.
This means that there are just $n!$ morphisms left in $Hom(n,1)$ corresponding to $n$--labelled planar corollas.
The composition is then given by grafting the corollas and then contracting the edges, see e.g.\ \cite{woods}.

\subsubsection{Commutative algebras}
For the commutative case, one further quotients by the $\SS_n$ action on morphisms, to obtain only one operation in each $Hom(n,1)$.

\begin{rmk}
There is a connection to $\FF_{surj}$ and its planar version. Indeed $\FF_{surj}$ is isomorphic to the Feynman category for commutative algebras. The one for associative algebras is then its planar version, viz.\ decorated version of \S\ref{decopar}. In that version, the objects are then finite sets with an order, or planar corollas, and this structure is passed on to the morphisms. Alternatively, one can also work with non--symmetric Feynman categories.
\end{rmk}

\begin{rmk}
These examples have in common that the morphisms are actually presented by combinatorial operads.

Many other examples such as Lie or pre-Lie are not of pure combinatorial type and thus need the enriched setup given in paragraph \ref{enrichedsec}. These examples are then generalized in \ref{operfeypar}, see also  \S\ref{hypersec}.
\end{rmk}

\subsection{Feynman categories with trivial $\V$}
\label{trivsec}

We consider $\V$  the category with one object $1$ with $Hom_{\V}(1,1)=id_1$.
Then  let $\V^{\otimes}$ be the  free symmetric category and let $\bar\V^{\otimes}$ be its strict version. It has  objects $n:=1^{\otimes n}$ and each of these has an $\Sn$ action: $Hom(n,n)\simeq \SS_n$.
Let $\bar\imath$ be the functor from $\V^{\otimes}\to\bar\V^{\otimes}$.
 $\F$ will be a category with $Iso(\F)$ equivalent to $\bar\V^{\otimes}$, which itself
 is equivalent to the skeleton of  $Iso(\FinSet)$ where $\FinSet$ is the category of finite sets with $\amalg$ as monoidal structure.  Here,  for  convenience, we do use the strictification of $\amalg$.

 \subsubsection{Finite sets with surjections} The example of finite sets and surjections given in subsection $\ref{surjpar}$ is of this type.  Recall this takes $\F$ to be the wide subcategory $Surj$ of $\FinSet$ with morphisms being only surjections. The inclusion $\imath$ is $\bar\imath$ followed
 by the identification of $n$ with the set $\bar n$. Then $\FF$ is naturally a Feynman category, since $f^{-1}(X)=\amalg_{x\in X} f^{-1}x$. This Feynman category is a basic building block and we denote it by $\FSurj$.
 We may of course also use the skeleton $\aleph_0$ of $\FinSet$ and surjections $\bar m \to \bar n$.

 \subsubsection{ $\F=\mathbf\Delta S_+$.}  Another possibility is to use $\mathbf{\Delta} S$ as defined by Loday \cite{Loday} and augment it by adding a monoidal unit (the empty set).
 This is the category with objects $[n]$, $Aut([n])\simeq \SS_{n+1}$ and morphisms that decompose as $\phi\circ f$ where $\phi$ is in $Aut([n])$
 and $f$ is a non--decreasing map.

\subsubsection{$FI$--modules} In \cite{Farb} the wide category $\Inj$ of $\aleph_0$ was considered, which has only the injections as morphisms.
This again yields a Feynman category $\FF$. Here we start with $\bar \V^{\otimes}$ and
 adjoin the morphism $j:1^{\otimes 0}=\emptyset \to 1$
to define $\F$, then $\F$ is equivalent to $\Inj$. Indeed let $i\colon X\to X'$ be an injection and let $I=\overline{|X\setminus X'|}$, then we can write it as
$X\simeq X\amalg \emptyset \amalg \dots\amalg \emptyset \stackrel{id\amalg j^I}{\longrightarrow}X\amalg I\stackrel{i\amalg id}
{\longrightarrow}Im(i)\amalg I\simeq X'$. Now $FI$--modules are simply functors from $FI$. This can be accommodated in our framework by passing to a free monoidal construction, see \S\ref{freemsec}.

\subsubsection{$\F=\FinSet$} In order to incorporate injections,  just like above,
we  adjoin a morphism $1^{\otimes 0}=\emptyset \to 1$
 to the morphisms of $\Surj$ and obtain a category equivalent (even isomorphic) to $\FinSet$. To show this notice that  any given map of sets $f=i\circ s$ can be decomposed into the surjection $s$ onto its image followed by the injection of the image into the target.

\subsubsection{Non--sigma versions, $\mathbf \Delta_+$} If we take the non--symmetric version of Feynman categories, by the above, we can realize non--sigma operads, but also the augmented simplicial category $\mathbf \Delta_+$. To obtain augmented simplicial objects, we should then use the free monoidal construction, \S\ref{freemsec} to realize these as $\FF^{\otimes}$-$\opcat$.

\subsection{Remarks on relations to similar notions}
Looking at these particular examples, there are connections to PROPs, to Lavwere
theories and crossed--simplicial groups.
\subsubsection{Crossed--simplicial groups}
There are two ways crossed simplicial groups can appear in the Feynman category context.
First, we found the crossed simplicial group $\mathbf {\Delta}S$ above. This is since we started out with the groupoid
given by the objects $[n]$ with their automorphisms and added the morphisms of  $\mathbf \Delta$ in a way that the
morphisms have the standard decomposition as set forth in \cite{Lodaycrossed}.
We arrived at the objects and automorphisms through a symmetric monoidal category construction on a trivial category.
We could likewise arrive at the braid groups if we were to consider braided monoidal categories instead. One could in this way, by altering the background of monoidal categories, achieve some of the symmetries in a free construction. From the fixed setup of symmetric monoidal categories this is not too natural.

The second way crossed simplicial groups can appear is in the wider context of non--trivial $\V$.
For instance the datum of objects $[n]$ with automorphism groups $G_n^{op}$ gives rise to a groupoid $\V$.
We can then use the category $\mathbf{\Delta}G$ as $(\imath\downarrow \imath)$.
This gives a  set  of basic morphisms, in general we may still choose morphisms for all of $(\F\downarrow \imath)$. The
existence of these implements both these symmetries {\em and} a co--simplicial structure on $\opcat$.
This is most naturally expressed as follows. Let $\FF_{G}=(\V,\F,\imath)$ be given by letting $\V$ be as above,
and $\F$ be the category with objects $\V^{\otimes}$ and morphisms monoidally generated by $\mathbf{\Delta}G$.
That is any morphism $Hom_{\F}(X,Y)$ is not empty only
if both $X$ and $Y$ have the same length and in that case
the decomposition of the morphism is given by $\phi=\amalg_i\phi_i$ with $\phi_i:[n_i]\to [m_i]$ for appropriate decompositions $X\simeq\amalg_i[n_i]$ and $Y\simeq\amalg_i[m_i]$.
 Since $\opcat$ are strong monoidal functors, we see that $\F_{\mathbf{\Delta}G}$--$\opcat_{\C}$ is isomorphic to ordinary functors in $Fun(\mathbf{\Delta}G, \C)$
 and the $\opcat^{co}$ are functors in $Fun(\mathbf{\Delta}G^{op},\C)$.

 We can implement these symmetries and the (co)--simplicial structure by looking at Feynman categories $\FF$ which have
 a faithful functor $\FF_{\mathbf{\Delta}G}\to\FF$ and we may also consider the strong version, namely that
 the functor is essentially surjective.

\subsubsection{Lavwere theories}
For these theories, the first formal similarity is that we have a functor $\V^{\otimes}\to \F$ whereas in a Lavwere theory
one has a functor $\aleph_0^{op}\to L$. Apart from that similarity there are marked differences. First, the basic datum is $\V$ not
$\V^{\otimes}$, and $\V$ and $\V^{\otimes}$ have to be groupoids. Secondly $\V$ is not fixed, but variable.

A second deeper level of similarity is given by regarding functors of Feynman categories.
Since $\FinSet$ is a part of the Feynman category $\FF_{\FinSet}$, we can consider morphisms of Feynman categories
$\FF_{\FinSet}\to \FF$ which are the identity on objects. Now this gives functors $\FF_{\FinSet}^{op}\to \FF^{op}$ and one can ask if $\F^{op}$ is a Lavwere theory.
This is the case if $\F$ has a co--product. The requirement of being a Feynman category is then an extra requirement.
If $\otimes$ in $\F$ is a co--product, which it is in $\F_{\FinSet}$, then $\F$--$\opcat^{op}_{\C}$ for a $\C$ with $\otimes$ a product are models, e.g.\ $\C=(\mathcal{T}op,\times)$ or $(\mathcal{A}b,\otimes)$.
Thus, we can say that some special Feynman categories give rise to some special Lavwere theories and models.

\section{General constructions}
\label{constructionsec}
In this section, we gather several constructions which turn a given Feynman category into another one.

\subsection{Free monoidal construction $\F^{\boxtimes}$}
\label{freemsec}
Sometimes it is convenient to construct a new Feynman category from a given one whose vertices
are the objects of $\F$.
Formally, we set $\FF^{\boxtimes}=(\V^{\otimes},\F^{\boxtimes},\imath^{\otimes})$ where $\F^{\boxtimes}$ is the free monoidal
category on $\F$ and we denote the ``outer'' free monoidal structure by $\boxtimes$.
This is again a Feynman category.
There is a functor $\mu:\F^{\boxtimes}\to \F$ which sends $\boxtimes_i X_i\mapsto \bigotimes_i X_i$
and by definition $Hom_{\F^{\boxtimes}}({\bf X}=\boxtimes_i X_i,{\bf Y}=\boxtimes_iY_i)=\bigotimes_{i}Hom_{\F}(X_i,Y_i)$. The only way that the index sets can differ, without the Hom--sets being empty, is if some of the factors are $\unit\in \F^{\boxtimes}$. Thus the one--comma generators are simply the elements of $Hom_\F(X,Y)$.
Using this identification  $Iso(\F^{\boxtimes})\simeq Iso(\F)^{\boxtimes}\simeq( \V^{\otimes})^{\boxtimes}$.
The factorization and size axiom follow readily from this description.

\begin{prop}
$\F^{\boxtimes}$-$\opcat_{\C}$ is equivalent
to the category of functors (not necessarily monoidal)
$Fun(\F,\C)$.
%Furthermore inclusion of $\V\to \V^{\otimes}$  and (after a choice of basis) $\F\to \F^{\otimes}$ induced by  $\bigotimes_v \imath(\ast_v)\to \boxtimes_v \imath(\ast_v)$ and ten yields a morphism of Feynman categories.
\end{prop}
\begin{proof}
Since $\F^{\boxtimes}$ is the free symmetric monoidal category on $\F$.
%The inclusion is clear from the definition.
\end{proof}

\begin{ex}
Examples are  $FI$ modules and (crossed) simplicial objects for the free monoidal Feynman categories for $FI$ and
$\mathbf \Delta_+$ where for the latter one uses the non--symmetric version.
\end{ex}
\subsection{NC--construction}
For any Feynman category one can define its nc version.

This plays a crucial role in physics and mathematics and manifests itself through the BV equation \cite{KWZ}.
Let $\FF=(\V,\F,\imath)$, then we set $\FF^{nc}=(\V^{\otimes},\F^{nc},\imath^{\otimes})$ where $\F^{nc}$ has objects $\F^{\boxtimes}$, the free monoidal product. We however add more morphisms. The one--comma generators will be
$Hom_{\F^{nc}}({\bf X},Y):=Hom_\F(\mu({\bf X}),Y)$,
where for ${\bf X}=\boxtimes_{i\in I}X_i$, $\mu({\bf X})=\bigotimes_{i\in I}X_i$.
This means that for ${\bf Y}=\boxtimes_{j\in J}Y_j$,
$Hom_{\F}({\bf X},{\bf Y})\subset Hom_{\F}(\mu({\bf X}),\mu({\bf Y}))$, includes only those morphisms for which there is a partition $I_j, j\in J$ of $I$ such that the morphism factors through $\bigotimes_{j\in J} Z_j$ where $Z_j\stackrel {\sigma_j}{\to}\bigotimes_{k\in I_j} X_k$ is an isomorphism. That is $\psi=\bigotimes_{j\in J} \phi_j\circ \sigma_j$ with $\phi_j:Z_j\to Y_j$. Notice that there is a map of ``disjoint union'' or ``exterior multiplication''  given by $\mu:X_1\boxtimes X_2\to X_1\otimes X_2$ via $id\otimes id$.

\begin{ex}
The terminology non--connected has its origin in the graph examples.
Examples can be found in \cite{KWZ}, where also a box--picture for graphs is presented. The connection is that morphisms in $\F^{nc}$, have an underlying graph that is disconnected and the connected components are those of the underlying $\F$.

\end{ex}

\begin{prop}
There is an equivalence of categories between $\F^{nc}$-$\opcat_\C$ and
symmetric lax monoidal functors $Fun_{lax\;\otimes}(\F,\C)$.
%Furthermore inclusion of  words of length one for $\V$  and one word symbols for $\F$ induces a morphism of Feynman categories.
\end{prop}
\begin{proof}
Recall that a lax monoidal functor is a functor $\O$ together with a
a morphism $\eps:\unit \to \O(\unit)$ and
natural transformation  $\mu_{X,Y}\colon\O(X)\otimes \O(Y)\to \O(X\otimes Y)$
satisfying natural diagrams. Symmetric means that the symmetries on both sides are compatible.

Let $\O \in \F^{nc}$-$\opcat_\C$. Since $\F$ can be embedded into $\F^{nc}$ as words of length one, we see that the restriction $\O^{res}\in Fun(\F,\C)$. This also yields the map $\eps$. Now the natural transformation $\mu_{X,Y}$ is given by $\O(\mu)$, for $\mu$ given above. The compatibility diagrams are all fulfilled since $\O$ already was a symmetric monoidal functor.

Vice--versa, given the datum of a lax symmetric monoidal functor $(\O,\eps,\mu)$,
we construct $\hat \O\in \F^{nc}$-$\opcat_\C$ as follows. Up to equivalence, we assume that $\FF$ is strict. Since $\hat \O$ will be symmetric monoidal, and $\FF^{nc}$ is a Feynman category, it will suffice to define $\hat \O$ on $\V$ and on one--comma generators. On $\V$: $\hat \O(X):=\O(X)$ and hence for ${\bf X}=X_1\boxtimes \cdots \boxtimes X_n$: $\hat \O({\bf X})=\O(X_1)\odo \O(X_k)$.
On one--comma generators $\phi: {\bf X}\to Y$ in $\FF^{nc}$, we let
$\hat\O(\phi)$ be the composition
$$
\hat \O ({\bf X}):=\O(X_1)\odo \O(X_l)\stackrel{\mu^{l-1}}{\to} \O(X_1\odo X_n)\stackrel{\O(\phi)}{\rightarrow}\O(Y)=\hat \O(Y)
$$

In order to extend to other morphisms, we use that $\hat\O$ will be monoidal.
The functoriality composition is then guaranteed by the diagrams for lax monoidal functors and the axioms of a Feynman category.
%$Y=\imath(\ast_1)\odo \imath(\ast_l)$
%and decompose $\phi\in Hom_{\F}(\mu({\bf X}),Y)$, as $\phi=\phi_1\odo \phi_l$ with $\phi_j\in Hom_\F(Z_j,\imath(\ast_j))$.
%The last statement is clear.
\end{proof}

\subsubsection{Mergers or $B_+$ operator}
Notice that there are no obvious functors between $\FF$ and $\FF^{nc}$. For instance starting with the inclusion $\V\to \V^{\otimes}$ we would need to have $\boxtimes_v \imath(\ast_v) $ as the image of $\bigotimes_v \imath(\ast_v)$. But not all maps in $Hom(\mu({\bf X}),\mu({\bf Y}))$ are in $\F^{nc}$.

Sometimes it is possible to define a $B_+$ operator, which is a collection of equivariant morphisms $\V^{\otimes n} \to \V$. By abuse of notation these are called  $\boxminus$.

\begin{df}
We call $\F^{nc}$, $B_+$-mergeable if there is a  $B_+$ operator, and a category $\F^{nc}_{merge}$ that has the following properties
\begin{enumerate}
\item $\FF^{nc}_{merge}=(\V,\F^{nc}_{merge},\imath)$, is a Feynman category.
    \item $\F^{nc}_{merge}$ has a  faithful functor from $\F$. $\F^{nc}_{merge}$ is obtained from $\F$ by adjoining the morphisms $\boxminus$ (not necessarily freely) and moreover $\FF$ is a sub-Feynman category of $\FF^{nc}_{merge}$.

    \item There is a morphism $B$ of Feynman categories $\FF^{nc}$
    to $\FF^{nc}_{merge}$ which is induced by $\mu$.
    %Better
\end{enumerate}

\end{df}
\begin{ex}
In the case of a Feynman category whose morphism have underlying connected graphs, the usual suspect for $\boxminus$ is the merger.
Among them are $(\properads)^{nc}_{merge}=\props$ and $\modular^{nc}=(\modular)^{nc}_{merge}$ whence the name.
Notice that there are relations between mergers and edge/loop contractions, see \S\ref{graphstrucsec}.
\end{ex}
\begin{ex}
The $\operads^{nc}$ is interesting and described first in \cite{KWZ}. Here a merger is provided by the $B_+$ operator, with $\boxminus:*_n\times *_m\to *_{n+m}$ given by merging the vertices and identifying the two outputs.

Concretely, given an operad $\O$,
we get $\hat\O(n_1,\dots,n_k)=\O(n_1)\odo \O(n_k)$.
Pushing this forward along $B$, $B_*(\hat O)(n)=\bigoplus_k \bigoplus_{(n_1,\dots,n_k)}\O(n_1)\odo\O(n_k)$.
The additional structure is a multiplication $\boxminus:\O(n)\otimes \O(m)\to O(n+m)$.
\end{ex}
%! check this example.
\subsection{$\Fdeco$: Decorated Feynman categories}
\label{decopar}
This theory was developed after the introduction of Feynman categories in \cite{decorated}. We refer to {\it loc.\ cit.} for the details and just give the main definitions and theorems here.
We do give new examples of decorations, for the procedures of \ref{unitmultsec} and non--trivial $\O(1)$

\begin{thm}\cite{decorated}
\label{decoexistthm}
Given an $\O\in \F$-$\opcat_\C$ with $\C$ Cartesian, there exists the following Feynman category $\FFdeco=(\Vdeco,\Fdeco,\ideco)$ indexed over $\FF$. The objects of $\F_{dec\O}$ are pairs $(X,dec\in \O(X))$ and $Hom_{\F_{dec\O}}((X,dec),(X',dec'))$ is the set of $\phi:X\to X'$, s.t. $\O(\phi)(dec)= (dec')$. $\Vdeco$ is given by the pairs $(\ast,a_{\ast})$, $\ast \in \V$, $a_\ast\in \O(\imath(\ast))$ and $\ideco$ is the obvious inclusion. The indexing
is given by forgetting the decoration.
\end{thm}

\subsubsection{Non--Cartesian case}
This construction works a priori for Cartesian $\C$, but with the following  modifications it also  works for the non--Cartesian case.
Choose a base functor $\jmath$, then the category $\F$ has objects $(X,(a_{v_1},\dots, a_{v_{|X|}}))$ where $a_{v_i}\in\O(\ast_{v_i})$ and $\jmath(X)=\ast_{v_1}\odo \ast_{v_{|X|}}$. The restriction on the morphisms then comes first from the monoidal structure of $\O$ and then applying $\O(\phi)$. That is, for a one--comma map $\phi$ with source $X$ and target $(\ast,a)$ we require that
$
\O(a_{v_1}\odo a_{v_{|X|}})=\O(a)
$

There is a general theorem saying that the decoration by the push--forward exists and how such push--forwards factor.

\begin{thm}\cite{decorated}
\label{theorem1deco}  Let $f\colon \Fe\to \Fe^\prime$.  There are the following commutative squares which are natural in $\O$.
\begin{equation}
\xymatrix{\Fepair \ar[r]^{f^{\O}} \ar[d]_{forget} & \Fe'_{dec\, f_{\ast}(\O)} \ar[d]^{forget'} \\
\Fe \ar[r]^f & \Fe'}
\quad
\xymatrix{\Fepair \ar[r]^{\sigma_{dec}} \ar[d]_{f^{\O}} & \Fe_{dec\Po} \ar[d]^{f^{\Po}} \\
\Fe'_{decf_{\ast}(\O)} \ar[r]^{\sigma'_{dec}}& \Fe'_{decf_{\ast}(\Po)} }
\end{equation}
 That is, for any $\O$, we have the first square and  for a natural transformation $\sigma:\O\to \P$ we have the second.
On the categories of monoidal functors to $\C$, we get the induced diagram of adjoint functors.
\begin{equation}
\xymatrix{\Fpair\text{-}\opcat \ar@/^/[rr]^{(f^{\O})_*}\ar@/_/[d]_{forget_*} && \F'_{dec\,f_{\ast}(\O)}\text{-}\opcat   \ar@/^/[ll]^{({f^{\O}})^*}\ar@/^/[d]^{forget'_*} \\
\F\text{-}\opcat \ar@/^/[rr]^{f_*}\ar@/_/[u]_{forget^*} && \F'\text{-}\opcat\ar@/^/[u]^{forget'^*}\ar@/^/[ll]^{f^*} }
\end{equation}
\end{thm}

\begin{ex}  Planar (aka non-$\Sigma$) versions of operads, cyclic operads, and modular operads are obtained by decoration in which $\O$ is $\mathcal{A}ss$ (the associative operad), $\mathcal{A}ss^{cyc}$ (the associative cyclic operad), and $\mathcal{M}od(\mathcal{A}ss^{cyc})$ (its modular envelope).

Other natural operads for decoration are given by various versions (cyclic, modular envelope, etc) of $Lie$, dihedral variants of these operads, or any of their $\infty$-versions.  Decorating with $Com$, its cyclic or modular version changes nothing as the operads are trivial.
\end{ex}

\begin{ex}
Further examples include genus decoration, directions, color decoration, etc. These recover such types of operad-like structures discussed previously. The details are in \cite{decorated}.

The genus decoration is the functor $G$ defined by $G(\ast)=\mathbb{N}$, the natural numbers. On morphisms the maps are $+:\N\times \N\to \N $ for edge contractions, and $+1:\N\to \N$ for loop contractions. For mergers the map is given by $(n,m)\to n+m-1$.

Colors are given by $C(\ast_S)=Hom(S,C)$ for the set of colors $C$. This means that $C(X)=Hom(F_X,C)$ (recall that $F_X$ is the set of flags of the aggregate $X$). For a morphism $\phi$, $C(\phi)=(\phi^F)^*$. To get colored versions, one then restricts to $\phi:(X,c)\to (Y,c')$ on $\FF_{dec C}$ which satisfy
$c\circ \imath_{\phi}=c$, that is only ghost edges whose flags have the same color are allowed.

For directions one uses the decoration $dir$, $dir(\ast_S)=Hom(S,\Z/2\Z)$, again with $dir(phi)=(\phi^F)^*$. The restriction on morphisms is that ghost edges have to have opposite flags. $dir(\imath_{\phi}(f))=1-dir(f)$ for any ghost edge $\{f,\imath_\phi(f)\}$.
\end{ex}

\begin{thm}\cite{decorated}
\label{decothm}
If $\final$ is the terminal element in $\fops$ and $forget:\Fdeco\to \F$ is the forgetful functor,
then $forget^*(\final)$ is a terminal object for $\Fdeco\text{-}\opcat$.
We have that $forget_*forget^*({\final})=\O$.
\end{thm}

%terminal/trivial?

\begin{df}
We call a morphism of Feynman categories $i:\FF\to \FF'$ a minimal extension over $\C$ if  $\F$-$\opcat_{\C}$ has a terminal/trivial functor $\final$ and $i_*{\final}$ is a terminal/trivial functor in $\F'$-$\opcat_{\C}$.
 \end{df}
%There are two examples that appear naturally. The first is $CycCom$ and \\
%$ModCycCom$ for $\CCyclic \to \modular$ and the second is the decorated version
%$\forget^*(CycAss)$ and $i^{\O}_*(\forget^*(CycAss))$.
\begin{prop}
\label{minextprop}
If $f:\Fe \to \Fe'$ is a minimal extension over $\C$, then $f^{\O}:\Fepair \to \Fepairtwo$ is as well.
\end{prop}

\begin{ex} The discussion in the example of non-$\Sigma$ modular operads is as follows.

 \begin{equation}
\label{modulardiag}
\xymatrix{\F_{dec(\mathcal{A}ss^{cyc})}=\CCyclic^{pl} \ar[rr]^{i^{\mathcal{A}ss^{cyc}}} \ar[d]_{forget} && \modular_{dec(\mathcal{M}od(\mathcal{A}ss^{cyc}))} =\modular^{\text{non-}\Sigma}\ar[d]^{forget} \\
\CCyclic \ar[rr]^i && \modular}
\end{equation}

For $\mathcal{C}=\mathcal{S}et$, the inclusion $i$ is a minimal extension.  This is a fact explained by basic topology.
Namely gluing together polygons in their orientation by gluing edges pairwise precisely yields all closed oriented surfaces, see e.g.\ \cite{munkres} and this is unique up to switching diagonals or Whitehead moves.  Hence $i^{\mathcal{A}ss^{cyc}}$ is also a minimal extension.

This explains the result of Markl (Theorem 35 of \cite{Marklnonsigma}) that the push-forward (aka non-$\Sigma$ modular envelope) of the terminal non-$\Sigma$ cyclic operad is terminal.  It reflects the fact that not gluing all edges pairwise, but preserving orientation, does yield all surfaces with boundary.

%\begin{enumerate}
%\item The commutative square exists simply by Theorem \ref{theorem1deco}.
%\item  On the left side,  if $*_C$ is final for $\CCyclic$ and hence $forget^*(*_C)=\underline{*}_C$
%is final for $\CCyclic^{\neg \Sigma}$. The pushforward $forget_*(\underline{*}_C)=CycAss$.
%\item On the right side,  if $*_M$ is final for $\modular$ and hence $forget^*(*_M)=\underline{*}_M$
%is final for $\modular^{\neg \Sigma}$. The pushforward $forget_*(\underline{*}_M)=ModAss$.
%\item The inclusion $i$ is a minimal extension.
%This is a fact explained by basic topology.
%Namely gluing together polygons in their orientation by gluing edges pairwise precisely yields all closed oriented surfaces, see e.g.\ \cite{munkres} and
%this is unique up to switching diagonals or Whitehead moves. %In the current understanding, this procedure guarantees the condition of Proposition \ref{critprop}.
%\item  Hence $i^{CycAss}$ is also a minimal extension.
%which explains why indeed the pushforward of the terminal $\oper$ is up to that point is still terminal.
%It also reflects the fact that not gluing all edges pairwise, but preserving orientation, does yield all surfaces with boundary.
%\end{enumerate}
\end{ex}

\begin{ex}
\label{bwex}
We can also treat the added morphisms of \S\ref{unitmultsec} as decorations and even generalize them to ``family''  decorations. First, the usual added structures are again of the type of the original structure, i.e.\ an $\O\in \fopsc$ the paradigmatic examples being units or multiplication for operads. These then are the operad $\O$ with the only non--vanishing $\O(n)$ is $\O(1)=\unit$ for units, and for multiplication the associative operad or the $A_{\infty}$ operad.

In general, say we have $\FF$ and $\O\in \fopsc$ we can proceed in three steps
\begin{enumerate}
\item Consider the Feynman category $\FF\times \FFdeco$.
\item Add morphisms according to $\forget$. This means that we add  morphisms $(X,dec) \to X$ compatible
with all the relations given by the data of the symmetric monoidal functor $\forget$ on morphisms, in particular, these morphisms are determined by the extra morphisms on the vertices $(\ast_v,dec)\to \ast_v$.  We call the resulting category $\FF \rtimes \FFdeco$. It is an easy check that this is again a Feynman category. $\FF \rtimes \FFdeco$-$\opcat$ will be the category whose objects triples $(\P,\P',N)$ of an $\P\in \fopsc$,  an $\P'\in\FFdeco$-$\opcat$ together with $N\in Hom_{\Fdeco-\opcat}(\P',\forget^*\P)\simeq
Hom_{\fops}(\forget_*\P',\P)$.
\item Restricting to $\O'$ to be terminal, that is $\O=\forget^*(\final)=\final'$ is a terminal object, we obtain the elements of the from $(\final,\P,N)$, with $N\in Hom_{\Fdeco-\opcat}(\final',\forget^*\P)$.

The Feynman category for these $\opcat$ then is simply given by adding morphisms from $\emptyset$ which can be identified with the unit in $\FF$. This is the description in \S\ref{unitmultsec}.

Thinking of  $N$ as an element in  $Hom_{\fops}(\forget_*\forget^*(\final),\P)=Hom_{\fops}(\O,\P)$, we obtain the Feynman category whose $\opcat$ are pairs of  $\P\in \fops$ and a morphism $\O\to \P$.
\end{enumerate}
In particular for the example of a unit for operads, to specify $N$ we need to fix one morphism $\unit\to \P(1)$ and for a multiplication for opearad, to specify $N$, we need to fix one morphism $\unit\to \P(2)$.

For the examples based on  graphs, the morphisms in step (2) are effectively given by  decorating the vertices by elements of $\P$, or not, that is undecorated vertices. Note that these are allowed to lie on the same graph. For example, this results in the b/w trees
in \cite{del,MSS} for multiplications or $A_{\infty}$ multiplications. In the latter, each black subtree is viewed as a decoration of one vertex. Step (3) then fixes that there is not a ``family'' of such multiplications, but only a specific choice.  The same technique applies for decorating with differentials hence explaining the ``twisted'' constructions of \cite{Willwacher}.
\end{ex}

\subsection{Iterating Feynman categories}
Given a Feynman category $\FF$ there is a simplicial tower of iterated categories built on the original one.

Let $\V'$ be the groupoid $\V'=(\imath^{\otimes}\downarrow \imath)$ and let $\F'=(id_{\F}\downarrow \imath^{\otimes})$.
There is an obvious inclusion $\imath'\colon\V'\to \F'$ and  $\V^{\prime \otimes} \simeq Iso(\F')$, since $\FF$ was a Feynman category
and condition (i) holds. The difference between the two is just
the choice of the representation
of the source and target.
$\F'$ is also equivalent to $(id_{\F}\downarrow i)$ where $i\colon \imath^{\otimes}(\V^{\otimes})\to \F$ is the inclusion of the image.
The difference being the choice of a representation of the target.
The morphisms in that category are given by diagrams
\begin{equation}
\label{moreq}
\xymatrix{
X\ar[r]^{\phi}\ar[d]_{\psi}&Y\ar[d]^{\simeq \sigma}\\
X'\ar[r]^{\phi'}&Y'
}
\end{equation}
and these morphisms can be factored into two morphisms, which we will call type I and type II:
\begin{equation}
\xymatrix{
X\ar[r]^{\phi}\ar[d]_{\psi}&Y\ar@{=}[d]\\
X'\ar@{=}[d]\ar[r]^{\sigma^{-1}\circ \phi'}&Y\ar[d]^{\simeq \sigma}\\
X'\ar[r]^{\phi'}&Y'
}
\end{equation}
and similarly there is a factorization in $\F'$.

Notice that $Hom_{\F'}(\phi_0,\phi_1)$ is empty if $t(\phi_0)\not\simeq t(\phi_1)$.
Also, if we are indexed over $\GG$, the underlying graphs satisfy $\gh(\phi)=\gh(\phi')\circ\gh(\psi)$.

There are also the two standard source and target functors
 $s,t:\F'\to\F$, which also restrict to $s,t:\V'\to \V$. They are just the restrictions of the functors on the full arrow category.
Explicitly any object $\phi\in \F'$ is sent to its source or its target, and $s$ sends a morphism (\ref{moreq})
to $\phi$ and $t$ sends it to $\sigma$.

\begin{dfprop}
The iterate of a Feynman category $\FF$ is the Feynman category $\FF'=(\V',\F',\imath')$. We set $\FF^{(k)}=
\FF^{(k-1)\prime}$. The morphisms $(s,s)$ and $(t,t)$ are morphisms of Feynman categories from $\FF'\to \FF$,
which by abuse of notation we will just call $s$ and $t$.
\end{dfprop}

\begin{proof}
We need to show that $\FF'$ is indeed a Feynman category. Since $\FF$ was a Feynman category,
the condition (\ref{morcond}) for $\F$ guarantees that indeed $\imath'^{\otimes}$ is an equivalence
of $\V^{\prime \otimes}$ and $\F'$. For  condition (\ref{morcond}) for $\F'$, one uses
(\ref{morcond}) on composition in $\F$. Given $\phi_1\colon X\stackrel{\psi}{\to}Y\stackrel{\phi_0}{\to}Z$
we can first decompose $Z\simeq \amalg_v\ast_v$ and hence $\phi_0\simeq \amalg_v \phi^0_v$
for the decomposition
$Y\simeq \amalg_v Y_v$ and $\phi_0\simeq \amalg_v \phi^1_v$ for the decomposition
$X\simeq \amalg_v X_v$. Further decomposing the $Y_v\simeq \amalg_{w_v\in I_v} \ast_{w_v}$
we get the desired decomposition of  $\psi$ by first splitting into the various $\psi_{w_v}$ and
then assembling them  into the $\amalg_{w_v\in I_v}\psi_w$, where now up to choosing
isomorphisms $\psi_w\colon X_v\to Y_v$.
The assertion about $(s,s)$ is straightforward.
\end{proof}

\begin{prop}
 The Feynman categories $\FF^{(k)}$ form a simplicial object in the category of Feynman categories with the simplicial morphisms
given by composition, the source and target maps and identities.
\end{prop}

\begin{proof}
For this, one just uses the nerve functor for categories and notices that this is compatible with the Feynman category structure.
\end{proof}

\begin{lem}\label{slicelemma}
The Feynman category $\FF'$ splits into subcategories or slices $\FF'_{[X]}$ whose objects have targets isomorphic to $[X]$.
Restricting to $\ast\in \V$ we even obtain
full Feynman subcategories $\FF'|_{[\imath(\ast)]}$ for $\imath(\ast)\in\F$.
Here splits means that $Hom_{\FF'}(\phi,\phi')=\emptyset$ unless $\phi$ and $\psi$ are in the same subcategory.
\end{lem}

\begin{proof}
The first statement is true by definition of the morphisms in $\FF'$: there are only morphisms between $\phi$ and $\phi'$
if they have isomorphic targets. The second assertion then follows from the same fact restricted to $\V'$.
\end{proof}

\begin{lem}
\label{FGlem}
If $\O\in \F$--$\opcat_{\C}$ then $FG\O=t_*s^*\CO$.
\end{lem}
\begin{proof}
On the l.h.s.\ evaluating at $Z$, we get $colim_{(\imath^{\otimes}\downarrow Z)}\CO\circ s$.
On the r.h.s.\ we get $colim_{(t\downarrow Z)}\CO\circ s\circ P$. We see that for a given object ($\phi\colon X\to Y$, $\chi\colon Y\to Z$),
we can use morphisms of type I with $\psi=\phi$ to reduce the colimit to just a colimit over objects $(id_Y,\chi\colon Y\to Z)$ and then use morphism
of type II to reduce to the colimit over automorphisms of these objects. Now $\CO\circ s\circ P (id_Y,\chi:Y\to Z)=\CO(Y)=\CO(s\chi)$.
The claim then follows, since we get an isomorphism between the two universal co-cones.
\end{proof}

\subsection{Arrow category}
For a category $\F$ we set $Ar(\F)$ to be the category $(\F\downarrow\F)$ which has as objects morphisms in $\F$ and as morphisms
commutative diagrams in $\F$. It is clear that any functor $\imath$ of categories induces a functor $Ar(\imath)$ on the level or arrow categories.
\begin{prop} If $\FF=(\V,\F,\imath)$ is a Feynman category, so is $Ar(\FF)=((\V^{\otimes} \downarrow \V),Ar(\F),Ar(\imath))$.
\end{prop}
\begin{proof}
It is clear that $(\V^{\otimes}\downarrow \V)$ is a groupoid and straightforward that since $\V^{\otimes}\simeq Iso(\F)$ that
$Iso(Ar(\F))\simeq   (\V^{\otimes}\downarrow \V)^{\otimes}$. The condition (\ref{morcond}) guarantees that $Ar(\F)$ is one--comma generated.
\end{proof}

\subsection{Feynman  level category $\FF^+$}
\label{Fplussec}
Given a Feynman category $\FF$, and a choice of basis for it, we will define its
Feynman level  category  $\FF^+=(\V^+,\F^+,\imath^+)$ as follows.
The underlying objects of $\F^+$ are the morphisms of $\F$. The morphisms
of $\F^+$ are given as follows: given $\phi$ and $\psi$, consider their decompositions
\begin{equation}
\xymatrix
{
X \ar[rr]^{\phi}\ar[d]_{\sigma}^{\simeq}&& Y\ar[d]^{\hat\sigma}_{\simeq} \\
 \bigotimes_{v\in I}\bigotimes_{w\in I_v} \ast_w\ar[rr]^{\bigotimes_{v\in I}\phi_{v}}&&\bigotimes_{v\in I} \ast_v
}
\quad
\xymatrix
{
X' \ar[rr]^{\psi}\ar[d]_{\tau}^{\simeq}&& Y'\ar[d]^{\hat\tau}_{\simeq} \\
 \bigotimes_{v'\in I'}\bigotimes_{w'\in I'_{v'}} \ast_{w'}\ar[rr]^{\bigotimes_{v'\in I'}\psi_{v'}}&&\bigotimes_{v'\in I'} \ast_{v'}.
}
\end{equation}
where we have dropped the $\imath$ from the notation, $\sigma,\hat\sigma,\tau$ and $\hat\tau$ are given  by the choice of basis
and the partition $I_v$ of the index set for $X$ and $I'_{v'}$ for the index set of $Y$
is given by the decomposition of the morphism.

 A morphism from $\phi$ to $\psi$ is a two level partition of $I: (I_{v'})_{v'\in I'}$, and partitions of $I_{v'}:(I^1_{v'}\dots,I_{v'}^{k_{v'}})$ such that
 if we set $\phi_{v'}^i:=\bigotimes_{v\in I_{v'}^i}\phi_v$ then $\psi_{v'}=\phi_{v'}^k\circ\dots\circ \phi_{v'}^1$.

 To compose two morphisms $f\colon\phi\to \psi$ and $g\colon\psi\to \chi$, given by partitions of $I:(I_{v'})_{v'\in I'}$ and of the $I_{v'}:(I^1_{v'}\dots,I_{v'}^{k_{v'}})$ respectively of
 $I':(I'_{v''})_{v''\in I''}$ and the $I_{v''}:(I^{\prime 1}_{v''}\dots,I_{v''}^{\prime k_{v''}})$,
 where $I''$ is the index set in the decomposition of $\chi$, we set the compositions to be
 the partitions of $I:(I_{v''})_{v''\in I''}$ where  $I_{v''}$ is the set partitioned by
 $(I_{v'})_{v'\in I^{\prime j}_{v''},j=1,\dots, k_{v''}}$.
 That is, we replace each morphism $\psi_{v'}$ by the chain $\phi_1^{v'}\circ\dots\circ\phi_k^{v'}$.

Morphisms alternatively correspond to  rooted forests of level trees thought of as flow charts. Here the vertices are decorated by the $\phi_v$ and the composition along the rooted  forest is $\psi$.
There is exactly one tree $\tau_{v'}$ per $v'\in I'$ in the forest and accordingly
the composition along that tree is $\psi_v'$.

Technically,   the vertices are the $v\in I$.
The flags are the union $\amalg_v \amalg_{w\in I_v} \ast_w\amalg \amalg_{v\in I}\ast_v$
with the value of $\del$ on $\ast_w$ being $v$ if  $w\in I_v$ and
$v$ on $\ast_v$ for $v\in I$. The orientation at each vertex is given by the target being out. The involution $\imath$  is given by matching source and target objects of the various $\phi_v$.
The level structure of each tree is given by the partition  $I_{v'}$.
 The composition is the composition of rooted  trees by gluing trees  at all vertices ---that is we blow up the vertex marked by $\psi_{v'}$ into the tree $\tau_{v'}$.

\begin{figure}
    \centering
   \includegraphics[width=0.8\textwidth]{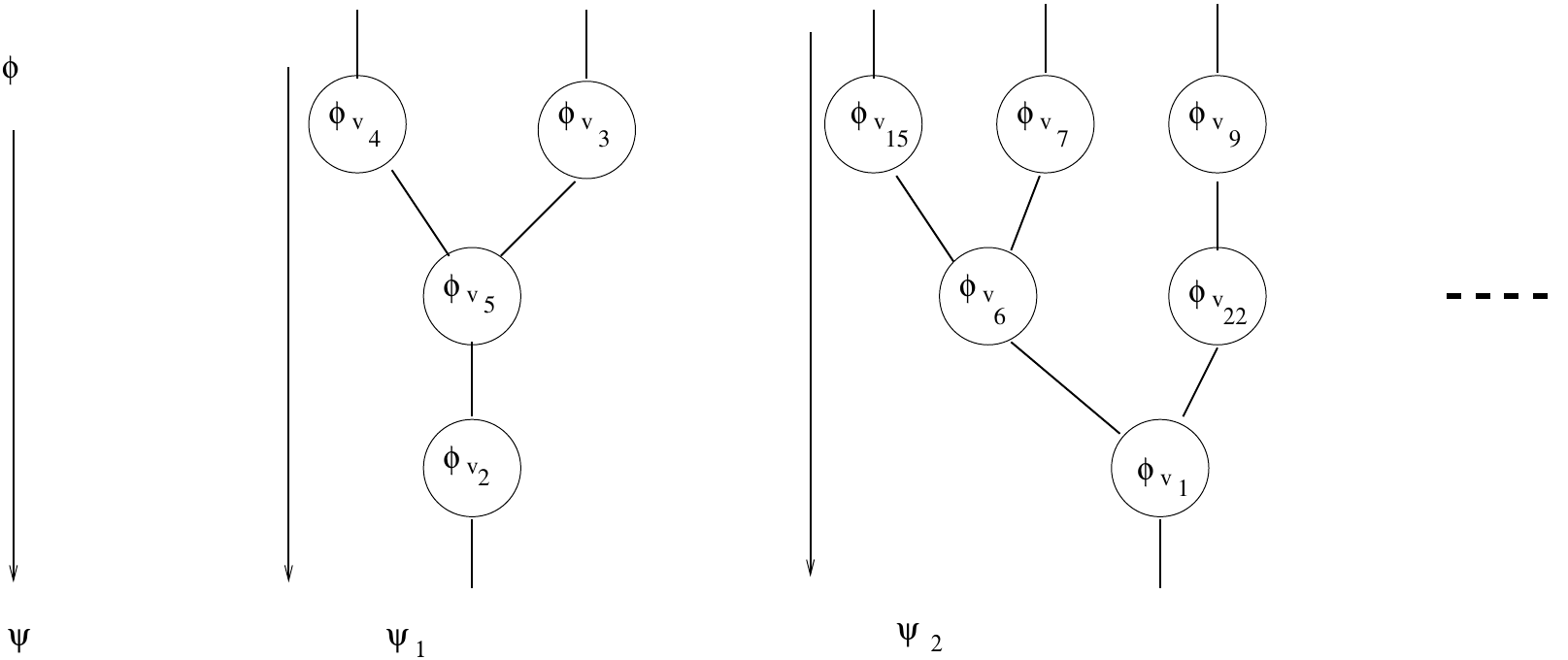}
    \caption{The level forest picture for morphisms in $\FF^+$. Indicated is a morphism from $\phi\simeq\bigotimes_v\phi_v$ to $\Psi\simeq \bigotimes_{i}\Psi_i$}
    \label{Fplusfig}
\end{figure}

The groupoid $\V^+$ are morphisms in $(\F\downarrow \V)$, with the one level trees as morhpisms, i.e.\  $(I)$ is the partition into the set itself and the second level partitions are given by the elements of $I$.
Notice that in this case $\phi\simeq\psi$ even in the comma-category via $(\tau^{-1}\sigma,\hat\tau^{-1}\circ\hat\sigma)$. We also recall that since we are indexing our tensor products by sets, there is always the permutation action of the automorphisms of this set.

The one--comma generating maps are those corresponding to
the morphisms from any morphism $\phi$ to a morphism $\chi\in (\F\downarrow \V)$.
The other axioms of a Feynman category are easily checked.

A set of generators for the one--comma generating maps is given by rooted two level trees. The root is on the $0$ level and the leaves are on the levels above. If we decorate the root by $\phi_0$ and the vertices at level $1$ by the $\phi_v$, then these morphisms  are morphism from $\phi=\phi_0\otimes \bigotimes_v \phi_v$ to $\phi_1=\bigotimes_v\phi_v\circ \phi_0$.

\subsubsection{$\F^+$-$\opcat$.}
After passing to the equivalent strict Feynman category, an element $\CalD$ in $\F^+$-$\opcat$ is a symmetric monoidal functor that has values on each morphism $\CalD(\phi)=\bigotimes\CalD(\phi_v)$ and has composition maps $\CalD(\phi_0\otimes \phi)\to \CalD(\phi_1)$ for each decomposition $\phi_1=\phi\circ\phi_0$. Further decomposing $\phi=\bigotimes \phi_v$ where the decomposition is according to the target of $\phi_0$, we obtain morphisms
\begin{equation}
\label{hypercompeq}
 \CalD(\phi_0)\otimes \bigotimes_v\CalD(\phi_v)\to \CalD(\phi_1)
 \end{equation}

It is enough to specify these functors for $\phi_1\in (\F \downarrow \V)$ and then check associativity for triples.

\begin{ex}
If we start from the tautological  Feynman category on the trivial category $\FF=(1,1^{\otimes}, \imath)$ then $\FF^+$ is the Feynman category $\FSurj$ of surjections. Indeed the possible trees are all linear, that is only have 2--valent vertices, and there is only one decoration. Such a  rooted tree is specified by its total length $n$ and the permutation which gives the bijection of its vertices with the set $n_i$. Looking at a forest of these trees we see that we have the natural numbers as objects with morphisms being surjections.
\end{ex}
\begin{ex}
We also have $\FSurj^+=\operads_{May}$, which is the Feynman category for  May operads. Indeed the basic maps (\ref{hypercompeq}) are precisely the composition maps $\gamma$. To be precise, these are May operads without units.

\end{ex}

\subsection{Feynman hyper category $\FF^{hyp}$}
\label{hypersec}
There is a ``reduced'' version of $\FF^+$ which is central to our theory of enrichment. This is the universal Feynman category through which any functor $\CalD$ factors, which satisfies the following restriction
{\it  $\CalD(\sigma)\simeq\unit$ for any isomorphism $\sigma$ where $\unit$ is the unit of the target category $\CalC$.}

For this, we invert the morphisms corresponding to composing with isomorphisms.
That is for all $\phi\colon X\to Y$ and $\sigma\in Aut(X)$ or respectively $\sigma'\in Aut(Y)$, we invert the morphisms
$\phi\otimes \sigma\to \phi\circ\sigma$ and respectively $\sigma'\otimes \phi\to \sigma'\circ \phi$ corresponding to the partitions yielding the composition.
And we furthermore add isomorphisms $id_{\emptyset}\stackrel{\sim}{\to} \sigma$ for any $\sigma \in \V$
where $id_{\emptyset}$ is the identity unit of $\F$ which is the empty word in $\V$.
This gives us an action of the isomorphisms of $(\F\downarrow \F)$. Let $\boldsymbol \sigma=(\sigma^{-1},\hat\sigma^{-1})$ be such an isomorphism from $\phi$ to ${}^{\boldsymbol \sigma}\phi=\sigma\circ\phi\circ\hat\sigma$.
Decomposing say $\sigma\otimes \hat \sigma=\bigotimes_v \sigma_v$ as usual and denoting the tree corresponding to $\hat\sigma\circ\phi\circ\sigma$ this gives us the following diagram
\begin{equation}
\label{hypcompateq}
\xymatrix{
\phi\otimes \bigotimes_v \sigma_v\ar[rr]^{\CalD(\tau)}_{\sim}&&{}^{\boldsymbol \sigma} \phi_0\\
\phi\otimes \bigotimes_v \emptyset
\ar[u]^{\sim}&& \ar[ll]_{id \otimes \bigotimes_v r_{\emptyset}}^{\sim}\phi\ar[u]_{\boldsymbol \sigma}
}
\end{equation}
where the right morphisms are given by the new isomorphisms, the top morphism is invertible by construction and
the bottom morphism is the unit constraint.

\subsubsection{$\F^{hyp}$-$\opcat$}
\label{hypopsec}
An element $\D\in \F^{hyp}$-$\opcat$ corresponds to the data of functors from $Iso(\F\downarrow\F)\to \C$ together with morphisms (\ref{hypercompeq}) which are associative and satisfy the
equality induced by (\ref{hypcompateq}), i.e. the following diagram commutes:
\begin{equation}
\label{isoactioneq}
\xymatrix{
\CalD(\phi)\otimes \bigotimes_v \CalD(\sigma_v)\ar[rr]^{\CalD(\tau)}_{\sim}&&\CalD({}^ {\boldsymbol \sigma}\phi)\\
\CalD( \phi)\otimes \bigotimes_v \unit  \ar[u]^{\sim}&& \ar[ll]_{id \otimes \bigotimes_v r^{-1}_{\unit}}^{\sim}\D(\phi)\ar[u]_{\CalD( \boldsymbol \sigma)}
}
\end{equation}

\begin{ex}
The paradigmatic examples are hyper--operads in the sense of \cite{GKmodular}. Here $\FF=\modular$ and $\FF^{hyp}$  is the Feynman category for hyper--operads.
\end{ex}

\subsubsection{Relation of $\F^{\prime op}$-$\opcat$ and $\F^{hyp}$-$\opcat$}
Any $\D\in \F^{hyp}$-$\opcat$ can be viewed as a universal $\F'$-$\opcat$ in the following way.
Let $\D'$ be  a symmetric monoidal functor from $\FF^{\prime op}$ and $\D$ such a functor from $\F^{hyp}$.

Looking at isomorphisms only, the two categories $Iso(\F')$ and $Iso(\F^{hyp})$ coincide and give the same data.  On objects there is a 1--1 correspondence between functors $\D'$ and $\D$ --- $\D(\phi)=\D'(\phi)$ and
the correspondence on isomorphism comes from the following: isomorphisms in $\FF^{\prime op}$ these give rise to isomorphism $\D'(id)=\unit\to \D(\sigma)$ coming from the square (\ref{moreq}) with top and right arrow $\sigma$ and the other two arrows being $id$.

In $\FF^{\prime op}$ we then have to consider type I morphisms. These  give maps
$\D'(\psi,id)\in Hom(\D(\phi), \D(\phi'))$ where $\phi'=\phi\circ\psi$.
which have to be compatible with composition.

\begin{equation}
    \xymatrix{
    \vdots&\vdots&\vdots\\
    X''\ar[r]^{\phi''}\ar[d]_{\psi'}&Y\ar@{=}[d]&\D(\phi'') \\
      X'\ar[r]^{\phi'}\ar[d]_{\psi'}&Y\ar@{=}[d]& \D(\phi')\ar[u]\\
        X''\ar[r]^{\phi}\ar[d]_{\phi}&Y\ar@{=}[d]&\D(\phi)\ar[u]\\
          Y\ar[r]^{id}&Y&\unit\ar[u]\\
    }
\end{equation}

In particular, picking the top arrow in the square (\ref{moreq}) to be $\psi$, we get a map $(\psi,id)$ from $id$ to $\psi$ and hence
$\D(\psi,id):\unit\to \D(\psi)$, that is an element of $\D(\psi)$. After picking these, we have to pick elements in $Hom(\D(\phi),\D(\phi'))$ which are compatible with composition and isomorphisms. Such a collection of elements defines the functor $\D'$.

Such a choice can be made automatically, if one has morphisms in
$$Hom(\D(\psi),Hom(\D(\phi),\D(\phi')))\simeq
Hom(\D(\psi)\otimes \D(\phi),\D(\phi'))$$ which are consistent with composition and compatible with isomorphisms. A functor from $\F^{hyp}$ provides exactly this type of universal information. That is, we can get the collection by choosing
a coherent set of elements for a generating set of morphisms and then use the structure maps of $\D$ to choose all the morphisms for $\D'$.

\subsubsection{A reduced version $\FF^{hyp,rd}$}  One may define $\FF^{hyp,rd}$, a Feynman subcategory of $\FF^{hyp}$ which is equivalent to it by letting $\F^{hyp, rd}$ and $\V^{hyp, rd}$ be the respective subcategories whose objects are morphisms that do not contains isomorphisms in their decomposition. In view of the isomorphisms $\emptyset\to \sigma$ this is clearly an equivalent subcategory.  In particular the respective categories of $\opcat$ and $\op{M}ods$ are equivalent.

The  morphisms are  described by rooted forests of  trees whose vertices are decorated by the $\phi_v$ as above --none of which is an isomorphism--, with the additional decoration of an isomorphism
per edge and tail. Alternatively one can think of the decoration as a black 2-valent vertex.
Indeed, using maps from $\emptyset \to \sigma$, we can introduce as many isomorphisms as we wish. These give rise to 2--valent vertices, which we mark black. All other vertices remain labeled by $\phi_v$.  If there are sequences of such black vertices, the corresponding morphism is isomorphic to the morphism resulting from composing the given sequence of these isomorphisms.

\begin{ex}
For $\FF^{hyp,rd}_{surj}=
\operads_0$, the Feynman category whose morphisms are trees with at least trivalent vertices (or identities) and whose $\opcat$ are operads whose $\O(1)=\unit$. Indeed the basic non--isomorphism
 morphisms are the surjections $\bar n\to \bar 1$, which we can think of as rooted corollas. Since for any two singleton sets there is a unique isomorphism between them, we can suppress the black vertices in the edges. The remaining information is that of the tails, which is exactly the map $\phi^F$ in the morphism of graphs.
\end{ex}
\begin{ex}
For the trivial Feynman category, we obtain back the trivial Feynman category, since the trees all collapse to a tree with one black vertex.
\end{ex}

\section{Indexed Enriched Feynman categories,  (odd) twists and Hopf algebras}
\label{enrichedsec}

In this chapter, we will introduce the indexed enriched version of Feynman categories.
This is needed in order to define the transforms of the next section.
These are given as dg-objects and we need non--Cartesian enrichment.

We show that there is a natural way to do this using enrichment indexed by an $\F^+$--$\oper$.
This allows us to recover the theory of twists of \cite{GKmodular} for the special case $\FF=\modular$ and generalize it to arbitrary Feynman categories.

A guiding motivation are the graph based examples.
For these ---much like in the case of simplicial (co)homology---
this allows us to pass from orientations to Abelian groups, where an orientation change induces a minus sign.

In its full generality, besides being a useful
extension of the notion of Feynman categories,  the construction also makes it possible to treat algebras over
$\CO\in\F^{hyp}$--$\opcat$ as $\opcat$ for an $\CO$ based--enriched version of $\FF$. The prime examples being algebras over operads.

%Another application are twisted modular operads, in their traditional form \cite{GKmodular}, here $\FF=\modular$.

We end the section with an application by presenting bi-algbera and Hopf structures defined by Feynman categories enriched over $\Ab$. This is a generalization of the Hopf--algebra  of Connes--Kreimer \cite{CK} as suggested by Kreimer. The details of this are in \cite{GKT}.

%\subsection{Indexed Enriched Feynman categories}

We now also  fix that the coproduct in the enrichment category $\CalE$, which we will denote by $\oplus$,  is distributive with respect to the monoidal structure $\otimes$ in both variables:
 $(X\oplus Y)\otimes Z\simeq (X\otimes Z) \oplus (Y\otimes Z)$ and similarly in the other variable.

\subsection{Enrichment functors}

The cleanest way to define an indexed enrichment is using 2--categories. We will disentangle the definition below. First, we can consider any category $\F$ to be a 2--category with the two morphisms generated by triangles of composable morphisms. If $\CalE$ is a monoidal category, let $\underline{\CalE}$ be the corresponding 2--category  with one object. I.e.\ the 1-morphisms of $\underline{\CalE}$ are the objects of $\CalE$ with the composition being $\otimes$, the monoidal structure of $\CalE$. The 2--morphisms are then the 2--morphisms of $\CalE$, their horizontal composition being $\otimes$ and their vertical composition being $\circ$.

\begin{df}
\label{enrichmentfunctordef}
Let $\FF$ be a Feynman category. An enrichment functor is a lax 2--functor $\D:\F\to \underline{\CalE}$ with the following properties
\begin{enumerate}
\item $\D$ is strict on compositions with isomorphisms.
\item $\D(\sigma)=\unit_{\CalE}$ for any isomorphism.
\item $\D$ is monoidal, that is $\D(\phi \otimes_{\F} \psi)=\D(\phi)\otimes_{\CalE}\D(\psi)$
\end{enumerate}
\end{df}

What this means is that for any morphism $\phi$ we have an object $\D(\phi)\in \CalE$ and
for any two composable morphisms $\phi$ and $\psi$ there is a morphism $\D(\psi)\otimes \D(\phi)\stackrel{\D(\circ)}{\to}\D(\phi\circ \psi)$. Condition (1) states that this morphism is an isomorphism whenever
 $\phi$ or $\sigma$ is an isomorphism. It follows that any $\D(\sigma)$ is invertible for an isomorphism $\sigma$.
 (2) then fixes that this invertible element is $\unit_{\CalE}$.

\begin{prop} An enrichment functor $\D$ for a Feynman category $\F$ corresponds 1--1 to
 $\tilde\D \in \F^{hyp}$-$\opcat_{\CalE}$. Thus these are  equivalent concepts.
\end{prop}
\begin{proof}
We define $\tilde \D(\phi):=\D(\phi)$.
Decomposing the morphism $\phi$, we get the composition morphisms
\begin{equation}
\label{twistingdata}
\D(\psi)\otimes \bigotimes_v \D(\phi_v) \to \D(\phi \circ \psi)
\end{equation}
The conditions (1)--(3) yield  an action of isomorphisms in $(\F\downarrow \F)$ given exactly by the diagram (\ref{isoactioneq}) and hence functors $\tilde\D$ from $Iso(\F\downarrow\F)$ which are automatically compatible.
According to \S\ref{hypopsec} this data fixes $\tilde \D$.
The other direction of the construction is similar.
\end{proof}

\subsubsection{Indexed enrichment}
Given a monoidal category $\F$ considered as a 2--category and lax 2--functor $\D$ to $\underline{\CalE}$ as above, we define an enriched monoidal category $\F_{\D}$ as follows. The objects of $\F_{\D}$ are those of $\F$. The morphisms are given as

\begin{equation}
Hom_{\F_{\CalD}}(X,Y):=\bigoplus_{\phi\in Hom_{\F}(X,Y)}\D(\phi)
\end{equation}
The composition is given by
\begin{multline}
Hom_{\F_{\CalD}}(X,Y)\otimes Hom_{\F_{\CalD}}(Y,Z) =
\bigoplus_{\phi\in Hom_{\F}(X,Y)}\D(\phi)\otimes \bigoplus_{\psi\in Hom_{\F}(Y,Z)} \D(\psi)\\
\simeq
\bigoplus_{(\phi,\psi)\in Hom_{\F}(X,Y)\times Hom_{\F}(Y,Z)} \D(\phi)\otimes \D(\psi)
\stackrel{\bigoplus \D(\circ)}{\longrightarrow} \bigoplus_{\chi\in Hom_{\F}(X,Z)}\D(\chi)=Hom_{\F_{\D}}(X,Z)
\end{multline}
The image lies in the components $\chi=\psi\circ\phi$. Using this construction on $\V$, pulling back $\D$ via $\imath$, we obtain $\V_{\D}=\V_{\CalE}$. The functor $\imath$ then is naturally upgraded to an enriched functor $\imath_{\CalE}:\V_{\D}\to \F_{\D}$.

\begin{df}
\label{enrichedfeydef}
Let $\FF$ be a Feynman category and let $\D$ be an enrichment functor.
We call $\FF_{\D}:=(\V_{\CalE},\F_{\D},\imath_{\CalE})$ a
Feynman category enriched
over $\CalE$ indexed by $\D$.
\end{df}

\begin{thm}
\label{enrichedtriplethm}
$\FF_{\D}$ is a weak Feynman category.  The forgetful functor from $\F_{\D}$-$\opcat$ to $\V_{\CalE}$-$\smodcat$ has a left adjoint and more generally push-forwards among indexed enriched Feynman categories exist. Finally there is an equivalence of
categories between algebras over the triple (aka.\ monad) $GF$ and $\F_{\D}$--$\opcat$.
\end{thm}
\begin{proof} This is a generalization of the arguments of \cite{GKmodular}.
(i') holds by construction and (iii) for $\FF$ implies (iii') for $\FF_{\D}$. It remains to check (ii').
For this, we notice that
\begin{multline}
\int^{Z,Z'}Hom_{\F_{\D}}(\imath^{\otimes}_{\CalE}(Z),X)\otimes_{\CalE} Hom_{\F_{\D}}(\imath^{\otimes}_{\CalE}(Z'),Y)\otimes_{\CalE}Hom_{\V_{\CalE}}(W,Z\otimes Z')\\
%=\int^{Z}Hom_{\F_{\D}}(\imath^{\otimes}_{\CalE}(Z),X)\otimes_{\CalE}Hom_{\V_{\CalE}}(W,Z\otimes Y')\\
=\int^{Z,Z'}\left[\bigoplus_{\phi \in Hom_{\F}(\imath^{\otimes}(Z),X)}\CalD(\phi)\right] \otimes_{\CalE}
\left[\bigoplus_{\phi \in Hom_{\F}(\imath^{\otimes}(Z'),Y)}\CalD(\psi)\right] \otimes_{\CalE}
\left[\bigoplus_{\phi \in Hom_{\V}(W, \imath^{\otimes}(Z)\otimes \imath^{\otimes}(Z'))}\CalD(\sigma)\right]
\\=\int^{Z,Z'}
\bigoplus_{(\phi,\psi,\sigma)\in Hom_{\F}(\imath^{\otimes}_{ }(Z),X)\times_{ } Hom_{\F}(\imath^{\otimes}_{ }(Z'),Y)\times_{ }Hom_{\V_{ }}(W,Z\otimes Z')}\D(\phi\otimes\psi\otimes \sigma)
%\int^{Z,Z'} (D\circ \tilde\imath^{\otimes}X)(Z)
\\=\bigoplus_{\chi\in Hom_{\F}(W,X\otimes Y)} D(\chi)=Hom_{\F_{\D}}(W,X\otimes Y)
\end{multline}
where we have used the conditions (1)-(3) on $\D$ to obtain the first equalities.  Then we
used  (ii') for $\F$ to perform the co--end. This replaces  the discrete category over which the inner ordinary colimit runs by a colimit over  representatives.
%where we used the notation \cite{kellybook}:
% $\tilde\imath^{\otimes}_{\CalE}: \F_{\D}\to Fun(\V_{\CalE}^{op},\CalE)$ given by
 %$\tilde\imath_{\CalE}^{\otimes}X=Hom_{\F_{\D}}(\imath^{\otimes}_{\CalE}(-),X)$
Then the analogue of Theorem \ref{freethm}, Theorem \ref{pushthm} and Theorem \ref{triplethm} follow readily.
\end{proof}

\begin{prop}
If $\CalC$ is tensored over $\CalE$ then the triple $\T_{\D}$  on $\V_{\D}$-$\smodcat_{\CalC}$
encoding $\F_{\D}$-$\opcat_{\CalC}$ satisfies
$\T_{\D}=\T\otimes \D:=\colim_{Iso(\imath^{\otimes} \downarrow \ast)} \D\otimes \O\circ s$, where $\T$ is the triple encoding $\F$-$\opcat_{\CalC}$.
\end{prop}
\begin{proof}
We have to show that the triple $\T$ for the Feynman category given by $\D$ has the form given in \cite{GKmodular}.
We denoted $G\O$ simply by $\O$ and as above let $B$ be the base functor $B:\F_{\D}\to \F$ which is identity on objects.
Then,
\begin{multline*}\T_{\D}(\ast)=\CalE(\imath_{\D}^{\otimes}  -, \ast)=\int^{X\in \imath^{\otimes}_{\D}\V^{\otimes}} \CalE(\imath_{\D}^{\otimes} X,\ast)\otimes \O(X)
\\
=\int^{X\in \imath^{\otimes}_{\D}\V^{\otimes}} \bigoplus_{\phi\in (B(X),B(\ast))}\D(\phi)\otimes \O(B(X))
=
 \colim_{Iso(\imath^{\otimes} \downarrow \ast)} \D\otimes \O\circ s\\
 \end{multline*}
where in the last line we consolidated the colimits, and $\D$ is considered as a functor by restricting it to the subcategory
with fixed target and isomorphism of the source; see also below.
\end{proof}
%\begin{rmk}
 %If there is a free construction, we can simply use associativity of colimits.
%\end{rmk}

\begin{rmk}
This can be rephrased without the assumption of being tensored over $\CalE$ in terms of weighted colimits as follows:
Consider  $Hom_{\mathcal B}(X, Y)$ as a discrete sub-category of $Iso({\mathcal B}\downarrow {\mathcal B})$ with the trivial
 indexing $\I\colon Hom_{\mathcal B}(X, Y)^{op}\to \CalE$
given by $\phi\mapsto \unit_{\CalE}$.  Then the definition postulates functors
 $\D_{X,Y}\colon Hom_{B(\F)}(B(X),B(Y))\to \CalE$ such that $Hom_{\F}(X,Y)=\I\star \D$, where $Hom_{B(\F)}(B(X),B(Y))$ is the full subcategory of $(\F\downarrow \F)$ with objects $B(X)$ and $B(Y)$ and we use the notation of \cite{kellybook}.
 The formula above then corresponds to:
$$
Lan_{i^{\otimes}}\O= \CalE(\imath_{\D}^{\otimes}  -, -)\star \O= (\I \star \D)\star \O=\I\star(\D\star\O)
$$
\end{rmk}

\subsubsection{$\V$--twists}
Given a functor $\fL\colon\V\to Pic(\CalE)$, that is the full subcategory of $\otimes$-invertible elements of $\CalE$, we can define a twist of a Feynman category indexed by $\D$ by setting the new twist-system to be $\D_{\fL}(\phi)=
\fL(t(\phi))^{-1}\otimes \D(\phi)\otimes \fL(s(\phi))$.

For the composition $\phi\circ\psi$ we just use $\fL(t(\psi))^{-1}\otimes \fL(s(\phi))=\fL(t(\psi))^{-1}\otimes \fL(t(\psi))\stackrel{\sim}{\to}\unit$ and the unit morphisms in $\CalE$. Associativity and compatibility with isomorphisms is clear.

\begin{prop}
If $\CalC$ is tensored over $\CalE$ there is an equivalence of categories between $\F_{\D}$--$\opcat_{\C}$ and
$\F_{\D_{\fL}}$--$\opcat_{\C}$ given by tensoring with $\fL$. That is, $\CO(-)\mapsto \fL^{-1}\otimes\CO(-)$.
\end{prop}

\begin{proof}
We see that $\CO(X)\mapsto\fL(X)^{-1}\otimes \CO(X)$ and hence for $\phi\colon X\to Y$,
$\CO_{\D}(\phi)\mapsto
 \fL(Y)^{-1}\otimes \D(\phi)\otimes \fL(X)=\CO_{\D_{\fL}}(\phi)$.

\end{proof}

\begin{ex}
We recover the original definitions of hyperoperads and twisted modular operads
 of \cite{GKmodular} if we use $\CalE=\CalC=\gVect$
and $\FF=\modular$.
%Since we are indexed over $\mathfrak{G}$, for a given morphism $\phi$, we set $B(\phi)=\bar \phi$ and let
%$\gh=\gh(\bar \phi)$ together with its genus marking.

In particular, since  any morphism $\phi$ is determined (up to isomorphism) by the $\gh_v(\phi)$
we get the form (equivariant with respect to isomorphism)
\begin{equation}
\label{GKeq}
\bigotimes_v \D(\gh_v)\otimes \D(\gh_0)\to \D(\gh_1)
\end{equation}
just like above, the functor is determined on the one--comma generators, restricts to isomorphisms,
is associative and unital.

Indeed $\O\in \modular_{\D}$--$\opcat$ is an algebra over the $\D$-twisted modular operad in the sense of \cite{GKmodular}.
\end{ex}
\begin{ex}
If we start with $\Surj:=\FF_{surj}$ then $\FF^{hyp,rd}=\operads_0$, that is the Feynman category for operads with trivial $\O(1)$ and hence an indexed enrichment is equivalent to a choice of such an operad $\O$. As we show below, these are exactly all enriched indexed Feynman categories
with trivial $\V_{\CalE}$. The $\F_{surj,\CO}$--$\opcat$ are just algebras over operads with reduced $\O(1)$.
\end{ex}

We will return to this example in subsection $\ref{moreexamples}$.  Motivated by this example we define:

\begin{df}
Let $\FF$ be a Feynman category and $\FF^{hyp,rd}$ its reduced hyper category, $\O$ an
$\FF^{hyp,rd}$-$\oper$ and $\D_{\O}$ the corresponding enrichment functor. Then we define an $\O$-algebra to be a $\FF_{\D_{\O}}$-$\oper$.
\end{df}
\subsection{Indexed Feynman $\Ab$-Categories:  Orientations and Odd $\opcat$}\label{odd42}
Besides the free construction there are standard constructions for the oriented
and ordered versions which give an indexing for an  $\Ab$ enrichment. These are obtained from the free construction by a relative construction.

\subsubsection{Oriented/ordered versions}
In the edge oriented examples of Feynman categories enriched over $\GG$ we define the
associated indexed $\Ab$-structure as follows:
Use the quotient category of the  free $\Ab$-construction above
with the relations $\sigma\sim -\bar\sigma$.
Here $\sigma$ is an orientation and $\bar\sigma$ is the opposite orientation.

Alternatively this means that one takes $\Z[Hom(X,Y)]\otimes_{\Z/2\Z} \Z/2\Z$, with $\Z/2\Z$ acting
by orientation reversal, as the enriched $Hom$. In the ordered version we can take
$\bigoplus_k (\Z[Hom(X,\ast)_k])\otimes_{\SS_{k}} \Z/2\Z$ where $Hom(X,\ast)_k$ is the degree $k$ part, that is those
$\phi$  with $|E_{\gh(\phi)}|=k$ and $\SS_{k}$ acting by permuting the edges and the action on $\Z/2\Z$ is the sign
representation. Both the ordered and the oriented construction lead to the same Feynman Ab-category.

Here the relevant functor $\D$ on the underlying non-edge oriented category is $\D(\phi)=sign(E_{\gh(\phi)})$,
the sign representation on the set of edges of the ghost graph.

This works the same way in $\Vect$ and $\gVect$. These constructions factor through the smaller Ab-enrichment.

\begin{df}
Given a Feynman category $\FF$ indexed over $\GG$, we define the odd version $\FF^{odd}$
to be the indexed Feynman $\Ab$-category resulting from the above operation on $\FF^{or}$ or equivalently on $\FF^{ord}$.
\end{df}

This allows us in one fell swoop to define odd versions for any of the graph examples
such as odd operads, odd PROPs, odd modular operads, odd NC modular operads, etc. that were treated in \cite{KWZ},
which are responsible for Lie brackets and BV operators, etc..  This realization allows us to define a Feynman transform/cobar construction in great generality.
See \S \ref{feytranspar}.
\subsection{Examples}\label{moreexamples}
We will discuss two types of examples. First
the algebra type examples given in \S\ref{algsec} can be generalized to algebras over operads.
We can also consider Lavwere theories, which amounts to adding degeneracies.

Both constructions are reminiscent of the original definition of a PROP \cite{PROPref}.
We consider $\V$ a trivial category, that is a category with one object $1$, enriched over $\C$, i.e.\ $Hom_{\V}(1,1)=\unit_{\C}$.
Then  let $\V^{\otimes}$ be the  free symmetric category and let $\bar\V^{\otimes}$ be its strict version. It has  objects $n:=1^{\otimes n}$ and each of these has an $\Sn$ action.  Let $\imath$ be the functor from $\V^{\otimes}\to\bar\V^{\otimes}$, and $\F$ a category with $Iso(\F)=\bar\V^{\otimes}$.

\begin{rmk}
The possible category structures with underlying $\V^{\otimes}$ from above as an underlying category are by their original
definition PROPs $P(n,m)=Hom(n,m)$, see \cite{PROPref}.
Additionally imposing the conditions of a Feynman category,
we restrict to only those PROPs which are actually generated by operads with $\O(1)$ having only constants as units (i.e. just a factor of $\unit$ corresponding to $id_1$) as we explain in the next sections.
\end{rmk}

\subsubsection{Trivial $\V$: aka algebras over operads}
\label{operfeypar}
To give a detailed example, we consider the possible weak Feynman categories $(\V,\F,\imath)$ with $\V$ a trivial category over $\C$ such that $\bar \V^{\otimes}\to \F$ is essentially surjective. Given such a category, if we set
$\O_{\F}(n):=Hom_{\F}(n,1)$ then
this is a (unital) operad. Indeed, composition $\gamma:Hom_{\F}(n,k)\otimes Hom_{\F}(k,1)\to Hom_{\F}(n,1)$ gives the structure of a May operad due to the hereditary condition and the identity in $Hom_{\F}(1,1)$ gives a unit. The $\Sn$ action is given by composing with $Aut(n)\simeq \Sn$.  If $\O(1)\neq\unit$ then the colimits defining the Kan extensions will not
be monoidal.
Thus we see that the invertible elements in $\O(1)$ are only the constants.

Vice-versa given an operad $\CO$ in $\CalC$ with trivial $\O(1)$, that is $\CO\in \operads$-$\opcat_{\CalC}$, let $\V$ be the trivial category enriched over $\C$ and
$\F=\bar\V^{\otimes}$ which is additionally one--comma generated by $Hom_{\F_{\CO}}(n,1):=\O(n)$ and the composition given by the operadic composition
$\gamma\colon\O(k)\otimes \O(n_1)\otimes\cdots \O(n_k) \to \O(\sum_{i=1}^k n_k)$.
Notice that we need the $\Sn$ action on $\O(n)$, in order to accommodate the action of precomposing with an isomorphism.
$\FF_{\CO}:=(\V,\F_{\CO},\imath)$ is a Feynman category and more precisely the Feynman category $\FSurj_{\CO}$.

Now $\F$-$\opcat$ are exactly algebras over the operad $\O$. Explicitly, let $\rho$ be a monoidal functor, then
it is fixed by $\rho(1)=X$ to  $\rho(n)=X^{\otimes n}$. For the morphisms we get $\rho:\O(n)\to Hom_{\CalC}(X^{\otimes n},X)$.
Since the functor is symmetric monoidal, the $\Sn$ action on $\O(n)$ is compatible with the permutation action on the factors of $X^{\otimes n}$.
We can also disregard the $\Sn$ action by no longer using the symmetric monoidal structure and obtain algebras over non--$\Sigma$ operads.

If  we consider a lax 2--functor $\D$, we see that we can also get operads which are a) non--unital and b) have non--trivial $\O(1)$. Likewise, we obtain non--trivial $\O(1)$ from
considering $\Surj^+$.

\begin{ex}\mbox{}

\begin{itemize}
\item Associative algebras in $\CalC$.  $\F(n,1)=\Sn=Aut(\unit^{\otimes n})$
 and the composition is given by the inclusion of products of symmetric groups. We will call the Feynman category $\Fass$. Notice that this is  defined without enrichment or equivalently with the trivial enrichment.
\item Commutative algebras in $\CalC$. In this case,
$\F(n,1)=\unit$ with the trivial $\SS_n$ action.
\item Lie and pre-Lie algebras. Notice that these are only defined for categories which are enriched over $\Ab$.
We will call the Feynman categories $\Flie$ and $\Fprelie$.
\end{itemize}

\end{ex}

\subsubsection{Graph insertion}
\label{graphinsertionsec}
Other examples that are useful as indexing Feynman categories are those coming from
graph insertions. The standard one is given by trivial $\V$, $Iso(\F)=\bar\V^{\otimes}$ as above
and the underlying operad given by $\O(n)$ graphs without tails and $n$-labeled vertices up to isomorphism.
The $\Sn$ action permutes the labels. The composition is given by the insertion composition for unlabeled
graphs --- see \ref{nolabcompsec}.
This example can be altered/modified by choosing different types of graphs, putting restrictions on the grafting
or the labeled vertices. For example the Feynman category $\Fprelie$ is the one obtained by choosing
rooted trees \cite{CL} with the operad structure given by insertion.

%???what was the meaning of the example

\subsection{A Connes--Kreimer style bi-algebra/Hopf algebra structure}
\label{hopfsec}
The following observation is essentially due to D. Kreimer. Applications of this theory and further details can be found in \cite{GKT}. Here we will give a brief survey.
\label{CKsec}
Consider a non--symmetric Feynman category $\FF$.
Let $\FF$ be enriched over $\mathcal{A}b$ and assume
that $\F$ is skeletal. Consider the direct sum of $Hom_\F(X,Y)$, denoted $B=\coprod_{X,Y\in \F} Hom(X,Y)$. Assume that $\F$ is composition finite, that is given any morphism $\phi$ there are only finitely many pairs $(\psi,\phi_0)$ s.t.\ $\phi=\psi\circ \phi_0$.
\begin{prop}\cite{GKT}
 $B$ has the structure of a bi-algebra, with multiplication given by the monoidal structure $\otimes$ of $\F$ and
the coproduct given by
\begin{equation}
 \Delta(\phi)=\sum_{(\psi,\phi_0):\phi=\psi\circ \phi_0}\phi_0 \otimes \psi
\end{equation}
where the sum is over all decompositions.
%and $\psi=\bigotimes_v \phi_v$ in a suitable choice of skeleton.
Furthermore defining $\eps$ by setting $\eps(id_X)=1$ and $\eps(\phi)=0$ otherwise gives a co--unit. If $\F$ is strict monoidal $\id_{\unit_\F}$ is a unit for the multiplication.
 If we use the grading which assigns to $\phi:X\to Y$ the degree (length) $l(\phi)=|X|-|Y|$, then the bialgebra is graded.
\end{prop}

\begin{proof} {\it (abbreviated)} \
 The fact that the coproduct is co-associative follows from associativity of the composition.  The fact that it is compatible with the product follows from the hereditary property (ii). In particular writing out the bi--algebra equation explicitly, each turns into a sum over diagrams, where the first is a sum over diagrams on the left  in \eqref{diagramseq} and the second is a sum over diagrams on the right  in \eqref{diagramseq}.

 \begin{equation}
 \label{diagramseq}
\xymatrix{X'\otimes X''\ar[rr]^{\phi'\otimes \phi''}\ar[dr]_{\phi_0'\otimes \phi_0''}&&Y'\otimes Y''\\
&Z'\otimes Z''\ar[ur]_{\phi'_1\otimes \phi''_1}&\\
}
\ \ \ \ \hfill \ \ \ \
\xymatrix{X'\otimes X''\ar[rr]^{\phi'\otimes \phi''}\ar[dr]_{\phi_0}&&Y'\otimes Y''\\
&Z\ar[ur]_{\phi_1}&\\
}
\end{equation}

 That these two sums coincide then can be seen from Lemma \ref{iterativelem}. The unit and counit are verified in a straightforward fashion.
\end{proof}

Note that the one--comma generators will be generators for the Bi--=algebra.
The general elements will be products of these. This is reflected in the fact that the Hopf algebra of Connes and Kreimer is a Hopf algebra of forests and not of trees. The trees  correspond to the one--comma generators, see below.

If $\FF$ was not enriched over $\Ab$, we can use the free enrichment \S\ref{freeenrichedpar}. This is equivalent to regarding the free Abelian group on the morphisms.

\subsubsection{Bi--algebra on coinvariants}
In the case that $\F$ is symmetric, and not necessarily skeletal, one has to be careful due to the symmetric structure.
Let
 $B^{iso}=colim_{Iso(\F \downarrow \F)}Hom(-,-)$. The means that $\sigma\circ f\circ \sigma'\sim f$
 for any two isomorphisms $\sigma,\sigma'$.

Let $\C$ be the ideal generated by elements $f-g$ with $f\sim g$. Then one can show, \cite{GKT} that
\begin{equation}
\Delta(\C)\subset B\otimes \C + \C\otimes B
\end{equation}
and hence the coproduct and product descend to $B/\C$.

Notice that when we mod out by isomorphisms,  as the commutativity constraints now
act as identities, the product becomes commutative and the tensor products are replaced by the symmetric product. For instance, if we decompose $\psi=\bigotimes_v\phi_v$,
then we obtain that the class $[\psi]$ is actually the symmetric product $\bigodot_v[\phi_v]$. There are still residual automorphisms even in the skeletal case given by the stabilizer groups of the chosen decomposition. For instance, for any $\ast\in \V$ the object $\imath(\ast)^{\otimes n}$ will have automorphism group  given by the wreath product $Aut(*)\wr \SS_n$.

Secondly,  one has to be careful with the automorphism groups as $Aut(Z)\otimes Aut(Z')\subset Aut(Z\otimes Z')$ may be a proper subgroup in the symmetric case. This means that the diagrams above are only in 1--1 correspondence after passing to coinvariants.

The short version to fix this problem is to consider ``channels'' of a decomposition, see \cite{GKT} for more details. For this call two decompositions of $\phi$ given by
$(\psi,\phi_0)$ and  $(\hat\psi,\hat\phi_0)$ equivalent if $[\phi_0]=[\hat\phi_0]$
and $[\psi]=[\hat\phi]$. This implies that there are
 isomorphisms $\sigma,\sigma',\sigma''$  such that $\phi=\psi\circ\phi_0=\sigma\hat\psi\sigma'\hat\phi_0\sigma''$.
%For an element/equivalence class $[\phi]\in B^{iso}$.

\begin{equation}
 \Delta([\phi])=\sum_{[(\psi,\phi_0)]:[\phi]=[\psi\circ \phi_0]}[\phi_0] \otimes [\psi]
\end{equation}
where the sum is over a system of representatives of equivalence  classes of decompositions.

\begin{prop}\cite{GKT}
 Together with the symmetric product as multiplication, the coproduct above gives $B^{iso}$ the structure of a bi-algebra. One can also define a counit, and the unit is the image of $id_\unit$. It inherits the grading from $B$.
\end{prop}
\begin{proof}
 This follows from a calculation. It can also be viewed as the quotient by the co--ideal, see \cite{GKT}.
 The grading descends since isomorphisms have $l(\sigma)=0$.
\end{proof}
\begin{rmk}
With hindsight, we discovered that the co-algebra structure can be traced back to at least \cite{JR}. There one can also find the idea to quotient out by co-ideals, say those generated by the equivalence under isomorphisms, although some modifications are necessary.
\end{rmk}
\subsubsection{Hopf algebra stucture}
Usually the bi-algebras above are not connected as all $[id_X]$ are group--like.
One can take another quotient by the graded coideal generated by ${\mathfrak C}=|Aut(X)|[id_X]-|Aut(Y)|[id_Y]\subset B$ in the non--symmetric case and ${\mathfrak C}=[id_X]-[id_Y]\subset B^{iso}$ in the symmetric case.
\begin{df}
We call a Feynman category  almost connected if $H=B^{iso}/{\mathfrak C}$ is connected in the symmetric case and if $H=B/{\mathfrak{C}}$ is connected in the non--symmetric case.
\end{df}

Thus, an almost connected Feynman category yields a graded connected bialgebra $H$, which is a Hopf algebra.

\begin{ex}
If we are in the properly graded Feynman category situation, see \S\ref{grfeysec}, then there is a  Hopf algebra structure on
$H$ since all non--isomorphism have degree greater that $0$.
\end{ex}
\subsubsection{Hopf algebra on Graphs}
If we look at Feynman categories indexed over $\GG$ we see that the decomposition is basically in terms of graphs and
subgraphs. But we have to be a little careful, see \S \ref{ggpar}. This nicely illustrates
the difference between the two constructions above. For instance the $\phi_v$ are not just the subgraphs $\gh_v(\phi)$
but also carry a labeling of the vertices via their source and target maps.
Forgetting this forces one to work with symmetric products.

\begin{prop}
 Considering ghost graph morphisms for $\GG^{ctd}$, $\Delta$ induces a bi-algebra structure on graphs. Furthermore dividing by the coideal $\mathfrak{C}$, and by grading with the number of edges, this bi-algebra is connected and hence a Hopf algebra with the unique antipode.
\end{prop}

\begin{proof}
 Indeed as discussed in \S\ref{ggpar}, the extra data needed to get from $\gh(\phi)$ to $\phi$, knowing the source and target, are the isomorphisms of $\phi^F$ restricted to its image. Since we are connected, the morphism $\phi_V$
is automatic. One lift for the restriction of $\phi^F$ is to fix strict identity
for the isomorphism. This, however, still leaves the order of the vertices. Recall that in $\V^{\otimes}$
the order matters. We can kill this by formally using symmetric products of graphs as done in Connes--Kreimer.
Indeed we then lift a symmetric product to a morphism by using identity for the isomorphism, use $\Delta$ as above and
project back to the symmetrized ghost graphs. The monoidal structure then becomes the symmetric product.
Since we kill isomorphisms this way, the quotient bi-algebra graded by the number of edges is indeed connected.

\end{proof}

\begin{rmk} \mbox{}
\begin{enumerate}
\item Instead of the condition of being connected, we could have imposed the condition of being one particle irreducible (1-PI),
which is also hereditary, see \cite{GKT}. This is a typical condition in physics.
 \item
Notice that the above Proposition easily translates to the indexed case.
%\item
\item If we use $\FF_{\O}$, with $\O$ the operad of trees with labeled leaves and
gluing at leaves, we arrive precisely at the Hopf algebra of rooted trees  of Connes--Kreimer \cite{CK} in the version with tails. For the version without tails, see below.
This can be seen as coming naturally from the construction above to the Feynman category of $\operads^{pl}$ or $\operads$, see \cite{GKT}.
\end{enumerate}

\end{rmk}

\subsubsection{Other variants: Connes--Kreimer algebra of trees}
There are actually several variants of this construction. % We refer to \cite{GKT} for details.

The paradigm to understand this situation is the Connes--Kreimer Hopf algebras over trees.
There are planar and non--planar versions. We have treated the non--planar version. For the planar version one has to start with a non-symmetric Feynman category.

A second variation is that  there are versions for trees with leaves and without for the Hopf algebra. So far, we have only dealt with the version with leaves. Here one has to distinguish between labeled leaves and unlabeled leaves. The version with labeled leaves corresponds to quotienting $B$ by a co-ideal. In order to forget the labels, one can use coinvariants and this is the construction via $B^{iso}$ in general.
Finally to get rid of leaves one needs the opposite of the foiliation operator of \cite{del}. This is given by taking a colimit over the semi--simplicial structure given by removing flags, see \S\ref{simppar}. This yields the Connes--Kreimer Hopf algebra of rooted trees (forests).

More generally, in the case of graphs, or Feynman categories indexed over graphs, there is usually also the additional structure of forgetting tails. This can be formalized further, see \cite{GKT}. If this additional structure is present, one can take  the colimit to obtain a bi-- and Hopf--algebra of graphs without tails. In the cases above  the result coincides with
 the Hopf algebras of Connes--Kreimer, see \cite{GKT} for details.

\section{Feynman categories given by generators and relations}
\label{genrelpar}
The monadicity theorem gives two ways of defining $\F$-$\opcat$. Often there is a third way, given by generators and relations. For example, operads can be defined by the $\circ_i$ operations and $\SS_n$ actions along with associativity and compatibility relations. In this paragraph, we consider the general setup for this.
Besides being of separate interest, this type of presentation is what we need to define ordered/oriented/odd versions of Feynman categories that
are essential for the definition of (co)bar/Feynman transforms.
\subsection{Structure of $\GG$}
\label{graphstrucsec}
We will first study $\GG$ as an archetypical example. It is basically generated by four types of operations:
simple edge contractions, simple  loop contraction, simple mergers and  isomorphisms.

\subsubsection{Generators}

There are useful numerical invariants for $\GG$:
Let $\deg(\ast_v)=\wt(\ast_v)=|F(\ast_v)|$ be the degree and weight of $\ast_v$
 and if $X\simeq \amalg_w \ast_w$ set $\wt(X)=\sum_w(\wt(*_w)+1)$
and $\deg(X)=\sum_w \deg(\ast_w)$, which is just the number of flags.
The degree and weight of $\phi\colon X\to Y$ are then defined as
\begin{equation}
\deg(\phi)=\frac{1}{2} (\deg(X)-\deg(Y)) \quad \wt(\phi)=\wt(X)-\wt(Y)
 \end{equation}
It is clear that the degree and weight are additive under composition.
Note that the degree is actually an integer and $\deg(\phi)=|E(\gh(\phi))|$.

\begin{prop}
\label{GGstructurelem}
All morphisms $\phi$ from the comma category $(\Agg \downarrow \Crl)$ can be factored into
morphisms of the following four types, which we call simple.
\begin{enumerate}
\item {\em Simple edge contraction.} The complement of the image $\phi^F$ is given by two flags $s,t$,
which form a unique ghost edge,
and the two flags are not adjacent to the same
vertex. This has degree $1$ and
weight $3$. We will denote this by $\scirct$.

\item {\em Simple loop contraction.} As above, but  the two flags of the ghost edge
are adjacent to the same vertex, this is called a simple loop contraction
and it has degree $1$ and weight $2$. We will denote this by $\circ_{st}$.

\item {\em Simple merger.} This is a merger in which $\phi_V$ only identifies
two vertices $v$ and $w$. Its degree is $0$ and the weight is $1$. We will denote this by $\mge{v}{w}$.

\item {\em Isomorphism}. The pure isomorphisms are of degree and weight $0$.
\end{enumerate}
 This factorization nor its length are in general unique, but
there is a minimal number $|\phi|$ of morphisms a factorization contains.
If $\phi$ is an isomorphism this number is $1$ and if it is as above and not an isomorphism this number is:
\begin{equation}
|\phi|=\sum_{v\in V}( |E(\gh_v)|+|\pi_0(\gh_v)|-1)
 \end{equation}
In addition:
\begin{enumerate}
\item
Any such minimal factorization is obtained by first contracting all ghost edges of $\phi$ step by step and
then performing mergers step by step and then performing an isomorphism.
\item Any such minimal factorization is uniquely determined by an order on the ghost edges of $\phi$ and
an order on the components of $\gh(\phi)$.
\end{enumerate}

\end{prop}
\begin{proof}
Given a morphism $\phi$ as above, we choose an order for the vertices $v\in V$,
an order for the components of $\gh_v:=\gh_v(\phi)$ and finally
an order for the edges in each component. This gives an order on all edges
and all components of the disjoint union of the $\gh_v$.
Then a factorization is given
by contracting the edges in the given order and then merging the components
in their order. Note that we need one less merger than components of $\gh_v$ since
we are left with one component $\ast_v$. Note during the last operation, we can simultaneously
perform the needed isomorphism, unless there is only an isomorphism.
This establishes the existence of factorizations and an upper bound.
The non--uniqueness is clear by the choices made.
Also note that one can replace a non--loop edge
contraction by a merger and a loop contraction.
Loop contractions and mergers commute as classes, but commutation
with  non--loop contractions is delicate.

However, since the weight is fixed, we obtain the least number of factors
precisely if we have the maximal number of non--loop edge contractions.
This is guaranteed by the given order and hence indeed this is the lower bound.

If we use a merger and contraction in lieu of a contraction, we only get a longer word.
Thus we see that in order to have a minimal factorization,
we have to contract the ghost edges, the choice here is an order of these edges,
and then merge the components  (now contracted to a vertex), where  now the choice is again an order of the mergers.
\end{proof}

\subsubsection{Relations}
\label{relsec}
All relations among morphisms in  $\GG$ are homogeneous in both weight and degree. We will not go into
the details here, since they follow directly from the description in the appendix. There are
the following types.

\begin{enumerate}
\item {\em Isomorphisms}. Isomorphisms commute with any $\phi$ in the following sense.
For any $\phi$ and any isomorphism $\sigma$ there are unique $\phi'$ and $\sigma'$ with $\gh(\phi\circ \sigma)=\gh(\phi')$
such
that
\begin{equation}
\phi\circ \sigma =\sigma'\circ \phi'
\end{equation}

\item {\em Simple edge/loop contractions}. All edge contractions commute in the following sense:
If two edges do not form a cycle, then the simple edge contractions commute on the nose
$\scirct \ccirc{s'}{t'} = \ccirc{s'}{t'} \scirct$.
The same is true if one is a simple loop contraction and the other a simple edge contraction:
$\scirct\circ_{s't'}=\circ_{s't'}\scirct$.
If there are two edges forming a cycle, this means that $\scirct \circ_{s't'}=\ccirc{s'}{t'} \circ_{st}$.
This is pictorially represented in Figure \ref{squarefig}.

\begin{figure}
    \centering
    \includegraphics[scale=.20]{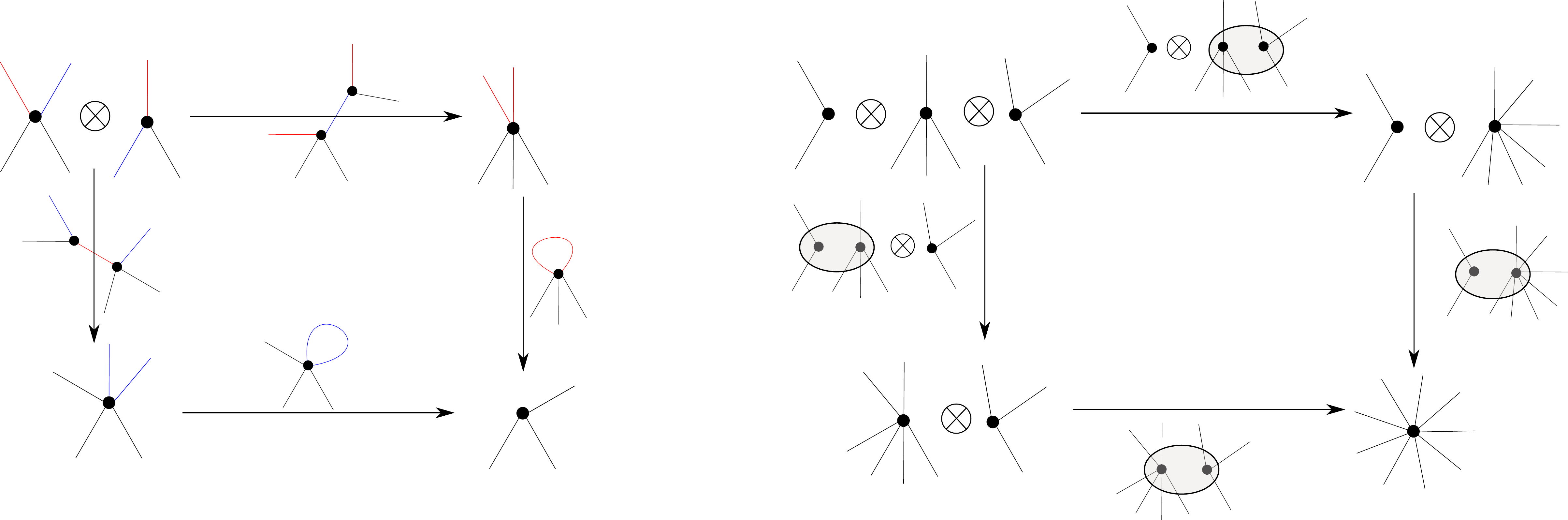}
    \caption{Squares representing commuting edge contractions and commuting mergers. The ghost graphs are shown. The shaded region is for illustrative purposes only, to indicate the merger}
    \label{squarefig}
\end{figure}

\item {\em Simple mergers.} Mergers commute amongst themselves $\mge{v}{w}\mge{v'}{w'}$ $=\mge{v'}{w'}\mge{v}{w}$.
If $\{\del(s),\del(t)\}\neq\{v,w\}$ then
\begin{equation}
\label{mergereq}
\scirct\mge{v}{w} = \mge{v}{w}\scirct, \quad \circ_{st}\mge{v}{w}=\mge{v}{w}\circ_{st}
\end{equation}

If $\del(s)=v$ and $\del(t)=w$ then for a simple edge contraction, we have the following relation
\begin{equation}
\label{triangleeq}
\scirct = \circ_{st}\mge{v}{w}
\end{equation}
This is pictorially represented in Figure \ref{trianglefig}.
\end{enumerate}

\begin{figure}
    \centering
 \includegraphics[scale=.25]{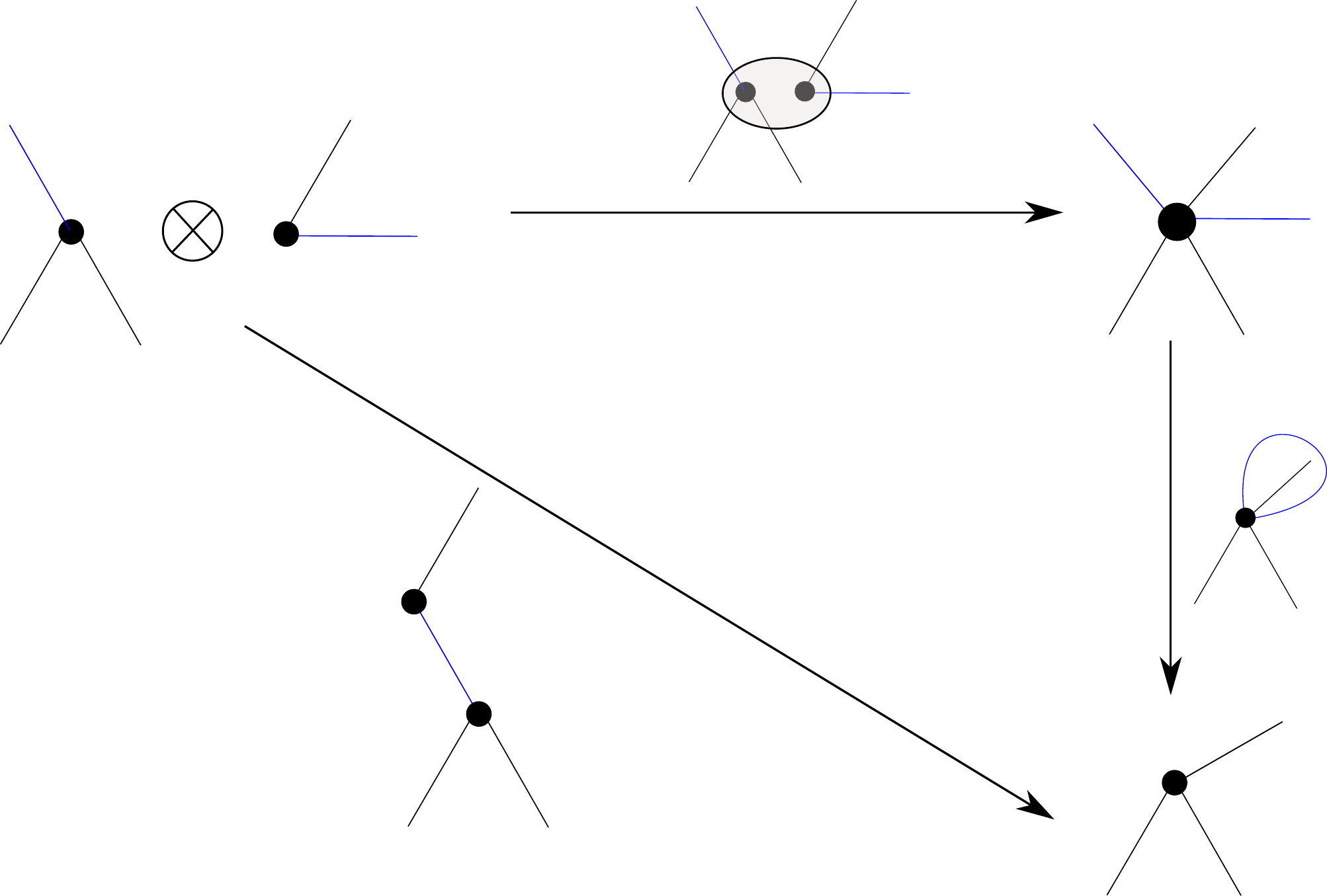}
    \caption{A triangle representing commutation between edge contraction and a merger followed by a loop contraction. The ghost graphs are shown. The shaded region is for illustrative purposes only, to indicate the merger.}
    \label{trianglefig}
\end{figure}

\subsubsection{$\GG^{or}$ and $\GG^{odd}$} Notice that in the case of words in a standard  viz.\ minimal form, two words
differ only by permutations on edges and components, see Proposition \ref{GGstructurelem}.
Let $Or(\phi)$ be the set of standard words for $\phi$ and
$Odd(\phi)$ be the Abelian subgroup of $Hom(W(\phi),\Z/2\Z)$ for whose elements  a switch in order of the edges
induces a minus sign and is invariant with respect to switches of orders of components.
For the composition we use the same technique as above. We then get the natural and twisted versions
discussed previously in \S\ref{orderedversionspar}.

\subsubsection{$\FF^{ord}$, $\FF^{or}$ and $\FF^{odd}$ in the case $\FF$ is indexed over $\GG$}
We can adapt the situation above for Feynman categories indexed over $\GG$ as in section $\ref{odd42}$.

\subsection{Odd versions for Feynman categories with ordered presentations}
The aim of this section is to find a more abstract setting for the constructions above. In particular, we want
to generalize $\FF^{ord}$ and $\FF^{odd}$ for Feynman categories not necessarily indexed over $\GG$.
We say that in a Feynman category a set of morphisms $\Phi\subset (\F\downarrow \imath(\V))$ one--comma generates
if any morphism $\psi\in Hom_{\F}(X,\ast)$, with $\ast\in \imath(\V)$  factors as
$\psi=\sigma \phi_0\dots\phi_k \sigma'$ with the $\phi_i\in\Phi'$, $\sigma\in Aut(X)$ and $\sigma'\in Aut(\ast)$,
where $\Phi'$ is the trivial extension of $\Phi$ by identities. That is the set of morphisms of the type
$p'\circ \phi\otimes id \odo id \circ p$ with $\phi\in \Phi$ and $p,p'$ any commutativity/associativity isomorphisms.
We say that $\phi$ is $n$-ary  if $\phi\in (\imath(\V)^{\otimes n}\downarrow\V)$.

We say that $\Phi$ is of crossed type if, for any $\phi\in\Phi, \phi\in Hom_{\F}(X,\ast)$ and for any
$\sigma\in Iso\F(X,X)$ there is a commutative square
\begin{equation}
\xymatrix{
X\ar[d]^{\sigma}\ar[r]^{\phi'}&\ast \ar[d]^{\sigma'}\\
X\ar[r]^{\phi}&\ast
}
\end{equation}
with  $\phi'\in \Phi,\sigma'\in Iso(\F)$ being unique.

For example:
\begin{lem}
\label{genlem}
In $\GG$, let $\Psi$ consist of the morphisms $\ccirc{i}{j}:\ast_{[n]}\amalg \ast_{[m]}
\to \ast_{[n+m]}$, $i\in [n],j\in[m]$, $\circ_{ii'}:\ast_{[n]}
\to \ast_{[n-2]},i\neq i' \in [n]$ and $\mge{n}{m}:\ast_{[n]}\amalg \ast_{[m]}\to \ast_{[n+m+1]}$, which is given by the merger
and enumerating the flags of $\ast_{[n]}$ before those of $\ast_{[m]}$ from $0$ up to $n+m+1$.

Let $\Phi$ consist of the morphisms $\ccirc{0}{0}:\ast_{[n]}\amalg \ast_{[m]}
\to \ast_{[n+m]}$, $\circ_{01}:\ast_{[n]}
\to \ast_{[n-2]}$ and $\mge{n}{m}:\ast_{[n]}\amalg \ast_{[m]}\to \ast_{[n+m+1]}$, which is given by the merger as above.

Then both $\Phi$ and $\Psi$
 one--comma generate and $\Phi$ is of crossed type.
\end{lem}

\begin{proof}
Indeed, any  simple edge contraction is conjugate by elements in $Iso(\F)$ to the first type of morphism, any simple loop
contraction to a morphism of the second type and finally any simple merger is conjugate to the third type.
Since any morphism can be factored into these three types of morphisms (after extending by identities) the first statement follows.

The fact that $\Phi$ is of crossed type is straightforward. We will deal with the second case. Fix  $\ast_S$ and
an iso $\sigma:\ast_{[n]}\to \ast_{[n]}$, Let $s_0$ and $s_1$ be the flags mapped to $0$ and $1$ and set
$S'=S\setminus \{0,1\}$ then $\phi'=\ccirc{s_0}{s_1}$ and $\sigma'$ is the restriction of $\sigma$ to $S'$ (or strictly speaking
the restriction of $\sigma$ to $\ast_{S'}$ followed by the renumbering as before).

\end{proof}
By restriction, we get that
\begin{cor}
 In $\operads$ the morphisms $\circ_1$ generate.  In $\modular$ the morphisms $\circ_1$ and $\circ_{01}$ generate
(cf.\cite{GKmodular,schw}, where $\circ_{nn+1}$ was chosen).
\end{cor}

In particular we find the $\Sn$ equivariance of the $\circ_i$ products equivalent to the property of the restriction of
$\Phi$ being crossed.
\begin{equation}
\label{snequieq}
\sigma_n(a)\circ_{\sigma_n(i)} \sigma_n(b)=(\sigma_n\circ_i\sigma_m) a \circ_i b
\end{equation}

\subsubsection{$\GG^{ord}$}

There are actually two ordered versions one could consider. The first is the naive one,
where a morphism is a pair of a morphism and a word in the chosen set of generators. The second
is a pair of a word and a word in a standard form.
For the second case, we can and will for instance
choose that, given $\Phi$, we only use words of minimal length, e.g. using Lemma \ref{GGstructurelem} with all the
isomorphisms at the end.

To compose two such minimal words, we use the relations explained above to move all the mergers and isomorphisms
to the right and to obtain the order as described in \S\ref{ordsec}.

For the Feynman category of $\modular$, $\FF^{odd}$ is  exactly the Feynman category for $\K$--twisted modular operads and
for twisted modular operads the indexing adds the $\K$ twist. The ``classical'' versions of $\FF^{odd}$ are
considered in depth in \cite{KWZ}.

\subsubsection{$\FF^{ord}$ for a choice of a one--comma generating set}
Let $\Phi$ be a one--comma generating set, then $\FF^{ord}$ is the indexed version over $\FF$ where
the morphisms are morphisms of $\F$ together with a decomposition into elements of $\Phi$ and isomorphisms.
If the generating set is crossed, then we can only retain the decomposition into elements of $\Phi$, the isomorphism
at the end being fixed.

\subsubsection{$\FF^{or}$, $\FF^{odd}$ for an ordered presentation}
Note that if we give a category by generators and relations, the relations can be depicted as polygons with oriented edges, which
form two simple edge paths to which every path belongs and which start and end at the same vertices.
If two edges are marked by the same morphism we will identify them and also their vertices. Thus we have
chains of polygons whose sides are indexed by different morphisms. We can of course combine relations by gluing along sides,
extending them by morphisms on both sides and composing these.
We call a polygon decomposable if it is the combination of two polygon relations, an extension or gluing of two relations.
A non-decomposable polygon is called simple.

\begin{df}\label{opdef}
 We call a presentation of a category ordered if for each relation, i.e. polygon chain
 as above, we can assign a value $+1,-1$ which is multiplicative under decomposition into two polygons.
Relations involving  isomorphisms are fixed to be $+1$, that is composition with isomorphisms and composition of isomorphisms.
\end{df}

It is clear that such an order is fixed on simple polygons, if they exist.

\begin{df}
 Given an ordered presentation with one--comma generators and isomorphisms,
we define  $\FF^{or}$ to be the Feynman category with unchanged $\V$ and objects of $\F$,
but with morphisms being a pair of a morphism and a class of representation, where two representations
are equivalent if they form part of a relation with value $1$.
The choice that isomorphisms have only $+1$ relations
guarantees that $\V^{or}=\V$ is well defined and $\FF^{or}$ is still a Feynman category.

Likewise we define $\FF^{odd}$, where now $\F^{odd}$ is  the free Abelian construction, quotiented out by the relations with
the given sign. That is, each morphism is a morphism of $\F$ plus a presentation modulo changing
the presentation and simultaneously changing the sign.
\end{df}

\subsubsection{Example: quadratic relations}
The first example is if we have quadratic equations among the elements of $\Phi$. In this case, we can associate
to each square not involving isomorphisms the sign $-1$.
All relations, after canceling isomorphisms, have polygons which are cubical. That is, a relation
involving $n$ generators can be obtained from a $n$--cube, by comparing the $n!$ ways of representing the diagonal
by edge morphisms. It is now clear that the induced signs are compatible. The total sign is $(-1)^f$ where $f$ is the number
of  flips across the diagonal.

\subsubsection{Example: multi--linear--quadratic relations}
A more complex example is if we have homogeneous relations
that are quadratic and linear (or higher order) and we can (as in $\GG$) search for certain standard forms. Here we would need something
like the minimal composition length. In the case of $\GG$ we used the grading in which contractions were odd and mergers are even.

The general theory is best expressed in 2--categories, see below.

\begin{df}\label{aaadef}
A set of one--comma generators $\Phi$ is $k$ at most quadratic,
if the set of generators splits into $k$ disjoint sets, such that
\begin{enumerate}
 \item every morphism has an expression in these generators of the form $\Phi_1\circ\dots \circ\Phi_k$
where the $\Phi_k$ are expressed in generators of the $k$--th set.
\item Relations among the sets separately are quadratic. And relations among the sets $\Phi_i$ and $\Phi_j$
are triangular for $i>j$. That is of the type $\phi_j\phi_i=\phi'_j$.
\end{enumerate}
\end{df}

\begin{prop}
 In the situation above, the presentation can be ordered by assigning $-1$ to each quadratic relation and $+1$ to
each triangular relation.
\end{prop}

\begin{proof}
 Using the triangular relations, we see that there is a well defined minimal word length and
we need to define the signs in relations only involving these. The relations there are now diagonal paths
on cubes stuck together along the diagonal, and the same reasoning as above applies.
\end{proof}

The archetypical example of course is $\GG$ where $\Phi_1$ is generated by simple edge/loop contractions and $\Phi_2$ is
generated by simple mergers.

\subsubsection{2--category construction for $\FF^{odd}$ with an ordered presentation}
\label{2catsec}
We can also consider the assignment of $+1,-1$ as the additional structure of 2--morphisms for a presentation $(\Phi,{\mathcal R})$
of a category. More precisely, given a presentation of a category, we can write the category as the category obtained from
the 2--category whose 1--morphisms are the free category in generators and whose 2--morphisms are the relations generated by
$\mathcal R$ which are treated as identity 2--morphisms.
The category is then obtained by reducing to a 1--category with morphisms being isomorphism classes under 2--morphisms.

In case we are enriched over $\Ab$, we can extend this to  2--morphisms that are $\pm id$.  We can then take the $\Ab$ category quotient.
As additional data we specify the 2--morphisms to be $\pm id$ for each 2--morphism in $\mathcal R$.
Such data is admissible if it is compatible with horizontal
and vertical compositions, that is if every relation 2--morphism has a well defined sign.

An {\em ordered presentation} is then a presentation together with an admissible assignment of signs.

\begin{rmk}
This consideration then easily gives rise to further generalizations, where  for instance 2--morphisms take values
in any Abelian group and there is a (double) groupoid morphism induced by horizontal and vertical composition etc..
We will not use this in the following as we are mainly interested in differentials arising from this situation and this
is tied to introducing only signs and not other characters or in physics parlance other statistics such as parafermions.
\end{rmk}

\section{Universal Operations}
\label{universalsec}
In this section, given a Feynman category, we construct a new Feynman category of universal operations. This is what conceptually explains the constructions of Gerstenhaber and the algebraic half of Deligne's conjecture by rendering them as the outcome of a calculable construction. This of course now translates to all contexts.

\begin{assump}
In this section, we assume that $\V$ is essentially small.
\end{assump}

\subsection{Cocompletion and the universal Feynman category}\label{cocompsec}  Given a Feynman category $\FF$ and an $\CO\in\F$-$\opcat_{\CalC}$, with $\CalC$ a cocomplete monoidal category, there
are natural operations on $\colim_{\V}\CO$. The example par excellence being the generalized
Deligne conjecture, which states that for an operad with multiplication in $\dgVect$
 there is a Gerstenhaber up to homotopy structure on $\bigoplus_n \O(n)$. The {\it ur}--operation
of this type is the structure of  pre-Lie algebra on $\bigoplus_{n}\CO(n)$ found
by Gerstenhaber. Here we show that this is a  universal feature. Any Feynman category $\FF$ gives rise to a new Feynman category $\FFV$ of coinvariants. This new Feynman category is of operadic
type, i.e.\ its underlying $\VV$ is trivial. If $\C$ is cocomplete then it naturally acts on the appropriate
colimit in $\C$, that is the colimit of any $\CO \in \F$-$\opcat_{\C}$ forms an element of
$\FV$-$\opcat_{\C}$.

A Feynman category  $\F$ need not be cocomplete itself. Of course if, as previously,
we map to a cocomplete category, this map factors through its cocompletion $\hat\F$.
Recall that $\hat\F$ is the category of accessible functors  $\hat F\subset[\F^{op},\Set]$, and in the case that $\F$ is small,
$\hat\F$ is the category of all presheaves $Fun(\F^{op},\Set)$. This is true as well in the enriched
case, where $\Set$ is replaced by $\CalE$. Here one assumes that the underlying category $\CalE_0$ of $\CalE$
is cocomplete, \cite{day} --- see also Remark \ref{cocompletermk} below.

Since $\F$ was monoidal its category of presheaves is also
 monoidal with the Day convolution product. For two presheaves $F$ and $G$

$$
F\day G=\int^{X,Y \in \F}F(X)\otimes G(Y)\otimes Hom_{\F}(-,X\otimes Y)
$$

We denote the Yoneda embedding to the cocompletion by $y:\F\to \hat\F$.

\begin{lem}
The Day convolution of $\hat \F$ preserves colimits in each variable if $\otimes$ does in $\CalE$.
\end{lem}

\begin{proof}
\begin{eqnarray*}
\int^{Z}F(Z)\day\int^{Z'}G(Z')&=&
\int^{X,Y\in\F}\int^Z F(Z)\times \int^{Z'}G(Z')\times Hom_{\F}(-,X\otimes Y)\\
& =&
\int^{X,Y\in \F}\int^{Z,Z'\in F} F(Z)\times G(Z') \times Hom_{\F}(-,X\otimes Y)\\
& =&
\int^{Z,Z'\in \F}\int^{X,Y\in F} F(Z)\times G(Z') \times Hom_{\F}(-,X\otimes Y)\\
&=&\int^{Z,Z'\in \F} F(Z)\day G(Z')
\end{eqnarray*}
\end{proof}

\begin{df}
We define the symmetric monoidal category $\FV$ of coinvariants of a Feynman category to be the full symmetric monoidal subcategory
of $([\F^{op},\Set],\day)$ generated by $1:=\colim_{\V}\,(y\circ\imath)$.
\end{df}

\begin{lem}
\label{mainfvlem}
If $\otimes$ in $\CalE$ preserves colimits in each variable, which we have assumed, then
$\FV$ is a subcategory of $\hat\F$. In particular $n:=1^{\day n}=\colim_{\V^{\times n}}(y\circ \imath^{\otimes n})$
and $ Hom_{\FV}(n,m):= \lim_{\V^{\times n}}\colim_{\V^{\times m}}
 Hom_{\F} ((\imath)^{\otimes n},(\imath)^{\otimes m})$.
\end{lem}

\begin{proof} We will prove the case for $2$, the general case being analogous.
Using the coend formula for the colimits and co--Fubini:
\begin{eqnarray*}
1\day 1&=&\int^{Z,Z'}\int^X Hom_{\F}(Z,X)\times \int^Y Hom_{\F}(Z',Y)\times Hom_{\F}(-,Z\times Z')\\
&=& \int^{Z,Z'}\int^{X,Y} Hom_{\F}(Z,X)\times Hom_{\F}(Z',Y)\times Hom_{\F}(-,Z\times Z')\\
&=&\int^{X,Y}\int^{Z,Z'} Hom_{\F}(Z,X)\times Hom_{\F}(Z',Y)\times Hom_{\F}(-,Z\times Z')\\
&=& \colim_{\V\times \V} \, y\circ \imath\day
y\circ \imath\\
&=& \colim_{\V\times \V} \, y\circ (\imath\otimes
 \imath)=2
\end{eqnarray*}
and the penultimate equation follows, since $y$ is strong monoidal. The general case follows
in the same fashion.
The last statement follows from the fact that colimits in the first variable pull out as
limits, and in the second variable the colimit comes from the fact that in $\hat \F$ the colimits are computed pointwise.
\end{proof}

The lemma above lets us characterize the morphisms  in $\FV$ in a practical manner. We first reduce to a
skeleton $sk(\V)$ of $\V$, that is pick representatives $[\ast_i], i\in I$ for each of the equivalence classes, and $X_{\bf j}, {\bf j}=(j_1,\dots, j_n)
\in J_n\subset I^n$
of representatives of $\imath(\ast_{i_1})\odo\imath(\ast_{i_n})$. For $\CalE$, we
denote the product by $\prod$, the coproduct by $\coprod$,
the tensor structure by $\otimes$, coinvariants by a subscript and invariants by a superscript.
\begin{lem}
\label{decomplem}
\begin{equation}
\label{decompeq}
 Hom_{\FV}(n,m)=
\prod_{{\bf j}\in J_n}
\coprod_{{\bf k}\in J_m}
\bigotimes_{i=1}^m
Hom_{\F}(X_{\bf j},\imath(\ast_{k_i}))_{Aut(\ast_{k_i})}^{Aut(\ast_{j_1})\times \dots \times Aut(\ast_{j_n})}
\end{equation}

\end{lem}

\begin{proof}
 This follows by unraveling the definitions and using the fact that $\FF$ is a Feynman category.
\end{proof}

\subsection{Enriched Versions} This construction readily generalizes to both enriched settings, Cartesian or indexed enriched. In this case, we have to replace the colimit by the appropriate
indexed colimit as in the previous paragraph.
In particular when index-enriched over $\Ab$ then this becomes simply:
\begin{equation}\label{nateq}
 Hom_{\FV}(n,m)=
\prod_{{\bf j}\in J_n}
\bigoplus_{{\bf k}\in J_m}
\bigotimes_{i=1}^m
Hom_{\F}(X_{\bf j},\imath(\ast_{k_i}))_{Aut(\ast_{k_i})}^{Aut(\ast_{j_1})\times \dots \times Aut(\ast_{j_n})}
\end{equation}
In this case, each morphism to $1$ has components $[\phi_{X_{\bf j},\ast_i}]$
of classes of coinvariants of morphisms $\phi_{X_{\bf j},i}\in Hom_{\F}(X_{\bf j},\imath(\ast_i))^{Aut(\ast_{j_1})\times \dots \times Aut(\ast_{j_n})}
\subset Hom_{\F}(X_{\bf j},\imath(\ast_i))$ and any morphism has components $[\phi_{X_{\bf j},X_{\bf k}}]$
which are tensor products of these.
By abuse of language we will call a choice of representatives in (\ref{decompeq}) ``components''; also in the general case.

Let $\VV$ be the  subcategory of $\FV$ given by
 the object $1:=\colim_{\V} \imath$ with only the identity morphism and $\iV$ be the inclusion.

\begin{thm}
\label{universalthm}
The category $\FFV=(\VV,\FV,\iV)$ is a Feynman category. This  holds analogously in the enriched
and enriched indexed case.
\end{thm}

\begin{proof}
First notice that
$\VV$ is indeed a groupoid and by construction all elements of $\FV$ are even equal to tensors of $1$, so that
the inclusion is essentially surjective. The isomorphisms from $\VV^{\otimes}$ (up to equivalence)
are the commutativity constraints and identities, and these are preserved yielding that the functor to $Iso(\FV$) is faithful. To show that it is full, we check
that these are the only isomorphisms.

If there would be any extra isomorphism in $\FV$, then the components $\phi_{X_{\bf j},X_{\bf k}}$
of the limit/colimit would need to be isomorphisms.
But these are taken care of by the colimit/choice of representative,
 since (i) holds for $\F$, leaving only the identity and the  permutations.
The condition (ii) holds by Lemma \ref{decomplem}, and the condition (iii) holds because it did in $\FF$.

The enriched versions and indexed enriched version follow in an analogous fashion, by making the necessary replacements, analogously using the constructions of \S4. In the Cartesian case this is straightforward, and in the indexed enriched case one has to insert the indexing functor $\D$ into the formulas.
\end{proof}

\begin{rmk}
\label{cocompletermk}
We could have alternatively  defined $\FV$ abstractly to have objects given by the natural numbers with $+$ as
a tensor product and morphisms given by the formula of Lemma \ref{mainfvlem}. The onus would then be to show
that this is indeed a monoidal category.
\end{rmk}

\subsection{Cocompletion for $\opcat$}
Let $\CO\in \F$-$\opcat_{\C}$  where $\C$ is cocomplete.  By its universal property, $\O$ factors through the cocompletion $\O=y\circ \hat \O$; $\hat \O$ being the cocontinuous extension.
Since $\F_\V$ is a subcategory of $\hat \F$,
we can restrict $\hat\O$ to $\FV$. This is just the pull back under the inclusion.
We call the resulting functor $\O_{\V}$ and this is then an element of $\FV$-$\opcat_{\C}$.

\subsection{Generators and weak generators}
\label{weaksec}
From the theorem above, we see that the morphisms in $\FV$ are one--comma
generated by the morphisms $\phi\in Hom_{\FV}(n,1)$, as they should be.
But even if $\F$ is generated in low arity, this might not be the case for $\FV$.
In fact this is seldom the case.

If $\CalE$ is Abelian,
we say $\FFV$ is weakly generated by morphisms $\phi\in \Phi$ if the summands of the components $[\phi_{X_{\bf j},i}]$
generate the morphisms of $\FFV$. Here different summands are indexed by different isomorphism classes of morphisms.

\begin{ex}
In the case of operads, see below \S\ref{preliesec}, the pre--Lie operation would generate the symmetric brace operations, while it weakly generates all the brace operations, which can be seen to be all the operations of $\operads_\V$ if one is enriched over a linear category.
\end{ex}

\subsection{Feynman categories indexed over $\GG$}
We will now treat the standard examples, those indexed over $\GG$,
in some more detail. We assume that we are enriched at least over $\Ab$.
There are several levels of Feynman categories which are obtained from $\FF$.
All of them are of operadic type. The first is $\FFV$, this is described by insertion operads for
graphs with tails. Forgetting tails, using the operator $\kill$,\footnote{See \S\ref{insertionsec}.}  $\FV$ is weakly indexed
over $\FV^{nt}$ which is based on graphs without tails and insertion. Here by weakly indexed,
we mean that $\kill(\phi\circ\psi)$ is a summand of $\kill(\phi)\circ\kill(\psi)$.

More importantly, there is a {\it bona fide} inclusion
of Feynman categories, given by $\leaf$ and the components of the section weakly generate.

Analyzing $\FV^{nt}$, we see that it is usually quite complex and many times not generated by its binary morphisms.
However, one can often find  sub--Feynman categories that correspond
to known operads which again weakly generate $\FV$ after applying $\leaf$.  Examples for $\FV^{nt}$ are the  operads used in \cite{KS,CL,del,cyclic,KSchw,Willwacher}.  The morphisms are given by the relevant type of trees and the objects of $\asts$ are the (white) vertices.

Before giving a tabular theorem, we will analyze the classical cases of an operad and less classically an
operad with multiplication in detail.  We will refrain from adding the additional technical definitions needed to give the most general framework in which the transition $\FV\leadsto \FV^{nt}$ can be done; although this is certainly possible.  We do wish to point out that $\FV^{nt}$ is canonical.

\subsection{Gerstenhaber's construction and its generalizations in terms of Feynman categories}

\subsubsection{Operads and pre-Lie in the Feynman category setup}
\label{preliesec}
Let us compute the standard example which is the Feynman category of
operads $\operads_{\V}$.

Let us first look at the unenriched case. In this situation, we can calculate the morphisms.
Let $1=\colim y\circ \imath=\colim_n \colim_{\Sn} y(\ast_{n})$. It is easy to see that
the outside colimit lets us consider the components $\lim_{\SS_{n_1}\times \dots\times \SS_{n_m}} \colim_{\SS_k}
Hom(\ast_{n_1} \amalg\dots\amalg \ast_{n_m},\ast_k)$. As a sample calculation let us consider a 2--ary morphism.
This will lift to a morphism $\circ_i$. Now acting on the left by $\Sn\times \SS_m$, and using equation (\ref{snequieq}) we
see that $\circ_i$ after acting on it by any element
$\SS_{n}\times \SS_m$ is equivalent under the action on the right by $\SS_{n+m-1}$ to $\circ_j$.
Thus there are no morphisms unless we consider enrichment over $\mathcal{A}b$.

In that case, however, by the same reasoning
the coefficients in the lift $\phi=\sum_i a_i \circ_i$ must all coincide.

\begin{prop}
In the case of enrichment over $\mathcal{A}b$ the components of the one--comma generating morphisms of $\operads_{\V}$
are in 1--1 correspondence with $\mathbb{S}_k$ coinvariant ghost trees of $\phi\in Hom(\ast_{n_1} \amalg \dots\amalg \ast_{n_m},\ast_k)$.
Furthermore $\FFV^{nt}=\Fprelie$ and under the inclusion $\Fprelie\hookrightarrow \operads_{\V}$ the components of $\Fprelie$ generate.
\end{prop}

\begin{proof}
Let us pick an $m$-ary morphism. Using the lim colim formula, means that we lift to an invariant morphism
and consider it up to coinvariants of the target. Picking a lift, we can consider the components in
$Hom(\ast_{n_1} \amalg \dots \amalg \ast_{n_m},\ast_k)^{\SS_{n_1}\times \dots\times \SS_{n_m}}_{\SS_k}$ as classes
of invariant maps. Now due to the description of morphisms in $\GG$ we see that a class of a morphism
up to the $\SS_k$ corresponds in a 1--1 fashion to a ghost tree. Since $\Phi$ is
of crossed type, generating as above the invariance up to $\SS_k$ action under the
 $\SS_{n_1}\times \dots\times \SS_{n_m}$ action means that the coefficients of all rooted trees
that are isomorphic must be equal. That is we can identify a component with a class of a ghost tree
under $\SS_k$ invariance, which means that the flags are unlabelled.

If we apply $\kill$, we end up with rooted trees without tails. These define a Feynman category under the insertion
operad structure, see e.g.\ \ref{graphinsertionsec}. Then $\leaf$ gives a functor that is identity on objects
and sends a tree to the direct sum of components. Here each summand is considered to have source $\amalg_v\ast_{F_{v}}$.

The fact that $\FFV^{nt}=\Fprelie$ follow from \cite{CL} and \cite{del}.
The last statement amounts to the fact
that $\operads$ is generated by 2--ary operations, and hence in the iterated pre-Lie structure every possible tree appears.
\end{proof}

The proposition lets us recover the results of \cite{KapMan}
\begin{cor}
For any Abelian category $\C$ and any operad $\O$ there is a pre-Lie and, by taking the commutator, a Lie structure
on $\bigoplus_n \O(n)_{\Sn}= \colim(\op{O})$.
\end{cor}
\begin{proof}
 Just pull back with respect to the inclusion $\FV^{nt}\to \FV$ above.
\end{proof}

\subsubsection{Non-$\Sigma$ Operads and the classical bracket}
For the Feynman category of non-$\Sigma$ operads, we see that any conjugate of any morphism
is a component of a universal morphism,
since here $\V$ is discrete. A more interesting question is given by the inclusions:
$$
\xymatrix
{
\operads^{pl}\ar[r]\ar[d]&\hat{\operads}^{pl}\ar[d]&\ar[l]\ar[d]\operads^{pl}_{\Crl^{pl,rt}}\\
\operads\ar[r]&\hat\operads&\ar[l]\operads_{\Crl^{rt}}
}
$$
and is: can we lift the universal operations, that is find a pre-image of the pre-Lie structure in $\operads^{pl}_{\Crl^{pl,dir}}$?  Practically this is the question, can we lift the bracket from the coinvariants to the direct sum $\bigoplus \O(n)$?
The answer is known to be ``yes'' and  indeed, we can, by using Gerstenhaber's classical formula \cite{G} without signs as in \cite{KapMan}.

\subsubsection{The cyclic case and the non-$\Sigma$ version}
In the cyclic case, there is a similar story; see \cite{KWZ}.  Here the universal operations of $\Cyclic_{\Crl}^{nt}$ include an (odd) Lie algebra which is not a commutator.  This is explained in detail in \cite{KWZ} and has predecessors \cite{Ginzburg,LeBruyn}.  Proceeding as above, we can calculate that there is a universal operation for each (now non-rooted) tree with unlabelled leaves.  The lifting interpretation applies to show the above bracket lifts to $C_{n+1}$ invariants.  One can further inquire if there is a lift for the inclusion $\operads\to \CCyclic$. Again the answer is ``yes'' as given in \cite{KWZ}.

\subsubsection{Operads with an ($A_{\infty}$)-multiplication}
In the case of operads with associative multiplication, the universal operations given by the $nt$ version
are given by the operad of rooted bi-partite trees as defined in \cite{del}. %called the Tamarkin operad in \cite{BBM}.
The foliage operator again provides a lift to the non-$\Sigma$ case. For an $A_{\infty}$-multiplication, the relevant operad is the one of \cite{KS}. In all these examples, the situation is analogous to that of operads, these operads include into the universal
operations via the foliage operator and their components generate.

The differential is an additional structure and may be viewed as follows.  The Maurer-Cartan functor associated to the natural Lie bracket mentioned above is represented in dg-Vect by the operad $\op{A}_\infty$.  This allows one to create/twist a differential using the multiplication.  The operad of dg natural operations associated to this twisted complex is then the result of applying a formal procedure (operadic twisting \cite{Willwacher}) to the operad of (non-dg) natural operations discussed above.

Describing the homotopy type of the dg natural operations for Lie structures associated to other Feynman categories is an interesting question.  The example of cyclic operads was studied in \cite{Ward} which showed these natural operations have the homotopy type of the operad of punctured Riemann spheres.

\subsubsection{Di-operads}
For di-operads, we can do a similar calculation. It is similar to the operad case, but the result is slightly different.
The $\FV^{nt}$ Feynman category is that of directed graphs without loops, and again there is weak generation under $\leaf$.
This Feynman category is no longer generated by its 2--ary part. The subcategory one--comma generated by it is the
Feynman category for
Lie admissible algebras \cite{Fiore}. However, since every graph is an iterate of inserting one edge, this subcategory still
weakly generates.

\subsubsection{Properads} For properads, there is an inclusion of Lie--admissible into the universal operations,
which was found in \cite{MerkVal}. Here the Lie--admissible structure is not equal to the 2--ary generated one,
but does weakly
generate.

\subsubsection{PROPs} Similarly for PROPs, in \cite{KapMan} a structure of associative algebra was found,
which gives rise to an inclusion of Feynman categories $\Fass\to \props_{\Crl^{dir}}$. Again the components do
weakly generate.

%\subsubsection{Signs}

\subsection{Collecting results}
%unused
%We will assume that we are enriched over $\Ab$.
%In order to phrase the results in general terms, given a morphism we will call the underlying ghost tree
%the topological type. We will call the binary generated part of $\FF_{\V}$ the sub--Feynman category one--comma generated
%by the binary morphisms and denote it by $\FF_{\V}^{bin}$.
% Finally we will call an inclusion of Feynman categories to a Feynman category
%type surjective if each combinatorial type appears (as a summand) in the binary generated part.

We can now rephrase the results of \cite{G,CL,KapMan,Bar,MerkVal,Fiore,KWZ} in this language.  Note that it is sometimes necessary to consider the odd versions of these structures in order to find the correct signs. This is explained in detail in \cite{KWZ}.
\begin{thm}
The  Table \ref{univtable} holds where $\operads$, $\operads_{mult}$, etc are the Feynman categories for operads, operads with multiplication, etc..
With the exception of PROPs and properads the weakly generating suboperad is the suboperad generated  by
the binary morphisms. In the last two cases, one sums over all binary operations for the inclusion.

\begin{table}
\label{universaltable}
\begin{tabular}{lllll}
$\FF$&Feynman category&$\FF_{\V}$ , $\FF_{\V}^{nt}$&weakly gen. subcat.&\\
\hline
$\operads$&operads&rooted trees&  pre-Lie\\
$\operads^{odd}$&odd operads&rooted trees  + orientation of set of edges& odd pre-Lie\\
$\operads^{pl}$&non-$\Sigma$ operads &planar rooted trees & all $\circ_i$ operations\\
$\operads_{mult}$&operads with mult.&b/w rooted trees&pre-Lie + mult.\\
$\CCyclic$&cyclic operads&trees& comm. mult.&\\
$\CCyclic^{odd}$&odd cyclic operads &trees + orientation of set of edges& odd Lie&\\
$\GG^{odd}$&unmarked modular&connected graphs &BV&\\
$\modular^{odd}$&$\K$--modular&connected + orientation on set of edges & odd dg Lie \\
&&+ genus marking&\\
$\modular^{nc,odd}$&nc $\K$-modular& orientation on set of edges & BV\\
&&+ genus marking&\\
$\dioperads$&dioperads&connected directed graphs w/o directed &Lie-admissible\\
&&loops or parallel edges&\\
$\props$&PROPs&directed graphs w/o directed loops&associative\\
$\properads$&properads&connected directed graphs w/o directed loops&Lie-admissible\\
$\props^{\circlearrowleft,ctd, odd}$&odd wheeled&connected directed graphs w/o parallel edges &odd dg\\
& properads &+ orientations of edges&Lie-admissible\\
$\props^{\circlearrowleft, odd}$& odd wheeled props & directed graphs w/o parallel edges &BV\\
&&+ orientations of edges&\\
\end{tabular}
\caption{\label{univtable}Here $\FF_{\V}$ and $\FF_{\V}^{nt}$ are
given as $\F_{\O}$ for the insertion operad. The former for the type of graph with unlabeled tails and
the latter for the version with no tails.}
\end{table}

\end{thm}
\subsection{Dual construction $\F^{\V}$}
\label{FhatVpar}
Dual to the Feynman category $\F_\V$ which acts on $colim_\V(\op{Q})$ for any $\op{Q}\in\fopsc$, one can define a Feynman category $\F^\V$ which acts on $lim_\V(\op{Q})$ for any $\op{Q}\in\fopsc$.  In particular if we are index--enriched over $\Ab$, we have:

\begin{equation}
 Hom_{\F^\V}(n,1):=\prod_{v\in \V}\bigoplus_{|X|=n}Hom_{\F}(X,\imath(\ast_{v}))_{Aut(X)}^{Aut(\ast_v)}
\end{equation}
which is suitably dual to the morphisms of $\F_\V$; compare equation $\ref{nateq}$.  To specify the action on $lim_\V(\op{Q})$, it suffices to describe the action on a target factor, which is given by projecting $lim_\V(Q)^{\tensor n}\to \tensor Q(v_i)\cong Q(X)$ and then composing with the given morphism $X\to \imath(\ast_v)$.
\subsection{Infinitesimal automorphism group, graph complex and $\mathfrak{grt}$}
Another nice example is given by $\Gra^{odd}$. This is the odd version of a Feynman category with trivial $\V$. The morphisms of this Feynman category are the elements of a graph operad, that is graphs/automorphisms with numbered vertices and substitution at the vertices.
There are several versions depending on whether one allows symmetries and flags, which are easy modifications. The most natural being $\GG$ itself. The odd version $\Gra^{odd}$ is then given by graphs together with an orientation of the vertices.

$\Gra^{\V}$ has a subspace $GC_2$ consisting of graphs with at least trivalent vertices, which is Kontsevich's graph complex.  The universal operations are generated by the one vertex loop graph, which gives rise to a BV operator $\Delta$ (if one allows disconnected graphs) and the two vertex graph with one edge which gives to a Lie algebra bracket and a differential as discussed in general above.
%Enriching over $k[[\hbar]]$ we can connect to a
It is a theorem of \cite{Willwacher} that $H^0(GC_2)=\mathfrak{grt}$ and
if one extends coefficients to $k[[\hbar]]$ and adds the BV operator $D=d+\hbar \Delta$ then also
$H^0(GC_2[[\hbar]])=\mathfrak{grt}$ \cite{MerkulovWillwacher}.
Now one can pull back or push--forward this construction with morphisms of Feynman categories.

\subsection{Universal operations in iterated Feynman categories}
Let us consider $\FF_{\V'}^{\prime op}$. Here any morphism $\psi\in Hom_{\F}(Y,X)$
is a component of a universal morphism in $Hom_{\F'_{\V'}}^{op}$ given by precomposition
with $\psi$. It is precisely a component with source any morphism $\phi_0\colon X\to X_0$
and target $\psi\circ \phi_0$, so that this is inside the slice $\FF|_{[X_0]}^{\prime op}$.
In particular, this is also true for the slices $\FF|_{[\imath (\ast)]}^{\prime op}$ for any $\ast\in \V$,
that is we obtain components in $Hom_{\F'_{\V'}}(1,1)^{op}$.
If we are enriched over $\Ab$, then we can simply sum over these components and obtain an endomorphism
$d_{\psi}\in Hom_{\F_{\V'}^{\prime op}}(n,n)$, if $|X|=n$.  This only depends on the isomorphism class of $\psi$. We will in particular be interested in $Hom_{\F_{\V'}^{\prime op}}(1,1)$.

\subsubsection{A universal morphism in the case of generators}
\label{unidiffsec}
If we have a set of morphisms $\Phi$, such as a (sub)set of one--comma generators,
 we can consider $d_{\Phi}=\sum_{\phi\in \Phi}d_{\phi}$.
Here $d_{\phi}$ acts on $\chi$ with $s(\chi)=\bigotimes_v \imath(\ast_v)$ as
$\sum_v id\odo id \otimes \phi\otimes id \odo id\circ \chi$ where
$\phi$ is in the $v$--th position if $t(\phi)\simeq \ast_v$ and $0$ else. Hence $d_{\Phi}$ gives a morphism on $H=\colim_{\F} Hom(-,-)$ and also gives a morphism in any $Hom_{\F_{\V'}^{\prime op}}(n,n)$, in particular in $Hom_{\F_{\V'}^{\prime op}}(1,1)$.  In the case that this morphism respects the slices $\FFV|_{[\imath(\ast)]}$, and so can be restricted to these, we make the following definition:

\begin{df}\label{resolvingdf}
We call a set $\Phi$ {\em resolving} if $d_{\Phi}^2=0$ as a morphism in $Hom_{\F_{\V'}^{\prime op}}(1,1)$.
\end{df}

\subsubsection{Differential in the odd case}
The morphism $d_{\Phi}$ is particularly interesting in the case of an ordered presentation in the form of Definition \ref{aaadef}, when we are dealing with $\FF^{odd}$.  In particular in this case we may assign $-1$ to each quadratic relation.

Thus, if we suppose that for any two elements $\phi, \phi^\prime \in \Phi^1$, composible in the sense of the following square, that there exists a unique $\tilde \phi,\tilde \phi'\in \Phi^1$ such that the following square relation can be completed:

\begin{equation}
 \xymatrix{
X_1\otimes X_{2}\ar[r]^{\phi\otimes id}\ar@{.>}[d]_{id\otimes \tilde \phi'}&\ast_1\otimes X_2\ar[d]^{\phi'}\\
X_1\otimes \ast_2\ar@{.>}[r]^{\tilde \phi}&\ast''}
\end{equation}
then it will be the case that if we apply $d_{\Phi^1}$ twice, we obtain pairwise summands that differ by the order of two elements that anti--commute, and hence a differential.  I.e. $\Phi^1$ is a resolving set.

\begin{ex}
 The prime example comes from Feynman categories indexed over $\GG$. For instance for connected graphs, $\Phi=\Phi^1$ is a set of edge insertions considered in Lemma \ref{genlem}.  An example where $\Phi^1\neq \Phi$ is the non-connected case.  Here $\Phi^1$ are the edge insertions above, while $\Phi$ also contains the mergers $\mge{n}{m}$.
\end{ex}

 \begin{rmk}
 We will revisit the construction of these differentials in Section $\ref{feytranspar}$, where they will provide a generalization of the usual edge insertion differentials for trees and graphs with oriented edges appearing in (co)bar constructions.  In particular, after Lemma \ref{FGlem}, one gets co-differentials $d:=d_{\Phi^1}$ on any $\imath_*\imath^*\O(\imath(\ast))$, where $\imath$ is the inclusion of $\F$ into $\F^{odd}$.  Things turn out to be more interesting if there is an additional duality allowing one to turn the co-differentials into differentials.  This setting will permit the definition of the Feynman transform in the next section.
\end{rmk}

\section{Feynman transform, the (co)bar construction and Master Equations}
\label{feytranspar}

In this section we study and generalize a family of constructions known alternately as the bar/cobar construction or the Feynman transform.  These constructions have been exhibited for particular Feynman categories in \cite{GinzKap}, \cite{GKcyclic},\cite{GKmodular},  \cite{Gan}, \cite{Vallette}, \cite{wheeledprops} and have two important properties:
\begin{enumerate}
\item  For well chosen $\C$, double iteration gives a functorial cofibrant replacement thus (often) describing a duality (by single application) on their respective categories of $\fopsc$.
\item  These constructions produce objects which represent solutions to the Maurer-Cartan/quantum master equation in associated dg Lie and BV algebras.
\end{enumerate}

The necessary input for the Feynman transform and the (co)bar construction is a Feynman category along with an ordered presentation, which permits the corresponding odd notion, and a resolving subset, which permits the construction of the differential (c.f. above).  In practice, one way to establish these structures is to assign an appropriate notion of degree to the morphisms of $\FF$.  To this end we introduce the notion of a graded Feynman category as a basis for applying the Feynman transform and the (co)bar construction.  The output of these constructions is a quasi-free object (and even cofibrant for well chosen $\C$; see Section $\ref{htsec}$) whose algebras are solutions to a general master equation that generalizes the usual Maurer-Cartan/quantum master equations mentioned above in property $(2)$. In particular, we recover the equations studied in depth in \cite{KWZ}.

Considering property $(1)$ in this generalized context is more subtle.  In particular a double iteration may be a partial or a full resolution depending on the nature of the associated degree function, see Theorem $\ref{resthm}$ and Corollary $\ref{cofcor}$.

\subsection{Preliminaries}

\subsubsection{Quasi-free $\opers$}  Recall that given a Feynman category $(\V,\F,\imath)$ an $\oper$ $\op{O}$ is free if it is in the image of $\imath_\ast\colon\vmodsc\to\fopsc$.  Informally, we say an $\oper$ is quasi-free if it is free after forgetting the differential.  To make this precise, suppose that $\op{C}$ is an additive category and let $Kom(\op{C})$ be the category of complexes on $\op{C}$ and let $\op{C}^\Z$ be the category of $\Z$ indexed sequences of objects/morphisms of $\op{C}$.  Then there are standard inclusion/forgetful functors:
\begin{equation}\label{qfeq}
\op{C}\leftrightarrows \op{C}^\Z\leftrightarrows Kom(\op{C})
\end{equation}
It is apparent that the functors in $\ref{qfeq}$ are strict symmetric monoidal functors between symmetric monoidal categories, and this allows us to move between their respective categories of $\fops$.

\begin{df}\label{qfdef1}  Let $\op{C}$ be an additive category.  An $\oper$ $\op{O}\in \fops_{Kom(\op{C})}$ is quasi-free if the push-forward of $\op{O}$ along the functor $Kom(\op{C})\to\op{C}^\Z$ is free in $\fops_{\op{C}^\Z}$.
\end{df}

\subsubsection{Dualizing}

For a groupoid $\V$ there is an equivalence, even isomorphism, of categories $\V$ and $\V^{op}$.
The functor is just given by identity on objects and inversion on morphisms.
This means that there is a functorial isomorphism between $Fun(\V,\CalC)$ and $Fun(\V^{op},\CalC)$.
Combining it with the usual isomorphism $Fun(\V^{op},\CalC)$ and $Fun(\V,\CalC)$ we get
an isomorphism between $Fun(\V,\CalC)$ and $Fun(\V,\CalC^{op})$.  More generally for an additive category $\C$ we get an isomorphism between $Fun(\V,Kom(\C))$ and $Fun(\V,Kom(\C^{op}))$ using the opposite grading and opposite morphisms to interpolate between $Kom(\C)$ and $Kom(\C^{op})$ and the aforementioned isomorphisms on morphisms of complexes.  We denote the image of a functor $\O$ under this isomorphism by $\O^{op}$ (in both directions).

\begin{df}
 A duality for $\C$ is a contravariant functor $\vee\colon \C\to \C$ such that $\vee\vee\colon \C\to \C$ is equivalent to the identity.  Note that a duality for $\C$ induces a equivalence $Kom(\C)\cong Kom(\C^{op})$ by taking the opposite complex grading when passing to the opposite.  These equivalences will also be denoted by $\vee$.
\end{df}

\subsection{Graded Feynman categories}
\label{grfeysec}
\begin{definition}  A degree function on a Feynman category $\FF$ is a map $\text{deg}\colon Mor(\F)\to \N_0$, such that:
\begin{itemize}

\item $\text{deg}(\phi\circ\psi)=\text{deg}(\phi)+\text{deg}(\psi)$
\item $\text{deg}(\phi\tensor\psi)=\text{deg}(\phi)+\text{deg}(\psi)$
\item Every morphism is generated under composition and monoidal product by those of degree $0$ and $1$.
\end{itemize}
\begin{flushleft}
In addition a degree function is called proper if
\end{flushleft}
\begin{itemize}
\item $\text{deg}(\phi)=0 \Leftrightarrow \phi \text{ is an isomorphism}$.
\end{itemize}
\end{definition}

Given a graded Feynman category $\FF$ admitting a degree function and objects $A,B \in \F$ we define $C_n(A,B)$ to be the set of sequences of $m$ composible morphisms for $m\geq n$, such that exactly $n$ of the morphisms have nonzero degree and such that the source of the composition is $A$ and the target of the composition is $B$, modulo the equivalence relation given by
\begin{equation}
\left[A\to\dots \to X_{i-1}\stackrel{f}\longrightarrow X_i \stackrel{g}\longrightarrow X_{i+1}\to\dots\to B\right] \sim \left[A\to\dots\to X_{i-1}\stackrel{g\circ f}\longrightarrow X_{i+1}\to\dots\to B\right]
\end{equation}
if $f$ or $g$ is of degree $0$.  Note that there is a natural equivalence relation on the set $C_n(A,B)$ given by equating sequences whose composition is equivalent.  We call the equivalences classes of this relation composition classes.  Finally, we define $C^+_n(A,B)\subset C_n(A,B)$ to be the subset whose constituent morphisms have degree $\leq 1$.

\begin{definition}\label{gradeddef}\label{cubicaldef}  A graded Feynman category $\FF$ is a Feynman category with a degree function such that there is a free $\SS_n$ action on the set $C_n(A,B)$ for every $A$ and $B$ in $\F$ for which composition defines an isomorphism of sets
\begin{equation*}
C^+_n(A,B)_{\SS_n}\stackrel{\cong}\longrightarrow Hom_{\F}(A,B)
\end{equation*}
which sends composition of sequences to composition of morphisms.  Explicitly this means:
\begin{itemize}
\item  The $\SS_n$ action preserves composition class.
\item  If $\phi$ is a sequence with no morphisms of degree $>1$, the $\SS_n$ action is transitive on the representatives of the composition class of $\phi$.
\item  The $\SS_n$ action respects composition of sequences, considering $\SS_n\times\SS_m\subset \SS_{n+m}$.
\end{itemize}

A graded Feynman category is called cubical if the underlying degree function is proper.
\end{definition}

Let us make several remarks about graded Feynman categories.  First, note that graded Feynman categories are common; the Feynman categories for non-unital operads, cyclic operads, properads, and modular operads are all graded, in fact cubical.  The unital and non-connected versions are no longer cubical but are still graded Feynman categories, see Example \ref{cubicalex}.

Second, if $\FF$ is graded, then for $\op{O}\in\vmodsc$ we may regard the free $\oper$ $F(\op{O})$ as being graded, i.e. $F(\op{O})\in \fops_{\C^\Z}$, with the grading induced from that of $Mor(\F)$.  This grading will support a differential in the (co)bar construction and the Feynman transform.

Third, any graded Feynman category $\FF$ can be given an ordered presentation as follows.  Since $\FF$ is graded, any polygonal relation can be expressed via two classes of chains of morphisms $[f],[g]\in C_n(A,B)$ for some $n$, with no morphisms in the chain of degree $>1$.  Now, again since $\FF$ is graded, there is a unique $\sigma\in \SS_n$ such that $\sigma([f])=[g]$.  We define the value of the polygonal relation to be $(-1)^{|\sigma|}$.  Note that this value is multiplicative since the symmetric group action respects composition by assumption, and so determines an ordered presentation.  As a consequence, in the $\Ab$ enriched context, we can consider the corresponding odd version of a graded Feynman category which takes the twisted $\SS_n$ representation on morphism chains of degree $n$.  Note the degree function on $\FF$ allows us to define the degree of morphisms in $\F^{odd}$, and there is a correspondence between the degree $1$ morphisms in $\F$ and the degree $1$ morphisms in $\F^{odd}$.  Also, any graded Feynman category has a resolving subset of the generators by taking $\Phi^1$ to be the one--comma generators of degree $1$.

Finally, we remark that the cubicality condition may be viewed as a symmetric Segal condition in that it says that relative to degree $0$ morphisms, the degree $n$ morphisms are given by coinvariants of chains of n degree 1 morphisms.  This will guide our intuition in Section \ref{Wsec}.

\begin{ex}
\label{cubicalex}
Let us study $\GG$ in particular. If we allow three types of generators, edge contractions, loop contractions and mergers, we can assign degree $1$ to edge contractions and loop contractions. The relation \eqref{triangleeq}, will force mergers to have degree $0$. Thus the degree function is not proper.

\begin{figure}
    \centering
    \includegraphics[scale=.35]{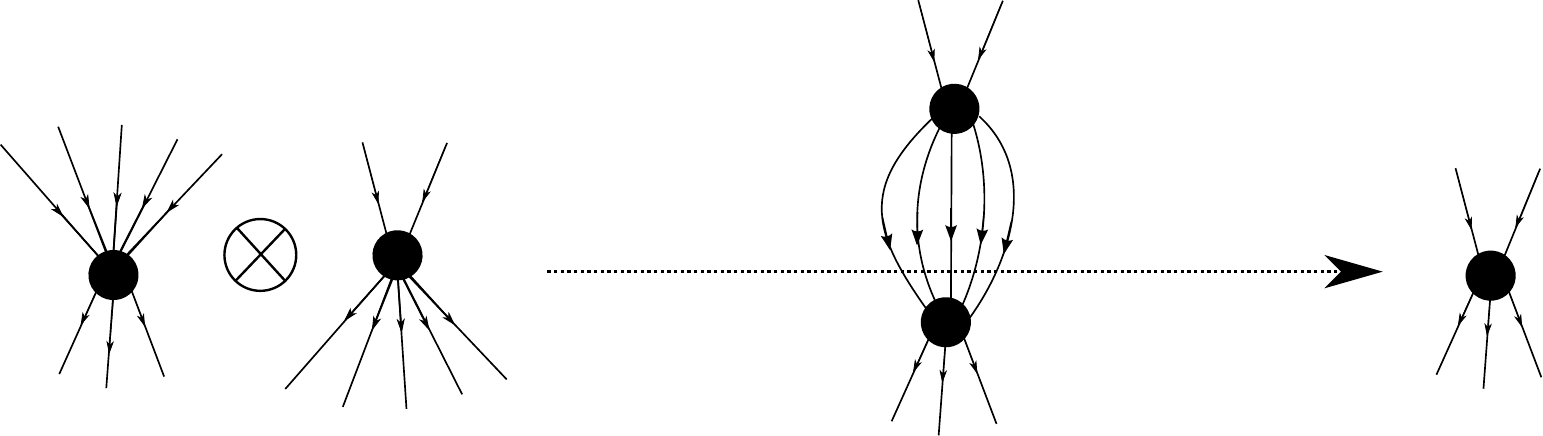}
    \caption{A $k=5$-multi-edge contraction.}
    \label{multiedge}
\end{figure}

Another interesting example are PROPs. A presentation is given by generators $k$--multi-edge contractions (Figure $\ref{multiedge}$) and mergers. Since there are no loop contractions, all relations are quadratic. Everything commutes. This means that we can assign degree $1$ to all the morphisms. However, in this version, the set of generators is not resolving, see Figure $\ref{propdiagram}$.

\begin{figure}
    \centering
    \includegraphics[scale=.35]{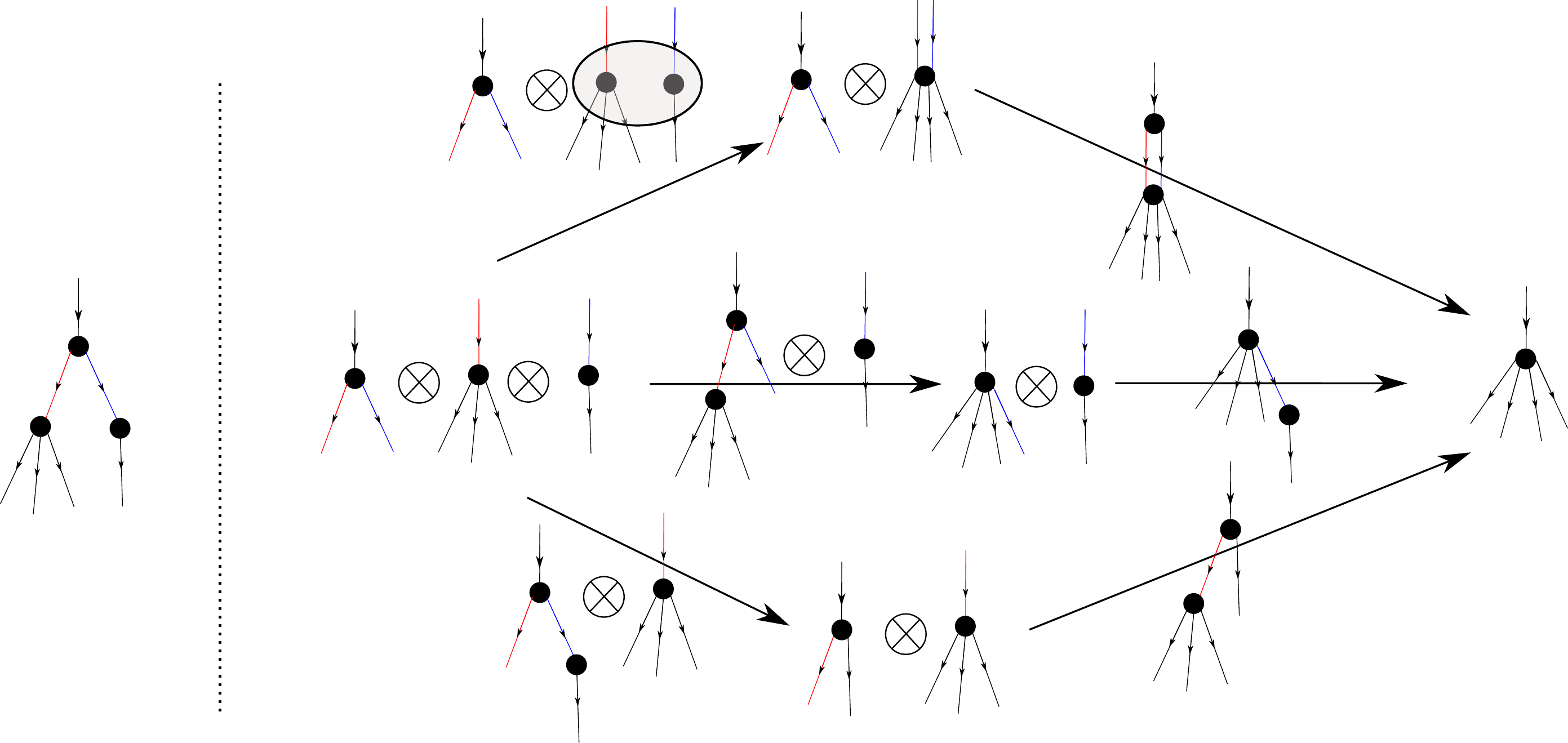}
    \caption{Giving (multi)-edge contractions and mergers degree $1$ is not cubical; the graph on the left has three consequent decompositions (pictured right).
    The ghost graphs are shown. The shaded region is for illustrative purposes only, to indicate the merger.  Here the (colored) edges are multi--edges.}
    \label{propdiagram}
\end{figure}

A second presentation is given by assigning degree $k$ to each $k$--multi--edge contraction. The same diagram induces a relation, which then forces mergers to be of degree $0$ and hence the degree function is not proper. For properads and wheeled properads, there are no problems.
\end{ex}

\subsection{The differential}
\label{diffsec}
In this subsection we will give a description of what will be the differentials in the (co)bar construction and the Feynman transform using the universality of the involved colimits.

\begin{definition}  Let $Y$ be an object of $\F$ and define $e_Y$ to be the category whose objects are degree $1$ morphisms $X\to Y$ and whose morphisms are isomorphisms between the sources.  A graded Feynman category is said to be of finite type if the category $e_Y$ has finitely many isomorphism classes of objects for each object $Y$.
\end{definition}
We now assume that $\FF$ is an $\Ab$ enriched graded Feynman category of finite type and that $\C$ is an additive category which is complete and cocomplete.

Fix $\op{O}\in\fops_{\C}$ and $\ast_v\in \V$.  First define $\Bar:=\ds\colim_{(\F^{odd}\downarrow \imath_{\FF^{odd}}(\ast_v))}(\imath\circ \op{O})^{op}\circ s$.  Next define a functor $L\colon Iso(\F^{odd}\downarrow \imath_{\FF^{odd}}(\ast_v)) \to \C^{op}$ by $\phi \mapsto lim_{e_{s(\phi)}}\imath(\op{O})^{op}\circ s$ and then define $\mathsf{A}:=\colim(L)$.  We will construct a square zero operator on $\Bar$ as a composite $\Bar \to \mathsf{A} \to \Bar$.

To construct these maps first observe that, using the $\F$-$\oper$ structure of $\op{O}$, $\op{O}^{op}(Y)$ admits an obvious cone over the functor $\op{O}^{op}\circ s \colon e_Y\to\op{C}^{op}$ and hence there is a map $\op{O}^{op}(Y)\to L(Y)$ for every object $Y\in\F$.

Next note that $\mathsf{A}$ admits a cocone over the functor $(\F^{odd}\downarrow \imath_{\FF^{odd}}(\ast_v))\colon(\imath\circ \op{O})^{op}\circ s\to \C^{op}$, for if $\phi\colon Y\to \imath_{\FF^{odd}}(\ast_v)$, then as a colimit $\mathsf{A}$ induces a map $L(Y)\to \mathsf{A}$, and hence a map $\op{O}^{op}(Y)\to \mathsf{A}$, by composition.  Thus by the universality of $\Bar$ we get a map $\Bar \to \mathsf{A}$.

Notice that in the composition $\op{O}^{op}(Y)\to L(Y)\to \mathsf{A}$ the first morphism came from something even and the second morphism came from something odd, so the picture to keep in mind in the $\GG$ indexed case is that a term in the image of the map $\Bar \to \mathsf{A}$ is an even graph inside an odd graph.

For the second morphism we show that $\Bar$ admits a cocone over the functor $L$, using the finiteness assumption.  Again let $\phi\colon Y\to \imath_{\FF^{odd}}(\ast_v)$.  As mentioned above the degree $1$ maps in $\F$ correspond to the degree $1$ maps in $\F^{odd}$, and such a map $\rho\colon X\to Y$ induces a map $X\to \ast_v$ and hence a sequence $L(Y)\to \op{O}^{op}(X) \to \Bar$.  Moreover the composite of this sequence is invariant under the action of $Mor(e_Y)$, (recall $e_Y$ is a groupoid), and so each isomorphism class $[\rho]$ in $e_Y$ gives us a map $L(Y) \to \Bar$.  Since $e_Y$ has finitely many isomorphism classes, and since $\C$ is additive, we may sum these maps to give a natural map $L(Y) \to \Bar$, and hence a cocone of $\Bar$ under $L$, whence a map $\mathsf{A}\to \Bar$.

\begin{definition}  We define $d_{\Phi^1}\colon \Bar \to \Bar$ via the composition $\Bar\to\mathsf{A}\to\Bar$ defined above.
\end{definition}

\begin{lemma} $d_{\Phi^1}^2=0$
\end{lemma}
\begin{proof}  For any morphism $Y\to \ast_v$, the induced composite $\op{O}^{op}(Y)\to \Bar\stackrel{d^2}\to \Bar$ is a sum over the odd degree $2$ morphisms with target $Y$.  Since $\FF$ is graded, the odd degree $2$ morphisms are indexed in pairs over the even degree $2$ morphisms and each pair adds to $0$, hence the claim.
\end{proof}

\begin{remark}  One can define a square zero operator in an analogous way for the cobar construction, and we will also refer to this operator as $d_{\Phi^1}$. It should be noted however that the fact that $d_{\Phi^1}$ squares to zero in the bar case comes from the fact that we took a free odd construction, whereas the fact that it squares to zero in the cobar case comes from the fact that we started with an odd $\oper$.
\end{remark}

\begin{remark}  The above construction of $d_{\Phi^1}$ also works if we replace the categories $\C$ and $\C^{op}$ with the categories $Kom(\C)$ and $Kom(\C^{op})$.  In this case the operator $d_{\Phi^1}$ has degree $1$ with respect to the induced grading on the free $\oper$ mentioned above.  In particular we will employ the total differential $d_{\op{O}^{op}}+d_{\Phi^1}$.
\end{remark}

\subsection{The (Co)bar construction and the Feynman transform}
In this section, we will define the named transformations in the case that we have an ordered presentation and a resolving subset,
e.g.\ if we are indexed over $\GG$.

\begin{df}\label{ftdef}
Let $\FF$ be a Feynman category enriched over $\Ab$ and with an ordered presentation and let $\FF^{odd}$ be its corresponding odd version.
Furthermore let $\Phi^1$ be a resolving subset of one--comma generators and let $\op{C}$ be an additive category.  Then:
\begin{enumerate}
\item  The bar construction is a functor
\begin{equation*}
\Bar \colon \fops_{Kom(\C)}\to \foddops_{Kom(\C^{op})}
\end{equation*}
defined by
\begin{equation*}
\Bar(\O):=\imath_{\FF^{odd} \; *}(\imath_{\FF}^*(\O))^{op}
\end{equation*}
together with the differential $d_{\op{O}^{op}}+d_{\Phi^1}$.

\item  The cobar construction is a functor
\begin{equation*}
\Cobar \colon \foddops_{Kom(\C^{op})}\to \fops_{Kom(\C)}
\end{equation*}
defined by
\begin{equation*}
\Cobar(\op{O}):=\imath_{\FF \; *}(\imath^\ast_{\FF^{odd}}(\op{O}))^{op}
\end{equation*}
together with the co-differential $d_{\op{O}^{op}}+d_{\Phi^1}$.

\item Assume there is a duality equivalence $\vee\colon \CalC\to \CalC^{op}$.
The Feynman transform is a pair of functors, both denoted $\FT$,
\begin{equation*}
\FT\colon \fops_{Kom(\C)}  \leftrightarrows \foddops_{Kom(\C)}\colon \FT
\end{equation*}
defined by
\begin{equation*}
 \FT(\O):=\begin{cases} \vee\circ \Bar(\O) & \text{ if } \O \in \fops_{Kom(\C)} \\ \vee\circ \Cobar(\O) & \text{ if } \O \in \foddops_{Kom(\C)}
\end{cases}
\end{equation*}
\end{enumerate}
\end{df}

Several remarks about this definition are in order.  First, for certain Feynman categories there is an equivalence of categories between $\fopsc$ and $\foddops_\C$.  For example this is the case for $\operads$ and $\CCyclic$, but not for $\modular$.  The general construction of the bar and cobar construction above agrees with the examples in the literature up to this equivalence.  These equivalences are given by various operadic suspensions and degree shifts, and were studied in detail in \cite{KWZ}.

Second, note that there are natural transformations $\Cobar\Bar\Rightarrow id$ and $id \Rightarrow \Bar\Cobar$ coming from the naturality of the left Kan extensions involved.  Moreover:
\begin{lemma}
\label{adjunctionlem}
The bar and cobar construction form an adjunction:
\begin{equation*}
\adj{\Cobar}{\foddops_{Kom(\op{C}^{op})}}{\fops_{Kom(\op{C})}}{\Bar}
\end{equation*}
\end{lemma}
\begin{proof}
This follows from Theorem $\ref{pushthm}$.  In particular this theorem tell us that there is an adjunction induced by $\imath_{\FF^{odd}}$ between $\overline{\op{V}}^\tensor\text{-}\opcat_{Kom(\C^{op})}$ and $\foddops_{Kom(\C^{op})}$.  Composing this adjunction with the equivalence of categories $\overline{\op{V}}^\tensor\text{-}\opcat_{Kom(\C^{op})}\cong \overline{\op{V}}^\tensor\text{-}\opcat_{Kom(\C)}$ we still have an adjunction, switching left and right.  Call this adjunction `A'.  Also, this theorem tells us that we have an adjunction induced by $\imath$ between $\overline{\op{V}}^\tensor\text{-}\opcat_{Kom(\C^{op})}$ and $\fops_{Kom(\C^{op})}$, call this adjunction `B'.  Now by definition the bar construction is the composite of the right adjoint of A after the right adjoint of B, and the cobar construction is the composite of the left adjoint of B after the left adjoint of A, whence the claim.
\end{proof}

In light of the previous lemma and the terminology ``bar/cobar''  it is natural to ask if $\Cobar\Bar\Rightarrow id$ gives a resolution.  The answer in general depends on the nature of $\FF$.

\begin{theorem}\label{resthm}  Let $\FF$ be a cubical Feynman category and $\op{O}\in \fops_{Kom(\C)}$.  Then the counit $\Cobar\Bar(\op{O})\to\op{O}$ of the above adjunction is a levelwise quasi-isomorphism.
\end{theorem}
\begin{proof}  We follow the proof in the case of classic operads given in \cite{GinzKap} Theorem 3.2.16.  Fix $\ast_v\in \imath(\V)$.  Assume that the internal differential is zero, and thus it is enough to show that the fibers in $\Cobar\Bar(\op{O})(\ast_v)$ are acyclic.

First, note that $\Cobar\Bar(\op{O})(\ast_v)$ is a colimit of $\op{O}\circ s$ over isomorphism classes of triangles:
\begin{equation*}
\xymatrix{X \ar[dr]_\phi \ar[rr] && Y \ar[dl] \\ & \ast_v &}
\end{equation*}
and it is enough to show that the colimit over such triangles is acyclic with respect to the differential taking
\begin{equation*}
\xymatrix{X \ar[dr]_\phi \ar[rr] && Y \ar[dl] \\ & \ast_v &} \mapsto \ds\sum_{\substack{\rho \colon Y\to Y^\prime \\ \text{deg}(\rho)=1}} \xymatrix{X \ar[dr]_\phi \ar[rr] && Y^\prime \ar[dl] \\ & \ast_v &}
\end{equation*}

Note that the complex associated to all such triangles splits over the isomorphism class of the source $X$, and as such we may restrict our attention to the subcomplex $\colim_{X\to Y\to \ast_v}\op{O}(X)$. Using the fact that colimits commute with $\tensor$ in each variable, we may factor out $\op{O}(X)$ to write said complex as a tensor product of $\op{O}(X)$ with a purely combinatorial complex.
	
We now use the necessary assumption that $\FF$ is cubical.  Since $\FF$ is cubical, isomorphism classes of such triangles with source $X$ are exactly the set $\coprod_{n\leq 2} C_n(X,\ast_v)$.  The fact that our Feynman category is graded gives this set the natural structure of a semi-simplicial set which recovers the differential.  We will show that this complex is acyclic.  In particular we will show that $\coprod_{n\leq 2} C_n(X,\ast_v)$ splits over composition class as a coproduct of semi-simplicial sets with associated chain complex equal to the augmented chain complex of a simplex of degree $1$ less than the degree of the composition class of the morphism.

Let $\phi\colon X \to \ast_v$ of degree $n$.  Define $\Delta_\phi(m)$ to be generated by the set of triangles
\begin{equation*}
 \xymatrix{X \ar[dr]_\phi \ar[rr] && Y \ar[dl]^\psi \\ & \ast_v &}
\end{equation*}
such that $\text{deg}(\psi)= m+1$.  In particular $\Delta_\phi(m)=0$ if $m\geq n$.  Since $\FF$ is graded, the morphism $\phi$ can be written as a composition of degree $1$ morphisms in exactly $n!$ distinct ways up to isomorphism.  Pick such a decomposition of $\phi$ into degree $1$ morphisms, labeled as:
\begin{equation*}
 X=:Y_{\emptyset} \stackrel{e^{id}_1}\longrightarrow Y_{\{1\}} \stackrel{e^{id}_2}\longrightarrow Y_{\{1,2\}} \stackrel{e^{id}_3}\longrightarrow\dots \stackrel{e^{id}_n}\longrightarrow Y_{\{1\cdc n\}}:=\ast_v
\end{equation*}
For $\sigma\in\SS_n$, label the action by $\sigma$ on the above sequence as:
\begin{equation}\label{deg1dec}
X=Y_{\emptyset} \stackrel{e^{\sigma}_{\sigma(1)}}\longrightarrow Y_{\sigma(\{1\})} \stackrel{e^{\sigma}_{\sigma(2)}}\longrightarrow Y_{\sigma(\{1,2\})} \stackrel{e^{\sigma}_{\sigma(3)}}\longrightarrow\dots \stackrel{e^{\sigma}_{\sigma(n)}}\longrightarrow Y_{\sigma(\{1\cdc n\})}=\ast_v
\end{equation}

Considering all such sequences simultaneously, given a set $S\subset \{1\cdc n\}$ of size $n-r$, there are exactly $r$ morphisms emanating from $Y_S$ and this set of morphisms has a total order induced by the order on $\{1\cdc n\}\setminus S$.  Define $r$ face maps by sending a sequence $X\to Y_S\to \ast_v$ to $X\to Y_{S\cup j}\to \ast_v$ for $j \in \{1\cdc n\}\setminus S$, along with the appropriate compositions of morphisms $Y_S\to Y_{S\cup j}$.  This gives $\Delta_\phi$ the structure of a semi-simplicial set whose chain complex is the augmented chain complex of an $n-1$ simplex.

We thus conclude that the combinatorial chain complex associated to $\coprod_{n\leq 2} C_n(X,\ast_v)$ is acyclic, from which the claim follows.
\end{proof}

\begin{remark}  The previous theorem can be used to show that when $\op{C}= \dgvect_k$, under the above hypotheses, the counit of the adjunction is a cofibrant replacement in a model category structure on $\fopsc$, see Corollary $\ref{cofcor}$.
\end{remark}

\subsection{A general master equation.}  In this subsection let $\op{C}$ be the category $\dgvect_k$ over a field of characteristic $0$.  Since the Feynman transform is quasi-free, a map from the underlying $\V$-module has only one obstruction to inducing a map from the Feynman transform, namely that the induced map respects the differentials on the respective sides.    In \cite{KWZ}, the following tabular theorem was compiled which states that in various studied cases this obstruction is measured by associated master equations.

\begin{theorem}\label{methm}(\cite{Bar},\cite{MerkVal},\cite{wheeledprops},\cite{KWZ})  Let $\op{O}\in \fopsc$ and $\op{P}\in \foddops_\op{C}$ for an $\F$ represented in Table $\ref{Ftable}$.  Then there is a bijective correspondence:
\begin{equation*}
Hom(\FT(\op{P}),\op{O})\cong ME(lim(\op{P} \tensor\op{O}))
\end{equation*}
\end{theorem}

\begin{table}[htb] \renewcommand{\arraystretch}{1.2}
\centering
\begin{tabular}{p{2.2cm}||l|p{1.9cm}|l}
Name of  $\fopsc$ & lim$(\op{P}\tensor\op{O})$ & Algebraic Structure & Master Equation (ME)\\ \hline\hline
operad \cite{GinzKap},\cite{GJ}& $\ds\prod_n(\op{P}(n)\tensor\op{O}(n))^{S_n}$ & odd pre-Lie & $d(-)+ -\circ- =0$  \\ \hline
cyclic operad \cite{GKcyclic} & $\ds\prod_n(\op{P}(n)\tensor\op{O}(n))^{S^+_n}$ & odd Lie & $d(-)+ \frac{1}{2}[-,-] =0$  \\ \hline
modular operad \cite{GKmodular} & $\ds\prod_{(n,g)}(\op{P}(n,g)\tensor\op{O}(n,g))^{S_n^+}$ & odd Lie + $\Delta$& $d(-)+ \frac{1}{2}[-,-]+\Delta(-) =0$  \\ \hline
properad  \cite{Vallette}& $\ds\prod_{(n,m)}(\op{P}(n,m)\tensor\op{O}(n,m))^{S_n\times S_m}$ & odd Lie-admissible & $d(-)+ -\circ- =0$  \\ \hline
wheeled properad \cite{wheeledprops} & $\ds\prod_{(n,m)}(\op{P}(n,m)\tensor\op{O}(n,m))^{S_n\times S_m}$ & odd Lie ad. + $\Delta$ & $d(-)+ -\circ- +\Delta(-) =0$  \\ \hline
wheeled prop \cite{KWZ} & $\ds\prod_{(n,m)}(\op{P}(n,m)\tensor\op{O}(n,m))^{S_n\times S_m}$ & dgBV & $d(-)+ \frac{1}{2}[-,-] +\Delta(-) =0$  \\
\end{tabular}
\caption{The $\opers$  in column 1 have a notion of Feynman transform.  A citation for the Feynman transform or closely related construction is also given in column 1.  In columns 2 and 3, we suppose $\op{P}\in \foddops_{\op{C}}$ and $\op{O}\in \fopsc$ and give the colimit of the product in column 2 and the algebraic structure that this dg vector space necessarily has.  Note that starting with $\op{P}$ odd results in the Lie brackets being odd.  The final column gives the master equation relevant to Theorem $\ref{methm}$.  Since the Lie brackets are odd, solutions to these master equations are of degree $0$.}
\label{Ftable}
\end{table}
Since we now have described the Feynman transform as a general construction, we can give the general version of the tabular Theorem $\ref{methm}$.

Fix a Feynman category $\FF$ which permits the Feynman transform (as in Definition $\ref{ftdef}$).  By definition, the differential $d$ specifies a map
\begin{equation*}
Hom_\F(Y,\imath(\ast_v))^{Aut(\ast_v)}_{Aut(Y)}\to \coprod_{X\stackrel{1}\to Y} Hom_\F(X, \imath(\ast_v))^{Aut(\ast_v)}_{Aut(X)}
\end{equation*}
and by pushing forward the orbit of the identity map of $\imath(\ast_v)$, when $Y=\imath(\ast_v)$, we get a distinguished element on the right hand side, which we can interpret as $Hom_{\F^{\V}}(1,1)$; cf.\ \ref{FhatVpar}.
Similarly for each $n$ the set $Hom_{\F^\V}(n,1)$ has a distinguished element by pushing forward the identity as in the above procedure for all $v\in \V$.  The action then specifies an $n$-cochain in $Hom(lim_\V(\op{Q})^{\tensor n}, lim_\V(\op{Q}))$.  Call this cochain $\Psi_{\op{Q},n}$.

\begin{definition}
\label{MEdef}
For a Feynman category $\FF$ admitting the Feynman transform and for $\op{Q}\in\fopsc$ we define the formal master equation of $\FF$ with respect to $\op{Q}$ to be the completed cochain $\Psi_{\op{Q}}:= \prod \Psi_{\op{Q},n}$.  If there is an $N$ such that $\Psi_{\op{Q},n}=0$ for $n>N$, then we define the master equation of $\FF$ with respect to $\op{Q}$ to be the finite sum:
\begin{equation*}
d_{\op{Q}}+\ds\sum_{n}\Psi_{\op{Q},n} = 0
\end{equation*}
We say $\alpha\in lim(\op{Q})$ is a solution to the master equation if $d_\op{Q}(\alpha)+\sum_{n}\Psi_{\op{Q},n}(\alpha^{\tensor n}) = 0$, and we denote the set of such solutions as $ME(lim(\op{Q}))$.
\end{definition}

With the groundwork now laid, the tabular Theorem $\ref{methm}$ can be stated in a general form:

\begin{theorem}\label{methm2}  Let $\op{O}\in \fopsc$ and $\op{P}\in \foddops_\op{C}$ for an $\F$ admitting a Feynman transform and master equation.  Then there is a bijective correspondence:
\begin{equation*}
Hom(\FT(\op{P}),\op{O})\cong ME(lim(\op{P} \tensor\op{O}))
\end{equation*}
\end{theorem}
\begin{proof}  First, there is a forgetful map $Hom_{dg}(\FT(\op{P}),\op{O})\to Hom_{gr}(F(\op{P}^\ast),\op{O})$ given by forgetting the differential.  By adjunction this specifies a map in $Hom_{\vmodsc }(\op{P}^\ast,\op{O})$, which in turn specifies an element of $lim_V(\op{P}\tensor\op{O})$ via the natural isomorphism\\ $$Hom_{Aut(v)}(\op{P}(v)^\ast,\op{O}(v))\cong \op{P}(v)\tensor^{Aut(v)}\op{O}(v)$$

Let $\eta\in Hom_{gr}(F(\op{P}^\ast),\op{O})$ and let $\bar{\eta}$ be the corresponding element in $lim(\op{P}\tensor \op{O})$.  It remains to show that $\eta$ induces a dg map if and only if $\bar{\eta}\in ME(lim(\op{P}\tensor\op{O}))$.  To see this fix $v=\imath(\ast_v)$.  Then the pushforward of $id_v$ as above via the differential produces finitely many isomorphism classes of degree $1$ morphisms $X\to v$ which we label $\{\gamma_i\}_{i\in I}$.  Now each map $\gamma_i$ contributes to $\Psi_{\op{Q}, |X|}$ via the map $Hom(\op{P}^\ast(X), \op{O}(X))\to Hom(\op{P}^\ast(v), \op{O}(v))$ given by convolution:
\begin{equation}
\op{P}^\ast(v)\stackrel{\gamma_i^\ast}\to\op{P}^\ast(X)\cong\tensor\op{P}^\ast(v_j) \stackrel{\tensor\eta_{v_j}}\to \tensor\op{O}(v_j) \cong\op{O}(X)\stackrel{\gamma_i}\to \op{O}(v)
\end{equation}
and the fact that $\F^\V$ acts by first projecting tells us that this convolution is equal to the $\gamma_i$ action on $\bar{\eta}$, from which the claim follows.

\end{proof}

\begin{remark} Interpreting master equation solutions as morphisms suggests a notion of homotopy equivalence of such solutions: namely two solutions are homotopic if the associated morphisms are homotopic.  However, in order to make sense of the notion of homotopy classes of morphisms we need to study the homotopical algebra of $\fopsc$.  This is done in Section $\ref{htsec}$; see Theorem $\ref{hcthm}$.
\end{remark}

\section{Homotopy theory of $\fopsc$.}\label{htsec}  In this section we give conditions on a symmetric monoidal category and model category $\op{C}$ which permit the construction of a model category structure on $\fopsc$.  In so doing we generalize prior work done in particular Feynman categories including operads \cite{spit}, \cite{hinich}, \cite{BM1}, props \cite{Fresse}, properads \cite{MerkVal2}, and colored props \cite{JY}.

After establishing the main theorem (Theorem $\ref{modelthm}$) we consider several implications.  For example our perspective of $\opcat$ as symmetric monoidal functors allows for consideration of the relationships between these model categories under the adjunctions induced by either morphisms of Feynman categories or adjunctions of the base categories.  As another example we show that, in a dg context, applying the bar construction/Feynman transform returns a cofibrant $\F$-$\oper$, and as a result the bar-cobar construction/double Feynman transform gives a functorial cofibrant replacement in the above model structure when $\FF$ is  cubical.  The fact that the bar construction/Feynman transform is cofibrant gives us, via the general theory, the notion of homotopy classes of maps from the bar construction/Feynman transform.  As seen above, such maps are given by Maurer-Cartan elements in a certain dg Lie algebra \cite{KWZ}.  There is then a notion of homotopy equivalence on both sides of this correspondence and we show these notions coincide.

Throughout this section we let $\C$ denote a category which is a model category and a closed symmetric monoidal category.  We do not assume any compatibility between the model and monoidal structures $\textit{a priori}$, but we will make such assumptions as necessary.  Throughout this section we also assume that $\V$ is a small category, an assumption which will be needed to access the known results on transfer of model structures across adjunctions.

\subsection{Preliminaries}  We will begin the section by establishing some preliminaries needed to prove this section's main result, Theorem $\ref{modelthm}$.

\subsubsection{Reduction}  In order to endow the category $\opcat$ over a Feynman category $\FF$ with a model structure it will be enough to consider what we call the reduction of $\FF$.  Recall (cf. Lemma $\ref{redlem}$) that the category  $\F$--$\opers$ over a Feynman category is equivalent to the category  $\tilde\F$-$\opers$ of its reduction.  In what follows we will establish a model structure in the context of strict Feynman categories, and the following lemma tells us that this is sufficient to establish a model structure on the categories $\fops$ over all Feynman categories.

\begin{lemma}  Let $C$ be a model category and let $\adj{F}{C}{D}{G}$ be an equivalence of categories.  Then $D$ is a model category by defining $\phi$ to be a weak equivalence/fibration/cofibration if and only if $G(\phi)$ is.
\end{lemma}
\begin{proof}
We prove this by direct verification of the axioms as enumerated in \cite{Hir}.

\textbf{M1:  }(limit axiom)  The functor $F$ is both a left adjoint and a right adjoint of $G$.  Thus it preserves all colimits and limits.  In particular $D$ is complete and cocomplete since $C$ is.

\textbf{M2:  }(two out of three axiom)  Let $f,g\in Mor(D)$ be composible with two of $f,g,fg$ weak equivalences.  Then the corresponding two of the three of $G(f),G(g),G(fg)=G(f)G(g)$ are weak equivalences in $C$, and hence so is the third.  Thus by definition of the weak equivalences in $D$, the third of the list $f,g,fg$ is also a weak equivalence.

\textbf{M3:  }(retract axiom)  Suppose $f$ is a retract of $g$ in $D$.  This means that there is a commutative diagram in $D$:
\begin{equation*}
\xymatrix{d_1 \ar[r] \ar[d]^f & d_2 \ar[r] \ar[d]^g & d_1 \ar[d]^f \\ d_3 \ar[r] & d_4 \ar[r] & d_3}
\end{equation*}
such that the top and bottom lines are the identity.  Applying $G$ to the diagram we see that $G(f)$ is a retract of $G(g)$.  If $g$ is a weak equivalence/fibration/cofibration then so is $G(g)$, then so is $G(f)$ by $\textbf{M3}$ in $C$, then so is $f$ by definition.  Hence $\textbf{M3}$ holds in $D$.

\textbf{M4:  }(lifting axiom)  Consider the following diagram in $D$
\begin{equation*}
\xymatrix{d_1 \ar[r] \ar[d]^i & d_2 \ar[d]^p \\ d_3 \ar[r] & d_4}
\end{equation*}
Suppose that either $i$ is a cofibration and $p$ is an acyclic fibration, or that $i$ is an acyclic cofibration and $p$ is a fibration.  Applying $G$ to the diagram we get a commutative diagram in $C$ which admits a lift due to $\textbf{M4}$ in $D$.  This lift is in the image of $G$ since the functor $G$ is full.  Moreover, the fact that the lift commutes in $C$ means that the preimage commutes in $D$ since $G$ is faithful.  Thus $\textbf{M4}$ holds in $D$.

\textbf{M5:  }(factorization axiom)  Let $f$ be a morphism in $D$.  Then $G(f)$ can be written as a cofibration followed by an acyclic fibration, say $G(f)=rs$.  Let $c$ be the source of $r$ and the target of $s$.  Since $G$ is essentially surjective, there exists a $c^\prime$ in the image of $G$ such that $c\cong c^\prime$.  Notice that by the lifting characterizations of acyclic fibrations and acyclic cofibrations (in the model category $C$), every isomorphism in $C$ is both an acyclic fibration and an acyclic cofibration.  Thus by composing with these isomorphisms, $G(f)$ can be written as a cofibration followed by an acyclic fibration as $G(f)=r^\prime s^\prime$ where the source of $r^\prime$ and the target of $s^\prime$ are $c^\prime$. Then since $c^\prime$ is in the image of $G$, and since $G$ is full and faithful, this factorization is in the image of $G$, hence $f$ factors as a cofibration followed by an acyclic fibration.  Similarly $f$ factors as an acyclic cofibration followed by a fibration.
\end{proof}

\subsubsection{Limits and colimits in $\fopsc$.}
Let $\FF$ be a strict Feynman category.

We make the following definition for convenience.
\begin{definition}  Given a Feynman category $(\V,\F,\imath)$ we define $\V_{id}$ to be the category with the same objects as $\V$ and only identity morphisms.  A $\V$ sequence in $\op{C}$ is defined to be a functor $\V_{id}\to \op{C}$.  The category of such is denoted $\vseq$.
\end{definition}

\begin{lemma}\label{limitslemma1}  The forgetful functor $\vmods_\C\to\vseq$ creates all limits and colimits.
\end{lemma}
\begin{proof}  Let $H\colon \vmodsc\to\vseq$ be the forgetful functor.  Then let $\alpha\colon J\to \vmodsc$ be a functor such that $colim(H\circ\alpha)$ exists.  Define $L:=colim(H\circ\alpha)$.  Then we will show that $L$ is naturally a $\V$-module.  Let $\psi\colon v\to w$ be a morphism in $\V$.  Then $L(w)$ naturally forms a cocone over the functor $H\alpha(-)(v)\colon J\to C$, and the colimit of this functor is $L(v)$.  As such we get a map  $L(v)\to L(w)$ which we define to be $L(\psi)$.  Then $L$ naturally forms a cocone over $\alpha$, which is limiting since its image under $H$ is.  A similar argument show that $H$ creates limits.
\end{proof}

\begin{lemma}\label{limitslemma2}  The forgetful functor $\fopsc\to\vmods_\C$ creates all limits, filtered colimits, and reflexive coequalizers.
\end{lemma}
\begin{proof}
Let $G\colon \fopsc\to\vmodsc$ be this forgetful functor.  First we consider limits. Consider a diagram
\begin{equation*}
J\stackrel{\alpha}\longrightarrow \fopsc \stackrel{G}\longrightarrow \vmodsc
\end{equation*}
such that the limit $lim(G\circ\alpha)$ exists in $\vmodsc$.  Call this limit $L$.  Now $L$ is $\textit{a priori}$ a $\V$-module, but it has a natural $\F$-$\oper$ structure as follows.  First extend $L$ to $Iso(\F)$ strict monoidally.  Let $\lambda\colon \tensor_{i=1}^n v_i\to v$ be a generating morphism in $\F$.  Then for every morphism $a\stackrel{f}\to b$ in $J$ we have the diagram:
\begin{equation*}
\xymatrix{\ds\tensor_{i=1}^n \alpha(a)(v_i) \ar[rr]^{\alpha(a)(\lambda)} && \alpha(a)(v) \ar[rr]^{\alpha(f)(v)} && \alpha(b)(v) && \tensor_{i=1}^n \alpha(b)(v_i) \ar[ll]_{\alpha(b)(\lambda)} \\ &&& L(v) \ar[ur] \ar[ul] &&& \\ &&& \tensor_{i=1}^nL(v_i) \ar[uurrr] \ar[uulll] \ar@{.>}[u]^{\exists  !} &&& }
\end{equation*}
where the diagonal arrows come from the cone morphisms in $\vmodsc$.  Using the universality of the limit the diagram gives us the morphism $\tensor_{i=1}^nL(v_i) \to L(v)$ which we define to be $L(\lambda)$.  This makes $L$ an $\F$-$\oper$, after extending monoidally, and makes the cone maps into morphisms in $\fopsc$.  Thus $L$ admits a cone over $\alpha$.  Finally, the fact that this cone is limiting follows from the fact that the corresponding cone over $G\circ\alpha$ is.  Hence $G$ creates limits.

Now let $\beta\colon K\to \fopsc$ be a functor such that $colim(G\circ\beta)$ exists and such that $K$ is either a filtered category or a reflexive coequalizing category.  Define $C$ to be this colimit.  Now $C$ is $\textit{a priori}$ a $\V$-module, and we will show it is naturally an $\F$-$\oper$.  First extend $C$ monoidally to $Iso(\F)$.  On morphisms it is enough to define the image by $C$ of the generating morphisms.  Let $\psi\colon \tensor_{i=1}^n v_i\to v_0$ be such a morphism (so $v_i\in \V$).  Notice that for each $v\in\V$, $colim(\beta(-)(v))\cong colim(G\circ\beta)(v)$, and hence $\times_{i}colim(\beta(-)(v_i))$ exists in $\C^{\times n}$ and is isomorphic to the colimit of the functor
\begin{equation*}
K\stackrel{\Delta}\to K^{\times n}\stackrel{\times\beta(-)(v_i)}\longrightarrow \op{C}^{\times n}
\end{equation*}
which for some $x\in K$ sends $x$ to $\times_i \beta(x)(v_i)$.  Since $\op{C}$ is monoidally closed, $\tensor$ preserves colimits in each variable separately, and since the category $K$ is supposed to be either filtered or reflexive coequalizing, we have that $\tensor$ preserves colimits in all variables (see \cite{Fressebook} section 1.2 or \cite{Rezk} lemma 2.3.2), that is:
\begin{equation*}
\tensor_{i}colim(\beta(-)(v_i))\cong colim(\tensor_i \beta(-)(v_i))
\end{equation*}
and using $\psi$ to construct a cocone, the universality of this colimit produces a map to $colim(\beta(-))(v_0)$.  Thus,
\begin{multline*}
C(\tensor_i v_i)= \tensor_{i}colim(G\circ\beta)(v_i)\cong \tensor_{i}colim(\beta(-)(v_i))\\
\cong colim(\tensor_i \beta(-)(v_i))\to colim(\beta(-)(v_0)) \cong colim(G\circ\beta)(v_0)=C(v_0)
\end{multline*}
gives the image of $\psi$.  Extending monoidally $C$ may be viewed as an object in $\fopsc$ and it is immediate from the construction that the cocone maps for $C$ in $\vmodsc$ lift to maps in $\fopsc$.

It thus remains to show that as a cocone over $\beta$ in $\fopsc$, $C$ is universal.  If $Z$ is any cocone over $\beta$ in $\fopsc$ then by universality in $\vmodsc$ there is a unique morphism of underlying $\V$-modules $C\to Z$.  Using this morphism of $\V$-modules and given a generating morphism $\phi\colon X\to v$, it is possible to write $Z(v)$ as a cocone over $\beta(-)(X)$ in two $\textit{a priori}$ distinct ways, but the uniqueness of the morphism $C(X)\to Z(v)$ given by the universality of $C(X)$ in $\C$ ensures that these two morphisms are in fact the same.  As such this morphism $C\to Z$ lifts to a morphism in $\fopsc$, completing the proof.
\end{proof}

\begin{lemma}\label{limitslemma3}  Let $\op{C}$ be a category with all small limits and colimits.  Then the categories $\vseq$, $\vmodsc$, and $\fopsc$ have all small limits and colimits.
\end{lemma}
\begin{proof}  Since $\op{C}$ has all small limits and colimits, the category $\vseq$ has all small limits and colimits levelwise.  Thus by Lemma $\ref{limitslemma1}$ the category $\vmodsc$ has all small limits and colimits.  Then applying Lemma $\ref{limitslemma2}$ we see that the category $\fopsc$ has all small limits as well.  Note in addition that by Lemma $\ref{limitslemma2}$ $\fopsc$ also has all reflexive coequalizers.  It thus remains to show that $\fopsc$ has all small colimits.  In order to show this we generalize an argument of Rezk in (\cite{Rezk} proposition 2.3.5).

As above, let $F$ and $G$ be the free and forgetful functors respectively between $\fopsc$ and $\vmodsc$.  Let $\alpha\colon J\to \fopsc$ be a functor from a small category $J$.  For a given $\F$-$\oper$ $\op{P}$ we define two functors $\Phi,\Psi\colon J^{op}\to \vmodsc$ as follows.
\begin{eqnarray*}
\Phi:=Hom_{\vmodsc}(GFG\alpha(-),G(\op{P})) & \text{ and } & \Psi:=Hom_{\vmodsc}(G\alpha(-),G(\op{P}))
\end{eqnarray*}
There exist two natural transformations $\Psi\rightrightarrows\Phi$ defined as follows:
\begin{enumerate}
\item  precompose $G\alpha(-)\to\op{P}$ with the natural transformation $G(FG)\alpha(-)\to G\alpha(-)$
\item  take the image of $G\alpha(-)\to G(\op{P})$ by $GF$ and postcompose with the natural map $FG(\op{P})\to\op{P}$.
\end{enumerate}
These two natural transformations give us the following:
\begin{eqnarray*}
lim(\Psi) & \rightrightarrows & lim(\Phi) \\
lim \left(Hom_{\vmodsc} ( G\alpha(-), G(\op{P}) )\right)  & \rightrightarrows & lim \left(Hom_{\vmodsc}( GFG\alpha(-), G(\op{P}) )\right) \\
Hom_{\vmodsc} \left( colim(G\alpha), G(\op{P}) \right)  & \rightrightarrows &  Hom_{\vmodsc} \left( colim(GFG\alpha), G(\op{P}) \right) \\
Hom_{\fopsc} \left( F(colim(G\alpha)), \op{P} \right)  & \rightrightarrows &  Hom_{\fopsc} \left( F(colim(GFG\alpha)), \op{P} \right) \\
\end{eqnarray*}
Now since the above pair of morphisms is natural in $\op{P}$, we can apply the Yoneda embedding theorem to see that these maps are given by maps of the sources, i.e. maps
\begin{equation}\label{refcoeq}
F(colim(GFG\alpha))\rightrightarrows F(colim(G\alpha))
\end{equation}
and these maps have an obvious candidate for a section, namely the map induced by the natural transformation $id\to GF$.  Checking this we see that this does give us a section, as an immediate consequence of the fact that the composites $G\to GFG\to G$ and $F\to FGF \to F$ are the identity.  Thus line $\ref{refcoeq}$ can be represented as a reflexive coequalizing diagram.  Define $Q$ to be its coequalizer, which exists by Lemma $\ref{limitslemma2}$.

To complete the proof we will show that $Q$ is the colimit of $\alpha$ in $\fopsc$.  First notice that $G(Q)$ admits a cocone over $G\circ \alpha$ by the composite $colim(G\circ \alpha) \to GF(colim(G\circ \alpha)) \to G(Q)$.  Call these cocone morphisms $\lambda$.  By adjointness, diagram $\ref{refcoeq}$ gives us morphisms
\begin{equation}\label{refcoeq2}
colim(GFG\alpha)\rightrightarrows GFcolim(G\alpha)\to G(Q)
\end{equation}
Examining the natural transformations above we see the two compositions can be described by the following diagram:
\begin{equation*}
\xymatrix{GFG\alpha(-)\ar[r] \ar[d]_{GF(\lambda)} & G\alpha(-) \ar[d]^{\lambda}\\ GFG(Q) \ar[r] & GQ }
\end{equation*}
which commutes since $Q$ is a coequalizer.  The fact that this diagram commutes tells us that the $\lambda$ are actually morphisms of $\F$-$\opers$, and in particular the cocone $G(Q)$ over $G\circ\alpha(-)$ lifts to a cocone $Q$ over $\alpha$.  The fact that this cocone is limiting follows immediately from the universality of the coequalizer.  Thus the colimit $colim(\alpha)$ exists in $\fopsc$, completing the proof of the lemma.
\end{proof}

\subsubsection{Monoidal model categories}  We assume the basics of model category theory and refer to \cite{Hir} and \cite{Hovey} thoroughout.  We make several recollections here for future use.

\begin{definition}  A category which is both symmetric monoidal and a model category satisfies the pushout product axiom (PPA) if for any pair of cofibrations $f_1\colon X_1\hookrightarrow Y_1$ and $f_2\colon X_2\hookrightarrow Y_2$ the induced map from the pushout
\begin{equation*}
X_1\tensor Y_2 \ds\coprod_{X_1\tensor X_2} X_2\tensor Y_1 \to Y_1\tensor Y_2
\end{equation*}
is a cofibration which is acyclic if $f_1$ or $f_2$ is.
\end{definition}

\begin{lemma}\label{idlevel}  Let $\op{D}$ be a model category and let $\op{E}$ be a small category whose only morphisms are identity morphisms.  Then the category of functors $\mathbf{Fun}(\op{E},\op{D})$ carries a model category structure where a morphism (natural transformation) $\phi\colon\op{P}\to\op{Q}$ is a weak equivalence/fibration/cofibration if and only if $\phi\colon\op{P}(X)\to\op{Q}(X)$ is for each $X\in\op{E}$.  Moreover if $\op{D}$ is cofibrantly generated so is $\mathbf{Fun}(\op{E},\op{D})$.
\end{lemma}
\begin{proof}  See e.g. \cite{Hir} proposition 7.1.7 and proposition 11.1.10.
\end{proof}

\subsubsection{Transfer Principle}  The technique that we will use to endow the category of $\F$-$\opers$ in $\op{C}$ with a model structure is to transfer the model structure across adjunctions from $\V$-sequences to $\V$-modules and then to $\F$-$\opers$.

\begin{theorem}\label{transferthm1} (\cite{Hir} Theorem 11.3.2)  Let $\op{C}$ be a cofibrantly generated model category with generating cofibrations $I$ and generating acyclic cofibrations $J$.  Let $\op{D}$ be a category which is complete and cocomplete and let $\adj{F}{\op{C}}{\op{D}}{G}$ be an adjunction.  Further suppose that
\begin{enumerate}
\item  Both $F(I)$ and $F(J)$ permit the small object argument and
\item $G$ takes relative $F(J)$-cell complexes to weak equivalences.
\end{enumerate}
Then there is a cofibrantly generated model category structure on $\op{D}$ in which $F(I)$ is the set of generating cofibrations, $F(J)$ is the set of generating acyclic cofibrations, and a map is a weak equivalence (resp. fibration) if and only if its image by $G$ is.
\end{theorem}

In order to apply Theorem $\ref{transferthm1}$ in our cases of interest we will reformulate the hypotheses to derive the following corollary.  This reformulation is inspired by \cite{SS}, \cite{BM1} and \cite{Fresse}.

\begin{corollary}\label{transferprinciple}  Let $\op{C}$ be a cofibrantly generated model category with generating cofibrations $I$ and generating acyclic cofibrations $J$.  Let $\op{D}$ be a category which is complete and cocomplete and let $\adj{F}{\op{C}}{\op{D}}{G}$ be an adjunction.  Further suppose the following:
\begin{enumerate}
\item[(i)]  All objects of $\op{C}$ are small,
\item[(ii)] $G$ preserves filtered colimits,
\item[(iii)]  $\op{D}$ has a fibrant replacement functor,
\item[(iv)]  $\op{D}$ has functorial path objects for fibrant objects.
\end{enumerate}
Then there is a cofibrantly generated model category structure on $\op{D}$ in which $F(I)$ is the set of generating cofibrations, $F(J)$ is the set of generating acyclic cofibrations, and a map is a weak equivalence (resp. fibration) if and only if its image by $G$ is.
\end{corollary}
\begin{proof}  Define in $\op{D}$ a class of weak equivalences and fibrations as in the statement, and define a class of cofibrations as morphisms having the left lifting property with respect to acyclic fibrations. It is enough to show that the conditions $(i-iv)$ in the statement imply conditions $(1)$ and $(2)$ of Theorem $\ref{transferthm1}$.  In particular we show that $(i),(ii)$ imply $(1)$ and $(iii),(iv)$ imply $(2)$.

First, since $G$ preserves filtered colimits, $F$ preserves small objects.  Since all objects in $\op{C}$ are small, all objects of the form $F(A)$ are small in $\op{D}$.  In particular, given a generating cofibration (resp. acyclic cofibration) $F(A)\to F(B)$ in $F(I)$ (resp. $F(J)$), $F(A)$ is small in $\op{D}$, and so in particular is small relative to the subcategory of $F(I)$-cell (resp. $F(J)$-cell) complexes.  Thus $F(I)$ (resp. $F(J)$) permits the small object argument.

Second, let $\phi$ be a relative $F(J)$-cell complex.  Then $\phi$ is an $F(J)$-cofibration (\cite{Hir} proposition 10.5.10) in $\op{D}$ and hence has the LLP with respect to $F(J)$-injectives.  One quickly sees that every fibration in $F(J)$ is an $F(J)$-injective, and thus $\phi$ is a cofibration in $\op{D}$.    Let $R$ denote the fibrant replacement functor in $\op{D}$ and $P$ denote the path objects for fibrants.  Let $\lambda\colon Y\to RX$ denote the resulting lift of $\phi$ and consider the diagram:
\begin{equation*}
\xymatrix{X \ar[r]^{\phi} \ar[d]^{\phi} & Y  \ar[r]^{R_Y} & R(Y)\ar[r]^{\sim} & P(RY) \ar@{>>}[d] \\ \ar@{.>}[urrr] Y\ar[rrr]_{(R_Y,R(\phi)\circ \lambda)} &&& RY\prod RY}
\end{equation*}
Since $R$ is functorial, the diagram commutes.  Note that since $G$ preserves limits, weak equivalences, and fibrations it also preserves path objects and thus $G(R_Y)$ is right homotopic to $G(R(\phi)\circ \lambda)$ in $\op{C}$.  Since $G(R_Y)$ is a weak equivalence it follows (\cite{Hir} proposition 7.7.6) that $G(R(\phi)\circ \lambda)$ is a weak equivalence.  It follows that $G(R(\phi))$ and $G(\lambda)$ induce an isomorphism in the homotopy category of $\op{C}$ and are thus weak equivalences.  It follows from the 2-out-of-3 axiom that $G(\phi)$ is a weak equivalence which completes the proof.
\end{proof}

\begin{remark}  The smallness requirement (condition (i)) does not hold in the category of topological spaces, but we will circumvent this problem following \cite{Fresse}, see Example $\ref{tsex}$.
\end{remark}

\subsubsection{Strictification}  In our arguments below we would like to compose an $\FF$-$\oper$ with a lax monoidal functor.  In general however, this composition would no longer be in $\fopsc$, and thus we introduce the following notation.

\begin{definition}\label{fctransfer}  Let $\op{C}$ and $\op{D}$ be symmetric monoidal categories, let $\FF$ be a strict Feynman category, and let $\gamma\colon\op{C}\to\op{D}$ be a lax symmetric monoidal functor.  We define a functor $\hat{\gamma}\colon\fopsc\to\fdops$ as follows.  For $\op{O}\in\fopsc$ we define $\hat{\gamma}(\op{O})$ restricted to $\V$ to be the composition $\gamma\circ\op{O}$.  We then extend strict monoidally to all objects,
\begin{equation*}
\hat{\gamma}(\op{O})(X):=\bigotimes_{i\in \Lambda}\gamma(\op{O}(v_i))  \ \text{ for } X=\ds\tensor_{i\in\Lambda} v_i \\
\end{equation*}
Finally for morphisms we define the image of generating morphisms by composing with the monoidal structure maps for the symmetric monoidal functor $\gamma$, that is for $\phi\colon X\to v$ we have,
\begin{equation*}
\hat{\gamma}(\phi)\colon \tensor_{i\in \Lambda}\gamma(\op{O}(v_i))\to \gamma(\tensor_{i\in \Lambda}\op{O}(v_i))\stackrel{\gamma(\phi)}\to \gamma(\op{O}(v))
\end{equation*}
and then extend strict monoidally to all morphisms.
\end{definition}

\subsubsection{$\tensor$-coherent path objects}

\begin{definition}\label{pathobjectdef}  We say that $\op{C}$ has functorial path objects for fibrant objects if there is a functor $P\colon \op{C}^{\text{fib}}\to\op{C}$, and, for each object $A$ in $\op{C}$, factorizations of the diagonal $\Delta_A=\psi_A\circ\phi_A$ as a weak equivalence followed by a fibration, such that the following diagram commutes:
\begin{equation*}
\xymatrix{A \ar[rr]_{\sim}^{\phi_A} \ar[d]^{f} && P(A) \ar@{>>}[rr]^{\psi_A} \ar[d]^{P(f)} && \ar[d]^{f\prod f} A\prod A \\ B \ar[rr]_{\sim}^{\phi_B} && P(B) \ar@{>>}[rr]^{\psi_B}  && B\prod B }
\end{equation*}
\end{definition}

Recall the notion of a symmetric monoidal natural transformation between symmetric monoidal functors.  In particular $\eta$ being a symmetric monoidal natural transformation requires:
\begin{equation*}
\xymatrix{F(A)\tensor F(B) \ar[d] \ar[r]^{\eta_A\tensor \eta_B} & G(A)\tensor G(B) \ar[d]\\ F(A\tensor B) \ar[r]^{\eta_{A\tensor B}}& G(A\tensor B)}
\end{equation*}

We will make use of the following example of a symmetric monoidal natural transformation:

\begin{example}\label{diagex}
Let $\op{C}$ be symmetric monoidal category with products.  Define a symmetric monoidal functor $\Delta_\op{C}\colon \op{C}\to\op{C}$ by taking $\Delta_\op{C}(X):=X\prod X$, the product on morphisms and the symmetric structure given by
\begin{equation*}
(\pi_1\tensor \pi_1) \prod (\pi_2\tensor \pi_2)\colon(A\prod A) \tensor (B \prod B) \to (A\tensor B)\prod (A\tensor B)
\end{equation*}
\end{example}
There is a canonical natural transformation $\text{id}_\op{C}\Rightarrow\Delta_\op{C}$, which we call the diagonal transformation, given by the diagonal maps $X\to X\small\prod X$, and it is easily checked that this natural transformation is symmetric monoidal.

\begin{definition}\label{podef}
Let $\op{C}$ be a symmetric monoidal category and a model category.  We say $\op{C}$ has $\tensor$-coherent path objects if there is a symmetric monoidal functor $P\colon \op{C}\to\op{C}$ along with symmetric monoidal natural transformations
\begin{equation*}
\text{id}_\op{C}\stackrel{\phi}\Rightarrow P\stackrel{\psi}\Rightarrow \Delta_\op{C}
\end{equation*}
which factor the diagonal natural transformation $id_\op{C}\Rightarrow\Delta_\op{C}$ (see Example $\ref{diagex}$) and such that $\phi_A$ is a weak equivalence and $\psi_A$ is a fibration for every $A\in\op{C}$.  Furthermore, we say $\op{C}$ has $\tensor$-coherent path objects for fibrant objects if there is a symmetric monoidal functor $P\colon \op{C}_\tensor^{\text{fib}}\to\op{C}$ (where $\op{C}_\tensor^{\text{fib}}:=(\op{C}^{\text{fib}})_\tensor$) along with symmetric monoidal natural transformations
\begin{equation*}
\text{id}_{\op{C}_\tensor^{\text{fib}}}\stackrel{\phi}\Rightarrow P\stackrel{\psi}\Rightarrow \Delta_{\op{C}}
\end{equation*}
which factor the diagonal natural transformation $\text{id}_{\op{C}_\tensor^{\text{fib}}}\Rightarrow\Delta_{\op{C}}$ (interpreted as a natural transformation of functors $\op{C}^{\text{fib}}_\tensor\to\op{C}$) and such that $\phi_A$ is a weak equivalence and $\psi_A$ is a fibration for every $A\in\op{C}^{\text{fib}}$.
\end{definition}

Note that the symmetric monoidal functor $P$ need not be strong monoidal.  In this lax case, composition with $P$ does not directly induce an endofunctor of $\fopsc$.  However, using the strictificaion procedure mentioned above we can get an endofunctor $\hat{P}$, which will be the path object functor for $\fopsc$.

\subsection{The Model Structure}

\begin{theorem}\label{modelthm}  Let $\FF$ be a Feynman category and let $\op{C}$ be a cofibrantly generated model category and a closed symmetric monoidal category having the following additional properties:
\begin{enumerate}
\item  All objects of $\op{C}$ are small.
\item  $\op{C}$ has a symmetric monoidal fibrant replacement functor.
\item  $\op{C}$ has $\tensor$-coherent path objects for fibrant objects.
\end{enumerate}
Then $\fopsc$ is a model category where a morphism $\phi\colon \op{O}\to\op{Q}$ of $\F$-$\opers$ is a weak equivalence (resp. fibration) if and only if $\phi\colon \op{O}(v)\to\op{Q}(v)$ is a weak equivalence (resp. fibration) in $\op{C}$ for every $v\in \V$.
\end{theorem}
\begin{proof}  Let $R$ denote the symmetric monoidal fibrant replacement functor for $\op{C}$ and let $P$ denote the path object functor for fibrant objects in $\op{C}$.  The proof of this theorem will be achieved via:
\begin{equation*}
\text{Lemma }\ref{idlevel}\Rightarrow \left\{\begin{tabular}{c}model structure\\on $\vseq$ \end{tabular}  \right\} \stackrel{\substack{\text{transfer}\\ \text{principle}}} \longrightarrow \left\{ \begin{tabular}{c}model structure\\on $\fopsc$\end{tabular}  \right\}
\end{equation*}
That is, since by Lemma $\ref{idlevel}$ the category of $\V$-sequences in $\op{C}$ has a levelwise cofibrantly generated model structure, the proof will amount to a verification that the hypotheses for the transfer principle (Corollary $\ref{transferprinciple}$) are satisfied with respect to the composite adjunction
\begin{equation*}
\vseq\leftrightarrows \vmodsc\leftrightarrows \fopsc
\end{equation*}
First by Lemma $\ref{limitslemma3}$ the category $\fopsc$ is complete and cocomplete.  Next we show that all objects of $\vseq$ are small.  Pick a cardinal $\kappa_1$ such that all objects of $\op{C}$ are $\kappa_1$-small (such a cardinal exists; see \cite{Hir} Lemma 10.4.6).  Then pick a cardinal $\kappa$ which is greater than both $\kappa_1$ and the cardinality of $\V$.  Let $X\to Z$ be the transfinite composition in $\vseq$ of a $\lambda$-sequence $X=X_0\to X_1\to..$ for any regular cardinal $\lambda\geq\kappa$.  Then for any $\V$-sequence $A$ we have
\begin{equation*}
\colim_{\beta<\lambda}Hom_{\vseq}(A,X_\beta) = \ds\colim_{\beta<\lambda}\prod_{v\in \V}Hom_\C(A(v),X_\beta(v))\cong  \prod_{v\in\V}colim_{\beta<\lambda}Hom_\C(A(v),X_\beta(v))
\end{equation*}
using the fact that $\lambda\geq\kappa>card(\V)$ to interchange the product and the colimit.  Now the fact that $A(v)$ is $\kappa$-small for each $v$ tells us that
\begin{equation*}
\prod_{v\in V}\colim_{\beta<\lambda}Hom_\C(A(v),X_\beta(v))\cong \prod_{v\in \V}Hom_\C(A(v),\colim_{\beta<\lambda}X_\beta(v))= Hom_{\vseq}(A,Z)
\end{equation*}
and hence $A$ is small.  Thus condition (i) is satisfied.  By Lemmas $\ref{limitslemma1}$ and $\ref{limitslemma2}$ the constituent forgetful functors both preserve filtered colimits.  Thus their composite does and condition (ii) of the transfer principle is satisfied.

For condition (iii) of the transfer principle, given $\op{O}\in\fopsc$ there is a functorial fibrant replacement given by $\hat{R}(\op{O})$, where $\hat{R}$ is as in Definition $\ref{fctransfer}$.  Indeed $\hat{R}(\op{O})(v)=R\circ\op{O}(v)$ is fibrant for each $v\in\V$.  Thus it remains to show that condition (iv) of the transfer principle is satisfied, so we now suppose $\op{O}$ to be a fibrant $\F$-$\oper$ and take $\hat{P}(\op{O})$ to be our candidate for a path object functor for fibrant $\F$-$\opers$.  Since $\op{O}$ is a fibrant $\F$-$\oper$, $\op{O}(v)$ is a fibrant object in $\op{C}$ for every $v\in\V$.  As such there are maps in $\op{C}$
\begin{equation*}
\op{O}(v)\stackrel{\phi_v}\longrightarrow P(\op{O}(v))=\hat{P}(\op{O})(v)\stackrel{\psi_v}\longrightarrow \op{O}(v)\prod \op{O}(v)
\end{equation*}
for each $v\in\V$ such that $\phi_v$ is a weak equivalence and $\psi_v$ is a fibration.  Since weak equivalences and fibrations in $\fopsc$ are induced levelwise, we can show that $\hat{P}(\op{O})$ gives us functorial path objects in $\fopsc$ if we can show that the levelwise diagrams above fit together to give a diagram in $\fopsc$.  To see this it is enough to see that these levelwise diagrams are compatible for a generating morphism $\alpha\colon X=\tensor_{i\in\Lambda}v_i\to v$, which follows from commutativity of the following diagram:
\begin{equation*}
\xymatrix{\op{O}(X)\ar[d]^{=} \ar[rr]^{\tensor_i \phi_{v_i}} && \hat{P}(\op{O})(X) \ar[d] \ar[rr]^{\tensor_i \psi_{v_i}} \ar@{.>}[drr] &&\ar[d]^{(\tensor_i\pi_1)\prod(\tensor_i\pi_2)} \tensor_i (\op{O}(v_i)\prod \op{O}(v_i)) \\ \op{O}(X) \ar[rr]^{\phi_X} \ar[d]^{\op{O}(\alpha)}&& P(\op{O}(X))\ar[rr]^{\psi_X} \ar[d]^{P(\alpha)} && \op{O}(X)\prod\op{O}(X) \ar[d]^{\alpha\prod\alpha} \\ \op{O}(v)\ar[rr]_{\phi_v}^\sim && \hat{P}(\op{O})(v) \ar@{>>}[rr]_{\psi_v} && \op{O}(v)\prod\op{O}(v)}
\end{equation*}
Note the two bottom squares commute by the assumption that $\phi$ and $\psi$ are natural transformations and the top two squares commute by the assumption that these natural transformations are symmetric monoidal.  Composing to the dotted arrow we get a diagram of $\F$-$\opers$
\begin{equation*}
\op{O}\stackrel{\sim}\to\hat{P}(\op{O})\twoheadrightarrow\op{O}\prod\op{O}
\end{equation*}
which is a weak equivalence followed by a fibration, because it is levelwise for each $v\in\V$.  Thus $\fopsc$ has functorial path objects for fibrant objects, which permits this application of the transfer principle and hence completes the proof.
\end{proof}

\begin{remark}\label{vmodmodelrmk}  Note that since the category of $\V$-modules is equivalent to the category of $Iso(\F)$-$\opers$, the above theorem also gives $\vmodsc$ a transferred model structure when the condition on $\op{C}$ are met.
\end{remark}

\subsubsection{Existence of $\tensor$-coherent path objects}

Before considering examples of categories which satisfy the conditions of Theorem $\ref{modelthm}$ we will consider how, in practice, one can establish the existence of $\tensor$-coherent path objects.  These conditions were inspired by \cite{BM1}.

\begin{definition}  We say $J$ is an interval in $\op{C}$ if there is a factorization of the folding map of $I$ as
\begin{equation*}
I\sqcup I \hookrightarrow J \stackrel{\sim}\to I
\end{equation*}
where $I$ is the monoidal unit, and $\sqcup$ is the coproduct in $\C$ (we will denote the product in $\C$ by $\sqcap$).  We say that $J$ is a cocommutative interval if $J$ is an interval with a cocommutative coassociative counital comultiplication $J\to J\tensor J$, with counit map as above $J \stackrel{\sim}\to I$.
\end{definition}

Note that if $\op{C}$ is Cartesian closed then $\op{C}$ necessarily has a cocommutative interval via the diagonal for any factorization of the folding map.

\begin{lemma}\label{polem1}  Suppose that $\op{C}$ satisfies PPA, that the monoidal unit is cofibrant, and that $\op{C}$ has a cocommutative interval $J$.  Then $P:=hom(J,-)$ defines $\tensor$-coherent path objects for fibrant objects in $\op{C}$.
\end{lemma}
\begin{proof}  First note $J$ having a cocommutative, coassociative multiplication makes $P\colon\op{C}\to\op{C}$ a symmetric monoidal functor by precomposition $J\to J\tensor J$ (and unit axiom by $I\cong Hom(I,I)$).  Second notice that since $\op{C}$ is closed symmetric monoidal, $X\tensor -$ is a left adjoint and so commutes with colimits, and so in particular $X\sqcup X\cong X\tensor (I\sqcup I)$ and hence $hom(I,X)\sqcap hom(I,X)\cong hom(I\sqcup I,X)$.

Note that if $X$ is a fibrant object we may apply the contravariant functor $hom(-,X)$ to the factorization above to get a path object for $X$ which is functorial:
\begin{equation*}
X\cong hom(I,X) \stackrel{\sim}\to hom(J,X)\twoheadrightarrow hom(I\sqcup I,X)\cong X\sqcap X
\end{equation*}
The fact that $hom(I,X) \stackrel{\sim}\to hom(J,X)$ is a weak equivalence follows from the PPA and the fact that $I$, hence $J$, is cofibrant and $X$ is fibrant (see \cite{BM1} lemma 2.3).  The fact that $hom(J,X)\twoheadrightarrow hom(I\sqcup I,X)$ is a fibration follows from the PPA by taking the exponential transpose, and the fact that $I\sqcup I\hookrightarrow J$ is a cofibration (see \cite{Fresse} p.14).

It is then straight forward to check that the induced natural transformations are symmetric monoidal.
\end{proof}

We can formulate adjoint conditions as well:

\begin{definition}  We say $K$ is a cointerval in $\op{C}$ if $K\in\op{C}^\text{op}$ is an interval.  That is, $K$ is a cointerval if there is a factorization of the diagonal map of $I$ as a weak equivalence followed by a fibration
\begin{equation*}
I \stackrel{\sim}\to K \twoheadrightarrow  I\sqcap I
\end{equation*}
where $I$ is the monoidal unit.  We say that $K$ is a commutative cointerval if $K$ is a cointerval with a commutative associative unital multiplication $K\tensor K\to K$ with the unit map as above $I \stackrel{\sim}\to K$.
\end{definition}

\begin{lemma}\label{polem3}  Suppose that $\op{C}$ is not Cartesian closed and that the monoidal product is distributive with respect to the categorical product.  Suppose also that $\op{C}$ satisfies the following conditions:
\begin{enumerate}
\item $I$ is fibrant.
\item $\op{C}$ has a commutative cointerval $K$.
\item For any fibrant object $X$, the functor $-\tensor X$ preserves weak equivalences between fibrant objects and preserves fibrations.
\end{enumerate}
Then $P:=-\tensor K$ defines $\tensor$-coherent path objects for fibrant objects in $\op{C}$.
\end{lemma}
\begin{proof}  Tensoring with the cointerval sequence we have, for any fibrant object $X$,
\begin{equation*}
X\cong X\tensor I \stackrel{\sim}\to X\tensor K\twoheadrightarrow X\tensor(I\sqcap I)\cong X\sqcap X
\end{equation*}
and the remaining details are easily checked.
\end{proof}

\subsubsection{Examples}
We will now consider a short list of the primary examples and nonexamples of categories which satisfy the conditions of Theorem $\ref{modelthm}$ that we wish to consider.

\begin{example}  (Simplicial Sets)  The category of simplicial sets is a Cartesian closed model category with injections as cofibrations and realization weak equivalences as weak equivalences (e.g. \cite{Hovey} section 3.2).  All simplicial sets are small (\cite{Hovey} Lemma 3.1.1). There is a fibrant replacement functor given by taking the singular complex of the realization.  All simplicial sets are cofibrant; in particular the monoidal unit is cofibrant, thus Lemma $\ref{polem1}$ applies and simplicial valued $\FF$-$\opers$ form a model category.
\end{example}

\begin{example}  (Vector Spaces in characteristic $0$)  Let $k$ be a field of characteristic zero and let $\dgvect_k$ be the category of differential graded $k$ vector spaces.  We consider $\dgvect_k$ to be a model category with weak equivalences / fibrations / cofibrations given by quasi-isomorphisms / surjections / injections.  Note that all objects in this model category are fibrant, so the identity gives a symmetric monoidal fibrant replacement functor.  Let $k[t,dt]$ be the unique commutative dga of polynomials in the variables $t$ and $dt$ which satisfies $d(t)=dt$ and $d(dt)=0$.  Consider the sequence
\begin{equation*}
k\stackrel{\phi}\longrightarrow k[t,dt]\stackrel{\psi}\longrightarrow k\oplus k
\end{equation*}
where $\phi(a)=a$ and $\psi(f(t)+g(t,dt)dt)=f(0)\oplus f(1)$.  Note that the composite $\psi\circ\phi$ is the diagonal.  Further notice that $\psi$ is surjective and that $\phi$ is a quasi-isomorphism, since $k[t,dt]$ is acyclic in characteristic $0$.  Thus $K=k[t,dt]$ is a commutative cointerval object.  Notice that the other conditions of Lemma $\ref{polem3}$ are clearly satisfied, thus Theorem $\ref{modelthm}$ applies and linear $\F$-$\opers$ in characteristic $0$ form a model category.
\end{example}

\begin{remark}\label{porem}
By the above work, given an $\F$-$\oper$ $\op{O}$ in $\dgvect_k$ (characteristic $0$) we get a factorization of the diagonal in the category of $\F$-$\opers$:
\begin{equation*}
\op{O}\stackrel{\sim}\to \op{O}[t,dt]\twoheadrightarrow \op{O}\oplus \op{O}
\end{equation*}
where $\op{O}[t,dt]$ is the $\F$-$\oper$ given by composition of the symmetric monoidal functors:
\begin{equation*}
\F\stackrel{\op{O}}\longrightarrow \op{C}\stackrel{-\tensor k[t,dt]}\longrightarrow \op{C}
\end{equation*}
In particular for an object $v\in\V$ we have $\op{O}[t,dt](v)=\op{O}(v)\tensor k[t,dt]$ and for a general object $X\in\F$ with $X\cong\tensor_iv_i$ we define $\op{O}[t,dt](X)= \tensor_i(\op{O}(v_i)\tensor k[t,dt])$.  Then for a generating morphism $\phi\colon X\to v$ we have
\begin{equation*}
\op{O}[t,dt](X)\cong \tensor_{i\in I}\op{O}(v_i)\tensor k[t,dt]^{\tensor I}\stackrel{\phi\tensor\mu}\longrightarrow \op{O}(v)\tensor k[t,dt]=\op{O}[t,dt](v)
\end{equation*}
where $\mu$ is the commutative associative multiplication in $k[t,dt]$.
\end{remark}

It is important to notice that the characteristic zero assumption in the previous example is necessary.  In particular the path object given above requires characteristic zero for the map $\phi$ to be a weak equivalence.  More generally, we have the following lemma.
\begin{lemma}  Let $\op{C}$ be the category of differential graded vector spaces in a field of characteristic $p>0$.  Then $\op{C}$ does not have $\tensor$-coherent path objects for fibrant objects.
\end{lemma}
\begin{proof}  By contradiction suppose $P(-)$ gives us $\tensor$-coherent path objects in $\op{C}$.  Choose a commutative dga $A$ which contains an element $x$ of even degree which represents a nonzero class in homology and such that $x^p$ represents a nonzero class in homology.  Suppose we have a factorization of the diagonal as:
\begin{equation*}
\xymatrix{A \ar[r]^{\sim} & P(A) \ar@{>>}[r]^{g} & A\oplus A}
\end{equation*}
Then $g$ is surjective so there is a $y\in P(A)$ such that $g(y)=(x,0)$.  Then, since $A$ is commutative and since $y$ is of even degree we have $d(y^p)=py^{p-1}=0$.  Tracing through Definition $\ref{pathobjectdef}$ we see that $g(y^p)=(x^p,0)$.  Now taking homology of the above diagram we have:
\begin{equation*}
H_\ast(A) \cong H_\ast(P(A)) \stackrel{g_\ast}\longrightarrow H_\ast(A)\oplus H_\ast(A)
\end{equation*}
On the one hand the composite is the diagonal, but on the other hand the element $([x^p],[0])$ is in the image of the composite.  This contradiction proves the lemma.
\end{proof}
\begin{rmk}
A similar argument can be made for e.g. chain complexes in $\mathbb{Z}$ modules.
\end{rmk}

\begin{example}\label{tsex}(Topological Spaces)  Let $\Top$ be the category of compactly generated spaces with the Quillen model structure: weak equivalences given by weak homotopy equivalences and fibrations given by Serre fibrations.  Then $\Top$ is a cofibrantly generated symmetric monoidal model category (see e.g. \cite{Hovey}) having all objects fibrant.  Since $\Top$ is Cartesian closed and since the monoidal unit, a point, is cofibrant, Lemma $\ref{polem1}$ applies to endow $\Top$ with $\tensor$-coherent path objects.  Thus conditions (ii),(iii),(iv) of Corollary $\ref{transferprinciple}$ are satisfied.  However condition (i) is not satisfied and so we can not directly apply the result.  In order to circumvent this problem we follow \cite{Fresse} and use the fact that all objects in $\Top$ are small with respect to topological inclusions.  The details are given in the appendix, but we record the result here:
\end{example}

\begin{theorem}
\label{feymodthm}
Let $\op{C}$ be the category of topological spaces with the Quillen model structure.  The category $\fopsc$ has the structure of a cofibrantly generated model category in which the forgetful functor to $\vseq$ creates fibrations and weak equivalences.
\end{theorem}

See Appendix $\ref{appb}$.

\subsection{Quillen adjunctions from morphisms of Feynman categories}
\label{Quillensec}
We assume $\op{C}$ is a closed symmetric monoidal and model category satisfying the assumptions of Theorem $\ref{modelthm}$.  Let $\fr{E}$ and $\FF$ be Feynman categories and let $\alpha\colon \fr{E}\to\FF$ be a morphism between them.  Recall (Theorem $\ref{pushthm}$) this morphism induces an adjunction
\begin{equation*}
\alpha_\ast\colon\eops \leftrightarrows \fopsc\colon \alpha^\ast
\end{equation*}
where $\alpha^\ast(\op{A}):= \op{A}\circ\alpha$ is the right adjoint and $\alpha_\ast(\op{B}):=Lan_\alpha(\op{B})$ is the left adjoint.  In this section we will see that several prominent examples of such adjunctions are in fact Quillen adjunctions.

\begin{lemma}\label{qalem} Suppose $\alpha^\ast$ restricted to $\vfmods\to \vemods$ preserves fibrations and acyclic fibrations (see Remark $\ref{vmodmodelrmk})$.  Then the adjunction  $(\alpha_\ast, \alpha^\ast)$ is a Quillen adjunction.
\end{lemma}
\begin{proof}  To show this adjunction is a Quillen adjunction it is enough to show that $\alpha^\ast$ is a right Quillen functor (see e.g. \cite{Hovey} Lemma 1.3.4), i.e. that $\alpha^\ast$ preserves fibrations and acyclic fibrations (on the entire domain $\fopsc$).  This follows by the assumption of the lemma and the commutativity of the following diagram,
\begin{equation*}
\xymatrix{\eops  \ar[d]^{G_\fr{E}} & \ar[l]^{\alpha^\ast} \fopsc \ar[d]^{G_\FF} \\ \vemods  & \ar[l]^{\alpha^\ast} \vfmods  }
\end{equation*}
along with the fact that $G_\fr{E}$ and $G_\FF$ preserve and reflect weak equivalences and fibrations.
\end{proof}

Common examples where the conditions of the Lemma $\ref{qalem}$ are satisfied come from the standard adjunctions between various $\opcat$, several of which we now describe.
\begin{example}  Lemma $\ref{qalem}$ immediately implies that the forgetful/free adjunction between $\vmodsc$ and $\fopsc$ is a Quillen adjunction.
\end{example}

\begin{example}\label{qaex1}  Recall that $\CCyclic$ and $\modular$ denote the Feynman categories whose $\opers$ are cyclic and modular operads respectively and that there is a morphism $i\colon\CCyclic\to\modular$ by including as genus zero.  As discussed above, this morphism induces an adjunction between cyclic and modular operads
\begin{equation*}
i_\ast\colon\CCyclic\text{-}\op{O}ps_\op{C} \leftrightarrows  \modular\text{-}\op{O}ps_\op{C}\colon i^\ast
\end{equation*}
and the left adjoint is called the modular envelope of the cyclic operad.  The fact that the morphism of Feynman categories is inclusion means that $i^\ast$ restricted to the underlying $\V$-modules is given by forgetting, and since fibrations and weak equivalences are levelwise, $i^\ast$ restricted to the underlying $\V$-modules will preserve fibrations and weak equivalences.  Thus by Lemma $\ref{qalem}$ this adjunction is a Quillen adjunction.
\end{example}

\begin{example}\label{qaex2}  In the directed case we have various morphisms of Feynman categories given by inclusion.  For example $\operads\to\properads\to \props$ or $\operads\to\dioperads$, and, as discussed above, the induced adjunctions recover various free constructions in the literature, (see e.g. \cite{Vallette},\cite{MVor}).  Since restriction to the underlying $\V$-modules is the identity or inclusion in each example, Lemma \ref{qalem} applies and each associated adjunction is a Quillen adjunction.  In particular we have a Quillen adjunction between the categories of operads and props in such a model category.
\end{example}

\begin{example}  Consider the morphism $\dioperads\to\CCyclic$ given by forgetting the directed structure on the morphisms.  On the underlying $\V$-modules, the right adjoint is given by restriction of the $S_{n+m}$-module to an $S_{n,m}$-module and so Lemma \ref{qalem} applies to show that the adjunction between dioperads and cyclic operads is a Quillen adjunction.

\end{example}

\begin{remark}  In this section we have given conditions on an adjunction induced by a morphism of Feynman categories (source categories) to be a Quillen adjunction.  Similarly it is possible to give conditions on a symmetric monoidal adjunction of base categories such that the induced adjunction is a Quillen adjunction or Quillen equivalence.  In particular one can show that the categories of simplicial and topological $\F$-$\opers$ are Quillen equivalent.
\end{remark}

\subsection{Cofibrant objects}  We continue to assume $\op{C}$ is a closed monoidal category and a model category such that the conditions of Theorem $\ref{modelthm}$ are satisfied.  Let $\adj{F}{\fopsc}{\vmodsc}{G}$ be the free and forgetful adjunction.

\begin{proposition}\label{cofprop}  Let $\Phi\in\vmodsc$ be cofibrant.  Then $F(\Phi)\in\fopsc$ is cofibrant.
\end{proposition}
\begin{proof}  Since $F$ is a left Quillen functor it preserves cofibrations and initial objects, hence preserves cofibrant objects.
\end{proof}
For the remainder of this section we fix $\op{C}$ to be the category of dg vector spaces over a field of characteristic $0$.  Recall that an object in $\fopsc$ is quasi-free if its image under the forgetful functor to graded vector spaces is free.  We now show that quasi-free $\F$-$\opers$ are cofibrant.  The argument is a generalization of that for operads \cite{Fressebook} and properads \cite{MerkVal2}.

\begin{theorem}\label{cofthm}  Let $\op{Q}=(F(\Phi),\delta)$ be a quasi-free $\F$-$\oper$.  Furthermore assume that $\Phi$ admits an exhaustive filtration
\begin{equation*}
\Phi_0\subset\Phi_1\subset\Phi_2\subset\dots\subset \Phi
\end{equation*}
in the category of $\V$-modules (so that these inclusions are split injections of Aut$(v)$-modules) such that $\delta(\Phi_i)\subset F(\Phi_{i-1})$.  Then $\op{Q}$ is cofibrant in $\fopsc$.
\end{theorem}

\begin{proof}  The proof follows as in $\textit{loc.cit}$.
The first step is to reduce the problem to the case of quasi-free objects built from free $\V$-modules; to this end we will show that $\Phi$ is a retract of a free $\V$-module.  Fix $v\in\V$ and let $A:=\Phi(v)$, $A_l:=\Phi_l(v)$ and $H:=Aut(v)$.  Then $A$ is an $H$-representation and we choose an isomorphism of representations $A_l\cong \oplus_{i\in J_l} A_l^i$ where each $A_l^i$ is irreducible.  The filtration allows us to consider $J_l\subset J_{l+1}$ and $A= \oplus_{i\in J} A^i$ where $J$ is the union of the $J_l$.  Then define maps
\begin{equation*}
A\stackrel{\epsilon}\to k[J]\tensor \text{Reg}(H) \stackrel{\eta}\to A
\end{equation*}
by taking
\begin{equation*}
\epsilon(A^i)=i\tensor A_i \ \text{   and   } \  \eta(j\tensor A_i)=\begin{cases} A_i \text{ if } i=j \\ 0 \text{  else} \end{cases}
\end{equation*}
where we have arbitrarily chosen for each $i\in J$ a copy of $A_i$ appearing as a summand of Reg$(H)$, the regular representation.  By construction the maps $\eta$ and $\epsilon$ are $H$-equivariant and the composition is the identity.  Define $\Gamma(v):=k[J]$ and then we have, by taking the above construction at each level, a retraction of $\V$-modules
\begin{equation*}
\Phi\to R(\Gamma)\to\Phi
\end{equation*}
where $R$ is the free functor $\vseq\to\vmods$.  Finally notice that the filtration on $\Phi$ induces a filtration on $\Gamma$ by $\Gamma_l(v):=k[J_l]$ and if we define $\delta^\prime \colon FR(\Gamma)\to FR(\Gamma)$ by $\delta^\prime= FR(\epsilon)\circ \delta\circ FR(\eta)$ then the filtration on $\Gamma$ satisfies $\delta^\prime(\Gamma_i)\subset FR(\Gamma_{i-1})$.

By the above construction we have a retraction of quasi-free $\F$-$\opers$
\begin{equation*}
\op{Q}=(F(\Phi),\delta)\to (FR(\Gamma),\delta^\prime) \to \op{Q}
\end{equation*}
and so if we can show that the quasi-free $\F$-$\oper$ $(FR(\Gamma),\delta^\prime)$ is cofibrant, it will follow that $\op{Q}$ is cofibrant.

Fix a natural number $j$ and define two $\V$-sequences $S^j$ and $D^j$ as follows.  For a fixed $j$ and $v\in\V$, choose a basis $\{e_b : b\in B\}$ for the vector space $\Gamma_{j+1}(v)\setminus\Gamma_j(v)$.  We then define
\begin{equation*}
S^j(v)=\ds\bigoplus_{b\in B} \Sigma^{|e_b|-1} k  \ \ \text{  and  } \ \ D^j(v)=\ds\bigoplus_{b\in B} (\Sigma^{|e_b|} k \oplus \Sigma^{|e_b|-1} k)
\end{equation*}
as vector spaces.  Define $x_b$ to be the generator of $S^j(v)$ corresponding to index $b$ and $y_b\oplus z_b$ to be the generator of $D^j(v)$ corresponding to index $b$.  Define differentials on these vector spaces by $d(x_b)=0$ and $d(y_b)=z_b$ and define the canonical inclusion $\text{in}_j\colon S^j(v)\to D^j(v)$ by $x_b\mapsto z_b$.

Taking this construction for all $v\in\V$ gives us $\text{in}_j\colon S^j\to D^j$ which is a cofibration in the category of $\V$-sequences.  As such $FR(\text{in}_j)$ is a cofibration.  Define a map $f\colon FR(S^j)\to (FR(\Gamma_j),\delta^\prime)$ to be the image under adjunction of the map $S^j(v) \to (FR(\Gamma_j)(v),\delta^\prime)$ taking $x_b\mapsto \delta^\prime(e_b)$.  We have the following pushout diagram in $\fopsc$:
\begin{equation*}
\xymatrix{ FR(S^j) \ar[dd]_{FR(\text{in}_j)} \ar[rr]^f && (FR(\Gamma_j), \delta^\prime) \ar@{.>}[dd] \\ && \\FR(D^j) \ar@{.>}[rr] && \ds\frac{(FR(\Gamma_j),\delta^\prime) \oplus FR(D^j)}{d(y_b)=\delta^\prime(e_b)}}
\end{equation*}
Then under the identification $e_b\mapsto y_b$, this pushout is isomorphic to $(FR(\Gamma_{j+1}), \delta^\prime)$.  Moreover, the morphism $(FR(\Gamma_{j}), \delta^\prime)\hookrightarrow(FR(\Gamma_{j+1}), \delta^\prime)$ is the pushout of the cofibration $FR(\text{in}_j)$ and hence a cofibration.  Taking the colimit over all $j$ we see that $(FR(\Gamma), \delta^\prime)$ is cofibrant, hence the proof.
\end{proof}

\begin{corollary}
\label{cofcor1}
The Feynman transform (Definition $\ref{ftdef}$) of a non-negatively graded dg $\F$-$\oper$ is cofibrant.
\end{corollary}

\begin{corollary}\label{cofcor}  The double Feynman transform of a non-negatively graded dg $\F$-$\oper$ in a cubical Feynman category is a cofibrant replacement.

\end{corollary}

\subsection{Homotopy classes of maps and master equations}
In this section we work over the category $\op{C}=\dgvect_k$ where $k$ is a field of characteristic $0$.  By the above work we know that the Feynman transform produces a cofibrant object and that all $\F$-$\opers$ are fibrant in $\dgvect_k$.  Thus we can consider homotopy classes of maps in this context, whose definition we now recall.

\begin{definition}\label{homdef2}  Let $\op{M}$ be a cofibrant $\F$-$\oper$, let $\op{O}$ be any $\F$-$\oper$, and let  $\gamma_1,\gamma_2 \in Hom(\op{M},\op{O})$.  Then $\gamma_1$ and $\gamma_2$ are homotopic if there exists a path object $P(\op{O})$ of $\op{O}$ along with a lift of the map $\gamma_1\oplus\gamma_2$ such that the following diagram commutes:
\begin{equation}\label{podi}
\xymatrix{&& P(\op{O}) \ar@{>>}[d] && \op{O} \ar[ll]_{\sim} \ar[dll]^{\Delta} \\ \op{M} \ar@{.>}[rru] \ar[rr]^{\gamma_1\oplus\gamma_2} && \op{O}\oplus\op{O} &&}
\end{equation}
The set of homotopy classes of such maps is denoted $[\op{M},\op{O}]$.
\end{definition}

Thus we have a notion of homotopy equivalence of morphisms from the Feynman transform due to Definition $\ref{homdef2}$.  By the above work (Theorem $\ref{methm2}$) such morphisms can be encoded by solutions to certain master equations.  Often these master equations can be interpreted as the Maurer-Cartan equation in a certain dg Lie algebra, in which case there is a notion of homotopy equivalence of MC solutions which we recall below.  The main result of this subsection will be to show that the notions of homotopy equivalence on both sides of the bijection of Theorem $\ref{methm}$ coincide.

\subsubsection{MC simplicial set}  Let $\Omega_n$ be the commutative dga of polynomial differential forms on the simplex $\Delta^n$,
\begin{equation}
\Omega_n:=\ds\frac{k[t_0\cdc t_n,dt_0\cdc dt_n]}{(\sum t_i-1,\sum dt_i )}
\end{equation}
having $d(t_i):=dt_i$ and with $|t_i|=0$ and $|dt_i|=1$ .  The collection of these commutative dgas $\Omega_\bullet$ is a simplicial object with face and degeneracy maps,
\begin{eqnarray*}
\delta_i\colon\Omega_{n+1}\to\Omega_n \ \ i=0\cdc n+1 & \ \ \ \  & \sigma_j\colon\Omega_{n}\to\Omega_{n+1} \ \ j=0\cdc n \\
\delta_i(t_k) =\begin{cases} t_k & \text{if  } k<i \\ 0 & \text{if  } k=i \\t_{k-1} & \text{if  } k>i \end{cases}
&\ \ \ \  &  \sigma_j(t_k) =\begin{cases} t_k & \text{if  } k<j \\ t_k+t_{k+1} & \text{if  } k=j \\t_{k+1} & \text{if  } k>j \end{cases}
\end{eqnarray*}

Following \cite{Hinich2} and \cite{Get}, given a dgLa $\fr{g}$ we define a simplicial set
\begin{equation*}
MC_\bullet(\fr{g}):=MC(\fr{g}\tensor\Omega_\bullet)
\end{equation*}
where $\fr{g}\tensor\Omega_n$ is a dgLa by taking, for $g,g^\prime\in\fr{g}$ and $f,f^\prime\in\Omega_n$,
\begin{equation*}
d(g\tensor f)= dg\tensor f + (-1)^{|g|}g\tensor d(f) \ \text{      and      } \  [g\tensor f, g^\prime\tensor f^\prime]= [g,g^\prime]\tensor (-1)^{|g^\prime||f|}ff^\prime
\end{equation*}
and with face and degeneracy maps coming from the functorial image of $-\tensor\Omega_\bullet$.  The fact that the image of an MC element is an MC element follows from the fact that the face and degeneracies on $\Omega_\bullet$ are maps of commutative dgas.

In the event that $\fr{g}$ is nilpotent, the simplicial set $MC_\bullet(\fr{g})$ is a Kan complex \cite{Hinich2}, \cite{Get}, i.e. a fibrant simplicial set, and we can consider its homotopy groups $\pi_\ast(MC_\bullet(\fr{g}))$ (see e.g. \cite{W}).  We then define the notion of homotopy equivalence of MC elements by defining $\pi_0(MC_\bullet(\fr{g}))$ to be the set of homotopy classes of MC elements.  In particular
\begin{equation*}
\pi_0(MC_\bullet(\fr{g}))=MC_\bullet(\fr{g})/\sim
\end{equation*}
where
\begin{equation}\label{ereq}
s_0\sim s_1 \Leftrightarrow \ \exists \ s\in MC_1(\fr{g}) \text{ such that } \delta_0(s)=s_0 \text{ and } \delta_1(s)=s_1
\end{equation}
This equivalence relation coincides with the notion of gauge equivalence in the nilpotent case.  Without the nilpotence assumption it is not clear that $\sim$ and its higher dimensional brethren are transitive.  One common approach to this problem is to tensor the dgLas of interest with the maximal ideal in a local Artin ring.  In our context we could also consider truncated operads.  However we will show that $\sim$ defined in Equation $\ref{ereq}$ is an equivalence relation in the context of operadic Lie algebras, and thus $\pi_0$ is well defined in our contexts of interest.  Further study of this simplicial set in the context of generalized operadic Lie algebras and their possible higher homotopy groups is a potentially interesting future direction.

\subsubsection{Homotopy classes theorem}  We can now state and prove the result linking homotopy classes of maps to homotopy classes of MC elements.  The proof will make use of the following lemma.

\begin{lemma}\label{polem2} Let $\op{M}$ be a cofibrant $\F$-$\oper$ and let $\gamma_1,\gamma_2$ be maps $\op{M}\to\op{O}$ for some $\F$-$\oper$ $\op{O}$.  Then $\gamma_1\sim\gamma_2$ if and only if there exist a lift
\begin{equation*}
\xymatrix{&& \op{O}[t,dt] \ar@{>>}[d]^{\psi} && \op{O} \ar[ll]^{\sim}_\phi \ar[dll]^{\Delta} \\ \op{M} \ar@{.>}[rru] \ar[rr]^{\gamma_1\oplus\gamma_2} && \op{O}\oplus\op{O} &&}
\end{equation*}
where $\phi$ and $\psi$ are induced levelwise as in Remark $\ref{porem}$.  In particular $\psi=\{\psi_v\}_v$ is given levelwise by $\psi_v(f(t)+g(t,dt)dt)=(f(0),f(1))$.
\end{lemma}
\begin{proof}  The $\Leftarrow$ direction follows immediately from the definition of $\sim$ and the fact that $\phi$ and $\psi$ make $\op{O}[t,dt]$ a path object for $\op{O}$ by the above work.  The $\Rightarrow$ implication is as follows.  If $\gamma_1\sim\gamma_2$ then there is a $P(\op{O})$ as in diagram $\ref{podi}$.  Since the map $\op{O}\stackrel{\sim}\to P(\op{O})$ is a weak equivalence we may factor it as an acyclic cofibration followed by an acyclic fibration.  Calling the object in the center of this factorization $P^\prime(\op{O})$ we have the following commutative diagram:
\begin{equation*}
\xymatrix{ && && \op{O}\ar@{^{(}->}[dll]_{\sim} \ar[d] \ar@/^2pc/[dd]^\Delta \\ \ast \ar@{^{(}->}[d] \ar[rr] && P^\prime(\op{O}) \ar@{>>}[d]^{\sim} \ar@{.>}[rr]  && \op{O}[t,dt] \ar@{>>}[d]  \\ \op{M} \ar@/_1pc/[rrrr]_{\gamma_1\oplus\gamma_2} \ar@{.>}[urr] \ar[rr] && P(\op{O}) \ar[rr] && \op{O}\oplus\op{O} }
\end{equation*}
The two lifts in the diagram combine to give the lift $\op{M}\to\op{O}[t,dt]$ that we require.

\end{proof}

\begin{theorem}\label{hcthm} For entries in Table $\ref{Ftable}$,
\begin{equation*}
[\FT(\op{P}),\op{O}]\cong\pi_0(ME_\bullet(colim(\op{P} \tensor\op{O})))
\end{equation*}
\end{theorem}

\begin{proof}  Let $\F$ be a Feynman category whose category of linear $\opers$ appears in the table.  Let $\op{P}$ be an odd $\F$-$\oper$ and let $\op{O}$ be an $\F$-$\oper$.  Then by Theorem $\ref{methm}$ we know there is a one to one correspondence
\begin{equation*}
Hom(\FT(\op{P}),\op{O})\stackrel{1\text{ to }1}\longleftrightarrow ME(colim(\op{P}\tensor\op{O}))
\end{equation*}
Let $\gamma_0$ and $\gamma_1$ be maps $\FT(\op{P})\to\op{O}$ and let $\tilde{\gamma}_0$ and $\tilde{\gamma}_1$ be their images via this bijection.  Then we want to show
\begin{equation*}
\gamma_0\sim\gamma_1 \Leftrightarrow \exists \  \tilde{\gamma}\in ME_1(colim(\op{P}\tensor\op{O})) \text{ such that } \delta_i(\tilde{\gamma})=\tilde{\gamma}_i \text{ for } i=0,1.
\end{equation*}
For the implication $\Rightarrow$, suppose that the maps $\gamma_1$ and $\gamma_2$ are homotopic.  Then, by Lemma $\ref{polem2}$ there is a lift of the sum of these maps, call it $\gamma$:
\begin{equation*}
\xymatrix{ & \op{O}[t,dt]\ar[d]^{\psi} \\ \FT(\op{P})\ar@{.>}[ur]^\gamma \ar[r]^{\gamma_0\oplus\gamma_1} & \op{O}\oplus\op{O}}
\end{equation*}
where, as above, $\psi(f(t)+g(t,dt)dt)=(f(0),f(1))$.  Next define an isomorphism of commutative dgas $\Omega_1\cong k[t,dt]$ by sending $t_0\mapsto t$ and $t_1\mapsto 1-t$.  Since $\op{O}[t,dt]$ is in particular a fibrant odd $\F$-$\oper$ we can apply the above $1$ to $1$ correspondence to get an element
\begin{equation*}
\tilde{\gamma}\in ME(colim(\op{P}\tensor(\op{O}[t,dt])))
\end{equation*}
Using the above isomorphism and the fact that this colimit is just a direct sum of coinvariant spaces, with the automorphism group acting trivially on the $\Omega_1$ factor, we have
\begin{equation*}
colim(\op{P}\tensor(\op{O}[t,dt]))\cong colim(\op{P}\tensor\op{O})\tensor\Omega_1
\end{equation*}
as Lie algebras.  As such we consider
\begin{equation*}
\tilde{\gamma}\in ME(colim(\op{P}\tensor\op{O})\tensor\Omega_1)=ME_1(colim(\op{P}\tensor\op{O}))
\end{equation*}
and it is enough to show that $\delta_i(\tilde{\gamma})=\tilde{\gamma}_i$ for $i=0,1$.  To see this note that under the above isomorphism $\Omega_1\cong k[t,dt]$ the morphism $id_\op{O}\tensor \delta_i\colon \op{O}\tensor\Omega_1\to\op{O}\tensor\Omega_0\cong\op{O}$ is sent to the morphism $\pi_{i+1}\circ\psi\colon\op{O}[t,dt]\to\op{O}$, where $\pi_1,\pi_2$ are projections.  Then the fact that $\pi_{i+1}\circ\psi(\gamma)=\gamma_i$ tells us that $id_{\op{O}}\tensor\delta_i(\tilde{\gamma})=\tilde{\gamma}_i$, where we have abused notation by identifying elements across the following isomorphisms:
\begin{equation*}
\xymatrix{colim(\op{P}\tensor\op{O})\tensor\Omega_1 \ar[d]^{\cong} \ar[rr]^{\delta_i} && colim(\op{P}\tensor\op{O}) \ar[d]^{\cong} \\ \ds\oplus_v Hom_{Aut(v)}(\op{P}^\ast(v),\op{O}(v)\tensor\Omega_1) \ar[rr]^{(id_{\op{O}}\tensor\delta_i)\circ -}&& \ds\oplus_v Hom_{Aut(v)}(\op{P}^\ast(v),\op{O}(v))}
\end{equation*}
The fact that this diagram commutes then tells us that $\delta_i(\tilde{\gamma})=\tilde{\gamma}_i$, proving the implication $\Rightarrow$.  Taking the argument in reverse yields $\Leftarrow$.

\end{proof}

The above theorem shows us that the equivalence relations coincide.  It was not however clear that the equivalence relation defining $\pi_0$ was transitive.  The above proof shows that in the case of the operadic Lie algebras, i.e. those arising in Table $\ref{Ftable}$, it is.

\subsection{W Construction}\label{Wsec}

In this subsection we work in the category $\op{C}=\textbf{Top}$ and with Feynman categories which we assume for convenience are strict.  In \cite{BV} Boardmann and Vogt give a construction, called the $W$ construction which, under suitable hypotheses, replaces a topological operad whose underlying $S$-module is cofibrant with a cofibrant operad in a functorial way.  In this section we will generalize the $W$ construction to $\fopsc$ when $\FF$ is cubical.  We will use the definitions set up in \S\ref{grfeysec}.

\begin{definition}  For a cubical Feynman category $\FF$ and an object $Y \in \F$, we define $w(\FF,Y)$ to be the category whose objects are the set $\coprod_n C_n(X,Y)\times [0,1]^n$.  An object in $w(\FF,Y)$ will be represented (uniquely up to contraction of isomorphisms) by a diagram
\begin{equation*}
X\xrightarrow[f_1]{t_1} X_1\xrightarrow[f_2]{t_2} X_2\to\dots\to X_{n-1}\xrightarrow[f_n]{t_n} Y
\end{equation*}
where each morphism is of positive degree and where $t_1,\dots,t_n$ represents a point in $[0,1]^n$.  These numbers will be called weights.  Note that in this labeling scheme isomorphisms are always unweighted.  The morphisms of $w(\FF,Y)$ are those generated by the following three classes:
\begin{enumerate}
\item  Levelwise commuting isomorphisms which fix $Y$, i.e.:
\begin{equation*}
\xymatrix{X \ar[r] \ar[d]^{\cong} & X_1 \ar[d]^{\cong} \ar[r] & X_2 \ar[d]^{\cong} \ar[r] & \dots \ar[r] & X_n \ar[d]^{\cong} \ar[r] & Y  \\ X^{\prime} \ar[r] & X^{\prime}_1\ar[r] & X^{\prime}_2\ar[r] &\dots \ar[r] & X^{\prime}_n \ar[ur] & }
\end{equation*}
\item  Simultaneous $\SS_n$ action.
\item  Truncation of $0$ weights: morphisms of the form $(X_1\stackrel{0}\to X_2\to\dots\to Y)\mapsto (X_2\to\dots\to Y)$.
\item  Decomposition of identical weights:  morphisms of the form $(\dots \to X_i\stackrel{t}\to X_{i+2} \to \dots) \mapsto (\dots\to X_i\stackrel{t}\to X_{i+1}\stackrel{t}\to X_{i+2}\to\dots)$ for each (composition preserving) decomposition of a morphism of degree $\geq 2$ into two morphisms each of degree $\geq 1$.
\end{enumerate}

\end{definition}

Note that since $\FF$ is assumed to be cubical every morphism of degree $n$ is generated by compositions of morphisms of degree $1$ uniquely upto $\SS_n$ action and composition of isomorphisms.  The graph based intuition takes the objects of $w(\FF,v)$ to be graphs with weighted edges along with a total order on the set of edges.  Here the total order is necessary to say which weight goes with which edge.  However, passing to the colimit, since we allow simultaneous $\SS_n$ action, removes the edge ordering, and here the role of an edge is played by the orbit of a morphism under the $\SS_n$ action.  Furthermore, the truncation of $0$ morphisms have the effect in the colimit of identifying an edge of weight $0$ with the graph having said edge contracted, as one expects from the classical $W$-construction.  As such we make the following definition.

\begin{definition}
\label{Wdef}
Let $\op{P}\in\fopst$.  For $Y \in ob(\F)$ we define
\begin{equation*}
W(\op{P})(Y):= colim_{w(\FF,Y)}\op{P}\circ s (-)
\end{equation*}
\end{definition}

\begin{proposition}\label{fopstprop}  $W(\op{P})$ is naturally an object in $\fopst$, that is $W(\op{P})(-)$ is a symmetric monoidal functor.
\end{proposition}
\begin{proof}  This proof follows similarly to the proof of the fact that the free $\F$-$\oper$ is in fact an $\F$-$\oper$.  In particular, using the fact that $\Top$ is closed monoidal, for $X=\tensor_{i\in I} v_i$ we have $W(\op{P})(X):=$
\begin{equation*}
colim_{ w(\FF,\times_{i\in I} v_i)}\op{P}\circ s (-)\cong colim_{ \times_{i\in I} w(\FF, v_i)}\op{P}\circ s (-) \cong \times_{i\in I} colim_{w(\FF,v_i)}\op{P}\circ s (-)
\end{equation*}
For morphisms we define the image of $\phi\colon X\to v$ via the functor $w(\FF, X)\cong\times_{i\in I} w(\FF, v_i)\to w(\FF, v)$ which composes with $\phi$ and which gives $\phi$ weight $1$.
\end{proof}

Note that there is a functor $Iso(\F\downarrow v) \to w(\F,v)$ which assigns weight $1$ to any non-isomorphism $X\to v$.  This functor induces a map on the respective colimits of $\op{P}\circ s$ and as a result induces a morphism $F(\op{P})(v)\to W(\op{P})(v)$.  The fact that the $\F$-$\oper$ structure on $W(\op{P})$ was defined by giving compositions weight $1$ ensures that these maps give a morphism in $\fopst$.  To show that $W(\op{P})$ is a cofibrant replacement for $\op{P}$, it will ultimately be necessary to determine conditions under which the map $F(\op{P})\to W(\op{P})$ is a cofibration.  To this end we make the following definitions.

\begin{definition}
\label{simpledef}
Let $\gamma\colon X\to v$ be a morphism in $\F$.  We define $Aut(\gamma)$ to be the subgroup of $Aut(X)$ such that $\sigma \in Aut(\gamma) \Leftrightarrow \gamma\sigma=\gamma$.  We define $\rho(\gamma)$ to be the subgroup of $Aut(v)$ such that $\tau \in \rho(\gamma) \Leftrightarrow \exists \ \sigma \in Aut(X) \text{ such that } \gamma\sigma=\tau\gamma$.  We define $\FF$ to be $\gamma$-simple if whenever $\gamma=\gamma^\prime\sigma$ for $\sigma\in Aut(X)$, there exists $\tau \in Aut(v)$ such that $\gamma = \tau \gamma^\prime$.  We define a cubical Feynman category $\FF$ to be simple if it is $\gamma$-simple for every $\gamma$.
\end{definition}
\begin{remark}  The graph based intuition for the above definitions is as follows.
The group $Aut(\gamma)$ consists of the automorphisms of the graph $\gamma$ which fix the tails of the graph element-wise, where a tail is a flag which is not part of an edge.  The group $\rho(\gamma)$ permutes the tails within each vertex, while fixing the non-tail flags.  The familiar graph based examples are simple since the condition that $\gamma=\gamma^\prime\sigma$ means that $\sigma$ doesn't interchange a tail with a non-tail, and thus the permutation of the tails can be performed after contracting edges instead of before.
\end{remark}

\begin{lemma}\label{lemcof}  Let $\op{P}\in\fopst$ and let $\gamma\colon X\to v$.  Then $\op{P}(X)_{Aut(\gamma)}$ is naturally a $\rho(\gamma)$-space.
\end{lemma}
\begin{proof}  For $a\in\op{P}(X)$ and $\tau\in\rho(\gamma)$, define $\tau([a]):= [\op{P}(\sigma)(a)]$ where $\sigma\in Aut(X)$ with $\gamma\sigma=\tau\gamma$ (which exists since $\tau\in\rho(\gamma)$).  It remains to show that this assignment is independent of the choice of such a $\sigma$ and is independent of the choice of representative of $[a]$.  These both follow from the fact that if $\sigma_1$ and $\sigma_2$ satisfy $\gamma\sigma_i=\tau\gamma$, then $\sigma^{-1}_2\sigma_1\in Aut(\gamma)$.
\end{proof}

\begin{definition}
\label{rhocofdef}
We say that $\op{P}\in\fopst$ is $\rho$-cofibrant if for each morphism $\gamma$, $(\op{P}\circ s (\gamma))_{Aut(\gamma)}$ is cofibrant in the category of $\rho(\gamma)$-spaces.
\end{definition}
\begin{remark}  Recall from \cite{BM1} that a (classic) operad is $\Sigma$-cofibrant if it is cofibrant in the category of $\Sigma$-modules after forgetting the operad structure.  When $\FF=\operads$, the Feynman category for operads, the notion of $\rho$-cofibrancy and $\Sigma$-cofibrancy coincide, since the group $Aut(\gamma)$ is trivial.  However in a context where there are tail fixing automorphisms, such as in $\modular$, $\rho$-cofibrancy is stronger than just asking that the image of the forgetful functor be cofibrant.  However for modular operads whose $\SS_{flag(v)}$ action corresponds to permuting labels, i.e. is free, $\rho$-cofibrancy is satisfied.
\end{remark}

\begin{theorem}
\label{Wthm}
Let $\FF$ be a simple Feynman category and let $\op{P}\in\fopst$ be $\rho$-cofibrant.  Then $W(\op{P})$ is a cofibrant replacement for $\op{P}$ with respect to the above model structure on $\fopst$.
\end{theorem}
\begin{proof}  First, there is a universal map $W(\op{P})(v) \to \op{P}(v)$ for each $v\in V$ and by definition this induces a morphism in $\fopst$; indeed for any weighted sequence $X\to\dots \to v$ there is a map $\op{P}(X)\to \op{P}(v)$ given by retaining only the composition type of the sequence and using the $\F$-$\oper$ structure of $\op{P}$.  These morphisms clearly give a morphism in $\fopst$.

To show the morphism $W(\op{P})\to\op{P}$ is a weak equivalence, it is enough to show that for any $v\in\V$, $W(\op{P})(v)\to\op{P}(v)$ is a weak homotopy equivalence.  Now taking the cocone morphism corresponding to $id_v$ we get a factorizaion of the identity $\op{P}(v)\stackrel{i}\to W(\op{P})(v)\stackrel{\pi}\to \op{P}(v)$.  It thus remains to show that $i\circ\pi\sim id_{W(\op{P})(v)}$, and we may define the homotopy $W(\op{P})(v)\times[0,1] \to W(\op{P})(v)$ as follows.

For $t\in[0,1]$ define $w(\F,v,t)\subset w(\F,v)$ to be the full subcategory whose weights are each less than or equal to $t$.  Define a functor $w(\F,v)\to w(\F,v,t)$ by multiplying all of the weights of a given sequence by $t$.  This functor and the inclusion functor induce maps,
\begin{equation*}
\alpha_t\colon colim_{w(\F,v)}\op{P}\circ s \leftrightarrows colim_{w(\F,v,t)}\op{P}\circ s\colon \beta_t
\end{equation*}
Clearly $\beta_1\alpha_1$ is the identity.  On the other hand, since we can truncate $0$ weights, one has $colim_{w(\F,v,0)}\op{P}\circ s \cong \op{P}(v)$ and $\beta_0\alpha_0 = i\circ \pi$.  We thus conclude that there is a weak equivalence $W(\op{P})\stackrel{\sim}\to \op{P}$.

To complete the proof we must show that $W(\op{P})$ is cofibrant.  We know, however, that $F(\op{P})$ is cofibrant, and that there is a morphism $F(\op{P})\to W(\op{P})$ induced by the functors $Iso(\F\downarrow v) \to w(\FF,v)$ taking all non isomorphisms to have weight $1$.  Thus it suffices to show that this morphism is a cofibration.  Using our assumptions, we may follow as in \cite{Vogt}, \cite{BM2}.

For each natural number $r$ we define $wt^r(\F, v)$ to be the subcategory of $w(\F,v)$ which omits the truncation morphisms and which omits those objects $X\to v$ of degree (in $\F$), greater than $r$.  We define $\partial wt^r(\F, v)\subset wt^r(\F, v)$ to be the full subcategory of those objects having at least one weight equal to $0$ or $1$.   Next define a $\V$-module $\Psi_r(\op{P})$ by:
\begin{equation*}
\Psi_r(\op{P})(v):=\substack{\text{colim} \\ wt^r(\F, v)} \ \op{P}\circ s(-)
\end{equation*}
and a $\V$-module $\partial\Psi_r(\op{P})$ by
\begin{equation*}
\partial\Psi_r(\op{P})(v):=\substack{\text{colim} \\ \partial wt^r(\F, v)} \ \op{P}\circ s(-)
\end{equation*}
We will show that the inclusion $\partial\Psi_r(\op{P})(v)\hookrightarrow \Psi_r(\op{P})(v)$ is a cofibration in the category of $Aut(v)$-spaces.  First, because we have removed truncations, the colimit $\Psi_r(\op{P})(v)$ splits over the isomorphism classes of the source.  Now, an isomorphism $X\stackrel{\cong}\to X^\prime$ fixes a bijection between objects in $wt^r(\F,v)$ with source $X$ and with source $X^\prime$, and so when considering the component of the colimit $\Psi_r(\op{P})(v)$ corresponding to the isomorphism class $[X]=[X^\prime]$, we may choose a representative, $X$ say, and consider only those sequences with source $X$.

The assumption that $\FF$ is cubical means that each composition class $[\gamma]$ contributes a factor of $\op{P}(X)_{Aut(\gamma)}\times [0,1]^{|\gamma|}$ to the colimit.  Now $Aut(v)$ acts on the colimit (by postcomposition), and by Lemma $\ref{lemcof}$ the action restricted to $\op{P}(X)_{Aut(\gamma)}$ is closed under the elements of $\rho(\gamma)$.  Thus the $Aut(v)$ orbit of $\op{P}(X)_{Aut(\gamma)}$ is the $\rho(\gamma)$ coinvariants of the free $Aut(v)$-space on $\op{P}(X)_{Aut(\gamma)}$.  Since $\op{P}(X)_{Aut(\gamma)}$ is cofibrant as a $\rho(\gamma)$-space, it follows that the $Aut(v)$ orbit of $\op{P}(X)_{Aut(\gamma)}$ is cofibrant as an $Aut(v)$-space (since it is the image of a cofibrant object by a left adjoint).  Then the assumption that $\FF$ is simple tells us that the only identifications in the $[X]$ component of the colimit occur within the $Aut(v)$ orbit of the morphism.  By \cite{BM2} Lemma 2.5.2, taking the product of this $Aut(v)$-cofibrant space with the topological cofibration $\partial[0,1]^n\to [0,1]^n$ is a cofibration of $Aut(v)$-spaces.  Thus we may write $\partial\Psi_r(\op{P})(v)\hookrightarrow \Psi_r(\op{P})(v)$ as a coproduct of cofibrations of $Aut(v)$-spaces; in particular this coproduct is taken over composition classes of morphisms modulo composing with automorphisms of the source.   It follows that $\partial\Psi_r(\op{P})(v)\hookrightarrow \Psi_r(\op{P})(v)$ is a cofibration in the category of $Aut(v)$-spaces, and from this we then conclude that $F(\partial\Psi_r(\op{P}))\hookrightarrow F(\Psi_r(\op{P}))$ is a cofibration in $\fopst$.

To complete the proof, define a sequence $\dots \to W^r(\op{P})\to W^{r+1}(\op{P})\to \dots$ as follows.  First, $W^0(\op{P}):=F(\op{P})$, and then inductively $W^r(\op{P})$ is defined to be the pushout:
\begin{equation*}
\xymatrix{F(\partial\Psi_r(\op{P})) \ar[r] \ar[d] & F(\Psi_r(\op{P})) \ar[d] \\ W^{r-1}(\op{P})\ar[r] & W^r(\op{P})}
\end{equation*}
Then it is straight forward to show that the direct limit is $F(\op{P})\to W(\op{P})$, see e.g. \cite{Vogt}, \cite{BM2}.  Thus $F(\op{P})\to W(\op{P})$ is a direct limit of pushouts of cofibrations, and hence a cofibration, implying in particular that $W(\op{P})$ is cofibrant, which implies the result.
\end{proof}

\begin{remark}  In analogy with the free operad, a reasonable question would be to ask if $W(\op{P})$ can be given as a Kan extension, and the answer is yes.  In particular if one considers $([0,1]^n)_{\SS_n}\cong \Delta^n$, we can achieve the truncation and $\SS_n$ identifications by considering the following nerve-like construction
\begin{equation}
|C_\ast(A,B)|=\coprod_{n\geq 0} C_n(A,B)\times \Delta^n/\sim
\end{equation}
where the equivalence relation identifies $(\delta_i(\gamma),p)\sim (\gamma, d^i(p))$ for $1\leq i \leq n-1$ for $\gamma \in C_n(A,B)$, and $p\in\Delta^{n-1}$, where $\delta_i$ composes morphisms as per usual.  Graphically, we consider ordering the edges from least to greatest weights, with the identifications accounting for any ambiguities.  The spaces $|C_\ast(A,B)|$ are then the morphisms of a category with composition given by composition of sequences followed by the unique permutation which reorders from least to greatest. If we call this category $W(\FF)$, we may define a family of morphisms
\begin{equation*}
\F\stackrel{\epsilon^t}\longrightarrow W(\FF)
\end{equation*}
for $t\in\Delta^1$ by setting $\epsilon^t$ to be the identity on objects and, if $A\stackrel{f}\to B$ is a morphism in $\F$, then
\begin{equation*}
\epsilon^t(f)= \begin{cases} (f,t) \in C_1(A,B)\times \Delta^1 & \text{ if } deg(f)\neq 0 \\ (f,\ast) \in C_0(A,B)\times \Delta^0 & \text{ if } deg(f)=0 \end{cases}
\end{equation*}
Then one can show
\begin{equation*}
W(\op{P})=\epsilon^1\circ Lan_{\epsilon_0}(\op{P})
\end{equation*}
\end{remark}

\begin{remark}  In assuming that $\FF$ is cubical we exclude from consideration any non-isomorphisms of the form $k\to X$ where $k$ is the monoidal unit.  For the $W$ construction to be cofibrant when such morphisms are included requires additional cofibrancy conditions.  For example considering operads with units $k\to \op{P}(1)$ requires the additional cofibrancy condition that $k\to \op{P}(1)$ is a cofibration.  Such operads are called well-pointed in \cite{BM1}, \cite{Vogt}.
\end{remark}

\begin{remark}  In \cite{BM2} Berger and Moerdijk give a generalization of the $W$ construction for classical operads in other monoidal model categories equipped with a suitable interval object.  Combining their arguments with the one above we expect an even more general notion of cofibrant replacement for $\fopsc$.  In particular one would expect an equivalence between the double Feynman transform and the generalized $W$-construction in $\dgvect_k$.
\end{remark}

\appendix

\section{Graph Glossary}
\subsection{The category of graphs}

\label{graphsec}
Most  of the known examples of Feynman categories used in operadic type
theories are indexed over a Feynman category built from graphs.
It is important to note that although we will first introduce a category of graphs $\Graphs$, the relevant
Feynman category is given by a full subcategory $\Agg$  whose
objects are disjoint unions (called {\it aggregates}) of corollas. The groupoid of corollas themselves
play the role of $\asts$.

Before giving more examples in terms of graphs, it will be useful to recall some terminology.  A nice presentation is given in \cite{BM} which we follow here.

\subsubsection{Abstract graphs}
An abstract graph $\G$ is a quadruple $(V_{\G},F_{\G},i_{\G},\del_{\G})$
of a finite set of vertices $V_{\G}$, a finite
set of half edges (aka {\it flags}) $F_{\G}$,
an involution on flags $i_{\G}\colon F_{\G}\to F_{\G}; i_{\Gamma}^2=id$, and
a map $\del_{\G} \colon F_{\G}\to V_{\G}$.
We will omit the subscripts $\G$ if no confusion arises.

Since the map $i$ is an involution, it has orbits of order one or two.
We will call the flags in an orbit of order one {\em tails} and denote the set of tails by $T_{\Gamma}$.
We will call an orbit of order two an {\em edge} and denote the set of edges by $E_{\Gamma}$. The flags of
an edge are its elements.
The function $\del$ gives the vertex a flag is incident to.
It is clear that the sets of vertices and edges index the cells in a 1--dim simplicial complex.
The realization of a graph is the realization of this simplicial complex. Alternatively this can be though of as a 1--dim CW complex.

A graph is (simply) connected if and only if its realization is.
Notice that the graphs do not need to be connected. Lone vertices, that
is vertices with no incident flags, are also possible.  We also allow the empty graph. That is the unique graph with $V_{\Gamma}=\egr$.
 It will serve as the monoidal unit.
\begin{ex}
A graph with one vertex and no edges is called a {\em corolla}.  Any set $S$ gives rise to a corolla. If $p$ is a one point set, we define the corolla $\crl_{p,S}:=(p,S, id,\del)$, where $\del$ is the constant map.

\end{ex}

Given a vertex $v$ of $\G$, we set
$F_v=F_v(\Gamma)=\del^{-1}(v)$ and call it {\em the
flags incident to $v$}. This set naturally gives rise to a corolla.
The {\em tails} at $v$ is the subset of tails of $F_v$.  As remarked above $F_{v}$ defines a corolla $\crl_{v}=\crl_{\{v\},F_v}$.

\begin{rmk}  The way things are set up, a graph is specifies particular finite sets.  Changing the sets even by a bijection changes the graph.
\end{rmk}

\begin{rmk} \label{disjointrmk}
Graphs do not need to be connected.  Moreover, given two graphs $\G$ and $\G'$
we can form their disjoint union: $\G\amalg\G'=(F_{\Gamma}\amalg F_{\Gamma'},V_{\Gamma}\amalg V_{\G'},
i_{\Gamma}\amalg i_{\G''}, \del_{\G}\amalg \del_{\G'})$.

One actually needs to be a bit careful about how disjoint unions are defined.
Although one tends to think that the disjoint union $X\amalg Y$
is commutative, this is not the case, which becomes apparent
if $X\cap Y\neq\emptyset$.  Of course
there is a bijection $X\amalg Y \stackrel{1-1}{\longleftrightarrow}Y\amalg X$.
Thus the categories here are symmetric monoidal, but not strict symmetric monoidal.
This is important, since we  consider functors into other not necessarily strict monoidal categories.

However, using MacLane's theorem it is possible to make a technical construction that makes the monoidal structure (on both sides) into a strict symmetric monoidal structure.
\end{rmk}

\begin{ex}
An {\em aggregate of corollas} or aggregate for short is a finite disjoint union of corollas, or equivalently,
a graph with no edges.  In an aggregate $X=\amalg_{v\in I} \ast_{S_v}$, the set of flags is automatically the disjoint union of the sets $S_v$. We will refer to specific elements in this disjoint union by their pre-image and just write $s\in F(X)$.
\end{ex}

\subsubsection{Category structure; Morphisms of Graphs}

\begin{df}\cite{BM}\label{gmdef}
Given two graphs $\G$ and $\G'$, a morphism from $\G$ to $\G'$ is a triple $(\phi^F,\phi_V, i_{\phi})$
where
\begin{itemize}
\item [(i)]
$\phi^F\colon F_{\G'}\hookrightarrow  F_{\G}$ is an injection,
\item [(ii)] $\phi_V\colon V_{\G}\twoheadrightarrow V_{\G'}$ is a surjection, and
\item [(iii)] $i_{\phi}$
is a fixed point free involution on the tails of $\G$ not in the image of $\phi^F$.
\end{itemize}

One calls the edges and flags that are
not in the image of $\phi$ the contracted edges
and flags. The orbits of $i_{\phi}$ are called ghost edges and denoted by $E_{ghost}(\phi)$.

Such a triple is {\it a morphism of graphs} $\phi\colon \G\to \G'$ if

\begin{enumerate}
\item The involutions are compatible:
\begin{enumerate}
\item An edge of $\G$
is either a subset of the image of $\phi^F$ or not contained in it.
\item
If an edge is in the image of $\phi^F$ then its pre--image
is also an edge.
\end{enumerate}
\item $\phi^F$ and $\phi_V$ are compatible with the maps $\del$:
 \begin{enumerate}
\item Compatibility with $\del$ on the image of  $\phi^F$:\\
\quad If $f=\phi^F(f')$ then
$\phi_V(\del f)=\del f'$
\item Compatibility with $\del$ on the complement of the image of  $\phi^F$:\\
 The two vertices of a ghost edge in $\G$ map to
the same vertex in $\G'$ under $\phi_V$.
\end{enumerate}

 \end{enumerate}

If the image of an edge under $\phi^F$ is not an edge,
we say that $\phi$ grafts the
two flags.

The composition $\phi'\circ \phi\colon\G\to \G''$
of two morphisms $\phi\colon\G\to \G'$ and $\phi'\colon\G'\to \G''$
is defined to be  $(\phi^F\circ \phi^{\prime F},\phi'_V\circ \phi_V,i)$
where $i$ is defined by its orbits viz.\ the ghost edges.
Both maps $\phi^{F}$ and $\phi^{\prime F}$ are injective,
so that the complement of their composition is in bijection
with the disjoint union of the complements of the two maps.
We take $i$ to be the involution whose
orbits are the union of the ghost edges of $\phi$ and $\phi'$
under this identification.
\end{df}

\begin{rmk}
One can define  {\em na\"ive morphisms} of graphs $\psi\colon\G\to \G'$
which are given by a pair of maps  $(\psi_F\colon F_{\G}\to F_{\G'},\psi_V\colon V_{\G}\to V_{\G'})$
compatible with the maps $i$ and $\del$ in the obvious fashion.
This notion is good to define subgraphs and automorphisms, but  this data {\em is not enough} to capture all the needed aspects
for composing along graphs.
For instance it is not possible to contract edges with such a map or graft
two flags into one edge. The basic operations of composition in an operad
viewed in graphs is however exactly grafting two flags and then contracting.
For this and other more subtle aspects, one needs the more involved definition above
which we will use.
\end{rmk}
\begin{df}
We let $\Graphs$ be the category whose objects are abstract graphs and
whose morphisms are the morphisms described above in Definition $\ref{gmdef}$.
We consider it to be a monoidal category with monoidal product $\amalg$ (see
Remark \ref{disjointrmk}).
\end{df}

\subsubsection{Decomposition of morphisms}
Given a morphism $\phi\colon X\to Y$ where $X=\amalg_{w\in V_X} \ast_w$ and $Y=\amalg_{v\in V_Y}\ast_v$
are two aggregates, we can decompose $\phi=\amalg \phi_v$ with $\phi_v\colon X_v\to \ast_v$ and $\amalg_v X_v=X$. Let $X_v$ be the sub--aggregate of $X$ consisting of  the  vertices $V_{X_v}=\phi_V^{-1}(v)$ together
with all the incident flags $F_{X_v}=\del^{-1}_X(V_{X_v})$. Let $(\phi_v)_V$, be the restriction of $\phi_V$ to $V_{X_v}$. Likewise let  $\phi_v^F$ be the restriction of
 $\phi^F$ to $(\phi^{F})^{-1}(F_{X_v}\cap \phi^F(F_Y))$. This is still injective.
Finally let $i_{\phi_v}$ be the restriction of $i_{\phi}$ to $F_{X_v}\setminus \phi^F(F_Y)$.
These restrictions are possible due to the condition (2) above.

\subsubsection{Ghost graph of a morphism}
\label{ghostpar}

Given a morphism $\phi\colon\G\to \G'$, the underlying ghost graph of $\phi$ is the graph
given by $\gh(\phi)=(V(\Gamma),F_{\Gamma},\hat i_{\phi})$ where $\hat i_{\phi}$ is $i_{\phi}$ on the complement of
$\phi^F(\Gamma')$ and identity on the image of  edges of  $\Gamma'$ under $\phi^F$.
The edges of $\gh(\phi)$ are exactly the ghost edges of $\phi$.

\begin{figure}
    \centering
    \includegraphics[width=.7\textwidth]{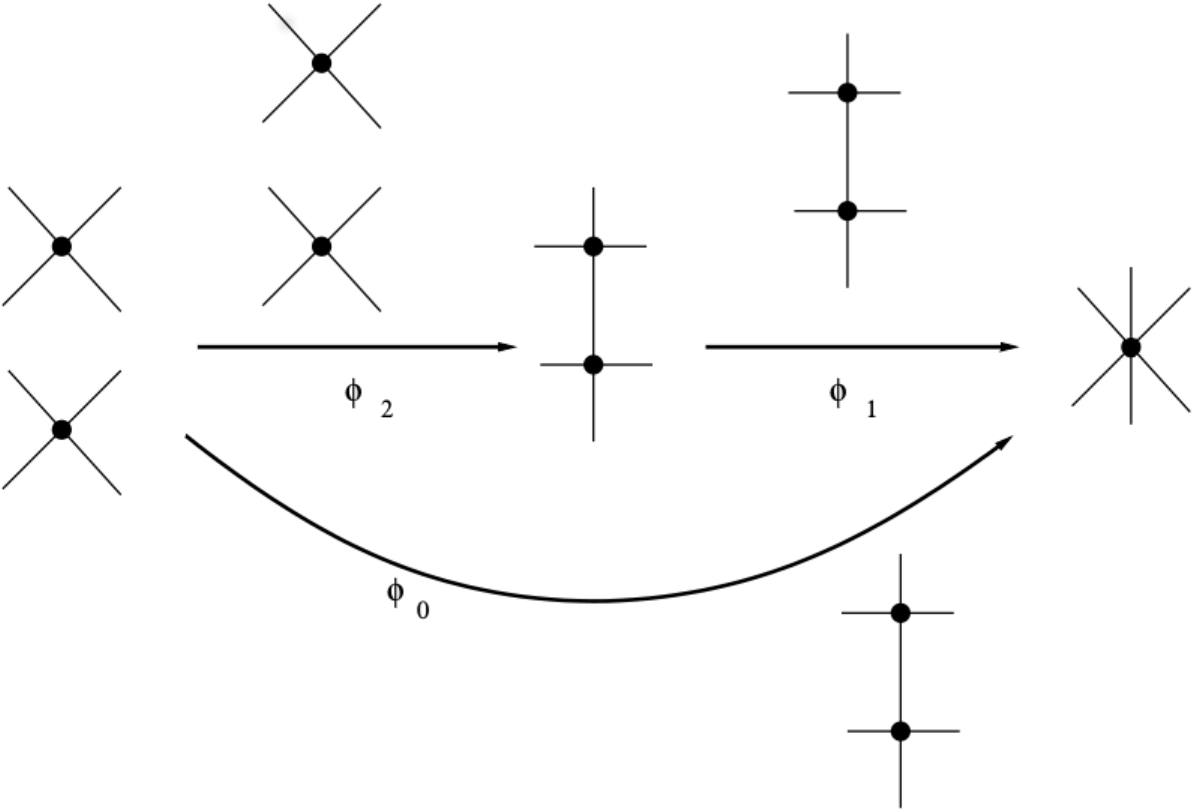}
    \caption{A composition of morphisms and the respective ghost graphs. The first morphism glues two flags to an edge, the second contracts an edge. The result is a morphism in $\Agg$.}
    \label{ghostgraphfig}
\end{figure}

\begin{figure}
    \centering
    \includegraphics[scale=.25]{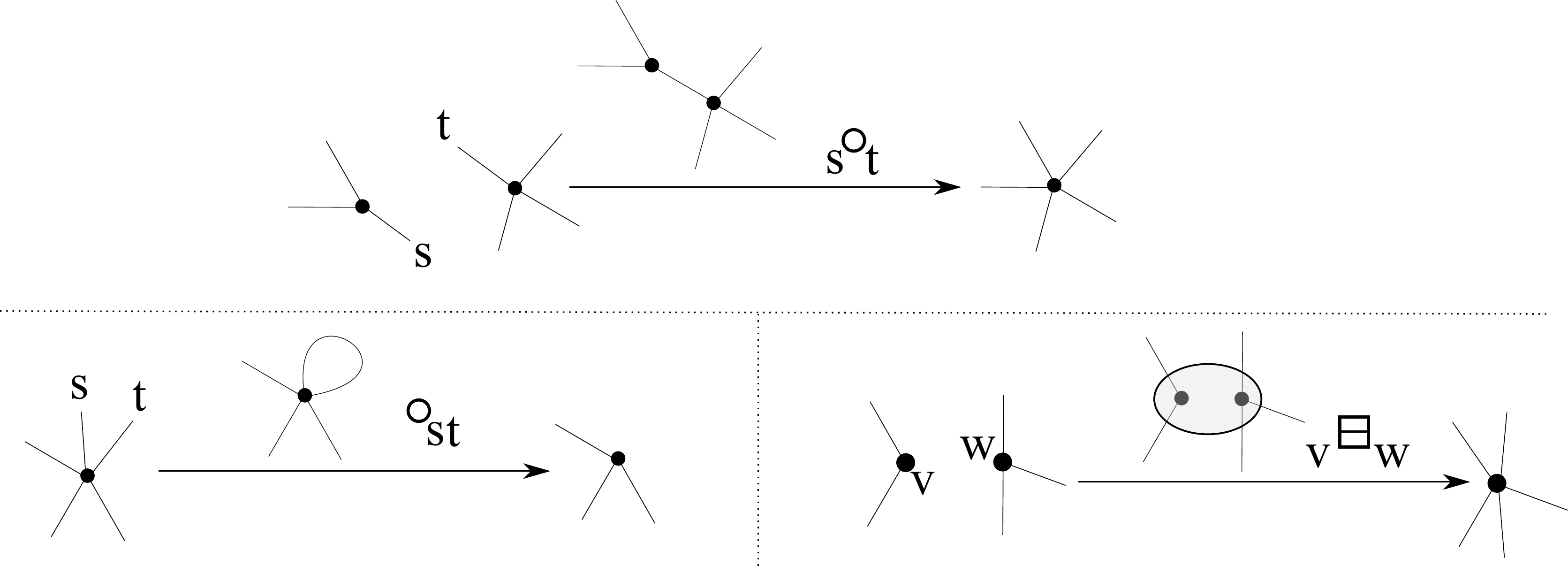}
    \caption{The three basic morphisms in $\GG$: an edge contraction (top), a loop contraction (left), and a merger (right). In the morphism, we give the ghost graph and label it by the standard notation. The shaded region is for illustration only.}
    \label{generatorfig}
\end{figure}

\begin{rmk}
Warning: many morphisms can have the same ghost graph. In fact even the data $(\phi_V,\phi^F,\imath_{\phi})$ appears for several morphisms.
Like always in mathematics  the source and the target of a map are part of its definition. For the ghost graph this discussion is given in detail in \S\ref{grsecondsec}.
\end{rmk}

\subsection{Extra structures}
\subsubsection{Glossary}
This section is intended as a reference section. All the following definitions are standard.

Recall that an order of a finite set $S$ is a bijection $S\to \{1,\dots ,|S|\}$.
Thus the group $\SS_{|S|}=Aut\{1,\dots,n\}$ acts on all orders.
  An  orientation of  a finite set $S$ is an
equivalence class of orders, where two orders are equivalent if they are
obtained from each other by an even permutation.

\begin{tabular}{l|l}
A tree& is
a connected, simply connected graph.\\

{A directed  graph $\G$}& is a graph together with  a map $F_{\G}\to \{in,out\}$\\
&such that the two flags of each edge are mapped\\
&to different values. \\

A rooted tree& is a directed tree such that each vertex has exactly one ``out'' flag.\\
& The unique ``out'' flag which is also a tail is called the root.\\
A level tree&is a rooted tree such that all the leaf vertices\footnotemark    {} are at \\
&the same distance (\# edges in shortest path) from the root.\\
A {ribbon or fat graph} &is a graph together with a cyclic order on each of \\&the
sets $F_{v}$.  \\
A planar graph& is a ribbon graph that can be embedded
into the\\
& plane such that the induced cyclic orders of the \\
&sets $F_v$ from the orientation of the plane  \\
&coincide with the chosen cyclic orders.\\
A planted planar tree&is a rooted planar tree together with a \\
&linear order on the set of ``in'' flags incident to the root's vertex.\\
An oriented graph& is a graph with an orientation on the set of its edges.\\
An ordered graph& is a graph with an order on the set of its edges.\\
A $\gamma$ labeled graph&is a graph together with a  map $\gamma:V_{\Gamma}\to \N_0$.\\
A b/w graph&is a graph $\G$ with a map $V_{\G}\to \{black,white\}$.\\
A bipartite graph& is a b/w graph whose edges connect only \\
&black to white vertices. \\
A $c$ colored graph& for a set $c$ is a graph $\G$ together with a map $F_{\G}\to c$\\
&s.t.\ each edge has flags of the same color.
\end{tabular}
\footnotetext{In a rooted tree, a leaf is a tail that is incoming and a leaf vertex is a vertex that has a leaf flag.}
\subsubsection{Remarks and language}\mbox{}
\begin{enumerate}

\item  In a directed graph one speaks about the ``in'' and the ``out''
edges, flags or tails at a vertex. For the edges this means the one flag of an edge
 is an ``in'' flag at a vertex. In pictorial versions the direction
is indicated by an arrow. A flag is an ``in'' flag if the arrow points to the vertex.

\item
As usual there are edge paths on a graph and the natural notion
of an oriented edge path. An edge path is a (oriented) cycle if it starts and
stops at the same vertex and all the edges are pairwise distinct. It is called simple if
each vertex on the cycle has exactly one incoming flag and one outgoing flag belonging to the cycle.
An oriented simple cycle will be called a {\em wheel}.
An edge whose two vertices coincide is called a {\em (small) loop}.

\item There is a notion of the genus of a graph, which is the minimal dimension
of the surface it can be embedded on. A ribbon graph is planar if this genus is $0$.
\item
 For any graph, its Euler characteristic is given by
$$\chi(\Gamma)=b_0(\Gamma)-b_1(\G)=|V_{\Gamma}|-|E_{\Gamma}|; $$
where $b_0,b_1$ are the Betti numbers of the (realization of) $\Gamma$.
Given a $\gamma$-labeled graph $\Gamma$, we define
the total $\gamma$ as
\begin{equation}
\label{gammaeq}
\gamma(\Gamma):=1-\chi(\Gamma)+\sum_{v \text{ vertex of $\Gamma$}}\gamma(v)
\end{equation}

If $\G$ is {\em connected}, that is if $b_0(\G)=1$, then a $\gamma$-labeled graph is traditionally called a genus labeled graph and
\begin{equation}\label{genuseq}
\gamma(\Gamma)=\sum_{v\in V_{\Gamma}}\gamma(v)+b_1(\Gamma)
\end{equation}
is called the genus of $\G$.
This is actually not the genus of the underlying graph, but the genus of
a connected Riemann surface with possible double points whose dual graph is the genus labeled graph.

A genus labeled graph is called {\em stable} if each vertex with genus labeling $0$ has at least 3 flags and
each vertex with genus label $1$ has at least one edge.
\item A planted planar tree induces a linear order on all sets $F_v$, by declaring the first
flag to be the unique outgoing one.
Moreover, there is a natural order on the edges, vertices and flags
given by its planar embedding.

\item A rooted tree is often taken to be a tree with a marked vertex. Note that a rooted tree as described above has exactly  one ``out'' tail. The unique vertex whose ``out'' flag is not a part of an edge is the root vertex.
The usual picture is obtained by deleting this unique ``out'' tail.
\end{enumerate}

\subsubsection{Category of directed/ordered/oriented graphs.}
\label{ordsec}
\begin{enumerate}

\item
Define the category of directed graphs $\Graphs^{dir}$ to be the category whose
objects are directed graphs. Morphisms are morphisms $\phi$ of the underlying graphs such that $\phi^F$ preserves orientation of the flags and $i_{\phi}$
only has orbits consisting of one ``in'' and one ``out'' flag, that is, the ghost graph is also directed.

\item

The category of edge ordered graphs $\Graphs^{ord}$ has as objects graphs with an order on the edges.
A morphism is a morphism of the underlying graphs together with an order $ord$ on all of the edges of the ghost graph.

The composition of orders on the ghost edges is as follows.
$(\phi,ord)\circ \amalg_{v\in V}(\phi_v,ord_v):=(\phi \circ \amalg_{v\in V}\phi_v,ord\circ\amalg_{v\in V}ord_v)$ where the order on the set of all ghost edges, that is $E_{ghost}(\phi)\amalg \amalg_vE_{ghost}(\phi_v)$,
is given by first enumerating  the elements of $E_{ghost}(\phi_v)$ in the
order $ord_v$, where the order of the sets  $E(\phi_v)$ is given by
the order on $V$  (i.e. given by the explicit  ordering of the tensor product in  $Y=\amalg_v \ast_v$)\footnote{Now we are working with ordered
tensor products. Alternatively one can just index the outer order
by the set $V$ by using \cite{TannakaDel}},
 and then enumerating the edges of $E_{ghost}(\phi)$ in their order $ord$.

\item The oriented version $\Graphs^{or}$ is then obtained by passing from orders to equivalence classes.

\end{enumerate}

\subsubsection{Category of planar aggregates and tree morphisms}
\label{planarsec}
Although it is hard to write down a consistent theory of planar graphs with planar morphisms, if not impossible,
there does exist a planar version of a special subcategory of $\Graphs$.

Let $\Crl^{pl}$ have as objects planar corollas --which simply means that there is a cyclic order on
the flags-- and as morphisms isomorphisms of these, that is isomorphisms of graphs which preserve the cyclic order.
 The automorphism group of a corolla $\ast_S$ are then isomorphic to $C_{|S|}$, the cyclic group of order $|S|$.
Let $\CCyclic^{pl}$ be the full subcategory of aggregates of planar
corollas whose morphisms are morphisms of the underlying
 corollas, for which the ghost graphs in their planar structure
 induced by the source  is compatible
 with the planar structure on the target via $\phi^F$. For this, we use the fact that the tails of a planar tree have a cyclic order.

Let $\Crl^{polypl}$ have as objects poly-cyclic corollas: that is corollas $*_S$ along with a bijection $N:S\to S$.

Let $\modular^{un,pl}$ be the category of aggregates of poly-cyclic  corollas and morphisms whose underlying  graphs have the induced structure of poly-cyclic order on the vertex, and the poly-cylic order on the target is the induced order obtained by contracting edges.

Let $\Crl^{pl,dir}$ be directed planar corollas with one output and let $\operads^{pl}$ be the  subcategory of $\Agg^{pl,dir}$ of aggregates of corollas of the type just mentioned, whose morphism are morphisms of the underlying directed corollas whose associated ghost graphs are compatible with the planar structures as above.

\subsection{Flag killing and leaf operators; insertion operations}
\label{insertionsec}
\subsubsection{Killing tails}
We define the operator $\kill$, which removes all tails from a graph. Technically, $\kill(\G)=(V_{\G},F_{\G}\setminus
T_{\G},\del_{\Gamma}|_{F_{\G}\setminus T_{\G}},
\imath_{\G}|_{F_{\G}\setminus T_{\G}})$.

\subsubsection{Adding tails}
Inversely, we define the formal expression
$\leaf$ which associates to each $\G$ without tails the formal sum
$\sum_n\sum_{\G':\kill(\G')=\G;F(\G')=F(\G')\amalg \bar n}\G'$, that is all possible
additions of tails where these tails are a standard set.
To make this well defined, we can consider the series as a power series in $t$:
$\leaf(\G)= \sum_n\sum_{\G':\kill(\G')=\G;F(\G')=F(\G')\amalg \bar n}\G't^{n}$.

This is a version of the foliage operator of \cite{KS,del} which was rediscovered in \cite{BBM}.

\subsubsection{Insertion}
Given graphs $\G$ and $\G'$, a vertex $v\in V_{\G}$, and an isomorphism $\phi$: $F_v\mapsto T_{\G'}$, we define $\G\circ_v\Gamma'$ to be the graph obtained by deleting $v$ and identifying the flags of $v$ with
the tails of $\G'$ via $\phi$. Notice that if $\G$ and $\G'$ are ghost graphs
of a morphism then this is just the composition of ghost graphs, with the morphisms at the other vertices being the identity.
\subsubsection{Unlabeled insertion} If we are considering graphs with unlabeled tails, that is classes $[\G]$ and $[\G']$
of coinvariants under the action of permutation of tails, then the insertion naturally lifts as
$[\G]\circ[\G']:=[\sum_{\phi} \G\circ_v\G']$ where $\phi$ runs through all the possible isomorphisms of two fixed lifts.

\subsubsection{No--tail insertion}
\label{nolabcompsec}
If $\G$ and $\G'$ are graphs without tails and $v$ a vertex of $\G$, then we
define $\G\circ_v\G'=\G\circ_v {\rm coeff}(\leaf(\G'),t^{|F_v|})$, the (formal) sum of graphs where $\phi$ is
one fixed identification of $F_v$ with $\overline{|F_v|}$.
In other words, one deletes $v$ and grafts all the tails to all possible positions on $\G'$.
Alternatively one can sum over all $\del:F_{\G}\amalg F_{\G'}\to V_{\G}\setminus v \amalg V_{\G'}$ where
$\del$ is $\del_{G}$ when restricted to $F_w,w\in V_{\G}$ and $\del_{\G'}$ when restricted to $F_{v'}, v'\in V_{\G'}$.

%\subsubsection{Compatibility}
%Let $\G$ and $\G'$ be two graphs without tails, then for any vertex $v$ of $\G$
% $\leaf(\G\circ_v\G')=\leaf(\G)\circ_v\leaf(\G')$.

\section{Topological Model Structure}\label{appb}

In this appendix we fix $\op{C}$ to be the category of topological spaces with the Quillen model structure.  The goal of this section is to show that $\fopsc$ has a model structure via Theorem $\ref{transferthm1}$.  The general corollary to this theorem does not apply in this case however because not all objects are small \cite{Hovey}.  Thus we will devise a separate argument to show that the conditions of the theorem are satisfied.  This work is inspired by \cite{Fresse}.  The key to this argument is the fact that all topological spaces are small with respect to topological inclusions \cite{Hovey}, whose definition we now recall.

\begin{definition}  A continuous map $f\colon X\to Y$ is called a topological inclusion if it is injective and if for every open set $U\subset X$ there exists an open set $V\subset Y$ such that $f^{-1}(V)=U$.  We will often refer to topological inclusions as simply `inclusions'.
\end{definition}

We will use the following facts about topological inclusions.

\begin{lemma}\label{inclemma1}  Inclusions are closed under pushouts, transfinite compositions, and finite products.
\end{lemma}
\begin{proof}  This is straight forward.  See \cite{Hovey} section 2.4 for the first two statements.  For the third statement, note that a product of injective maps is injective, and that every open set in a product can be written as a union of products of open sets, from which the claim quickly follows.

\end{proof}

\begin{lemma}\label{inclemma2}  Let $J$ be a groupoid, let $\alpha,\beta$ be functors from $J$ to $\op{C}$ and let $\eta\colon \alpha\Rightarrow \beta$ be a natural transformation such that $\eta_j\colon\alpha(j)\to\beta(j)$ is an inclusion for each $j\in J$.  Then the induced map $colim(\alpha)\to colim(\beta)$ is an inclusion.
\end{lemma}
\begin{proof}  Define $\phi$ to be the induced map between the colimits.  We first show $\phi$ is injective.  Recall that the forgetful functor from spaces to sets preserves colimits and thus as a set,
\begin{equation*}
colim(\alpha) = \coprod_{j\in J} \alpha(j)/\sim_\alpha
\end{equation*}
where $\sim_\alpha$ is the equivalence relation defined by $x_1\sim_\alpha x_2$ if and only if there exists $f\in Mor(J)$ such that $x_2=\alpha(f)(x_1)$.  Note that since $J$ is a groupoid this is an equivalence relation.  Similarly we define $colim(\beta)$ via an equivalence relation $\sim_\beta$.

Let $[x_1]$ and $[x_2]$ be elements in $colim(\alpha)$ with $\phi([x_1])=\phi([x_2])$.  Suppose $x_i\in \alpha(a_i)$.  Then $\phi([x_i]):=[\eta_{a_i}(x_i)]$, and by assumption $\eta_{a_1}(x_1)\sim_{\beta}\eta_{a_2}(x_2)$.  Thus there exists $f \in Mor(J)$ such that $\beta(f)(\eta_{a_1}(x_1))= \eta_{a_2}(x_2)$.  Since $\eta_{a_2}$ is an injection by assumption, the naturality of $\eta$ implies that $\alpha(f)(x_1)=x_2$, and thus $[x_1]=[x_2]$.  So $\phi$ is injective.

Next we must show that any open set in $colim(\alpha)$ is the pullback via $\phi$ of an open set in $colim(\beta)$.  Let $U\subset colim(\alpha)$ be open.  Then by definition of the topology of the colimit, $\lambda_a^{-1}(U)=:U_a$ is open in $\alpha(a)$ for every $a\in J$, where $\lambda_\bullet\colon \alpha(-)\to colim(\alpha)$ is the cocone.  Now by assumption for each $a\in J$ there exists an open set $V_a\subset \beta(a)$ such that $\eta_a^{-1}(V_a)=U_a$.  Define $V:=\cup_{a\in J} \gamma_a(V_a)$, where $\gamma_\bullet$ is the cocone for $\beta$.  This will be our candidate for an open set whose pull back under $\phi$ is $U$.

To show $V$ is open it is necessary and sufficient to show that for each $a\in J$, $\gamma_{a}^{-1}(V)\subset \beta(a)$ is open, and to show this we show
\begin{equation}\label{seteq1}
\gamma_a^{-1}(V)=\ds\bigcup_{\substack{f\in Mor(J) \\ t(f)=a}} \beta(f)(V_{s(f)}).
\end{equation}
First let $x\in \gamma_a^{-1}(V)$.  Then $\gamma_a(x)\in V$ and so there exists $b\in J$ and $y\in V_b$ with $\gamma_a(x)=\gamma_b(y)$, and thus there exists a morphism $f\colon b\to a$ in $J$ with $x = \beta(f)(y)$.  As a result, $x\in \beta(f)(V_b)$ and hence the inclusion $\subset$ in equation $\ref{seteq1}$.  For the other inclusion, suppose $w$ is an element of the rhs of equation $\ref{seteq1}$.  Then $w= \beta(f)(z)$ for some $z\in V_b$ and $f\colon b\to a$.  Using commutativity in the cocone this implies $\gamma_a(w)=\gamma_b(z)$ and hence $w\in \gamma_a^{-1}(\gamma_b(V_b))\subset \gamma_a^{-1}(V)$,  from which equation $\ref{seteq1}$ follows.  It follows that $V\subset colim(\beta)$ is open.

It remains to show that $\phi^{-1}(V)=U$.  For the $\supset$ inclusion:  if $[u]\in U$ then there exists $u\in U_a$ such that $\lambda_a(u)=[u]$ and by definition of $\phi$, $\gamma_a\eta_a(u) = \phi([u])$.  Now $\eta_a(u)\in V_a$ implies that $\phi([u])\in V$ and hence $[u] \in \phi^{-1}(V)$.  For the $\subset$ inclusion:  if $[w]\in \phi^{-1}(V)$, where $w\in \alpha(a)$, then there exists $[v]\in V$ with $\phi([w])=[v]$.  Since $[v]\in V$ there exists $b\in J$ with $v\in V_b$ and by commutativity of the relevant diagram we have $[v]=\gamma_b(v)=\gamma_a\eta_a(w)$.  Thus there exists a morphism $f\colon a\to b$ such that $v=\beta(f)(\eta_a(w))$, and by naturality, $v=\eta_b(\alpha(f)(w))$ so that there exists a unique $u:=\alpha(f)(w)\in U_b$ such that $\eta_b(u)=v$.  Then $[w]=[u]$ by definition of $\sim_\alpha$ and the fact that $[u]\in U$ permits the conclusion.
\end{proof}

\begin{corollary}\label{inccor1}  For $A\in\vseq$, the natural map $A\to GF(A)$ is a levelwise inclusion.

\end{corollary}
\begin{proof}  Fix $v \in \V$ and let $J$ be the groupoid $Iso(\F\downarrow v)$, let $\beta$ be the functor $A\circ s(-)(v)$ and let $\alpha$ be the functor which sends an isomorphism object in $J$ to $A(v)$, with morphisms in $J$ sent to the identity, and which sends all other objects in $J$ to $\emptyset$.  Note that $colim(\alpha)=A(v)$ and that there is an obvious natural transformation $\alpha\Rightarrow\beta$ which is an inclusion for each $j\in J$.  Thus Lemma $\ref{inclemma2}$ applies, from which the conclusion follows.

\end{proof}

Fix a Feynman category $\FF=(\V,\F, \imath)$ and let $\adj{F}{\vseq}{\fopsc}{G}$ be the forgetful free adjunction.

\begin{proposition}\label{soaprop}  Let $I$ be a collection of levelwise inclusions in $\vseq$. Then $F(I)$ permits the small object argument.
\end{proposition}
\begin{proof}  The proof will be given in several steps.
\begin{enumerate}
\item  Show the objects of $\vseq$ are small with respect to levelwise inclusions.
\item  Show $GF$ preserves levelwise inclusions.
\item  Show pushouts in $\fops$ preserve levelwise inclusions.
\item  Conclude that all objects in $\fopsc$ are small relative to $F(I)$, from which the claim follows.
\end{enumerate}

{\bf Step 1:}  This follows exactly as in the first part of the proof of Theorem $\ref{modelthm}$.

{\bf Step 2:}  Suppose $\eta\colon A\to B$ is a levelwise inclusion in $\vseq$.  For a fixed $v\in V$, we will apply Lemma $\ref{inclemma2}$ with $J= Iso (\F\downarrow v)$, $\alpha = A\circ s$ and $\beta = B\circ s$.  To do so note that the morphism of spaces $F(A)(v)\to F(B)(v)$ is induced by the natural transformation $A\circ s \to B\circ s$ given by mapping $\times_iA(v_i)\stackrel{\times_i\eta_{v_i}}\to \times_iB(v_i)$, and that $\times_i\eta_{v_i}$ is an inclusion by Lemma $\ref{inclemma1}$ since each $\eta_v$ is.  Lemma $\ref{inclemma2}$ thus implies $F(A)(v)\to F(B)(v)$ is an inclusion for each $v\in \V$.  In particular  $GF(\eta)$ is a levelwise inclusion.

{\bf Step 3:}  Let $J$ be the category with three objects whose nonidentity morphisms are $\ast_A\leftarrow \ast_C\rightarrow \ast_B$ and let $\alpha\colon J\to \fopsc$ be a functor whose image is of the form $ A\stackrel{s}\leftarrow C \stackrel{r}\rightarrow B$, where the induced map $G(C)(v)\to G(B)(v)$ is a topological inclusion for each $v\in V$.  In this step we want to show that for a given $v\in \V$, the natural map $\gamma\colon GA(v)\to Gcolim(\alpha)(v)$ is an inclusion.  Consider the diagram
\begin{equation}\label{d1}
\xymatrix{   && colim(G\alpha)(v)  \ar[d]^{\beta_1} &  \ar[dl]_{\beta:=\beta_1\circ\beta_2} GA(v) \ar[l]_{\beta_2}\ar[d]^{\gamma= \lambda\circ\beta} \\ GFcolim(GFG\alpha)(v) \ar@/^1pc/[rr]^f \ar@/_1pc/[rr]_g &&  \ar[ll]_h GFcolim(G\alpha)(v) \ar[r]^\lambda& Gcolim(\alpha)(v) }
\end{equation}
where the bottom line is given by taking the image under $G$ of the reflexive coequalizer of equation $\ref{refcoeq}$.  Note that by the construction of colimits in $\fopsc$ and the fact that $G$ preserves reflexive coequalizers, the bottom line of this diagram is still a reflexive coequalizer.  Also by the above construction of $colim(\alpha)$ we know the map $\gamma$ is equal to $\lambda\circ\beta$.  Finally note that $\beta$ is an inclusion by Lemma $\ref{inclemma1}$ and Corollary $\ref{inccor1}$.

{\it Substep 1: Analyze the morphisms $f,g,\beta$.}

Let us now consider the spaces $GFcolim(GFG\alpha)(v)$ and $GFcolim(G\alpha)(v)$ and the maps between them in more detail.  First note that since the map $r$ is injective, we can write $colim(G\alpha)(v)$ as  $A(v)\coprod (B(v)\setminus im(r(v)))$.  Therefore, writing $B^\prime(v):=B(v)\setminus im(r(v))$, we have
\begin{equation}
GFcolim(G\alpha)(v) = \ds\colim_{Iso(\F\downarrow v)} (A \coprod B^\prime)\circ s
\end{equation}
Similarly, if we define $B^{\prime\prime}(v):=FB(v)\setminus im(Fr(v))$, then
\begin{equation}
GFcolim(GFG\alpha)(v) = \ds\colim_{Iso(\F\downarrow v)} (FA \coprod B^{\prime\prime})\circ s
\end{equation}
Now these two spaces are subspaces of slightly less mysterious spaces, namely
\begin{equation}
GFcolim(G\alpha)(v) \subset \ds\colim_{Iso(\F\downarrow v)} (A \coprod B)\circ s
\end{equation}
and
\begin{align}
GFcolim(GFG\alpha)(v)\subset \ds\colim_{Iso(\F\downarrow v)} (FA \coprod FB)\circ s \nonumber \\
\cong\ds\colim_{Iso(\F\downarrow v)} (F(A \coprod B))\circ s \cong \ds\colim_{Iso(\F^2\downarrow v)} (A \coprod B)\circ s
\end{align}
So to define maps $GFcolim(GFG\alpha)(v)\rightarrow GFcolim(G\alpha)(v)$, it suffices to define maps
\begin{equation}
\xymatrix{\ds\colim_{Iso(\F^2\downarrow v)} (A \coprod B)\circ s \ar@/_1pc/[rr]_f \ar@/^1pc/[rr]^g && \ds\colim_{Iso(\F\downarrow v)} (A \coprod B)\circ s}
\end{equation}
and to show they land in the appropriate subspace of the target when restricted to the appropriate subspace of the source.

Recall that the objects of $Iso(\F^2\downarrow v)$ are sequences of morphisms $X\to Y\to v$ and that the morphisms are levelwise isomorphisms.  Define $g$ to be the map that is induced on the colimits by forgetting the middle term; $(X\to Y\to v) \mapsto (X\to v)$, and which takes the identity on the argument $(A\coprod B)(X)$.  Define $f$ to be the map that is induced on colimits by forgetting the left hand term; $(X\to Y\to v) \mapsto (Y\to v)$ and which uses to operad structure on the argument; $(A\coprod B)(X) \to (A\coprod B)(Y)$.  Note that although $FB\setminus im(F(r))\supset F(B\setminus im(r))$ are not equal, for each point in $im(F(r)(v))$ there is a unique point in $F(A)(v)$ which is equivalent, and so the maps will restrict appropriately via this relabeling.

Now it must be checked that this definition of $f$ and $g$ recover the morphisms defined in the proof of Lemma $\ref{limitslemma3}$.  The morphism $f$ is the morphism $(1)$ of the lemma, which is easily seen since contraction via the operad structure. corresponds to the natural transformation $FG\Rightarrow id$.  More delicate is to check that the morphism $g$ is the morphism $(2)$ of the lemma.  To see this note that when the Yoneda lemma was applied in the proof, we pushed forward the identity morphism of $Fcolim(G\alpha)$.  Therefore, taking $G\op{P}:=GFcolim(G\alpha)$, the morphism $(2)$ in the lemma can be described as the image of the cocone $GFG\alpha(-)\to colim(GFG\alpha)\to GFcolim(G\alpha)$ by $GF$ along with postcomposition by $FG\Rightarrow id$, i.e.
\begin{equation}
GFcolim (GFG\alpha)\to G(FG)Fcolim (G\alpha)\to GFcolim (G\alpha)
\end{equation}

On the other hand notice that
\begin{equation*}
GFGFcolim (G\alpha) = colim_{Iso(\F^2\downarrow v)}(A\coprod B^\prime)\circ s \subset colim_{Iso(\F^2\downarrow v)}(A\coprod B)\circ s
\end{equation*}
and so $g$ as described above is precisely restriction of the map $G(FG)Fcolim (G\alpha)\to GFcolim (G\alpha)$ to the subspace $GFcolim (GFG)\alpha\hookrightarrow G(FG)Fcolim (G\alpha)$, and hence the maps coincide.

There are two other maps that we wish to consider via colimits.  First, notice that we may write
\begin{equation}
colim(G\alpha)(v)= A(v)\coprod B^\prime(v) =colim_{Iso(\V\downarrow v)} (A\coprod B^\prime)\circ s
\end{equation}
and therefore can realize the map $\beta_1$ as being induced by the inclusion functor $Iso(\V\downarrow v)\hookrightarrow Iso(\F\downarrow v)$.  Second notice that there is a map
\begin{equation}
\epsilon\colon \ds\colim_{Iso(\F\downarrow v)} (A \coprod B) \circ s \to colim_{Iso(\V\downarrow v)} (A\coprod B)\circ s
\end{equation}
induced by composition in the $\F$-$\opers$ $A$ and $B$.  Moreover $\epsilon\beta_1$ is the identity.

{\it Substep 2:  Show $\gamma$ is injective.}

For simplicity of notation we will rewrite diagram $\ref{d1}$ as simply

\begin{equation}\label{d2}
\xymatrix{&&& Z \ar[d]^\gamma \ar@{_{(}->}[dl]_\beta &\\ X \ar@/^1pc/[rr]^f \ar@/_1pc/[rr]_g &&  \ar[ll]_h Y \ar[r]^\lambda& Y/\sim }
\end{equation}
where the equivalence relation is that generated by setting $f(x)\sim g(x)$ for each $x\in X$.

Let $z_1, z_2\in Z$ and suppose $\gamma(z_1)=\gamma(z_2)$.  Then $\beta(z_1)\sim\beta(z_2)$.  Since $\beta$ is an injection, it is sufficient to show that $\beta(z_1)=\beta(z_2)$.  Now since $\beta(z_1)\sim \beta(z_2)$, there exists a finite series of points $x_1\cdc x_m\in X$ along with an $m$-tuple $(l_1\cdc l_m) \in (\mathbb{Z}/2\mathbb{Z})^{m}$ such that:
\begin{equation}
\xymatrix{ \ar[d]^{f_{l_1}} x_1 \ar[dr]^{f_{l_1+1}} & \ar[d]^{f_{l_2}} \ar[dr] x_2 & \dots &\ar[d] \ar[dr]^{f_{l_m+1}} x_m &  \\ \beta(z_1)=f_{l_1}(x_1)& f_{l_1+1}(x_1)=f_{l_2}(x_2) & \dots &  & \beta(z_2)=f_{l_m+1}(x_m) }
\end{equation}
where we define $f_1:=f$ and $f_2=f_0:= g$.

Now clearly any adjacent pair of points in the bottom row have the same image via $\epsilon$ and as a result $\epsilon\beta(z_1)=\epsilon\beta(z_2)$, but $\epsilon\beta$ is the identity, from which we conclude $z_1=z_2$, and hence $\gamma$ is injective.

{\it Substep 3:  Show the injective map $\gamma$ is an inclusion.}
Let $U\subset A(v)$ be an open set.  We want to show that there is an open set $V\subset G(colim(\alpha))(v)$ such that $\gamma^{-1}(V)=U$.  In order to define $V$, we first define a subspace
\begin{equation}
W\subset GFcolim(G\alpha)(v) = \ds\colim_{Iso(\F\downarrow v)} (A \coprod B^\prime)\circ s
\end{equation}
as follows.  Let $(A\coprod B^\prime)\circ s \stackrel{\epsilon_\bullet}\to GFcolim(G\alpha)(v)$ be the cocone maps.  Then $W:=\cup W_\phi$ where $\phi$ runs over the objects in $Iso(\F\downarrow v)$ and where $W_\phi$ is defined in two cases.  First, if $\phi$ is not an isomorphism then define $W_\phi$ to be the image of $\epsilon_\phi$.  Second if $\phi$ is an isomorphism, from $v^\prime\to v$ say, then it induces a homeomorphism $A(v)\cong A(v^\prime)$ and we define $U^\prime$ to be the image of $U$ via this morphism.  We then define $W_\phi$ to be the image of $(U^\prime \coprod B^\prime)\circ s$ via $\epsilon_\phi$.

Now to show $W$ is open it suffices to show $W_\phi$ is open for each $\phi$, and to show this it suffices to show $\epsilon_\psi^{-1}(W_\phi)$ is open for every pair of objects $\phi$ and $\psi$.  However, the fact that $Iso(\F\downarrow v)$ is a groupoid tells us that $\epsilon_\psi^{-1}(W_\phi)$ is empty unless $\phi\cong \psi$, in which case $\epsilon_\psi^{-1}(W_\phi)$ is clearly open.  So $W$ is an open set and it was chosen to be large enough to contain the entirety of any of its equivalence classes.  Indeed the only potential snag would be having $\beta(z_1)\sim \beta(z_2)$ with $z_1\in U$ and $z_2\in A(v)\setminus U$, but this is precluded by the above argument, thus we may conclude $\lambda^{-1}\lambda(W)=W$.

Define $V:= \lambda(W)$ and note $V$ is open by the definition of the topology of the colimit of a reflexive coequalizer.  Further note that $\gamma^{-1}(V) = \beta^{-1}(W) = \beta_2^{-1}\epsilon_{id}^{-1}(W) = U$ as desired.  We thus conclude that $\gamma$ is an inclusion.

{\bf Step 4:}  Let $P\to Q$ be a relative $F(I)$-cell complex in $\fopsc$.  Then by definition there exists a sequence of pushouts:
\begin{equation}
\xymatrix{ F(A_0) \ar[d] \ar[r]^{F(i_0)} & \ar[dr] F(B_0) & F(A_1) \ar[d] \ar[r]^{F(i_2)} & \ar[dr] F(B_1) & \ar[d] F(A_2) \ar[r] &\dots \\  P=P_0 \ar[rr] && P_1 \ar[rr] && P_2 \ar[r] & \dots  }
\end{equation}
whose composition is the morphism $P\to Q$.  By the above work we know that the bottom line is a sequence of levelwise inclusions, and it thus follows from Lemma $\ref{inclemma1}$ that $GP\to GQ$ is a levelwise inclusion.  Thus for any domain of a morphism in $F(I)$, call it $F(A)$, we have
\begin{align*}
Hom(F(A), colim_i P_i)&\cong Hom(A, Gcolim_i P_i)  \cong Hom(A, colim_i(GP_i)) \\ & \cong colim_iHom(A, GP_i) \cong colim_i Hom(F(A), P_i)
\end{align*}
since $A$ is small with respect to levelwise inclusions.  Thus, $F(I)$ permits the small object argument.

\end{proof}

\begin{theorem}  Let $\op{C}$ be the category of topological spaces with the Quillen model structure.  The category $\fopsc$ has the structure of a cofibrantly generated model category in which the forgetful functor to $\vseq$ creates fibrations and weak equivalences.
\end{theorem}
\begin{proof}  We will apply Theorem $\ref{transferthm1}$.  Since both the generating cofibrations $I$ and the generating acyclic cofibrations $J$ in $\vseq$ are levelwise inclusions, it follows from Proposition $\ref{soaprop}$ that $F(I)$ and $F(J)$ permit the small object argument.  In addition, as discussed above in Example $\ref{tsex}$, conditions (ii),(iii) and (iv) of Corollary $\ref{transferprinciple}$ are satisfied, which imply (as seen in the proof of said corollary) that $G$ takes relative $F(J)$-cell complexes to weak equivalences.  Thus the conditions for transfer are verified, and Theorem $\ref{transferthm1}$ applies to imply the desired result.
\end{proof}

\bibliography{fcbib}
\bibliographystyle{halpha}
\end{document}